\documentclass[11pt]{article}
\usepackage{amssymb}
\usepackage{amsmath}
\usepackage{amsthm}
\usepackage{float}
\usepackage{fullpage}
\usepackage{graphicx}
\usepackage{mathtools}
\usepackage{mathrsfs}
\usepackage[round]{natbib}
\usepackage{subcaption}
\usepackage{tikz}
\usetikzlibrary{arrows.meta,calc,fit,positioning}
\usepackage[title]{appendix}
\usepackage[affil-it]{authblk}
\usepackage{hyperref}
\usepackage[capitalize,nameinlink]{cleveref}
\crefname{figure}{Figure}{Figures}
\Crefname{figure}{Figure}{Figures}

\newtheorem{definition}{Definition}
\newtheorem{theorem}{Theorem}
\newtheorem{lemma}{Lemma}
\newtheorem{assumption}{Assumption}

\title{Uncovering Symmetry Transfer in Large Language Models via Layer-Peeled Optimization}

\author[1]{Zhehang Du\thanks{Email: \texttt{duz@wharton.upenn.edu}.}}
\author[2]{Hangfeng He\thanks{Email: \texttt{hangfeng.he@rochester.edu}.}}
\author[1]{Weijie Su\thanks{Email: \texttt{suw@wharton.upenn.edu}.}}

\affil[1]{The Wharton School, University of Pennsylvania}
\affil[2]{University of Rochester}

\date\today

\begin{document}
\maketitle

\begin{abstract}
Large language models (LLMs) are pretrained by minimizing the cross-entropy loss for next-token prediction. In this paper, we study whether this optimization strategy can induce geometric structure in the learned model weights and context embeddings. We approach this problem by analyzing a constrained layer-peeled optimization program, which serves as a mathematically tractable surrogate for LLMs by treating the output projection matrix and last-layer context embeddings as optimization variables. Our analysis of this nonconvex optimization program demonstrates that symmetries in the target next-token distributions are transferred to the global minimizers of the layer-peeled model in a precise group-theoretic sense. Specifically, we prove that when the target tokens exhibit a cyclic-shift symmetry (such as the seven days of the week or the twelve months of the year), the optimal logit matrix is exactly circulant, and the Gram matrices of both the output projections and the context embeddings form circulant geometries as well. Next, for exchangeable target distributions invariant under the symmetric group and, more generally, under two-transitive group actions, we show that the global optimal output projection matrix forms a simplex equiangular tight frame, while the optimal logit matrix and context embeddings inherit the permutation symmetries present in the input data. A key technical step is to reduce the constrained nonconvex factorized problem to an explicit logit-level convex characterization for cyclic symmetry and to a symmetry-based lower bound for permutation symmetry, together with a sharp characterization of the optimal factorization. Finally, we empirically demonstrate that open-source LLMs naturally exhibit symmetries consistent with our theoretical predictions, despite being trained without any explicit regularization promoting such geometric structure.
\end{abstract}

\section{Introduction}

The past four years have witnessed rapid progress in large language models (LLMs)~\citep{brown2020language,achiam2023gpt4}, which are now ubiquitous in applications ranging from creative writing to coding to workflow automation. Remarkably, LLMs have recently begun to exhibit a level of mathematical reasoning capability that was beyond the imagination of nearly all researchers just one year ago~\citep{bubeck2025early}. These rapid advancements range from achieving gold medal level performance at the International Mathematical Olympiad (IMO) in July 2025~\citep{googledeepmind2025imo}, to resolving a long-standing optimization conjecture regarding Nesterov's accelerated gradient descent in October 2025~\citep{jang2025point}, and, just a few weeks ago, to cracking the celebrated Erd\H{o}s problem \#1196 in number theory~\citep{alexeev2026primitive}.

Interestingly, their widespread versatility and astounding capabilities contrast sharply with their simple training paradigm, which seeks merely to optimize the loss for predicting the next token~\citep{radford2018improving}. Despite this conceptual simplicity, unveiling the black box of LLMs has proven immensely difficult. This difficulty is driven largely by the massive scale of the models and the highly nonlinear nature of the Transformer architecture~\citep{vaswani2017attention}. From an optimization viewpoint, in particular, directly analyzing the full training dynamics of LLMs remains largely intractable with existing tools due to the nonconvexity of the loss landscape. In spite of these challenges, the optimization and machine learning communities are interested in understanding whether the optimized model parameters exhibit any geometric structures, and crucially, how such structures are transferred from the training data. The motivation to uncover these geometric phenomena is driven, on the one hand, by mathematical curiosity to explain the implicit biases of optimization algorithms, and on the other hand, by the practical desire to enhance training efficiency and improve model interpretability~\citep{soudry2018implicit,kaplan2020scaling,hoffmann2022training,elhage2021mathematical,zou2023representation}.

In this paper, we take an optimization perspective on the geometry of the output projection matrix---the last-layer linear map that produces vocabulary logits---and the context embeddings---the last-layer hidden states associated with given contexts. In particular, we investigate how these emergent structures explicitly relate to the underlying training data. This
question is closely related to the phenomenon of neural collapse~\citep{papyan2020prevalence}, which unveils surprising geometric simplicities in the last layer of optimized feedforward neural networks for classification. However, our problem arises in the context of next-token prediction within Transformer architectures, which presents several technical challenges. First, unlike the single-label classification setting, next-token prediction inherently deals with soft target distributions. Because a given context may have several plausible continuations, the target is naturally a smooth probability distribution over the vocabulary, which renders the data structure mathematically more complex to analyze. Moreover, the context embeddings themselves exhibit considerably more intricate geometric structures driven by these soft targets and the sequential nature of language modeling.

To overcome these challenges, we formulate a layer-peeled approach by constructing a mathematically tractable surrogate optimization program for LLM training, and analyze this surrogate to unveil the geometric symmetries of the last-layer representations. Our layer-peeled surrogate model builds upon~\citet{fang2021exploring} and extends it from classification with one-hot labels to LLM next-token prediction with soft target distributions. Building on this distributional viewpoint, we use smooth probability distributions to represent the target tokens in the training data~\citep{zhao2024implicit}. This formulation allows us to investigate how the token distributions, output projections, and context embeddings transform equivariantly under group actions formed by subgroups of the permutation group of the tokens. Formally, considering the problem of predicting among $m$ possible candidate tokens, let $W \in \mathbb{R}^{m \times d}$ denote the output projection matrix, let $H \in \mathbb{R}^{d \times n}$ collect the context embeddings across $n$ distinct contexts, and let $Y \in \mathbb{R}^{m \times n}$ denote the target probability distributions for these $n$ contexts. Our layer-peeled optimization program takes the form
\begin{equation}\label{eq:intro}
\min_{W,H}
\sum_{j=1}^n \mathcal L\left(\sigma\left((WH)_{:,j}\right),  Y_{:,j}\right)
\end{equation}
subject to $\|W\|_F^2 \le E_W$ and $\|H\|_F^2 \le E_H$, where $\|\cdot\|_F$ denotes the Frobenius norm, and $E_W$ and $E_H$ are scalar constants depending on weight decay, among other hyperparameters. In this nonconvex optimization program, $\mathcal L$ is the cross-entropy loss, $X_{:,j}$ denotes the $j$-th column of a matrix $X$, and \(\sigma(\boldsymbol z)_i=\exp(z_i)/\sum_{k=1}^{m}\exp(z_k)\) for \(i=1,\dots,m\) is the softmax function. For a matrix \(X\), \(\boldsymbol{\sigma}(X)\) denotes the matrix obtained by applying \(\sigma\) to each column of \(X\). In simple terms, this layer-peeled approach isolates the interaction between the last-layer weights and features and seeks to optimally approximate the target distributions $Y$ by minimizing the cross-entropy loss under norm constraints.

A central feature in~\cref{eq:intro} is that it captures \textit{local} and \textit{distributional} geometric structures inherent in LLM training. Rather than analyzing the entire token vocabulary, we focus on semantically coherent subsets of tokens whose internal structure can be described through group actions. This local geometric perspective stands in contrast to approaches that analyze the entire vocabulary simultaneously~\citep{zhao2024implicit}. For example, we consider tokens representing the days of the week---\texttt{Monday} through \texttt{Sunday}---which possess a natural cyclic structure. Our layer-peeled model in~\cref{eq:intro} adopts a \textit{local} viewpoint in the sense that it operates on a small collection (for example, \(m = 7\)) of conceptually related tokens. Furthermore, recognizing that the same prefix may admit different next-token completions,\footnote{For example, given the prompt ``\texttt{When will you finish the project if you start on Saturday? \textvisiblespace}'', the completion token might be \texttt{Sunday} with 50\% probability, \texttt{Monday} with 30\%, and \texttt{Tuesday} with 20\%.} we focus on a distributional setting in which each column \(Y_{:,j}\) is a probability vector aggregated across many contextual occurrences, rather than a one-hot label vector. Consequently, the number of distinct context distributions \(n\) in~\cref{eq:intro} is relatively small, thereby rendering a global optimality analysis mathematically tractable.

Despite the nonconvexity of its objective function, our analysis of the layer-peeled model establishes the following geometric phenomenon: if the training data, which is represented by the target distribution matrix $Y$, exhibits invariance under the action of a permutation group, then both the globally optimal output projections and the context embeddings precisely inherit this identical symmetry. More formally, we say that the training data \(Y\) has symmetry under a permutation group of degree $m$ if the unordered set of columns \(\{Y_{:,1}, \ldots, Y_{:,n}\}\) is identical to the set \(\{g\circ Y_{:,1}, \ldots, g\circ Y_{:,n}\}\) for any element $g$ in the group.

Specifically, we present the following two findings:

\vspace{0.5em}
\noindent{\bf Symmetry Transfer I: Cyclic-shift symmetry leads to circulant geometry.}
Consider prompts of the form ``\texttt{Starting on Monday, when will you finish the project? \textvisiblespace}''. By varying the context---replacing \texttt{Monday} with \texttt{Tuesday}, and so on up to \texttt{Sunday}---the target matrix $Y$ naturally exhibits a cyclic-shift symmetry. In this scenario, the number of distinct contexts is $n = m$, and any column of $Y$ can be obtained by applying a cyclic shift to the first column. For such locally symmetric training data, we prove that the globally optimal solutions $W^*$ and $H^*$ to the layer-peeled optimization program in~\cref{eq:intro} exhibit a circulant geometry. Specifically, the Gram matrices \(W^*(W^*)^\top\) and \((H^*)^\top H^*\) of the output projections and context embeddings, respectively, are exactly circulant matrices, where every column is a cyclic shift of the first column. Furthermore, the optimal logit matrix \(W^*H^*\) is circulant, implying that the learned logits depend entirely on relative positional differences modulo \(m\). This result illustrates that the cyclic geometry inherent in the data $Y$ is structurally transferred to the optimal model parameters $W^*$ and context embeddings $H^*$.

\vspace{0.5em}
\noindent{\bf Symmetry Transfer II: Permutation symmetry leads to a simplex equiangular tight frame (ETF).} 
Consider the prompt, ``\texttt{Alice, what is your favorite color? \textvisiblespace}''. Suppose the model must predict among $m = 3$ candidate colors---such as \texttt{red}, \texttt{green}, and \texttt{blue}---and assume these options are symmetrically exchangeable across a variety of contexts (e.g., \texttt{Bob}, \texttt{Charlie}, \texttt{David}, \texttt{Emily}, or \texttt{Frank} in addition to \texttt{Alice}). By considering all permutations, we obtain $n = m!$ ($n = 6$ in this case) distinct target distributions, yielding a target matrix $Y$ whose columns form an orbit that is invariant under the symmetric group of degree $m$. Given such a permutation-symmetric target matrix $Y$, we prove that this geometry is transferred to the global optimal solutions $W^*$ and $H^*$ of the nonconvex optimization program in~\cref{eq:intro}. First, we show that the $m$ rows of the optimal projection matrix $W^*$ form a simplex ETF; that is, the row vectors have an equal Euclidean norm, and every pair spans the identical, maximally possible angle. This structural finding effectively generalizes the neural collapse phenomenon~\citep{papyan2020prevalence} to the setting of next-token prediction under soft, distributional targets. Furthermore, the optimal logit matrix $W^* H^*$ has columns that form the exact orbit of a single vector under the action of the symmetric group, just as the columns of $Y$ do. The same group-theoretic orbit structure applies to the context embeddings $H^*$, up to an appropriate linear projection into an $m$-dimensional space.

\vspace{0.5em}

\begin{figure}[ht]
\centering
\[
\begin{tikzpicture}[
  >=Stealth,
  baseline=(current bounding box.center),
  every node/.style={inner sep=1pt, outer sep=0pt}
]

\node (Yi) at (0,0) {$Y_{:,i}$};
\node[anchor=base west] (commai) at ([xshift=0.04em]Yi.base east) {$,$};
\node[anchor=base west] (Hti) at ([xshift=0.04em]commai.base east) {$\widetilde H^*_{:,i}$};
\node[anchor=base east] (lpari) at ([xshift=-0.04em]Yi.base west) {$\bigl($};
\node[anchor=base west] (rpari) at ([xshift=0.04em]Hti.base east) {$\bigr)$};
\node[above=2.8em of Hti] (Hi) {$H^*_{:,i}$};
\node[fit=(lpari)(Yi)(commai)(Hti)(rpari), inner sep=1pt] (pairi) {};
\draw[->] (Hi.south) -- node[right] {$P_W^\top$} (Hti.north);

\begin{scope}[xshift=7.2cm]
\node (Yj) at (0,0) {$Y_{:,j}$};
\node[anchor=base west] (commaj) at ([xshift=0.04em]Yj.base east) {$,$};
\node[anchor=base west] (Htj) at ([xshift=0.04em]commaj.base east) {$\widetilde H^*_{:,j}$};
\node[anchor=base east] (lparj) at ([xshift=-0.04em]Yj.base west) {$\bigl($};
\node[anchor=base west] (rparj) at ([xshift=0.04em]Htj.base east) {$\bigr)$};
\node[above=2.8em of Htj] (Hj) {$H^*_{:,j}$};
\node[fit=(lparj)(Yj)(commaj)(Htj)(rparj), inner sep=1pt] (pairj) {};
\draw[->] (Hj.south) -- node[right] {$P_W^\top$} (Htj.north);
\end{scope}

\draw[->]
  ($(pairi.east)+(0.25,0.08)$)
  --
  node[above] {$g_jg_i^{-1} $}
  ($(pairj.west)+(-0.25,0.08)$);

\draw[->]
  ($(pairj.west)+(-0.25,-0.08)$)
  --
  node[below] {$g_ig_j^{-1} $}
  ($(pairi.east)+(0.25,-0.08)$);

\node at ($(pairi.south)!0.5!(pairj.south)+(0,-0.95)$)
  {$W^*\xrightarrow{\ \textstyle P_W\ }\widetilde W^*$};

\node at ($(pairi.south)!0.5!(pairj.south)+(0,-1.65)$)
  {$\widetilde W^*\bigl(g\circ(\cdot)\bigr)
  =
  g\circ\bigl(\widetilde W^*(\cdot)\bigr)
  \quad\text{for every group element }g$};

\end{tikzpicture}
\]
\caption{This diagram shows the common form behind both symmetry transfers~I and~II. Let
\(P_W:=W^{*\top}(W^*W^{*\top})^{\dagger/2}\in\mathbb R^{d\times m}\), where \((\cdot)^{\dagger/2}\) denotes the pseudoinverse square root, and define \(\widetilde H^*:=P_W^\top H^*\) and \(\widetilde W^*:=W^*P_W\). Multiplying by \(P_W^\top\) or \(P_W\) maps \(H^*\) and \(W^*\) from the \(d\)-dimensional space to the \(m\)-dimensional space, where the group action is defined. Let \(g_i\) and \(g_j\) be the group elements associated with columns \(i\) and \(j\). If \(Y_{:,i}=g_i\circ Y_{:,1}\), then \((g_jg_i^{-1})\circ\) maps \(Y_{:,i}\) to \(Y_{:,j}\), and the same action maps \(\widetilde H^*_{:,i}\) to \(\widetilde H^*_{:,j}\). Thus \(Y\) and \(\widetilde H^*\) share the same orbit structure. Moreover, \(\widetilde W^*\) is equivariant with respect to the group action: applying \(g\circ\) before \(\widetilde W^*\) gives the same result as applying \(g\circ\) after \(\widetilde W^*\), for every group element \(g\). In particular, commuting with any cyclic-shift action gives the circulant structure, and commuting with any permutation gives the simplex ETF structure.}
\end{figure}

These two symmetry transfer results are formally stated in~\cref{thm:single_block_cyclic} and in~\cref{thm:single_block_perm,thm:embedding_products_structure}, in~\cref{sec:single_block_cyclic,sec:single_block_perm}, respectively. Our proofs introduce several technical novelties to the optimization literature. These include leveraging logit symmetrization and Jensen's inequality to explicitly enforce orbit structures, and utilizing nuclear-norm bounds alongside active constraints to systematically eliminate nonsymmetric components from the global minimizers.

The significance of these symmetries---if they are indeed observed in empirical settings---extends beyond mathematical elegance. More importantly, they demonstrate that LLM weights self-organize into refined geometric configurations during gradient-based optimization, without the need for explicit regularization designed to promote such symmetries. Uncovering these optimal geometric configurations sheds light on how semantic information is structurally encoded within network weights, offering a connection between optimization and mechanistic interpretability for LLMs.

While these structural symmetry transfer phenomena are derived purely through the analysis of an idealized nonconvex optimization program, we find that open-source LLMs naturally exhibit these geometric symmetries in both their output projection matrices and context embeddings, as shown in~\cref{fig:teaser}. Our empirical evaluations employ carefully designed prompt templates and semantic word sets to naturally instantiate both cyclic and permutation symmetries, with details in~\cref{sec:exp}. 

\begin{figure}[ht]
  \centering
  \begin{subfigure}[t]{0.42\linewidth}
    \centering
    \includegraphics[width=\linewidth]{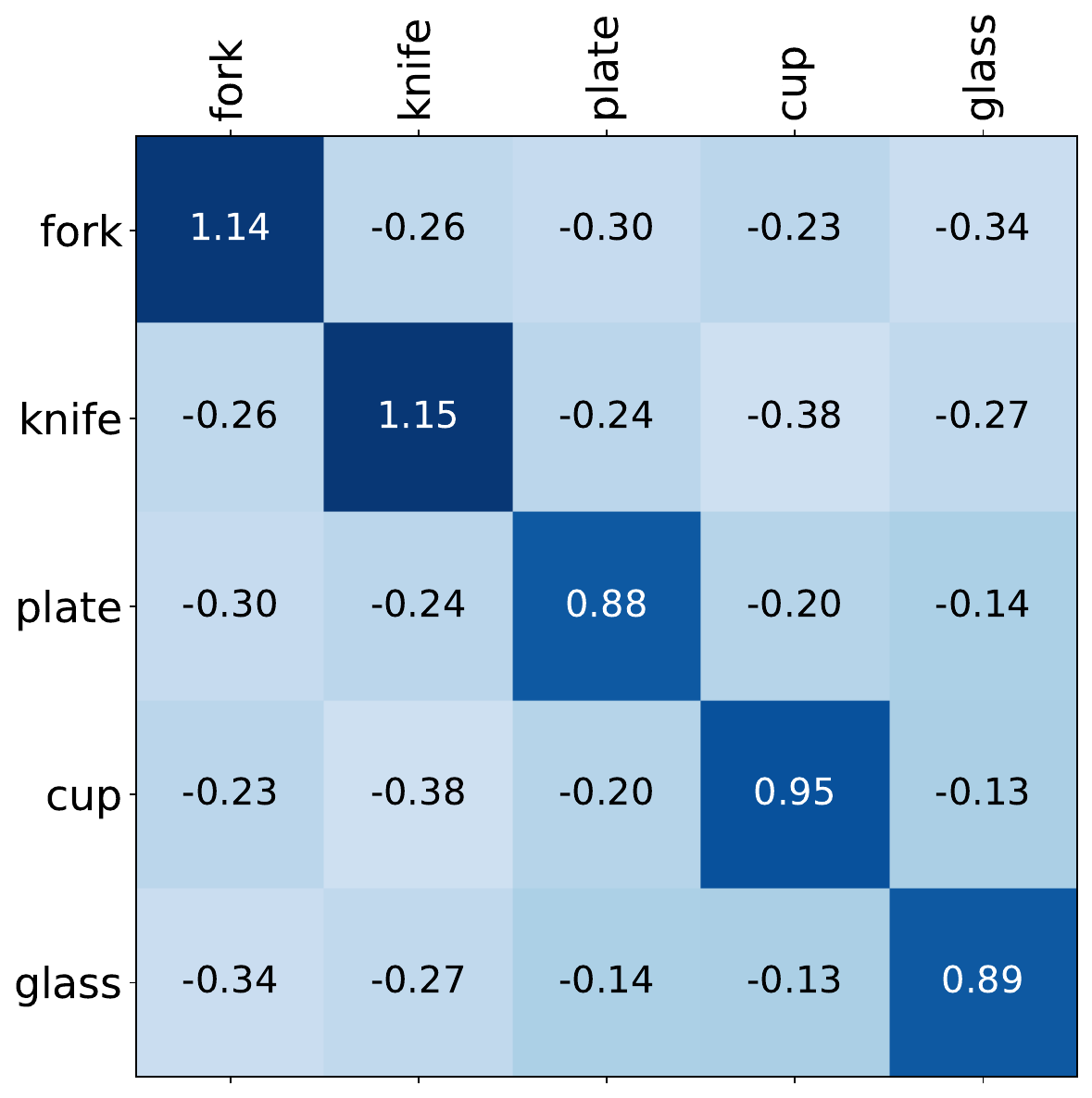}
  \end{subfigure}
  \begin{subfigure}[t]{0.49\linewidth}
    \centering
    \includegraphics[width=\linewidth]{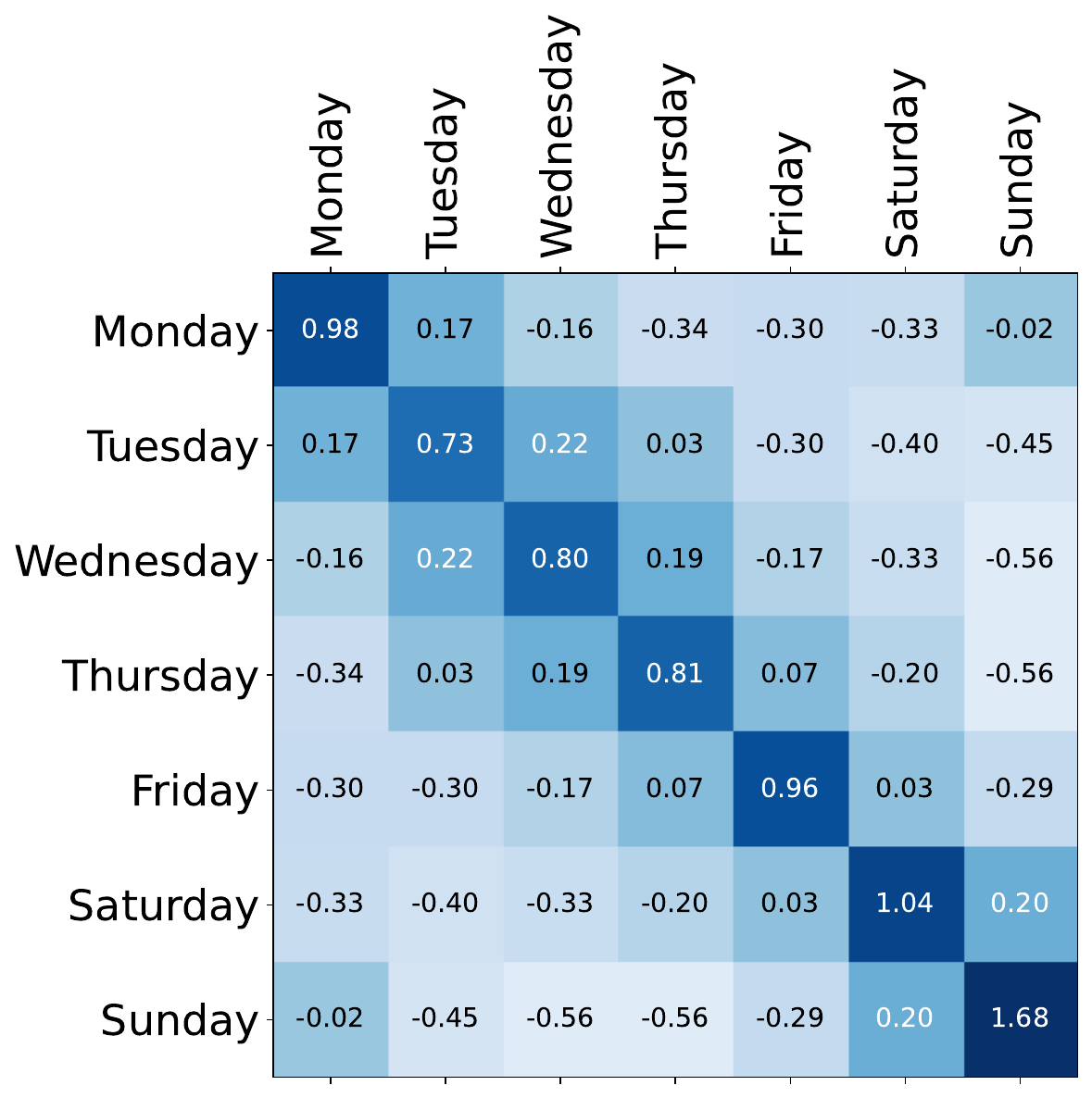}
  \end{subfigure}
  \caption{Normalized output projection Gram matrices for Mistral-7B-Instruct-v0.3~\citep{jiang2023mistral7b}.
  Left: full permutation symmetry implies a naturally occurring simplex ETF structure.
  Right: cyclic-shift symmetry implies a circulant structure. The corresponding next-token distribution matrices and context embedding Gram matrices are shown in~\cref{fig:perm-probs,fig:perm-gram,fig:ctx-cyclic-probs,fig:ctx-cyclic-gram}.}
  \label{fig:teaser}
\end{figure}

\subsection{Related Work}

Neural collapse~\citep{papyan2020prevalence} shows that class means and final-layer weights exhibit approximately a simplex ETF. Layer-peeled models~\citep{fang2021exploring}, also known as unconstrained feature models~\citep{zhu2021geometric,ji2021unconstrained,mixon2022neural}, analyze this geometry by treating the last-layer features and classifier weights as free optimization variables; this is motivated by the universal approximation property of overparameterized networks~\citep{hornik1989multilayer}. Since then, theory has developed in several ways: dynamics and landscape for cross-entropy~\citep{ji2021unconstrained}; for MSE the central-path picture and benign landscape of ETF minima~\citep{han2021neural,zhou2022optimization}; loss-design perspectives~\citep{zhou2022all,guo2025cross,fisher2024pushing}; robustness to data imbalance and long-tail structure~\citep{fang2021exploring,thrampoulidis2022imbalance,dang2023neural}; prevalence beyond standard classification~\citep{andriopoulos2024prevalence}; architectural/algorithmic regimes where neural collapse is provable~\citep{jacot2024wide}; persistence of collapse under low-rank bias~\citep{garrod2026the}, extended models~\citep{tirer2022extended}, and neural collapse emerging in intermediate layers~\citep{rangamani2023feature,he2023law}. Recent efforts extend the framework to soft-label targets~\citep{wu2024linguistic} and quantitative laws in next-token prediction in LLMs~\citep{he2024law}. Recent work also links next-token prediction to semantic embedding geometry via neural collapse-inspired analysis~\citep{zhao2025geometrysemantics}. In particular, prior studies model each context with a sparse probability vector over the next tokens and show that the learned logits decompose into a sparse-plus-low-rank component~\citep{zhao2024implicit,thrampoulidis2024implicit}.

Burer--Monteiro methods represent semidefinite programs through low-rank factors, reducing dimension while introducing nonconvexity~\citep{burer2003nonlinear,burer2005local}. Later work studies when such factorizations have benign landscapes or sharp rank requirements~\citep{boumal2016non,waldspurger2020rank}, with related results for matrix completion, matrix sensing, and low-rank recovery~\citep{sun2016guaranteed,ge2016matrix}. The proof also connects to nuclear-norm methods for low-rank recovery~\citep{recht2010guaranteed}. Because the language model output layer factorizes logits through hidden states and an output projection matrix, our formulation is also related to recent work studying low-rank structure directly at the level of LM logit matrices~\citep{golowich2025sequences} and identifying the LM head as an optimization bottleneck for backpropagated gradients~\citep{godey2026lmhead}. Our use of group symmetry is related to symmetry reduction in semidefinite and sums-of-squares optimization~\citep{gatermann2004symmetry,deklerk2010exploiting,deklerk2010qap}.

Foundational work on word representations spans latent semantic analysis~\citep{deerwester1990indexing}, neural language models~\citep{bengio2003neural}, and modern distributional embeddings~\citep{mikolov2013distributed,pennington2014glove,levy2014neural}, alongside theoretical accounts of isotropy and linear structure~\citep{arora2016latent}. ETFs arise in optimal packing and sensing~\citep{strohmer2003grassmannian}. A long line of research shows that when the data or task exhibits symmetries, constraining models to respect those symmetries yields structured, sample-efficient representations: permutation-invariant/equivariant set functions~\citep{zaheer2017deep}, group-equivariant convolutions~\citep{cohen2016group}, and cycle based and rotational parameterizations in knowledge graphs~\citep{sun2019rotate}. Recent work sharpens the geometric picture of LLM representations: Removing common components from word vectors strengthens linguistic regularities~\citep{mu2017all}; contextual embeddings are anisotropic~\citep{ethayarajh2019contextual}; Toy Models of Superposition explain how many features can be packed into fewer dimensions via interference~\citep{elhage2022toy}; correlations in data statistics can shape feature geometry, giving rise to semantic clusters and cyclic structures~\citep{prieto2026data}; Representation Engineering offers top-down tools for reading and controlling high-level features~\citep{zou2023representation}; knowledge editing~\citep{zhong2023mquake}; concept erasure~\citep{belrose2023leace}; identifying hallucination and assessing truthfulness from internal activations~\citep{burns2022discovering,azaria2023internal}; causal inner product~\citep{park2023linear}; categorical concepts as polytopes and hierarchical relations as geometric structure~\citep{park2025categorical}; linguistic collapse~\citep{wu2024linguistic}; translation symmetry in pairwise word co-occurrence statistics~\citep{karkada2026symmetry}; circular features for days of the week and months~\citep{engels2025not}; and representation manifolds~\citep{modell2025origins}.

Transformer circuits work views the residual stream as a shared communication space between model components~\citep{elhage2021mathematical}. Related methods decode intermediate states or project parameters through embedding space~\citep{belrose2023eliciting,dar2023analyzing}. Feedforward layers have also been interpreted as key-value memories and as mechanisms that promote concepts in vocabulary space~\citep{geva2021transformer,geva2022transformer}. Closest to our setting are mechanistic studies of ordered or algebraic structure: successor heads~\citep{gould2024successor}, GPT-2's greater-than behavior~\citep{hanna2023does}, modular addition models~\citep{nanda2023progress}, and group composition models~\citep{chughtai2023toy}. We analyze a layer-peeled optimization objective and show that group-orbit structure in the target distributions forces corresponding structure in the optimal last-layer geometry.

\subsection{Preliminaries}

\paragraph{Background on language models.}
Modern LLMs tokenize text into a sequence of tokens from a large vocabulary. In our analysis, we restrict to a local set of \(m\) candidate tokens and index them by \(\{1,\dots,m\}\).
An input embedding matrix converts each token index to a vector in \(\mathbb{R}^{d}\), which is combined with positional information and then fed into a stack of Transformer layers~\citep{vaswani2017attention} to produce a last-layer \emph{context embedding} \(\boldsymbol{h}\in\mathbb{R}^{d}\), which is used to predict the next token.
With the \emph{output projection} matrix \(W\) and bias \(\boldsymbol b\in\mathbb{R}^{m}\), the model forms logits \(\boldsymbol z=W\boldsymbol h+\boldsymbol b\in\mathbb{R}^{m}\), and the next-token probability for token index \(i\) is \(\sigma(\boldsymbol z)_i=\exp(z_i)\big/\sum_{k=1}^{m}\exp(z_k)\) for \(i=1,\dots,m\), from which the next token is sampled. Let each training sentence define a context indexed by \(j\).
For each context \(j\), the next-token target is a one-hot vector \(\boldsymbol{y}^{(j)}\in\{0,1\}^{m}\) with correct index \(i^*(j)\), and the logits are \(\boldsymbol{z}^{(j)}=W \boldsymbol{h}^{(j)}+\boldsymbol{b}\). The total loss is \(-\sum_{j=1}^{n}\log \sigma(\boldsymbol z^{(j)})_{i^*(j)}\), which sums the negative log-probability assigned to the correct next token in each context. Minimizing this total loss therefore encourages high probability on the correct token and is performed with gradient-based optimization on large text corpora. For simplicity of our analysis, we set \(\boldsymbol{b}=\boldsymbol{0}_{m}\).

\paragraph{Layer-peeled model.}
The classical layer-peeled model~\citep{fang2021exploring}, also known as the unconstrained feature model~\citep{mixon2022neural}, was introduced to study the geometry of deep classifiers in the terminal phase of training. The model isolates the final linear classifier from the rest of the network, and the last-layer feature produced by the network is treated as a decision variable for each sample. For \(n\) samples, we collect these features as \(H\in\mathbb{R}^{d\times n}\) (the context embedding matrix in our NTP setting), let \(W\in\mathbb{R}^{m\times d}\) be the classifier (the output projection matrix in our NTP setting), and let \(Y\in\{0,1\}^{m\times n}\) be the one-hot label matrix, with one target column per sample (the next token in our NTP setting). The layer-peeled optimization problem takes the form
\[
\min_{W,H} 
\sum_{j=1}^{n}
\mathcal L\left(\sigma\left((WH)_{:,j}\right),Y_{:,j}\right) \quad
\mathrm{s.t.}\quad \|W\|_F^2\le E_W,\quad
\|H\|_F^2\le E_H .
\]
Here \(E_W,E_H>0\) are norm constraints, and \(\mathcal L(\boldsymbol q,\boldsymbol p):=-\sum_{i=1}^{m}p_i\log q_i\) is the cross-entropy loss.

\section{Symmetry Transfer~I: Cyclic-Shift Symmetry}\label{sec:single_block_cyclic}

Full LLM training involves a vocabulary of size on the order of \(10^5\) and an enormous number of contexts. We instead restrict our analysis to a subset of \(m\) candidate tokens and a subset of \(n\) representative contexts. Although this simplifies the full optimization landscape, it captures the local geometry within the selected candidate tokens and context embeddings. This gives the single-block cyclic-shift and permutation subproblems studied in this section and~\cref{sec:single_block_perm}. In addition, \cref{sec:multi_block} considers the multi-block setting, which combines several such local subproblems by summing their losses, and the same symmetry transfer mechanisms extend to that setting.

The key difference from the classical one-hot setting in classification is that a context in language modeling may have several plausible next tokens across repeated occurrences. For example, for the context ``\texttt{A project starts on Monday and finishes on \textvisiblespace}'', over many occurrences the next token within the local candidate set of weekdays may be \(\texttt{Tuesday}\) in \(0.5\) of the cases, \(\texttt{Wednesday}\) in \(0.3\) of the cases, and \(\texttt{Thursday}\) in \(0.2\) of the cases. Thus, instead of assigning a single one-hot label to the context, we aggregate repeated next-token observations and use the resulting empirical distribution as the target column of \(Y\).

Specifically, suppose context \(j\) appears \(n_j\) times with observed next-token indices \(c^{(j,1)},\dots,c^{(j,n_j)}\in\{1,\dots,m\}\). We define the target column directly by
\[
Y_{:,j}
=
\frac{1}{n_j}
\left(
\sum_{r=1}^{n_j}\mathbf 1\{c^{(j,r)}=1\},
\dots,
\sum_{r=1}^{n_j}\mathbf 1\{c^{(j,r)}=m\}
\right)^\top.
\]
Since \(n_jY_{i,j}=\sum_{r=1}^{n_j}\mathbf 1\{c^{(j,r)}=i\}\), the total loss for context \(j\) is
\[
\sum_{r=1}^{n_j}
\mathcal L\left(\sigma\left((WH)_{:,j}\right),\boldsymbol e_{c^{(j,r)}}\right)
=
-\sum_{i=1}^{m} n_jY_{i,j}\log \sigma\left((WH)_{:,j}\right)_i
=
n_j
\mathcal L\left(\sigma\left((WH)_{:,j}\right),Y_{:,j}\right).
\]
Therefore, replacing one-hot labels by empirical next-token distributions preserves the same cross-entropy training objective after aggregating repeated occurrences of the same context.
For the example above, with the weekday tokens ordered as \((\texttt{Monday},\texttt{Tuesday},\dots,\texttt{Sunday})\), the context ``\texttt{A project starts on Monday and finishes on \textvisiblespace}'' may give \(Y_{:,1}=(0,0.5,0.3,0.2,0,0,0)^\top\), while the same context starting on \texttt{Tuesday} gives the cyclic shift \(Y_{:,2}=(0,0,0.5,0.3,0.2,0,0)^\top\).

Therefore, we first consider the unweighted layer-peeled objective (weighted variants can be viewed as a special case of the multi-block formulation in~\cref{sec:multi_block})
\begin{equation}
\label{eq:obj_function_constrained}
\min_{W,H} \mathcal J(W,H;Y)
:=
\sum_{j=1}^{n}
\mathcal L\left(\sigma\left((WH)_{:,j}\right),Y_{:,j}\right),
\quad
\mathrm{s.t.}\quad
\|W\|_F^2\le E_W,\quad
\|H\|_F^2\le E_H .
\end{equation}
This is an extension of the classical layer-peeled model, where the optimization variables and prediction rule are unchanged, while the target matrix \(Y\) contains empirical next-token distributions instead of one-hot labels. For simplicity of our analysis, we also impose the following active constraint assumption.
\begin{assumption}[Active Constraints]
\label{ass:active_constraint}
Assume \(E_W,E_H>0\). We assume that the constraints in~\cref{eq:obj_function_constrained} are active at the optimum. In other words, any optimal solution \((W^*,H^*)\) satisfies
\[
\|W^*\|_F^2=E_W
\quad\text{and}\quad
\|H^*\|_F^2=E_H .
\]
\end{assumption}

Many token sets are naturally ordered around a cycle, e.g., weekdays, months, or seasons. We model this by the one-step cyclic-shift permutation matrix \(\Pi\):
\begin{equation}
\label{eq:cyclic}
\Pi=\begin{bmatrix}
0 & 0 & \cdots & 0 & 1\\
1 & 0 & \cdots & 0 & 0\\
0 & 1 & \cdots & 0 & 0\\
\vdots & \vdots & \ddots & \vdots & \vdots\\
0 & 0 & \cdots & 1 & 0
\end{bmatrix},
\qquad
\Pi(y_1,\ldots,y_m)^\top=(y_m,y_1,\ldots,y_{m-1})^\top,
\end{equation}
so that the \(j\)-th column of \(Y\) is \(Y_{:,j}=\Pi^{j-1}\boldsymbol{y}\) for \(j=1,\ldots,m\).

\begin{definition}[Circulant Matrix]
A \(m\times m\) matrix \(A\) over \(\mathbb F\in\{\mathbb R,\mathbb C\}\) is called circulant with respect to \(\Pi\) if there exists a generating vector \(\boldsymbol a\in\mathbb F^m\) such that
\[
A_{:,j}=\Pi^{j-1}\boldsymbol a,\qquad j=1,\ldots,m.
\]
\end{definition}

\begin{theorem}[Optimal Embeddings for the Cyclic-Shift Group]
\label{thm:single_block_cyclic}
Let \(\boldsymbol{y}\) be a nonuniform target distribution, and let \(Y=[\boldsymbol{y},\Pi\boldsymbol{y},\ldots,\Pi^{m-1}\boldsymbol{y}]\) be the circulant target distribution matrix. For \(d \ge m\), under~\cref{ass:active_constraint}, any minimizer pair \((W^*,H^*)\) of the optimization problem~\cref{eq:obj_function_constrained} must satisfy the following properties:
The product logit matrix \(Z^*=W^*H^*\) is a circulant matrix with \(j\)-th column \(Z^*_{:,j}=\Pi^{j-1}\boldsymbol z^*\), where the generating vector \(\boldsymbol{z}^*  = (z^*_{1}, \ldots, z^*_{m})^\top\) is obtained as a solution to the convex optimization problem:
\[
\begin{aligned}
\min_{\boldsymbol{z} \in\mathbb R^{m}}
 & \mathcal L \left(\sigma(\boldsymbol{z} ), \boldsymbol{y}\right)
\\
\text{s.t. } & \sum_{k=1}^{m}\left| \sum_{l=1}^{m} z_{l} e^{-2\pi \mathrm{i} (k-1)(l-1)/m} \right| \le \sqrt{E_W E_H}.
\end{aligned}
\]
Furthermore,
\[ W^*(W^*)^\top = \sqrt{\frac{E_W}{E_H}} (Z^* (Z^*)^\top)^{1/2}, \qquad (H^*)^\top H^* = \sqrt{\frac{E_H}{E_W}} (Z^* (Z^*)^\top)^{1/2}. \]
Here the common matrix \((Z^* (Z^*)^\top)^{1/2}\) is circulant.
\end{theorem}

\begin{proof}[Proof of~\cref{thm:single_block_cyclic}] 
We begin with a brief roadmap.
The total loss depends only on the logit matrix \(Z=WH\), so we begin by analyzing \(Z\).
In Steps~1--3, we show that any optimal \(Z\) must be circulant.
Because a circulant matrix is determined by its first column, Step~4 reduces the original nonconvex problem to a convex program in that column vector.
Step~5 then characterizes the optimal Gram matrices \(WW^\top\) and \(H^\top H\).

Consider the original optimization problem~\cref{eq:obj_function_constrained}:
\begin{equation} \label{eq:prob_original}
\min_{\|W\|_F^2 \le E_W, \|H\|_F^2 \le E_H} \mathcal J(W,H;Y).
\end{equation}
Consider also the restricted optimization problem, where the product \(WH\) is constrained to be circulant:
\begin{equation} \label{eq:prob_circulant}
\min_{\|W\|_F^2 \le E_W, \|H\|_F^2 \le E_H,   WH \text{ is circulant}} \mathcal J(W,H;Y).
\end{equation}
We will prove that any optimal solution \((W,H)\) to~\cref{eq:prob_original} has a circulant product \(WH\), which implies that the solution sets of~\cref{eq:prob_original} and~\cref{eq:prob_circulant} coincide.

\paragraph{Step~1 (Symmetrization of logits leads to nonincreasing loss).}
For any logit matrix $Z=WH$, we construct its symmetrization $\widetilde{Z}$ and show that replacing $Z$ with $\widetilde{Z}$ does not increase the total loss.

First, let \((W,H)\) be any feasible solution to~\cref{eq:prob_original} with \(Z = WH\). Let \(\Pi\) be the \(m \times m\) cyclic-shift matrix defined in~\cref{eq:cyclic}. Since \(Y_{:,j}=\Pi^{j-1}\boldsymbol{y}\), we align the \(j\)-th logit column by the inverse shift: \( \boldsymbol z_j := \Pi^{-(j-1)} Z_{:,j}\), for \(j=1,\ldots,m\).
We then define an average aligned logit vector \(\bar{\boldsymbol z}\) from these:
\[ \bar{\boldsymbol z} := \frac{1}{m}\sum_{j=1}^{m} \boldsymbol{z}_j. \]
From this vector \(\bar{\boldsymbol z}\), we construct a new circulant logit matrix whose \(j\)-th column is \(\widetilde Z_{:,j}=\Pi^{j-1}\bar{\boldsymbol z}\). Equivalently, 
\[
\widetilde Z =\frac{1}{m}\sum_{t=1}^{m} \Pi^{t-1} Z \Pi^{-(t-1)}.
\]
Next, we prove that this symmetrization procedure does not increase the loss. The total loss for the original logits \(Z\) is
\[ \mathcal J(W,H;Y) = \sum_{j=1}^{m}\mathcal L(\sigma(Z_{:,j}), \Pi^{j-1}\boldsymbol{y}) = \sum_{j=1}^{m}\mathcal L(\sigma(\Pi^{-(j-1)}Z_{:,j}),  \boldsymbol{y}) = \sum_{j=1}^{m}f(\boldsymbol{z}_j), \]
where we define \(f(\boldsymbol u) \coloneqq \mathcal L(\sigma(\boldsymbol u), \boldsymbol{y})\). The loss for the new predictor \(\widetilde{Z}\) is \(m \cdot \mathcal L(\sigma(\bar{\boldsymbol z}),\boldsymbol{y}) = m \cdot f(\bar{\boldsymbol z})\). By~\cref{lem:softmax_ce_strict_convexity_V0}, \(f\) is convex, and Jensen's inequality gives
\begin{equation}
    \label{eq:jensen2}
f(\bar{\boldsymbol z}) = f\left(\frac{1}{m}\sum_{j=1}^{m} \boldsymbol{z}_j\right) \le \frac{1}{m}\sum_{j=1}^{m} f(\boldsymbol{z}_j). 
\end{equation}
Multiplying by \(m\) proves that the loss never increases: \(m \cdot f(\bar{\boldsymbol z}) \le \sum_{j=1}^{m}f(\boldsymbol{z}_j) \implies \mathcal J(\widetilde{W},\widetilde{H};Y) \le \mathcal J(W,H;Y)\). We now construct \( \widetilde{W}, \widetilde{H}\) in Step~2.

\paragraph{Step~2 (Symmetrization in Step~1 is feasible).}
With the symmetrized matrix $\widetilde{Z}$ from Step~1, we construct $\widetilde{W}$ and $\widetilde{H}$ that satisfy the constraints $E_W$ and $E_H$ and realize $\widetilde{Z}=\widetilde{W}\widetilde{H}$. Combined with Step~1, this shows that every feasible pair $(W,H)$ admits a feasible counterpart $(\widetilde W,\widetilde H)$ with no larger loss.

Consider the chain of inequalities involving the logit matrices \(Z=WH\) and \(\widetilde{Z}\):
\begin{equation}
\label{eq:budget_chain}
\sqrt{E_W E_H} \ge \|W\|_{F}\|H\|_{F} \ge \|WH\|_{*} = \|Z\|_{*} \ge \|\widetilde{Z}\|_{*}.
\end{equation}
The second inequality follows from~\cref{lem:nuclear_frobenius_combined}. The last inequality follows because the nuclear norm is unitarily invariant and convex, and $\widetilde Z$ is an average of orthogonal conjugates of $Z$:
\[
\|\widetilde Z\|_* = \left\|\frac{1}{m}\sum_{t=1}^{m} \Pi^{t-1} Z \Pi^{-(t-1)}\right\|_*
\le \frac{1}{m}\sum_{t=1}^{m} \left\|\Pi^{t-1} Z \Pi^{-(t-1)}\right\|_* = \|Z\|_*.
\]
This establishes that \(\|\widetilde{Z}\|_* \le \sqrt{E_W E_H}\). Now we show that a factor pair \((\widetilde{W}, \widetilde{H})\) can always be constructed to realize \(\widetilde{Z}\) and meet the budget constraint \(E_W\) and \(E_H\). 
Let the singular value decomposition of \(\widetilde{Z}\) be \( \widetilde{Z} = U\Sigma V^\top \), where \(U \in \mathbb{R}^{m \times r}\) and \(V \in \mathbb{R}^{m \times r}\) are matrices with orthonormal columns, \(r = \mathrm{rank}(\widetilde{Z})\), and \(\Sigma = \mathrm{diag}(s_1, \dots, s_r)\) is the diagonal matrix of positive singular values. Let \(Q\in\mathbb{R}^{d\times r}\) satisfy \(Q^\top Q=I_r\).
We define \((\widetilde{W}, \widetilde{H})\) by:
\begin{equation}
    \label{eq:wh_construction}
\widetilde{W} := \sqrt[4]{\frac{E_W}{E_H}} U\Sigma^{1/2}Q^\top \qquad \text{and} \qquad \widetilde{H} := \sqrt[4]{\frac{E_H}{E_W}} Q\Sigma^{1/2}V^\top,
\end{equation}
where \(\Sigma^{1/2} = \mathrm{diag}(\sqrt{s_1}, \dots, \sqrt{s_r})\). Then \(\widetilde W\widetilde H=\widetilde Z\).
Their Frobenius norms are:
\begin{equation}
\label{eq:budget_frobenius_W}
\|\widetilde{W}\|_F^2 = \sqrt{\frac{E_W}{E_H}} \|\Sigma^{1/2}\|_F^2 = \sqrt{\frac{E_W}{E_H}} \|\widetilde{Z}\|_{*} \le E_W,
\end{equation}
\begin{equation}
\label{eq:budget_frobenius_H}
\|\widetilde{H}\|_F^2 = \sqrt{\frac{E_H}{E_W}} \|\Sigma^{1/2}\|_F^2 = \sqrt{\frac{E_H}{E_W}} \|\widetilde{Z}\|_{*} \le E_H.
\end{equation}
The construction is complete. 

\paragraph{Step~3 (The optimal logit matrix is circulant via contradiction).}
If an optimal \(Z\) were non-circulant, we could construct a pair \((\widetilde{W}, \widetilde{H})\) achieving no larger loss but strictly smaller Frobenius norms, contradicting the assumption that any optimal solution saturates the budgets \(E_W\) and \(E_H\). Hence the optimal \(Z\) is circulant.

Let \((W^*, H^*)\) be any optimal solution for~\cref{eq:prob_original}, and let its logit matrix be \(Z^* = W^*H^*\). From the symmetrization argument, we can construct a feasible pair \((\widetilde{W},\widetilde{H})\) with a circulant product \(\widetilde Z^*\) such that its loss is no greater. Given that \((W^*, H^*)\) is optimal, we must have
\[ \mathcal J(\widetilde{W},\widetilde{H};Y) = \mathcal J(W^*,H^*;Y). \]
Assume \(Z^*\) is not circulant.
The equality of losses implies that the inequality in~\cref{eq:jensen2} must hold with equality. By the equality condition associated with strict convexity on \(\mathcal V_0 \coloneqq \{\boldsymbol v\in\mathbb R^m : \boldsymbol{1}_m^\top \boldsymbol v = 0\}\) from~\cref{lem:softmax_ce_strict_convexity_V0}, the aligned columns \(\boldsymbol{z}^*_j = \Pi^{-(j-1)}Z^*_{:,j}\) must satisfy \(\boldsymbol{z}^*_j = \boldsymbol{z}^*_1 + \beta_j^* \boldsymbol{1}_m\), where \(\boldsymbol{1}_m^\top \boldsymbol{z}^*_1 = 0\) and the vector of shifts \(\boldsymbol{\beta}^* = (\beta_1^*, \ldots, \beta_m^*)^\top\) is nonzero.
This allows us to decompose \(Z^*\) into a circulant part \(N^*\), defined by \(N^*_{:,j}=\Pi^{j-1}\boldsymbol z_1^*\), and a non-circulant shift part:
\[ Z^* = N^* + \boldsymbol{1}_m (\boldsymbol{\beta}^*)^\top. \]
The loss function is invariant to the shifts \(\beta_j^*\), so the objective value is unchanged when the product \(WH=Z^*\) is replaced by \(N^*\).

It remains to show that if \(\boldsymbol{\beta}^* \neq \boldsymbol{0}_m\), then the nuclear norm of \(Z^*\) is strictly greater than the nuclear norm of its circulant part \(N^*\). Let \(\mu_i\) be the eigenvalues of \(N^{*\top}N^*\) and \(\nu_i\) be the eigenvalues of \(Z^{*\top}Z^*\), both ordered nonincreasingly. The singular values are \(s_i = \sqrt{\mu_i}\) and \(t_i = \sqrt{\nu_i}\), respectively.
We compute \(Z^{*\top}Z^*\):
\[ Z^{*\top}Z^* = (N^* + \boldsymbol{1}_m (\boldsymbol{\beta}^*)^\top)^\top (N^* + \boldsymbol{1}_m (\boldsymbol{\beta}^*)^\top) = N^{*\top}N^* + \boldsymbol{\beta}^*(\boldsymbol{1}_m^\top N^*) + N^{*\top}\boldsymbol{1}_m(\boldsymbol{\beta}^*)^\top + \boldsymbol{\beta}^*(\boldsymbol{1}_m^\top \boldsymbol{1}_m)(\boldsymbol{\beta}^*)^\top. \]
The columns of the circulant matrix \(N^*\) are shifts of \(\boldsymbol{z}_1^*\), and since \(\boldsymbol{1}_m^\top\boldsymbol{z}_1^*=0\), all columns of \(N^*\) are in \(\mathcal V_0\) and are orthogonal to \(\boldsymbol{1}_m\). Thus, \(\boldsymbol{1}_m^\top N^* = \boldsymbol{0}_m^\top\), and the cross-terms vanish. The expression simplifies to
\[ Z^{*\top}Z^* = N^{*\top}N^* + m \boldsymbol{\beta}^*(\boldsymbol{\beta}^*)^\top. \]
Let \(G_N = N^{*\top}N^*\) and \(G_\beta = m \boldsymbol{\beta}^*(\boldsymbol{\beta}^*)^\top\). Both \(G_N\) and \(G_\beta\) are real, symmetric, and positive semidefinite (\(G_\beta \succeq 0\)). The eigenvalues of \(G_N\) are \(\{\mu_i\}\) and the eigenvalues of \(Z^{*\top}Z^*\) are \(\{\nu_i\}\). We can now apply Weyl's inequality (\cref{lem:weyl_inequality}) for the eigenvalues of a sum of Hermitian matrices:
\[ \nu_i = \lambda_i(G_N+G_\beta) \ge \lambda_i(G_N) + \lambda_m(G_\beta) \ge \lambda_i(G_N) = \mu_i \quad \text{for all } i=1, \dots, m. \]
This shows that no eigenvalue can decrease. To prove a strict increase for at least one eigenvalue, we examine the trace. Since \(\boldsymbol{\beta}^* \neq \boldsymbol{0}_m\), the matrix \(G_\beta\) is nonzero and has a strictly positive trace
\[ \mathrm{tr}(G_\beta) = \mathrm{tr}(m \boldsymbol{\beta}^*(\boldsymbol{\beta}^*)^\top) = m \cdot \mathrm{tr}((\boldsymbol{\beta}^*)^\top\boldsymbol{\beta}^*) = m\|\boldsymbol{\beta}^*\|_2^2 > 0. \]
The trace of the sum is the sum of traces
\[ \sum_{i=1}^m \nu_i = \mathrm{tr}(G_N+G_\beta) = \mathrm{tr}(G_N) + \mathrm{tr}(G_\beta) = \sum_{i=1}^m \mu_i + m\|\boldsymbol{\beta}^*\|_2^2 > \sum_{i=1}^m \mu_i. \]
We have a set of nonnegative differences \(\nu_i - \mu_i \ge 0\) whose sum \(\sum(\nu_i-\mu_i)\) is strictly positive. This is only possible if at least one difference is strictly positive. Thus, there exists at least one index \(j\) such that \(\nu_j > \mu_j\).
Since the singular values are the square roots of these eigenvalues (\(t_i = \sqrt{\nu_i}, s_i=\sqrt{\mu_i}\)), we have \(t_i \ge s_i\) for all \(i\), and \(t_j > s_j\) for at least one \(j\).
This leads to a strict inequality for the nuclear norms:
\[ \|Z^*\|_* = \sum_{i=1}^m t_i > \sum_{i=1}^m s_i = \|N^*\|_*. \]
Since \(\|N^*\|_* < \|Z^*\|_* \le \sqrt{E_W E_H}\), by the argument in Step~2, we can construct a feasible solution \((\widetilde{W}, \widetilde{H})\) (as in~\cref{eq:wh_construction} with the SVD of \(N^*\)) with \(\|\widetilde{W}\|_F^2 = \sqrt{E_W/E_H}\|N^*\|_{*} < E_W\) (as in~\cref{eq:budget_frobenius_W}) and \(\|\widetilde{H}\|_F^2 = \sqrt{E_H/E_W}\|N^*\|_{*} < E_H\) (as in~\cref{eq:budget_frobenius_H}). 
This contradicts~\cref{ass:active_constraint}, which requires both norm budgets to be active at the optimum.
Therefore, the initial assumption that an optimal solution \(Z^*\) could be non-circulant (i.e., \(\boldsymbol{\beta}^* \neq \boldsymbol{0}_m\)) must be false.
With \(\boldsymbol{\beta}^* = \boldsymbol{0}_m\), the decomposition \(\boldsymbol{z}^*_j = \boldsymbol{z}^*_1 + \beta_j^*\boldsymbol{1}_m\) implies \(\boldsymbol{z}^*_j = \boldsymbol{z}^*_1\) for all \(j=1,\ldots,m\). Undoing the alignment, \(Z^*_{:,j} = \Pi^{j-1}\boldsymbol{z}^*_1\), which is precisely the definition of a circulant matrix. This completes the proof that every optimal logit matrix of~\cref{eq:prob_original} must be circulant. 

\paragraph{Step~4 (Convex reduction via first column).}
With optimal \(Z\) circulant, we parameterize \(Z\) by its first column \(\boldsymbol z\) and rewrite the optimization as a convex problem in \(\boldsymbol z\).

Let the first column of the circulant logit matrix \(Z\) be the vector \(\boldsymbol{z}  \in \mathbb{R}^m\). The objective is to find the optimal \(\boldsymbol{z} \). First, we express the total loss function in terms of \(\boldsymbol{z} \). Since both the logit matrix \(Z\) and the target distribution matrix \(Y\) are circulant, the cross-entropy loss for each column is identical due to the shift-invariance property (\(\mathcal L(\sigma(\Pi^k\boldsymbol{z} ), \Pi^k\boldsymbol{y}) = \mathcal L(\sigma(\boldsymbol{z} ), \boldsymbol{y})\)). The total loss is thus
\[ \mathcal J(W,H;Y) = \sum_{j=1}^{m} \mathcal L(\sigma(Z_{:,j}), Y_{:,j}) = m \cdot \mathcal L(\sigma(\boldsymbol{z} ), \boldsymbol{y}). \]
Next, we re-parameterize the budget constraint in terms of \(\boldsymbol{z} \). From~\cref{eq:budget_chain}, we have \( \|Z\|_* \le \sqrt{E_W E_H} \). By~\cref{lem:fourier_diagonalization_circulant}, a circulant matrix \(Z\) is normal, so its singular values \(s_k\) are the magnitudes of its eigenvalues \(\lambda_k\). The eigenvalues are the components of the Discrete Fourier Transform of the first column \(\boldsymbol{z} \):
\[ \lambda_k = \sum_{l=1}^m z_{l} e^{-2\pi \mathrm{i} (k-1)(l-1)/m}. \]
The nuclear norm is therefore
\[ \|Z\|_* = \sum_{k=1}^m s_k = \sum_{k=1}^m |\lambda_k| = \sum_{k=1}^m \left| \sum_{l=1}^m z_{l} e^{-2\pi \mathrm{i} (k-1)(l-1)/m} \right|. \]
The budget constraint thus becomes an explicit constraint on the vector \(\boldsymbol{z} \).
Combining these results, the original nonconvex problem over matrices \((W,H)\) is equivalent to the following convex optimization problem over the single vector \(\boldsymbol{z}  \in \mathbb{R}^m\):
\[
\begin{aligned}
\min_{\boldsymbol{z} \in\mathbb R^{m}}
 & \mathcal L \left(\sigma(\boldsymbol{z} ), \boldsymbol{y}\right)
\\
\text{s.t. } & \sum_{k=1}^{m}\left| \sum_{l=1}^{m} z_{l} e^{-2\pi \mathrm{i} (k-1)(l-1)/m} \right| \le \sqrt{E_W E_H}.
\end{aligned}
\]
The objective function is convex in \(\boldsymbol{z} \), and the constraint set defined by the nuclear-norm bound is convex.
Once the optimal vector \(\boldsymbol{z} \) is found, the optimal logit matrix \(Z\) is fully determined as \(Z = [\boldsymbol{z} , \Pi^1\boldsymbol{z} , \dots, \Pi^{m-1}\boldsymbol{z} ]\).

\paragraph{Step~5 (Geometry of Gram matrices from equality conditions).} Tightness of the budget chain implies that \(W\) and \(H\) share singular directions and have proportional singular values, giving the circulant structure of the Gram matrices.

Let \(r=\mathrm{rank}(Z)\), and let \(Z=U\Sigma V^\top\) be a compact SVD, where \(U,V\in\mathbb R^{m\times r}\) have orthonormal columns and \(\Sigma\in\mathbb R^{r\times r}\) has positive diagonal entries. Equality in~\cref{eq:budget_chain} (\cref{lem:nuclear_frobenius_combined}) implies that \(W\)'s right singular vectors coincide with \(H\)'s left singular vectors and that the two sets of singular values are proportional.  Thus we may write
\[
  W = U \Sigma_W Q^\top,
  \qquad
  H = Q \Sigma_H V^\top,
  \qquad
  \Sigma_W = \eta \Sigma_H,
  \qquad
  Q^\top Q = I_r,
\]
where \(Q\in\mathbb R^{d\times r}\) has orthonormal columns and \(\Sigma_W,\Sigma_H\in\mathbb R^{r\times r}\) are positive diagonal matrices.
Moreover, we have
\[
\|W\|_F^2 = \|\Sigma_W\|_F^2 = E_W, \qquad
\|H\|_F^2 = \|\Sigma_H\|_F^2 = E_H,
\qquad\Longrightarrow\qquad
\Sigma_W = \sqrt{\frac{E_W}{E_H}}\Sigma_H.
\]
Therefore, 
\[ Z = WH = U\Sigma_W\Sigma_H V^\top.\]
This is also the compact SVD of \(WH\). 
We compute the matrices \(WW^\top\) and \(H^\top H\):
\[
WW^\top = (U \Sigma_W Q^\top)(U \Sigma_W Q^\top)^\top = U\Sigma_W^2 U^\top = \sqrt{\frac{E_W}{E_H}} U \Sigma_W\Sigma_H U^\top.
\]
Similarly, for \(H^\top H\):
\[
H^\top H = (Q \Sigma_H V^\top)^\top(Q \Sigma_H V^\top) = V\Sigma_H^2 V^\top = \sqrt{\frac{E_H}{E_W}} V \Sigma_W\Sigma_H V^\top.
\]
Extending this compact SVD to a full real SVD by padding zero singular values, we now apply~\cref{lem:svd_factor_symmetry}. 
Here, the circulant matrix is \(Z\), and the nonzero singular-value block is \(\Sigma_W\Sigma_H\). We conclude that
\[ WW^\top = \sqrt{\frac{E_W}{E_H}} (Z Z^\top)^{1/2}, \qquad H^\top H = \sqrt{\frac{E_H}{E_W}} (Z Z^\top)^{1/2}, \]
where the common matrix \((Z Z^\top)^{1/2}\) is a circulant matrix. This completes the proof.
\end{proof}

\section{Symmetry Transfer~II: Permutation Symmetry}\label{sec:single_block_perm}
In this section, we study the case where the target distribution matrix \(Y\) is the orbit of a single probability vector under a group of permutations of the indices. Concretely, let \(\mathcal G\) be a finite group acting on probability vectors \(\boldsymbol{y}\in\Delta^{m-1}\) by permuting their entries; we write this action as \(g\circ\boldsymbol{y}\). Given a base vector \(\boldsymbol{y}\), we collect its \(\mathcal G\)-orbit as columns to form
\[
Y = [  g\circ \boldsymbol{y} \mid g\in \mathcal G  ] \in \mathbb{R}^{m\times |\mathcal G|},
\]
so each column \(Y_{:,j}=g_j\circ \boldsymbol{y}\) for some enumeration \(\{g_j\}\) of \(\mathcal G\). 

Many token sets in natural language behave as interchangeable alternatives in families of analogous contexts, e.g., \{\texttt{dog}, \texttt{cat}, \texttt{rabbit}\}, \{\texttt{car}, \texttt{bus}, \texttt{bicycle}, \texttt{train}\}, or \{\texttt{fork}, \texttt{knife}, \texttt{plate}, \texttt{cup}, \texttt{glass}\}. Typical prompts such as ``\texttt{He barks and wags his tail. My pet is a \textvisiblespace}'', ``\texttt{She purrs and loves milk. My pet is a \textvisiblespace}'', or ``\texttt{He hops around the garden. My pet is a \textvisiblespace}'' yield next-token distributions concentrated on the same animal set; aggregated over many and varied contexts, swapping members within the set leaves the family of next-token distributions essentially unchanged. In our modeling, the label columns are permutations of a single base probability vector, forming an orbit under the symmetric group \(S_m\) that captures exchangeability within the set.

For \( \mathcal G = S_m \), the symmetric group of all permutations of \( \{1, 2, \dots, m\} \), we have \( |S_m| = m! \). Fix an enumeration \(\{g_j\}_{j=1}^{m!}\) of \(S_m\). The target distribution matrix becomes
\[
Y=[g_1\circ\boldsymbol{y}\mid g_2\circ\boldsymbol{y}\mid\cdots\mid g_{m!}\circ\boldsymbol{y}].
\]
Equivalently, for each column \(j\), define \(\tau_{j,i}\) by \((g_j\circ\boldsymbol{y})_i=y_{\tau_{j,i}}\) for \(i=1,\ldots,m\). Then the entries of \(Y\) are
\[
Y = \begin{bmatrix}
y_{\tau_{1,1}} & y_{\tau_{2,1}} & \cdots & y_{\tau_{m!,1}} \\
y_{\tau_{1,2}} & y_{\tau_{2,2}} & \cdots & y_{\tau_{m!,2}} \\
\vdots & \vdots & \ddots & \vdots \\
y_{\tau_{1,m}} & y_{\tau_{2,m}} & \cdots & y_{\tau_{m!,m}}
\end{bmatrix}.
\]

\begin{definition}[Simplex ETF~\citep{papyan2020prevalence}]
For \(m\ge 2\), the standard simplex ETF is the collection of points in \(\mathbb{R}^m\) specified by the rows of
\begin{equation}
\label{eq:simplex-etf-mstar}
M^*=\sqrt{\frac{m}{m-1}}\left(I_m-\frac{1}{m}\boldsymbol{1}_m\boldsymbol{1}_m^\top\right).
\end{equation}
In this paper, we allow rotations and rescalings, so a general simplex ETF is specified by the rows of \(M=c M^* U^\top\in\mathbb{R}^{m\times d}\), where \(c>0\) is a scale factor and \(U\in\mathbb{R}^{d\times m}\) is a partial orthogonal matrix satisfying \(U^\top U=I_m\).
\end{definition}

\begin{theorem}[Optimal Embeddings for the Symmetric Group]
\label{thm:single_block_perm}
Let \(\boldsymbol{y}\) be a nonuniform target distribution, and let \(Y\) be the target distribution matrix generated by the symmetric group \(S_m\).
For \(d\ge m\), under~\cref{ass:active_constraint}, the matrices \(W^*\) and \(H^*\) solve the constrained problem~\cref{eq:obj_function_constrained} if and only if 
\[
  W^*=\sqrt{\frac{E_W}{m-1}} UQ^\top,
  \qquad
  H^*=-\sqrt{\frac{E_H}{m-1}} QV^\top,
\]
for any partial orthogonal matrix \(Q\in\mathbb{R}^{d\times(m-1)}\) (\(Q^\top Q=I_{m-1}\)), where
\begin{itemize}
  \item \(C=A-Y\), with \(A\in\mathbb{R}^{m\times m!}\) constructed analogously to \(Y\) so that \(A_{:,j}=g_j\circ\boldsymbol{\alpha}\), where \(\boldsymbol{\alpha}=(\alpha_1,\ldots,\alpha_m)\); equivalently, \( A_{i,j}=\alpha_{\tau_{j,i}} \).
  Here \(\boldsymbol{\alpha}\) is the unique solution to
  \[
    k(\alpha_i-y_i)+\log\alpha_i=\bar{\ell}\ (i=1,\dots,m),\quad \sum_{i=1}^m\alpha_i=1,\quad
    k=\sqrt{\frac{E_W E_H}{m!(m-1)\sum_{i=1}^m(\alpha_i-y_i)^2}},
  \]
  where \(\bar{\ell}=(1/m)\sum_{r=1}^m \log \alpha_r\).
  \item \(C=U(\gamma I_{m-1}) V^\top\) is the rank-\((m-1)\) SVD of \(C\), with \(U\in\mathbb{R}^{m\times(m-1)}\) and \(V\in\mathbb{R}^{m!\times(m-1)}\).
  \item \(\gamma = \sqrt{m!/(m-1)}\|\boldsymbol{\alpha}-\boldsymbol{y}\|_2\).
  Equivalently, \(k=\sqrt{E_W E_H}/((m-1)\gamma)\).
\end{itemize}
\end{theorem}

\begin{proof}[Proof of~\cref{thm:single_block_perm}]
The proof is divided into five steps. Steps~1--3 derive an explicit lower bound for the total loss. In Step~4 (sufficiency), we show that \((W,H)\) with the form stated in the theorem attains this lower bound. In Step~5 (necessity), we prove that any optimal \((W,H)\) must have the form stated in the theorem.

\paragraph{Step~1 (Bounding the total loss as the sum of two terms).}
In this step, we apply Jensen's inequality to the total loss using an auxiliary vector \(\boldsymbol{\alpha}=(\alpha_1,\ldots,\alpha_m)\) (to be determined later). Define \(A\) by \(A_{:,j}=g_j\circ\boldsymbol{\alpha}\), using the same group element enumeration as \(Y_{:,j}=g_j\circ\boldsymbol{y}\), and set \(C:=A-Y\). The total loss is then lower bounded by $\mathrm{tr}(C^\top WH)-m!\sum_{i=1}^m \alpha_i\log \alpha_i$.

Expanding the cross-entropy loss for each column \( j \), we have:
\begin{align*}
\mathcal L\left( \sigma\left( \left( W H \right)_{:,j}\right), Y_{:,j} \right ) &= - \sum_{i=1}^m y_{\tau_{j,i}} \log \left( \frac{ \exp\left( \left( W H \right)_{i,j} \right ) }{ \sum_{k=1}^m \exp\left( \left( W H \right)_{k,j} \right ) } \right )\\
& = - \sum_{i=1}^m y_{\tau_{j,i}} \left( W H \right)_{i,j} + \log \left( \sum_{k=1}^m \exp\left( \left( W H \right)_{k,j} \right ) \right ).
\end{align*}
To further analyze the loss function, introduce a vector \( \boldsymbol{\alpha}=(\alpha_1,\ldots,\alpha_m) \) such that \( \alpha_i > 0 \) and \( \sum_{i=1}^m \alpha_i = 1 \). We can rewrite the denominator term in the logarithm as follows:
\[
\sum_{k=1}^m \exp\left( \left( W H \right)_{k,j} \right ) = \sum_{k=1}^m \alpha_{\tau_{j,k}} \cdot \frac{1}{\alpha_{\tau_{j,k}}} \exp\left( \left( W H \right)_{k,j} \right ).
\]
Applying Jensen's inequality to \( \log x \), we get
\begin{align}
\label{eq:jensen}
\log \left( \sum_{k=1}^m \alpha_{\tau_{j,k}} \cdot \frac{1}{\alpha_{\tau_{j,k}}} \exp\left( \left( W H \right)_{k,j} \right ) \right ) &\ge \sum_{k=1}^m \alpha_{\tau_{j,k}} \cdot \log \left( \frac{1}{\alpha_{\tau_{j,k}}} \exp\left( \left( W H \right)_{k,j} \right ) \right )\\
&= \sum_{k=1}^m \alpha_{\tau_{j,k}} \cdot \left( \left( W H \right)_{k,j} - \log \alpha_{\tau_{j,k}} \right ).  \notag
\end{align}
Thus, the cross-entropy loss for the \( j \)-th column is lower bounded by
\[
\mathcal L\left( \sigma\left( \left( W H \right)_{:,j}\right), Y_{:,j} \right ) \ge - \sum_{i=1}^m y_{\tau_{j,i}} \left( W H \right)_{i,j} + \sum_{k=1}^m \alpha_{\tau_{j,k}} \left( W H \right)_{k,j} - \sum_{k=1}^m \alpha_{\tau_{j,k}} \log \alpha_{\tau_{j,k}}.
\]
Summing over all columns \( j = 1, \ldots, m! \), we obtain a lower bound on the total loss
\[
\mathcal J(W,H;Y)
\ge
\sum_{j=1}^{m!}
\left[
- \sum_{i=1}^m y_{\tau_{j,i}} \left( W H \right)_{i,j}
+ \sum_{k=1}^m \alpha_{\tau_{j,k}} \left( W H \right)_{k,j}
- \sum_{k=1}^m \alpha_{\tau_{j,k}} \log \alpha_{\tau_{j,k}}
\right].
\]
Define \( C = A - Y \), where \( A_{:,j}=g_j\circ\boldsymbol{\alpha} \) and \( Y_{:,j}=g_j\circ\boldsymbol{y} \). Equivalently, each entry of \( C \) is \( C_{i,j} = \alpha_{\tau_{j,i}} - y_{\tau_{j,i}} \). 
Therefore, the lower bound reduces to
\[
\mathcal J(W,H;Y) \ge \mathrm{tr}\left( C^\top W H \right ) - m!\sum_{i=1}^m\alpha_i\log\alpha_i.
\]

\paragraph{Step~2 (Bounding the first term of the total loss in Step~1).}
In this step, we bound the first term of the total loss in Step~1, obtaining \(\mathrm{tr}(C^\top WH)\ge -\gamma(\boldsymbol{\alpha})\sqrt{E_W E_H}\), where $\gamma(\boldsymbol{\alpha})=\sqrt{m!/(m-1)}\|\boldsymbol{\alpha}-\boldsymbol{y}\|_2$.

Define the singular value decomposition \( C = U \Sigma V^\top \), where \( U \in \mathbb{R}^{m \times r} \), \( V \in \mathbb{R}^{m! \times r} \) are column-orthogonal matrices, and \( \Sigma \in \mathbb{R}^{r \times r} \) is a diagonal matrix with \( r = \text{rank}(C) \).
By~\cref{lem:2_transitive_svd}, $C$ has $m-1$ equal singular values $\gamma(\boldsymbol{\alpha})=\sqrt{m!/(m-1)}\|\boldsymbol{\alpha}-\boldsymbol{y}\|_2$.
From~\cref{lem:von_neumann_rectangular}, we have:
\begin{equation}
\label{eq:von_neumann}
\mathrm{tr}\left(C^\top W H\right)
\ge
- \sum_{i=1}^{\mathrm{rank}(C)} s_i(C)s_i(W H).
\end{equation}
Since \(C\) has rank \(m-1\) and each of its nonzero singular values is \(\gamma(\boldsymbol{\alpha})\), it follows that
\[
\mathrm{tr}\left(C^\top W H\right)
\ge
- \gamma(\boldsymbol{\alpha})\sum_{i=1}^{m-1} s_i(W H).
\]
But \(\sum_{i=1}^{m-1}s_i(W H) \le \sum_{i=1}^{\mathrm{rank}(W H)}s_i(W H) = \left\|W H\right\|_{*}\),
where \(\left\|\cdot\right\|_{*}\) denotes the nuclear norm. Hence,
\begin{equation}
\label{eq:trace_bound_1}
\mathrm{tr}\left(C^\top W H\right)
\ge
- \gamma(\boldsymbol{\alpha})\left\|W H\right\|_{*}.
\end{equation}
By~\cref{lem:nuclear_frobenius_combined}, one has
\begin{equation}
\label{eq:nuclear_frobenius_combined}
\left\|W H\right\|_{*}
\le
\left\|W\right\|_F \left\|H\right\|_F
\le
\sqrt{E_W E_H}.
\end{equation}
Combining these results, we obtain
\begin{equation}
\label{eq:trace_bound}
\mathrm{tr}\left(C^\top W H\right)
\ge
- \gamma(\boldsymbol{\alpha})\left\|W H\right\|_{*}
\ge
- \gamma(\boldsymbol{\alpha})\sqrt{E_W E_H}.
\end{equation}
Thus, for every feasible pair $\left(W,H\right)$, we have
$\mathrm{tr} \left(C^\top W H\right) \ge - \gamma(\boldsymbol{\alpha})\sqrt{E_W E_H}$. Therefore, we have shown for every $(W,H)$ and any $\boldsymbol{\alpha}\in\Delta^{m-1}$:
\begin{equation}
  \mathcal J(W,H;Y) \ge \phi(\boldsymbol{\alpha}) \coloneqq
  -\sqrt{\frac{m! E_W E_H}{ m-1 }} 
     \left\|\boldsymbol{\alpha}-\boldsymbol{y}\right\|_2
   - m!\sum_{i=1}^m \alpha_i\log\alpha_i.
  \label{eq:phi_def}
\end{equation}

\paragraph{Step~3 (Obtaining an explicit lower bound for the total loss).}
Substituting the estimate from Step~2 into the Jensen-based bound from Step~1 yields a lower bound depending only on $\boldsymbol{\alpha}$, namely the function on the right-hand side of~\cref{eq:phi_def}. We show that it is strictly concave. Maximizing this function over $\boldsymbol{\alpha}$ yields an explicit lower bound for the total loss.

Impose the simplex constraint $\sum_{i=1}^m\alpha_i=1$ with multiplier $\nu$ and
consider the Lagrangian of the right-hand side of~\cref{eq:phi_def}
\[
  \mathscr{J}_{\boldsymbol{\alpha}}(\boldsymbol{\alpha},\nu)
  =-\sqrt{\frac{m! E_W E_H}{ m-1 }} 
     \left\|\boldsymbol{\alpha}-\boldsymbol{y}\right\|_2
   - m!\sum_{i=1}^m\alpha_i\log\alpha_i
   + \nu\left(\sum_{i=1}^m\alpha_i-1\right).
\]
Taking derivatives,
\[
  \frac{\partial\mathscr{J}_{\boldsymbol{\alpha}}}{\partial\alpha_i}
  =-\sqrt{\frac{m! E_W E_H}{ m-1 }}
    \frac{\alpha_i-y_i}{\|\boldsymbol{\alpha}-\boldsymbol{y}\|_2}
   - m!\left(\log\alpha_i+1\right)
   + \nu
   = 0,
  \qquad i=1,\dots,m.
\]
Summing these $m$ equations and using $\sum_i(\alpha_i-y_i)=0$ gives
\[
  - m!\sum_{i=1}^m(\log\alpha_i+1)+m \nu=0 \qquad
   \Longrightarrow \qquad
  \nu=m!\left(1+\frac1m\sum_{i=1}^m\log\alpha_i\right).
\]
Subtract this average equation from the $i$-th one to eliminate $\nu$:
\begin{equation}
  -\sqrt{\frac{m! E_W E_H}{ m-1 }}
    \frac{\alpha_i-y_i}{\|\boldsymbol{\alpha}-\boldsymbol{y}\|_2}
   - m!\left(\log\alpha_i-\frac1m\sum_{r=1}^m\log\alpha_r\right)=0,
  \qquad i=1,\dots,m.
  \label{eq:alpha_eq1}
\end{equation}
Define
\[
  k:= \sqrt{\frac{E_W E_H}{ m!(m-1) \|\boldsymbol{\alpha}-\boldsymbol{y}\|_2^2}},
\]
so that~\cref{eq:alpha_eq1} is equivalent to the closed system
\begin{equation}
  \label{eq:alpha_eq2}
    k (\alpha_i-y_i) + \log\alpha_i
    = \frac1m\sum_{r=1}^m\log\alpha_r,
    \quad i=1,\dots,m,
    \qquad
    \sum_{i=1}^m\alpha_i=1.
  \end{equation}
We now show that the solution to~\cref{eq:alpha_eq2} exists and is unique. $\phi$ is continuous on the compact simplex
$
  \Delta^{m-1}
$
so by the Weierstrass theorem it attains a maximum; let
$\boldsymbol{\alpha}^*$ be one maximizer. Since $-\|\boldsymbol{\alpha}-\boldsymbol{y}\|_2$ is concave and
$-\sum_k\alpha_k\log\alpha_k$ is strictly concave on the simplex
$\Delta^{m-1}$, their positive linear combination $\phi$ is strictly concave. A strictly concave function
has at most one maximizer, hence $\boldsymbol{\alpha}^*$ is unique. Now, we prove that $\alpha_i^* >0$ for all $i$. Assume, toward a contradiction, that $\alpha_i^*=0$ for some coordinate
$i$.  Because $\sum_k\alpha_k^*=1$, there is at least one index
$j\neq i$ with $\alpha_j^*>0$. Move an infinitesimal mass $\delta>0$ from $j$ to $i$:
\[
  \tilde{\alpha}_k(\delta)=
  \begin{cases}
       \delta,            & k=i,\\[4pt]
       \alpha_j^*-\delta,& k=j,\\[4pt]
       \alpha_k^*,   & k\notin\{i,j\}.
  \end{cases}
\]
Now, consider the directional derivative of $\phi$.  
Set $\boldsymbol{\eta}=\boldsymbol{e}_i-\boldsymbol{e}_j$.  Then
\[
  \mathcal{D}_{\boldsymbol{\eta}} \phi(\boldsymbol{\alpha}^*)
  =-\sqrt{\frac{m! E_W E_H}{ m-1 }} \mathcal{D}_{\boldsymbol{\eta}}\|\boldsymbol{\alpha}-\boldsymbol{y}\|_2
     + m! \mathcal{D}_{\boldsymbol{\eta}}\mathcal{H}(\boldsymbol{\alpha})\Big|_{\boldsymbol{\alpha}^*}.
\]
\begin{itemize}
\item Entropy part.
      $\mathcal{H}(\boldsymbol{\alpha})=-\sum_k\alpha_k\log\alpha_k$ has
      $\partial_{\alpha_i}\mathcal{H}=-(\log\alpha_i+1)\to+\infty$ as
      $\alpha_i\to0^{+}$, while $\partial_{\alpha_j}\mathcal{H}$ is finite because
      $\alpha_j^*>0$.  Hence
      $m! \mathcal{D}_{\boldsymbol{\eta}}\mathcal{H}(\boldsymbol{\alpha}^*)=+\infty$.
\item $\ell_2$-norm part.
      $r(\boldsymbol{\alpha})=\|\boldsymbol{\alpha}-\boldsymbol{y}\|_2$ is
      $C^{1}$ on $\Delta^{m-1}\setminus\{\boldsymbol{y}\}$. If \(\boldsymbol{\alpha}^*\neq\boldsymbol{y}\), this gives \( |\mathcal{D}_{\boldsymbol{\eta}}r(\boldsymbol{\alpha}^*)|<\infty \). 
\end{itemize}
Combining the two pieces,
\[
  \mathcal{D}_{\boldsymbol{\eta}}\phi(\boldsymbol{\alpha}^*)=+\infty>0.
\]
Thus $\phi(\tilde{\boldsymbol{\alpha}}(\delta))>\phi(\boldsymbol{\alpha}^*)$
for all sufficiently small $\delta$, contradicting maximality.  Hence every
component of the unique maximizer is strictly positive:
$\alpha_i^*>0\text{ for all }i$. Therefore, the first-order Lagrange conditions are both necessary and sufficient for
optimality. Writing them out gives exactly the closed system~\cref{eq:alpha_eq2} whose unique solution inside the simplex is the maximizer
$\boldsymbol{\alpha}^*$. Then a lower bound of the loss is given by
\begin{equation}
  \label{eq:lower_bound}
\mathcal J(W,H;Y)
\ge -\sqrt{\frac{m! E_W E_H}{ m-1 }}  \left\|\boldsymbol{\alpha}^*-\boldsymbol{y}\right\|_2
-m!\sum_{i=1}^m\alpha_i^*\log\alpha_i^*.
\end{equation}
Let \(\mathcal J^*\) denote the right-hand side of~\cref{eq:lower_bound}.  Since \(\mathcal J^*\) depends only on the fixed vector \(\boldsymbol{\alpha}^*\), Steps~1--3 prove \(\mathcal J^*\le \mathcal J(W',H';Y)\) for every feasible competitor \((W',H')\); thus \(\mathcal J^*\) is a universal lower bound. Finally, we show that \( \boldsymbol{\alpha}^* \neq \boldsymbol{y} \). If \( \boldsymbol{\alpha}^* = \boldsymbol{y} \), then \( \alpha_i^* = y_i \) for all \( i \). Then~\cref{eq:alpha_eq2} becomes
\[
\log y_i =\frac1m\sum_{r=1}^m\log y_r.
\]
That forces \(y_1=\cdots=y_m=1/m\), contradicting the assumption that \(\boldsymbol{y}\) is nonuniform. Hence \(\boldsymbol{\alpha}^*\neq\boldsymbol{y}\). In particular, if we choose \(\boldsymbol{\alpha}=\boldsymbol{\alpha}^*\), then \(\gamma(\boldsymbol{\alpha}^*)>0\), the matrix \(C\) has exactly \(m-1\) nonzero singular values, and the division by \(\|\boldsymbol{\alpha}^*-\boldsymbol{y}\|_2\) in this step is valid.

\paragraph{Step~4 (Sufficiency: constructing $(W,H)$ that attain the Step~3 lower bound).}
If $(W,H)$ have the form stated in the theorem, each inequality in the preceding steps holds with equality; hence this construction attains the lower bound from Step~3.
It remains to verify: (E1) equality in Jensen~\cref{eq:jensen}; (E2) equality in the trace--nuclear bound~\cref{eq:trace_bound}; and (E3) that $\boldsymbol{\alpha}$ satisfies the closed system~\cref{eq:alpha_eq2}.

\medskip

\noindent\textit{(E3) System for $\boldsymbol{\alpha}$.}  First, choose $\boldsymbol{\alpha}=\boldsymbol{\alpha}^*$ to be the unique solution of the closed system~\cref{eq:alpha_eq2} and set $A, C$ as in the statement of the theorem. Then~\cref{eq:alpha_eq2} holds by construction. 

\medskip

\noindent\textit{(E2) Trace--nuclear equality~\cref{eq:trace_bound}.} 
Next, we verify that the trace bound~\cref{eq:trace_bound} holds with equality, i.e., \(\mathrm{tr}(C^\top WH) = -\gamma \sqrt{E_W E_H}\). We compute the product \(WH\) with the form stated in the theorem:
\[
WH 
= \left(\sqrt{ \frac{E_W}{m-1}} UQ^\top\right)
  \left(-\sqrt{ \frac{E_H}{m-1}} QV^\top\right)
= -\frac{\sqrt{E_W E_H}}{m-1} U (Q^\top Q) V^\top
= -\frac{\sqrt{E_W E_H}}{m-1} U V^\top.
\]
Hence,
\[
C^\top WH
= (U \Sigma V^\top)^\top \left(- \frac{\sqrt{E_W E_H}}{m-1} U V^\top\right)
= -\frac{\sqrt{E_W E_H}}{m-1} V \Sigma (U^\top U) V^\top.
\]
Taking the trace and using \(\Sigma=\gamma I_{m-1}\), \(U^\top U=I_{m-1}\), and \(V^\top V=I_{m-1}\):
\[
\mathrm{tr}(C^\top WH)
= -\frac{\sqrt{E_W E_H}}{m-1} \mathrm{tr} \left(V (\gamma I_{m-1}) V^\top\right) \\
= -\frac{\gamma \sqrt{E_W E_H}}{m-1} \mathrm{tr} (V^\top V) \\
= -\gamma \sqrt{E_W E_H}.
\]
Additionally, we verify that the individual constraints on the factors are satisfied. Let \(W^*\) and \(H^*\) be the constructed matrices. We compute each Frobenius norm separately:
\[
\left\|W^*\right\|_F^2 = \left\| \sqrt{\frac{E_W}{m-1}} U Q^\top \right\|_F^2 = \frac{E_W}{m-1} \|UQ^\top\|_F^2 = \frac{E_W}{m-1} \mathrm{tr}\left(Q U^\top U Q^\top\right) = \frac{E_W}{m-1} \mathrm{tr}(Q Q^\top).
\]
Since \(Q\) is a \(d \times (m-1)\) matrix with orthonormal columns, \(\mathrm{tr}(QQ^\top) = \mathrm{tr}(Q^\top Q) = \mathrm{tr}(I_{m-1}) = m-1\). Thus:
\[
\left\|W^*\right\|_F^2 = \frac{E_W}{m-1} (m-1) = E_W.
\]
Similarly, for \(H^*\):
\[
\left\|H^*\right\|_F^2 = \left\| -\sqrt{\frac{E_H}{m-1}} Q V^\top \right\|_F^2 = \frac{E_H}{m-1} \|QV^\top\|_F^2 = \frac{E_H}{m-1} \mathrm{tr}(V Q^\top Q V^\top) = \frac{E_H}{m-1} (m-1) = E_H.
\]
The construction thus satisfies the constraints \(\|W\|_F^2 = E_W\) and \(\|H\|_F^2 = E_H\). 

\medskip

\noindent\textit{(E1) Jensen equality~\cref{eq:jensen}.}  Finally, it suffices to show that Jensen's inequality~\cref{eq:jensen} holds with equality for our choice of $W,H$ and $\boldsymbol{\alpha}^*$. Recall~\cref{eq:jensen} asserts for each column $j$:
\[
      \log \left(\sum_{k=1}^m \alpha^*_{\tau_{j,k}}
      \frac{\exp\left((WH)_{k,j}\right)}{\alpha^*_{\tau_{j,k}}}\right)
      \ge \sum_{k=1}^m \alpha^*_{\tau_{j,k}}
      \log \left( \frac{\exp\left((WH)_{k,j}\right)}{\alpha^*_{\tau_{j,k}}}\right).
\]
Equality holds if the ratios \( \exp\left((WH)_{k,j}\right)/\alpha^*_{\tau_{j,k}} \) are the same for all $k=1,\dots,m$. But by~\cref{eq:alpha_eq2}, for each $i$
\[
\log\alpha^*_{i}
  = - k \left(\alpha^*_{i}-y_{i}\right)
    +\frac1m\sum_{r=1}^m\log\alpha^*_{r}
  \quad\Longrightarrow\quad
- k (\alpha^*_{i}-y_{i})
  = \log\alpha^*_{i}
    -\frac1m\sum_{r=1}^m\log\alpha^*_{r}.
\]
Since $WH=-k C$ and $C_{i,j}=\alpha^*_{\tau_{j,i}}-y_{\tau_{j,i}}$, we obtain
\[ (WH)_{k,j}
= - k\left(\alpha^*_{\tau_{j,k}}-y_{\tau_{j,k}}\right)
= \log\alpha^*_{\tau_{j,k}}
  -\frac1m\sum_{r=1}^m\log\alpha^*_{r}.
\]
Therefore
\[
\frac{\exp\left((WH)_{k,j}\right)}{\alpha^*_{\tau_{j,k}}}
= \frac{\exp \left(\log\alpha^*_{\tau_{j,k}}
                     -\frac1m\sum_{r=1}^m\log\alpha^*_{r}\right)}
       {\alpha^*_{\tau_{j,k}}}
= \exp\left(-\frac1m\sum_{r=1}^m\log\alpha^*_{r}\right),
\]
which is independent of $k$.  Hence the condition for equality in Jensen's
inequality is satisfied, and~\cref{eq:jensen} holds with equality for every
column $j$.  Therefore equality holds throughout the Step~1--Step~3 lower-bound
chain for the displayed factorization, so the constructed pair satisfies
\(\mathcal J(W,H;Y)=\mathcal J^*\).

\paragraph{Step~5 (Necessity: characterizing all $(W,H)$ that attain the Step~3 lower bound).}
At optimality, combining the equality conditions with the KKT stationarity equations implies: (1) $C$ and $WH$ share common left/right singular vectors; (2) each of $W$ and $H$ has rank $m-1$ and each has the same nonzero singular values; and (3) $W$ and $H$ take the stated forms with an arbitrary partial orthogonal matrix $Q$.

From the sufficiency proof, any optimizer must attain the lower bound in~\cref{eq:lower_bound}. Consequently, \cref{eq:jensen,eq:trace_bound_1,eq:trace_bound} must hold with equality. First, by the equality condition of Jensen's inequality~\cref{eq:jensen}, for each column \(j\) the ratios \( \exp\left((WH)_{k,j}\right)/\alpha^*_{\tau_{j,k}} \) must be equal over all \(k\).  Hence \( \boldsymbol{\sigma}(WH)_{:,j}
=(\alpha^*_{\tau_{j,1}},\ldots,\alpha^*_{\tau_{j,m}})^\top=A_{:,j} \). Thus \(\boldsymbol{\sigma}(WH)=A\).  
We now consider the KKT conditions for~\cref{eq:obj_function_constrained}. Since the active Frobenius constraints are saturated at the optimum, there exist nonnegative multipliers \(\lambda_1,\lambda_2\).  
With the Lagrangian
\[
  \mathscr{J}(W,H,\lambda_1,\lambda_2)
  = \sum_{j=1}^{m!}\left[-Y_{:,j}^\top (WH)_{:,j}
      +\log \sum_{i=1}^m e^{(WH)_{i,j}}\right]
   + \frac{\lambda_1}2 \left(\|W\|_F^2-E_W\right) + \frac{\lambda_2}2 \left(\|H\|_F^2-E_H\right),
\]
stationarity $\nabla_W\mathscr{J}=0$ gives
\[
  (-Y + \boldsymbol{\sigma}(WH)) H^\top + \lambda_1 W = 0.
\]
Since $\boldsymbol{\sigma}(WH)=A$ and \(A-Y=C\),
\begin{equation}\label{eq:K1}
  C H^\top + \lambda_1 W = 0.
\end{equation}
Similarly, stationarity $\nabla_H\mathscr{J}=0$ yields
\begin{equation}\label{eq:K2}
  W^\top C + \lambda_2 H = 0.
\end{equation}
The active constraints imply \(W\neq 0\) and \(H\neq 0\), so the nonzero equality
case applies. Equality in~\cref{eq:nuclear_frobenius_combined} (by~\cref{lem:nuclear_frobenius_combined}) implies that \(W\)'s right singular vectors coincide with \(H\)'s left singular vectors and that the two sets of singular values are proportional. Thus we can write the SVDs of \(W\) and \(H\) as
\begin{equation}\label{eq:general_SVD}
  W = U_W \Sigma_W Q^\top,
  \qquad
  H = Q \Sigma_H V_H^\top,
  \qquad
  \Sigma_W = \eta \Sigma_H,
  \qquad
  Q^\top Q = I_{m-1}.
\end{equation}
Moreover, we have the following observation:
\[
\|W\|_F^2 = \|\Sigma_W\|_F^2 = E_W, \qquad
  \|H\|_F^2 = \|\Sigma_H\|_F^2 = E_H
\qquad\Longrightarrow\qquad
\Sigma_H = \sqrt{\frac{E_H}{E_W}}\Sigma_W.
\]
By~\cref{lem:von_neumann_rectangular}, equality in~\cref{eq:von_neumann}
implies that \(C\) and the product \(WH\) share their singular directions, so
\begin{equation}\label{eq:shared_singular}
  C  = U \gamma I_{m-1} V^{\top},
  \qquad
  WH = U \Sigma_{WH} V^{\top}.
\end{equation}
The equality in~\cref{eq:trace_bound_1} implies that $\mathrm{rank}(WH) \le m-1$. Also, we have established $\boldsymbol{\sigma}(WH) = A$. And the columns of $A$ and $C$ are the $m!$ permutations of a single vector.
Therefore, by applying~\cref{lem:orbit_S_softmax_2_transitive}, we conclude that $\Sigma_{WH}$ must be a scalar multiple of the $(m-1) \times (m-1)$ identity matrix.
Let this scalar be $k_0$. Thus, \( \Sigma_{WH} = k_0 I_{m-1}\). Because \(U\) and \(U_W\) are both orthonormal bases for the same subspace
\(\mathrm{col}(WH)=\mathrm{col}(W)\), there exist orthogonal matrices \(S, T\in\mathbb{R}^{(m-1)\times(m-1)}\) such that
\begin{equation}\label{eq:bridge}
  U_W = U S, \qquad V_H = V T^{\top}.
\end{equation}
Inserting~\cref{eq:bridge} into~\cref{eq:general_SVD} and using \(\Sigma_H=\sqrt{E_H/E_W}\Sigma_W\) we obtain
\begin{equation}\label{eq:W_H_expanded}
  W = U S \Sigma_W Q^{\top},
  \qquad
  H = \sqrt{\frac{E_H}{E_W}} Q \Sigma_W T V^{\top},
  \qquad
  WH = \sqrt{\frac{E_H}{E_W}} U S \Sigma_W^2 T V^{\top}.
\end{equation}
Comparing the two singular value decompositions
\[
  WH = U \Sigma_{WH} V^\top
  \quad\text{and}\quad
  WH = U S \left(\sqrt{\frac{E_H}{E_W}} \Sigma_W^2\right) T V^{\top},
\]
the uniqueness of the singular values forces
\(
\Sigma_{WH} \) and \( \sqrt{E_H/E_W} \Sigma_W^2 \) to have the same diagonal entries (up to permutation). Since \(\Sigma_{WH}\) and \(\sqrt{E_H/E_W} \Sigma_W^2\) have the same diagonal entries, 
\[
  \|W\|_F^2  = \sum_{i=1}^{m-1}s_i(W)^2
                = \sqrt{\frac{E_W}{E_H}}\sum_{i=1}^{m-1}k_0
                = \sqrt{\frac{E_W}{E_H}} k_0 (m-1).
\]
The Frobenius norm constraint \(\|W\|_F^2=E_W\) then gives
\[
  \sqrt{\frac{E_W}{E_H}} k_0 (m-1) = E_W
  \quad\Longrightarrow\quad
  k_0=\frac{\sqrt{E_W E_H}}{m-1}.
\]
Hence the singular values themselves satisfy
\(s_i(W)=\sqrt{E_W / (m-1)}\) and \( s_i(H)=\sqrt{E_H / (m-1)}\).
Plug~\cref{eq:shared_singular,eq:W_H_expanded} into~\cref{eq:K1},
\[
  U \gamma I_{m-1} T^{\top}\Sigma_H Q^{\top}
  + \lambda_1 U S \Sigma_W Q^{\top} = 0.
\]
Multiplying on the left by \(U^{\top}\) and on the right by \(Q\) and using \(\Sigma_W=\sqrt{E_W/(m-1)} I_{m-1}\)  and \(\Sigma_H=\sqrt{E_H/(m-1)} I_{m-1}\) gives
\begin{equation}\label{eq:K1_matrix}
  \gamma\sqrt{E_H} T^{\top}
  + \lambda_1\sqrt{E_W} S 
  = 0.
\end{equation}
Plug~\cref{eq:shared_singular,eq:W_H_expanded} into~\cref{eq:K2},
\[
  Q \Sigma_W S^{\top}\gamma I_{m-1} V^{\top}
  + \lambda_2 Q \Sigma_H T V^{\top}=0.
\]
Left-multiplying by \(Q^{\top}\) and right-multiplying by \(V\) and using \(\Sigma_W=\sqrt{E_W/(m-1)} I_{m-1}\)  and \(\Sigma_H=\sqrt{E_H/(m-1)} I_{m-1}\) yields
\begin{equation}\label{eq:K2_matrix}
  \gamma \sqrt{E_W} S^{\top}
  + \lambda_2 \sqrt{E_H} T 
  = 0.
\end{equation}
Solving~\cref{eq:K1_matrix} for \(T^{\top}\) gives
\[
  T^{\top}  =  -\frac{\lambda_1}{\gamma}\sqrt{\frac{E_W}{E_H}} S.
\]
Transposing and substituting \(T = -(\lambda_1/\gamma)\sqrt{E_W/E_H}\, S^{\top}\) into~\cref{eq:K2_matrix} gives
\[
  \gamma \sqrt{E_W} S^{\top}
  + \lambda_2\left(-\frac{\lambda_1}{\gamma}\sqrt{E_W} S^{\top}\right)
   = S^{\top} \sqrt{E_W} \left(\gamma - \frac{\lambda_1 \lambda_2}{\gamma}\right) = 0.
\]
Thus, we have \( \gamma = \sqrt{\lambda_1 \lambda_2} \) and
\[
T = -\sqrt{\frac{\lambda_1 E_W}{\lambda_2 E_H}} S^{\top} = - S^\top.
\]
Since \( T \) and \( S^\top \) are both orthogonal matrices, the scalar factor \(\sqrt{(\lambda_1 E_W)/(\lambda_2 E_H)}\) must be 1.
Substituting into~\cref{eq:W_H_expanded} gives
\[
  W = U S \left(\sqrt{ \frac{E_W}{m-1}} I_{m-1}\right) Q^{\top}
    = \sqrt{ \frac{E_W}{m-1}} U S Q^{\top},
\]
\[
  H = - Q \left(\sqrt{ \frac{E_H}{m-1}} I_{m-1}\right) S^{\top}V^{\top}
    = -\sqrt{ \frac{E_H}{m-1}} Q S^{\top} V^{\top}.
\]
Set \(\widetilde Q = Q S^{\top}\).  Then
\(\widetilde Q^{\top}\widetilde Q = S Q^{\top}Q S^{\top}=I\), so
\(\widetilde Q\) is orthonormal. 
Renaming \(\widetilde Q\) back to \(Q\) gives exactly
\[
  W=\sqrt{\frac{E_W}{m-1}} UQ^\top,
  \quad
  H=-\sqrt{\frac{E_H}{m-1}} QV^\top.
\]
The proof is now complete.
\end{proof}

\begin{theorem}[Structure of Optimal Embeddings for the Symmetric Group]
\label{thm:embedding_products_structure}
Under the conditions of~\cref{thm:single_block_perm}, the optimal \(W^* \in \mathbb{R}^{m\times d}\) and \(H^* \in \mathbb{R}^{d\times m!}\) and their products satisfy
\[
  W^*(W^*)^\top=\frac{E_W}{m-1} \left(I_m-\frac{1}{m}\boldsymbol{1}_m\boldsymbol{1}_m^\top\right),\qquad
  (H^*)^\top H^*=\frac{E_H}{(m-1)\gamma^2} C^\top C,\qquad
  W^*H^*=-kC.
\]
Here the rows of \(W^*\) form a simplex ETF, and the context embedding Gram is a scalar multiple of \(C^\top C\). Finally, the optimal logit matrix satisfies the matrix identity \(W^*H^*=-kC\).
\end{theorem}

\begin{proof}[Proof of~\cref{thm:embedding_products_structure}]
From~\cref{thm:single_block_perm}, the optimal forms are given by \( W = \sqrt{E_W/(m-1)} U Q^\top \) and \( H = -\sqrt{E_H/(m-1)} Q V^\top \). We derive the structure of their products based on these forms.
For \(WW^\top\):
\[
WW^\top = \left( \sqrt{\frac{E_W}{m-1}} U Q^\top \right) \left( \sqrt{\frac{E_W}{m-1}} U Q^\top \right)^\top = \frac{E_W}{m-1} U Q^\top Q U^\top = \frac{E_W}{m-1} U U^\top. \]
We use the SVD of \(C = U\Sigma V^\top\), where \(\Sigma = \gamma I_{m-1}\). Given that \( C \) has zero column sums, we have:
    \[
    \boldsymbol{1}_m^\top U \Sigma V^\top = \boldsymbol{0}_{m!}^\top.
    \]
    Multiplying both sides by \( V \) gives \( \boldsymbol{1}_m^\top U \Sigma = \boldsymbol{0}_{m-1}^\top \).
    Since \( \Sigma \) is a diagonal matrix with \( m-1 \) nonnegative entries, the above equation implies
    \( \boldsymbol{1}_m^\top U = \boldsymbol{0}_{m-1}^\top \).
This implies that \(\boldsymbol{1}_m\) is orthogonal to every column of \(U\). Since \(U^\top U=I_{m-1}\), the \(m-1\) columns of \(U\in\mathbb{R}^{m\times(m-1)}\) form an orthonormal basis for a \((m-1)\)-dimensional subspace, which must be \(\boldsymbol{1}_m^\perp\). Therefore \(UU^\top\) is the orthogonal projector onto \(\boldsymbol{1}_m^\perp\), 
    \[
    UU^\top = I_m - \frac{1}{m} \boldsymbol{1}_m \boldsymbol{1}_m^\top.
    \]
Substituting this into the expression for \( WW^\top \) yields
\[
WW^\top = \frac{E_W}{m-1} \left( I_m - \frac{1}{m} \boldsymbol{1}_m \boldsymbol{1}_m^\top \right ).
\]
For \(H^\top H\):
\[
H^\top H = \left( -\sqrt{\frac{E_H}{m-1}} Q V^\top \right)^\top \left( -\sqrt{\frac{E_H}{m-1}} Q V^\top \right) = \frac{E_H}{m-1} V Q^\top Q V^\top = \frac{E_H}{m-1} VV^\top.
\]
To relate this to \(C\), we use the SVD of \(C = U\Sigma V^\top\), where \(\Sigma = \gamma I_{m-1}\). From this, we have \(C^\top C = \gamma^2 VV^\top\).
Therefore, \(VV^\top = C^\top C/\gamma^2\). Substitution into the expression for \(H^\top H\) gives
\[ H^\top H = \frac{E_H}{(m-1)\gamma^2} C^\top C. \]
For \( WH \):
\[
WH = \left( \sqrt{\frac{E_W}{m-1}} U Q^\top \right ) \left( -\sqrt{\frac{E_H}{m-1}} Q V^\top \right ) = -\frac{\sqrt{E_W E_H}}{m-1} U Q^\top Q V^\top = -\frac{\sqrt{E_W E_H}}{m-1} U V^\top.
\]
From the SVD of \(C = \gamma U V^\top\), we have \( UV^\top = C/\gamma \). Hence
\[
WH = -\frac{\sqrt{E_W E_H}}{(m-1)\gamma} C.
\]
As shown in~\cref{thm:single_block_perm}, the parameter \(k\) in the system for \(\boldsymbol{\alpha}\) is precisely \(k = \sqrt{E_W E_H}/((m-1)\gamma)\). Thus, we can write \(WH = -kC\).
This completes the proof.
\end{proof}

The preceding result extends from \(S_m\) to any 2-transitive group action on \(m\) elements.

\begin{definition}[$2$-Transitive Group Action]
Let \(\mathcal G\) be a group acting on a set \(\mathcal X\) with \(|\mathcal X| \ge 2\). The action is $2$-transitive if, for any two ordered pairs of distinct elements \((x_1,x_2)\) and \((y_1,y_2)\) in \(\mathcal X\), there exists \(g\in\mathcal G\) such that \(g\circ x_1=y_1\) and \(g\circ x_2=y_2\). When \(\mathcal G\le S_m\) is a permutation group, we say that \(\mathcal G\) is $2$-transitive if its natural action on \(\{1,\ldots,m\}\) is $2$-transitive.
\end{definition}  

The symmetric group \(S_m\) (for \(m\ge2\)) is 2-transitive: given two ordered pairs of distinct elements \((x_1,x_2)\) and \((y_1,y_2)\) in a \(m\)-element set \(\mathcal X\), define a permutation \(\tau\in S_m\) by \(\tau(x_1)=y_1\) and \(\tau(x_2)=y_2\); then extend \(\tau\) arbitrarily to a bijection between \(\mathcal X\setminus\{x_1,x_2\}\) and \(\mathcal X\setminus\{y_1,y_2\}\), which is always possible, hence \(S_m\) acts 2-transitively. 

\begin{theorem}[Optimal Embeddings and Their Structure for 2-Transitive Groups]
\label{thm:2_transitive_optimal_embeddings}
Let \(\boldsymbol{y}\) be a nonuniform target distribution, and let \(Y\) be the target distribution matrix generated by a finite 2-transitive permutation group \(\mathcal G\).
For \(d\ge m\), under~\cref{ass:active_constraint},
\(W^*\in\mathbb{R}^{m\times d}\) and \(H^*\in\mathbb{R}^{d\times |\mathcal G|}\) solve the constrained
problem~\cref{eq:obj_function_constrained} if and only if
\[
  W^*=\sqrt{\frac{E_W}{m-1}} UQ^\top,\qquad
  H^*=-\sqrt{\frac{E_H}{m-1}} QV^\top,
\]
for any partial orthogonal matrix \(Q\in\mathbb{R}^{d\times(m-1)}\) with \(Q^\top Q=I_{m-1}\), where
\begin{itemize}
\item \(C=A-Y\), where \(A\in\mathbb{R}^{m\times |\mathcal G|}\) is obtained by the action of \(\mathcal G\) on \(\boldsymbol{\alpha}=(\alpha_1,\ldots,\alpha_m)\) where
\( A=\big[ g\circ\boldsymbol{\alpha}\mid g\in \mathcal G \big],
\)
with the same enumeration of \(\mathcal G\) as used to construct \(Y\) from \(\boldsymbol{y}\).
Here \(\boldsymbol{\alpha}\) is the unique solution to
\[
k(\alpha_i-y_i)+\log\alpha_i=\bar{\ell}\ (i=1,\dots,m),\quad \sum_{i=1}^m\alpha_i=1,\quad
k=\sqrt{\frac{E_W E_H}{|\mathcal G|(m-1)\sum_{i=1}^m(\alpha_i-y_i)^2}},
\]
with \(\bar{\ell}=(1/m)\sum_{r=1}^m \log \alpha_r\).
\item \(C=U\Sigma V^\top\) is the rank-\((m-1)\) SVD of \(C\), with \(U\in\mathbb{R}^{m\times(m-1)}\), \(V\in\mathbb{R}^{|\mathcal G|\times(m-1)}\), and \(\Sigma=\gamma I_{m-1}\).
\item  \( \gamma = \sqrt{|\mathcal G|/(m-1)}\|\boldsymbol{\alpha}-\boldsymbol{y}\|_2\). Equivalently, \(k=\sqrt{E_W E_H}/((m-1)\gamma)\).
\end{itemize}
Moreover,
\[
  W^*(W^*)^\top=\frac{E_W}{m-1} \left(I_m-\frac{1}{m}\boldsymbol{1}_m\boldsymbol{1}_m^\top\right),\qquad
  (H^*)^\top H^*=\frac{E_H}{(m-1)\gamma^2} C^\top C,\qquad
  W^*H^*=-k C.
\]
\end{theorem}

See~\cref{pf:2_transitive_optimal_embeddings} for the proof of~\cref{thm:2_transitive_optimal_embeddings}.

\section{Symmetry Transfer in Multi-Block Models}\label{sec:multi_block}

In this section, we extend the single-block framework to target matrices formed by concatenating multiple group-orbit blocks. We consider this case since, in real-world corpora, a fixed word group appears under many contexts that induce different base target distributions. Each base distribution generates its own orbit under the group action and therefore contributes a separate block of columns to the overall target matrix \(Y\). Thus \(Y\) is formed by concatenating several group orbits, each obtained by applying a group action to a base probability vector.

\begin{definition}[Multi-Block Group-Orbit Target Distribution Matrix]
\label{def:multi_block_group_orbit_Y}
Let \(n_{\mathrm{blk}}\) be the number of blocks, and let \(\mathcal G_1, \mathcal G_2, \ldots, \mathcal G_{n_{\mathrm{blk}}}\) be finite groups, where each group \(\mathcal G_i\) acts on \(\mathbb{R}^m\); write \(g\circ \boldsymbol{x}\) for this action.
Let \(\boldsymbol{y}_1, \boldsymbol{y}_2, \ldots, \boldsymbol{y}_{n_{\mathrm{blk}}}\) be base probability vectors in \(\mathbb{R}^m\).
A multi-block group-orbit target distribution matrix \(Y \in \mathbb{R}^{m \times n}\) is defined as the horizontal concatenation of \(n_{\mathrm{blk}}\) blocks:
\[
Y = [Y_1 \mid Y_2 \mid \cdots \mid Y_{n_{\mathrm{blk}}}],
\]
where \(n = \sum_{i=1}^{n_{\mathrm{blk}}} n_i\), and each block \(Y_i \in \mathbb{R}^{m \times n_i}\) is constructed from its respective group \(\mathcal G_i\) and base vector \(\boldsymbol{y}_i\) according to the following properties:
\begin{itemize}
    \item Closure: The set of columns forming block \(Y_i\) consists exclusively of vectors generated by the action of the group \(\mathcal G_i\) on the base vector \(\boldsymbol{y}_i\). That is, every column in \(Y_i\) is of the form \(g\circ \boldsymbol{y}_i\) for some \(g \in \mathcal G_i\). Furthermore, if \(\boldsymbol{y}\) is any column vector present in block \(Y_i\), then for any group element \(g' \in \mathcal G_i\), the transformed vector \(g'\circ \boldsymbol{y}\) must also be a column within the same block \(Y_i\) (it could be an identical column or a column at a different position within \(Y_i\)).
    \item Balance: Each distinct vector appears an equal number of times as a column within block \(Y_i\).
\end{itemize}
\end{definition}

\paragraph{Motivating example.}
Consider the animal set \{\texttt{dog}, \texttt{cat}, \texttt{rabbit}\} and the four contexts:
\begin{itemize}
  \item \(\mathbb{P}(\texttt{dog})=1\): ``\texttt{He barks and wags his tail. My pet is a \textvisiblespace}''
  \item \(\mathbb{P}(\texttt{cat})=1\): ``\texttt{She purrs and loves milk. My pet is a \textvisiblespace}''
  \item \(\mathbb{P}(\texttt{rabbit})=1\): ``\texttt{He hops around the garden. My pet is a \textvisiblespace}''
  \item \(\mathbb{P}(\texttt{dog})=\mathbb{P}(\texttt{cat})=\mathbb{P}(\texttt{rabbit})=1/3\): ``\texttt{I have a pet at home. My pet could be a \textvisiblespace}''
\end{itemize}
Collecting the animal rows gives the \(3\times4\) submatrix
\[
\widetilde{Y}_{\mathrm{animals}}
=\begin{bmatrix}
1   & 0   & 0   & \frac{1}{3} \\[4pt]
0   & 1   & 0   & \frac{1}{3} \\[4pt]
0   & 0   & 1   & \frac{1}{3}
\end{bmatrix}.
\]
Here, the first three columns form Block 1: an \(S_3\)-orbit generated by the base vector \(e_{\texttt{dog}}\) (its permutations yield the one-hot vectors for \texttt{cat} and \texttt{rabbit}). The fourth column forms Block 2: the uniform vector \((1/3)\boldsymbol{1}_3\), which is a fixed point of the \(S_3\) action (orbit size one). Thus, \(Y\) naturally decomposes into multiple blocks, each with a base distribution. 

\begin{theorem}[Optimal Embeddings and Their Structure for Multi-Block Cyclic-Shift Groups]
\label{thm:multi_block_cyclic}
Let \( Y = [Y_1 | Y_2 | \cdots | Y_{n_{\mathrm{blk}}}] \in \mathbb{R}^{m \times n_{\mathrm{blk}}m} \) be a multi-block group-orbit target distribution matrix as defined in~\cref{def:multi_block_group_orbit_Y} where each group \( \mathcal G_i \) is a copy of \(C_m\) acting by cyclic shifts.
Here \(C_m=\langle\Pi\rangle=\{I_m,\Pi,\ldots,\Pi^{m-1}\}\) is the cyclic group of order \(m\) generated by the cyclic-shift matrix \(\Pi\).
Suppose at least one of the underlying probability vectors \(\boldsymbol{y}_i\) is nonuniform. For \(d \ge m\), under~\cref{ass:active_constraint}, any minimizer pair \((W^*,H^*)\) of the optimization problem must satisfy the following properties.
The product logit matrix \(Z^*=W^*H^*\) is a block-structured matrix formed by concatenating \(n_{\mathrm{blk}}\) circulant blocks, \(Z^* = [Z_1^* | Z_2^* | \cdots | Z_{n_{\mathrm{blk}}}^*]\), where each block is defined by the cyclic shifts of an optimal generating vector:
\[ (Z_i^*)_{:,j}=\Pi^{j-1}\boldsymbol z_i^*,\qquad j=1,\ldots,m. \]
The set of optimal generating vectors \(\{\boldsymbol{z}_{1}^*, \ldots, \boldsymbol{z}_{n_{\mathrm{blk}}}^*\}\) is obtained as a solution to the convex optimization problem:
\[
\begin{aligned}
\min_{\boldsymbol{z}_{1}, \ldots, \boldsymbol{z}_{n_{\mathrm{blk}}} \in \mathbb{R}^m} \quad & \sum_{i=1}^{n_{\mathrm{blk}}} \mathcal L \left(\sigma(\boldsymbol{z}_{i}), \boldsymbol{y}_i\right) \\
\text{s.t.} \quad & \left\| [ [\Pi^{j-1} \boldsymbol{z}_{1}]_{j=1}^{m} | \cdots | [\Pi^{j-1} \boldsymbol{z}_{n_{\mathrm{blk}}}]_{j=1}^{m} ] \right\|_* \le \sqrt{E_W E_H}.
\end{aligned}
\]
Furthermore,
\[ W^*(W^*)^\top = \sqrt{\frac{E_W}{E_H}} (Z^* (Z^*)^\top)^{1/2}, \]
which is a circulant matrix.
\end{theorem}

See~\cref{pf:multi_block_cyclic} for the proof of~\cref{thm:multi_block_cyclic}.

\begin{theorem}[Optimal Embeddings and Their Structure for Multi-Block 2-Transitive Groups]
\label{thm:2_transitive_multiblock}
Let \( Y \) be a multi-block group-orbit target distribution matrix as defined in~\cref{def:multi_block_group_orbit_Y}, and assume that each group \( \mathcal G_i \) is 2-transitive.
Suppose at least one of the underlying probability vectors \(\boldsymbol{y}_i\) is nonuniform, and for each block, the columns of \(Y_i\) comprise the full set of \(n_i = |\mathcal G_i|\) vectors \(\{g \circ \boldsymbol{y}_i \mid g \in \mathcal G_i\}\).
For each block, introduce a probability vector $\boldsymbol{\alpha}_i=(\alpha_{i1},\dots,\alpha_{im})^{\top}\in\Delta^{m-1}$.
The vectors $\{\boldsymbol{\alpha}_i\}_{i=1}^{n_{\mathrm{blk}}}$ are jointly determined
as the unique solution of the system:
\begin{equation}
\label{eq:alpha_system_2trans}
  \begin{cases}
    \displaystyle k\left(\alpha_{i\ell}-y_{i\ell}\right)
      +\log\alpha_{i\ell}
     =\displaystyle\frac1m\sum_{r=1}^{m}\log\alpha_{ir},
    &  \ell=1,\dots,m,   i=1,\dots,n_{\mathrm{blk}},\\[1em]
    \displaystyle\sum_{\ell=1}^{m}\alpha_{i\ell}=1,
    &  i=1,\dots,n_{\mathrm{blk}},\\[1em]
    \displaystyle \gamma = \sqrt{\frac{1}{m-1}\sum_{j=1}^{n_{\mathrm{blk}}}|\mathcal G_j|\sum_{s=1}^m (\alpha_{js}-y_{js})^2}, \\[1em]
    \displaystyle k = \frac{\sqrt{E_W E_H}}{(m-1)\gamma}.
  \end{cases}
\end{equation}
Denote \(A_i=[g \circ \boldsymbol{\alpha}_i\mid g\in \mathcal G_i]\), \(A=[A_1|\dots|A_{n_{\mathrm{blk}}}]\), and \(C:=A-Y\).
Then, the matrix $C$ has the rank-$(m-1)$ singular-value decomposition
\[
  C
  =U \gamma I_{m-1} V^{\top},
  \quad
  U\in\mathbb{R}^{m\times(m-1)}, \quad
  V\in\mathbb{R}^{n\times(m-1)}.
\]
Let $Q\in\mathbb{R}^{d\times(m-1)}$ be any partial orthogonal matrix
($Q^{\top}Q=I_{m-1}$).
Then, for \(d \ge m\),  $W^* \in \mathbb{R}^{m \times d}$ and $H^* \in \mathbb{R}^{d \times n}$ solve the optimization problem~\cref{eq:obj_function_constrained} if and only if
\[
  W^*=\sqrt{\frac{E_W}{m-1}} UQ^{\top},
  \qquad
  H^*=-\sqrt{\frac{E_H}{m-1}} QV^{\top}.
\]
Furthermore, the following products hold:
\begin{equation}\label{eq:opt_products_2trans}
  W^*(W^*)^{\top}
    =\frac{E_W}{m-1}
        \left(I_{m}-\frac1m \boldsymbol1_{m}\boldsymbol1_{m}^{\top}\right),\qquad
  (H^*)^{\top}H^*
    =\frac{E_H}{(m-1)\gamma^{2}} C^{\top}C,\qquad
  W^*H^* =-k C.
\end{equation}
\end{theorem}

See~\cref{pf:2_transitive_multiblock} for the proof of~\cref{thm:2_transitive_multiblock}.

\paragraph{Remark.}
\Cref{thm:2_transitive_multiblock} assumes that for each block, the number of columns \( n_i \) equals the group order \( |\mathcal G_i| \). 
We have not explicitly detailed the cases where, for example, a block \(Y_i\) might be formed by taking only the distinct vectors in the orbit of \(\boldsymbol{y}_i\) under \(\mathcal G_i\), or where \(Y_i\) represents multiple full repetitions of all transformations \(\{g\circ \boldsymbol{y}_i \mid g \in \mathcal G_i\}\).
However, we argue that these variations do not fundamentally alter the structural conclusions of the main results (e.g., that \(C=U\gamma I_{m-1}V^\top\), that \(\Sigma_{WH}=k_0 I_{m-1}\), and the resulting forms for \(W, H, WW^\top, H^\top H, WH\)). Such changes in the detailed construction of \(n_i\) columns for block \(i\) would primarily affect the scaling constant \(\lambda'_i\) associated with \(C_iC_i^\top\), and consequently the overall singular value \(\gamma\) of the total matrix \(C\). This change in \(\gamma\) would then proportionally adjust the parameter \(k\) that appears in the system defining \(\boldsymbol{\alpha}_i\) and in the expression \(WH = -kC\). The structural properties would remain the same. The proofs can be readily adapted for these cases by correctly accounting for the scaling factor in \(\gamma\) that arises from the precise number of distinct elements and their repetitions within each block \(Y_i\).

\section{Experiments}
\label{sec:exp}

In this section we report experiments on open-source LLMs to measure the extent to which the predicted geometric patterns are visible in these models. We use three models: GPT-OSS-20B~\citep{openai2025gptoss120bgptoss20bmodel}, Mistral-7B-Instruct-v0.3~\citep{jiang2023mistral7b}, and RWKV7-7.2B~\citep{peng2025rwkv}. Additional experiments appear in~\cref{app:more_exp}. We plot the normalized Gram matrices of the output projections \(WW^\top\) and the context embeddings \(H^\top H\), as well as the predicted next-token probability matrices for our word sets of interest (obtained by applying \(\sigma\) to the logit matrix \(Z=WH\), so the idealized theory predicts the same structural symmetries as the logits). 

For every Gram matrix we first mean-center the vectors (subtract the coordinate-wise mean) and rescale all vectors by a common factor so that the mean \(\ell_2\) norm equals one. We then report two scale-invariant diagnostics. The ETF distance \(\delta_{\mathrm{ETF}}\) (\cref{def:dis_etf}) measures how close \(G\) is to a simplex ETF: we choose the scalar that best fits \(G\) to the canonical simplex ETF in Frobenius norm and report the resulting relative residual. The circulant distance \(\delta_{\mathrm{circ}}\) (\cref{def:circ-distance}) measures how far \(G\) is from the subspace of circulant matrices: we project \(G\) onto that subspace and report the relative residual. Both metrics are invariant to multiplying \(G\) by a positive scalar, and smaller values indicate closer agreement with the predicted geometry. Full formulas and properties appear in~\cref{app:measure}.

\subsection{Output Projections}

\paragraph{Simplex ETF pattern.}
For exchangeable semantic categories, the rows of the output projection matrix are expected to approximate a simplex ETF (\cref{thm:single_block_perm}): equal norms on the diagonal and nearly constant off-diagonals. The kitchen utensils set \{\texttt{fork}, \texttt{knife}, \texttt{plate}, \texttt{cup}, \texttt{glass}\} in~\cref{fig:word-kitchen} is compared with this pattern using the reported \(\delta_{\mathrm{ETF}}\) values. 

\paragraph{Circulant pattern.}
For the weekday set \{\texttt{Monday}, \texttt{Tuesday}, \texttt{Wednesday}, \texttt{Thursday}, \texttt{Friday}, \texttt{Saturday}, \texttt{Sunday}\}, \cref{thm:single_block_cyclic} predicts a circulant output projection Gram: each day is a successor of another, so correlations should repeat with a fixed offset and wrap around at the boundary. Concretely, \(\left(WW^\top\right)_{ij}\) should depend primarily on the day difference \((j-i)\pmod{7}\), producing a banded, near-constant diagonal structure. We assess this behavior across models using the \(\delta_{\mathrm{circ}}\) values in~\cref{fig:word-weekdays}.

\begin{figure}[htbp]
  \centering
  \captionsetup[subfigure]{justification=centering}
  \begin{subfigure}[t]{0.32\textwidth}
    \centering
    \includegraphics[width=\linewidth]{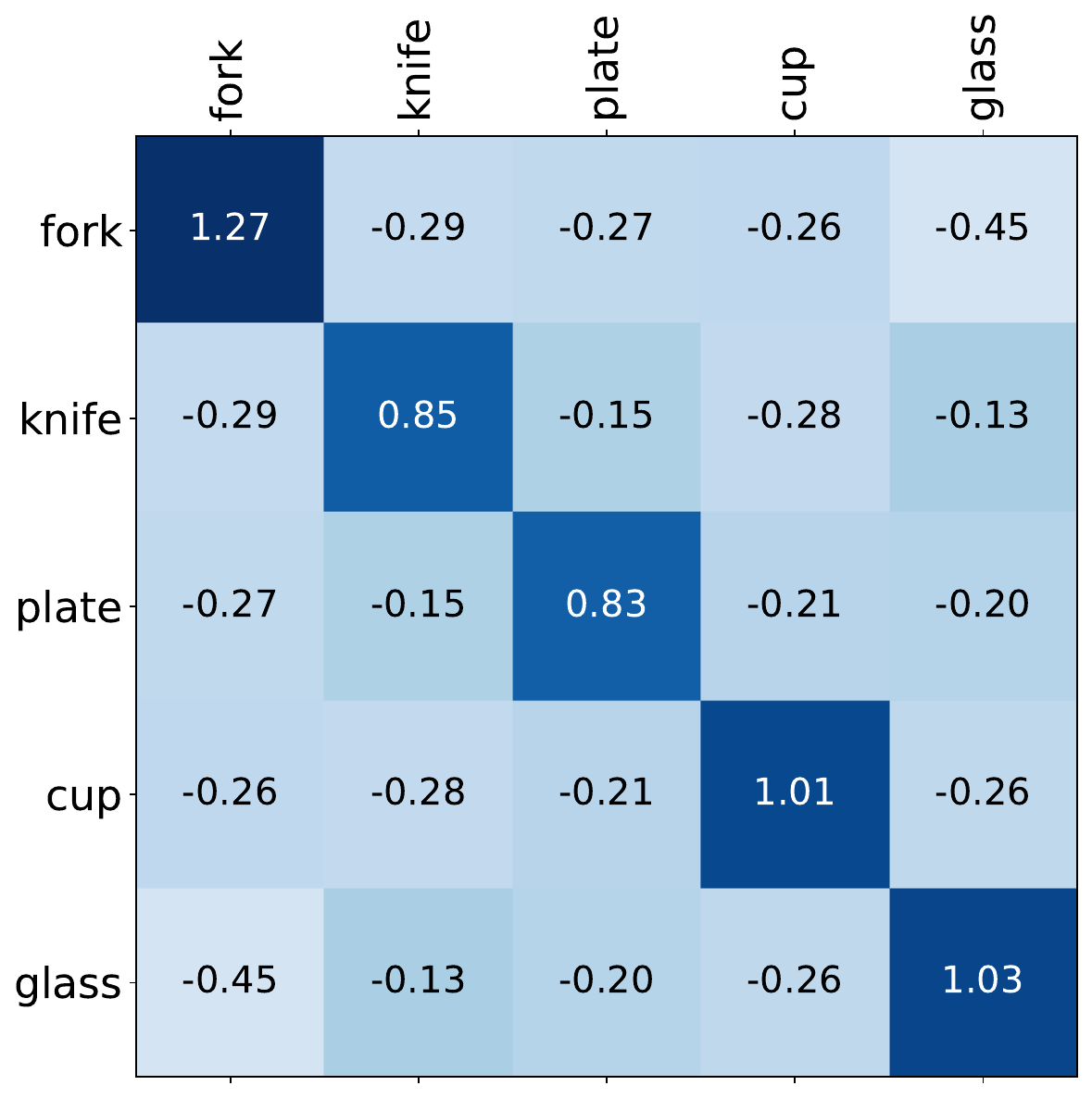}
    \caption{GPT-OSS-20B (\(\delta_{\mathrm{ETF}}=0.205\))}
  \end{subfigure}\hfill
  \begin{subfigure}[t]{0.32\textwidth}
    \centering
    \includegraphics[width=\linewidth]{figures/word/mistralai_Mistral-7B-Instruct-v0.3/Kitchen_Utensils.pdf}
    \caption{Mistral-7B-Instruct-v0.3 (\(\delta_{\mathrm{ETF}}=0.170\))}
  \end{subfigure}\hfill
  \begin{subfigure}[t]{0.32\textwidth}
    \centering
    \includegraphics[width=\linewidth]{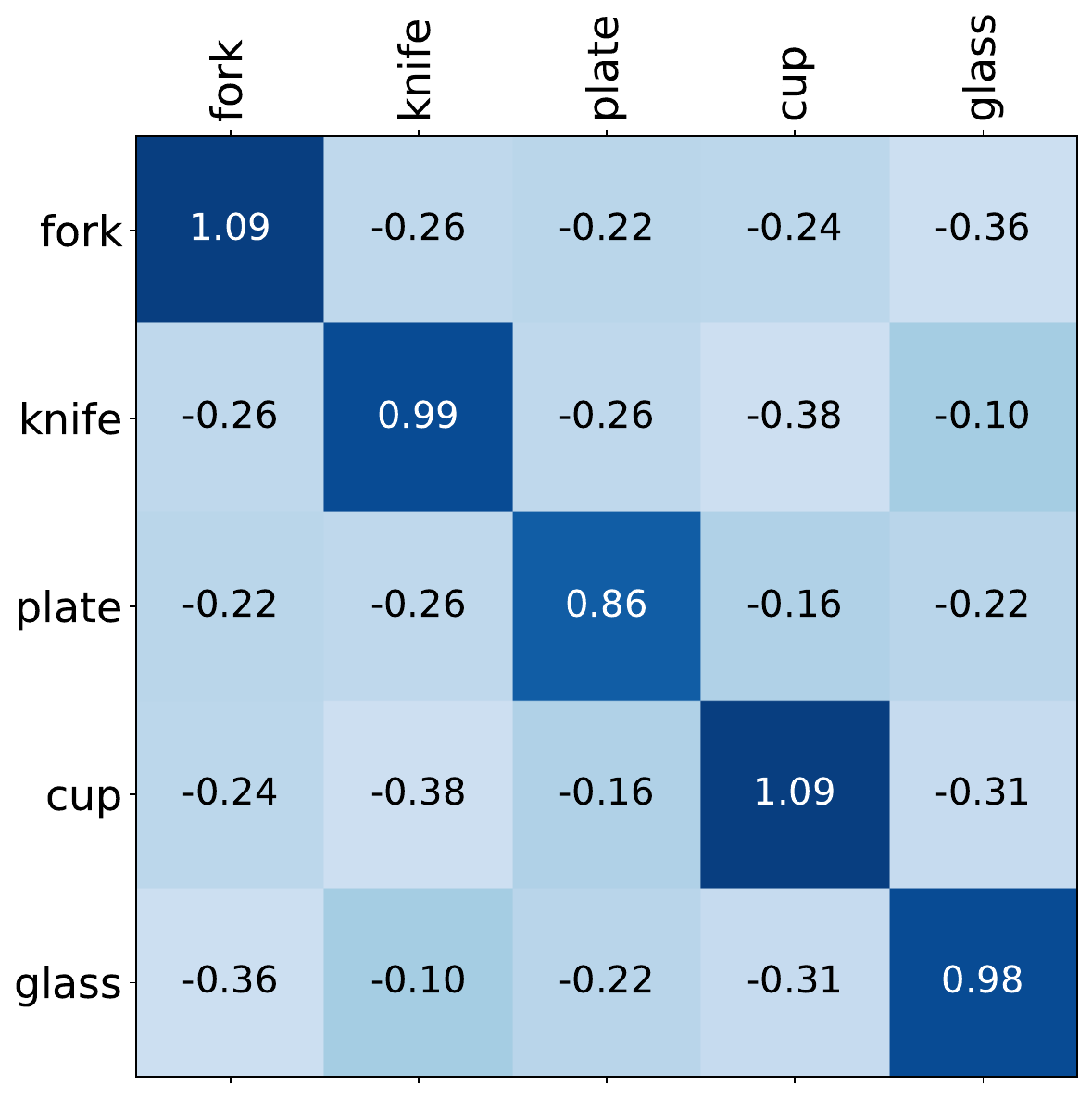}
    \caption{RWKV7-7.2B (\(\delta_{\mathrm{ETF}}=0.163\))}
  \end{subfigure}
  \caption{Output projection Grams for kitchen utensils.}
  \label{fig:word-kitchen}
\end{figure}

\begin{figure}[htbp]
  \centering
  \captionsetup[subfigure]{justification=centering}
  \begin{subfigure}[t]{0.32\textwidth}
    \centering
    \includegraphics[width=\linewidth]{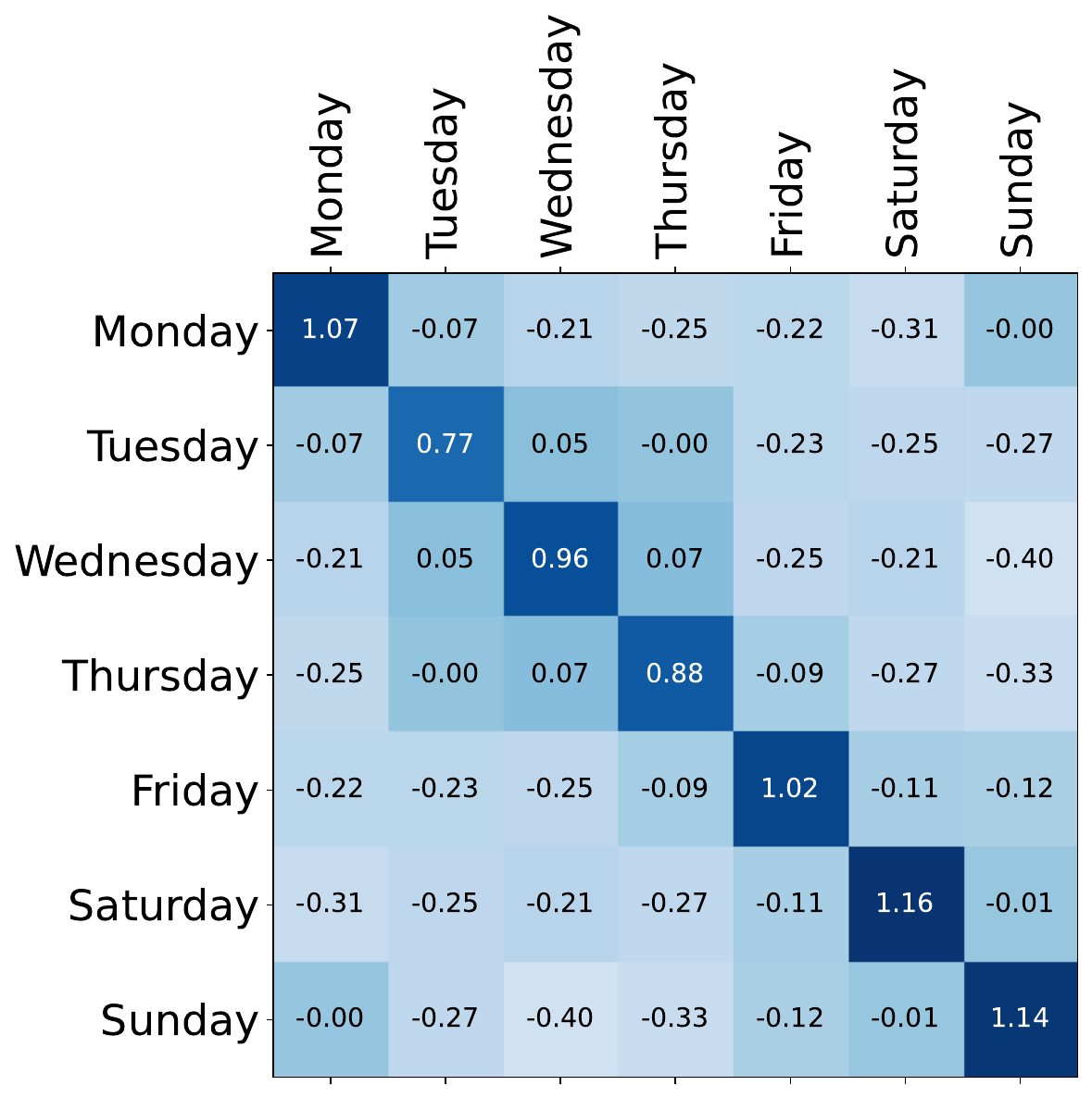}
    \caption{GPT-OSS-20B (\(\delta_{\mathrm{circ}}=0.202\))}
  \end{subfigure}\hfill
  \begin{subfigure}[t]{0.32\textwidth}
    \centering
    \includegraphics[width=\linewidth]{figures/word/mistralai_Mistral-7B-Instruct-v0.3/Days_of_the_Week.pdf}
    \caption{Mistral-7B-Instruct-v0.3 (\(\delta_{\mathrm{circ}}=0.319\))}
  \end{subfigure}\hfill
  \begin{subfigure}[t]{0.32\textwidth}
    \centering
    \includegraphics[width=\linewidth]{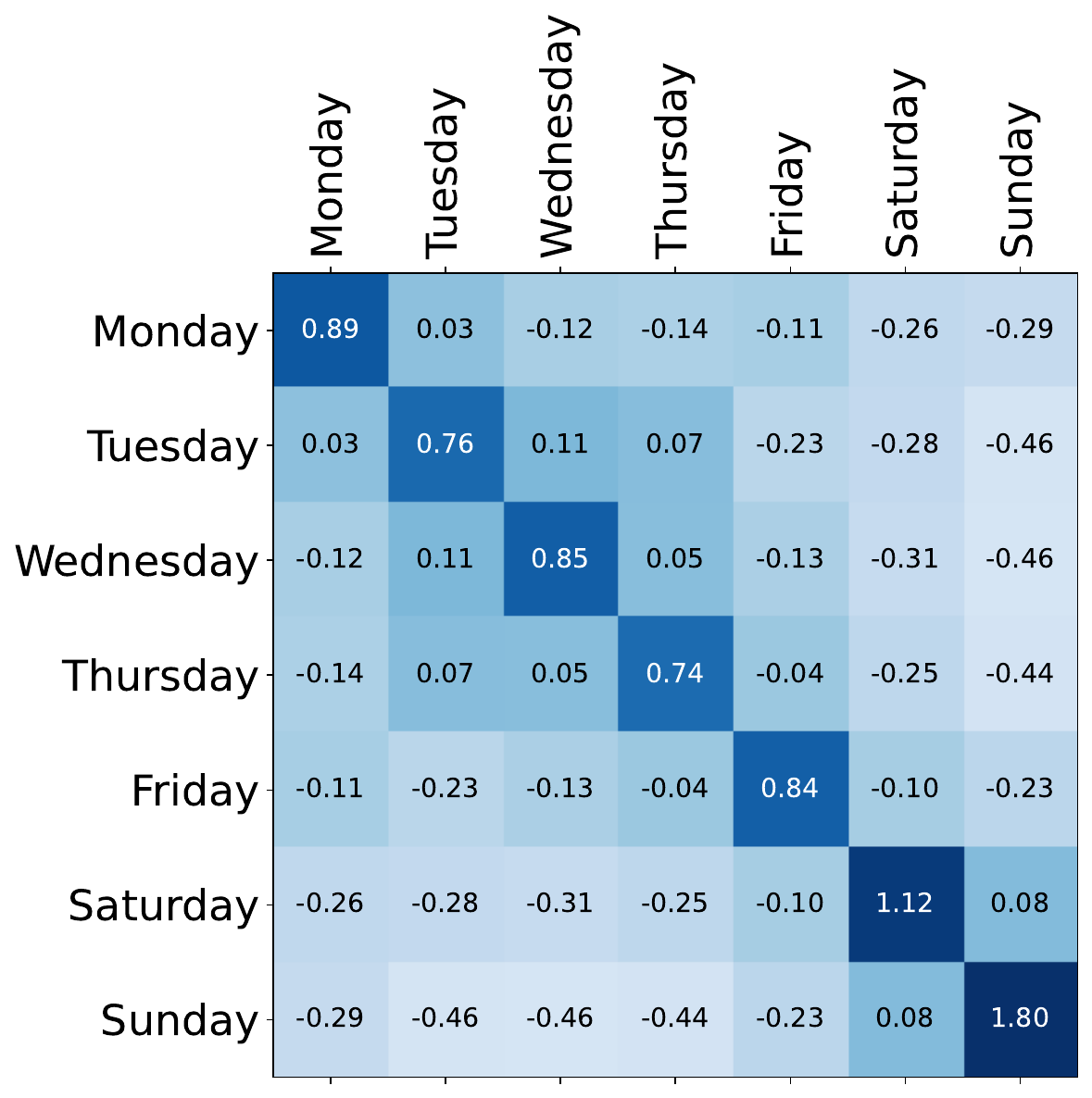}
    \caption{RWKV7-7.2B (\(\delta_{\mathrm{circ}}=0.396\))}
  \end{subfigure}
  \caption{Output projection Grams for weekdays.}
  \label{fig:word-weekdays}
\end{figure}

\subsection{Context Embeddings and Logits}

\paragraph{Permutation.}
For the permutation template ``\texttt{Among three primary colors, <TOKEN1> is a mix of <TOKEN2> and \textvisiblespace}'', the six contexts permute (\texttt{<TOKEN1>}, \texttt{<TOKEN2>}) as follows:
(\texttt{orange}, \texttt{red}), (\texttt{orange}, \texttt{yellow}), (\texttt{green}, \texttt{yellow}), (\texttt{green}, \texttt{blue}), (\texttt{purple}, \texttt{red}), (\texttt{purple}, \texttt{blue}).
\Cref{fig:perm-probs} shows the predicted next-token probability matrices; all three models place most of the probability mass on the intended color. \Cref{fig:perm-gram} shows the context embedding Grams. According to~\cref{thm:single_block_perm}, the theoretical context embedding Gram for the six contexts is
\[
\begin{bmatrix}
1&0&0&1&0&0\\
0&1&0&0&0&1\\
0&0&1&0&1&0\\
1&0&0&1&0&0\\
0&0&1&0&1&0\\
0&1&0&0&0&1
\end{bmatrix}.
\]
Here the entries labeled ``1'' and ``0'' indicate relative magnitudes: ``1'' denotes larger values and ``0'' denotes smaller values. This is the prediction in the idealized case where the model assigns probability one to the intended token. In practice, predictions are not deterministic, so the off-diagonal entries marked ``1'' may shrink in value; however, the prediction is that they should generally be \emph{larger than other off-diagonal entries marked ``0''}. The matrices in~\cref{fig:perm-gram} are consistent with this ordering.

\begin{figure}[htbp]
  \centering
  \captionsetup[subfigure]{justification=centering}
  \begin{subfigure}[t]{0.32\textwidth}
    \centering
    \includegraphics[width=\linewidth]{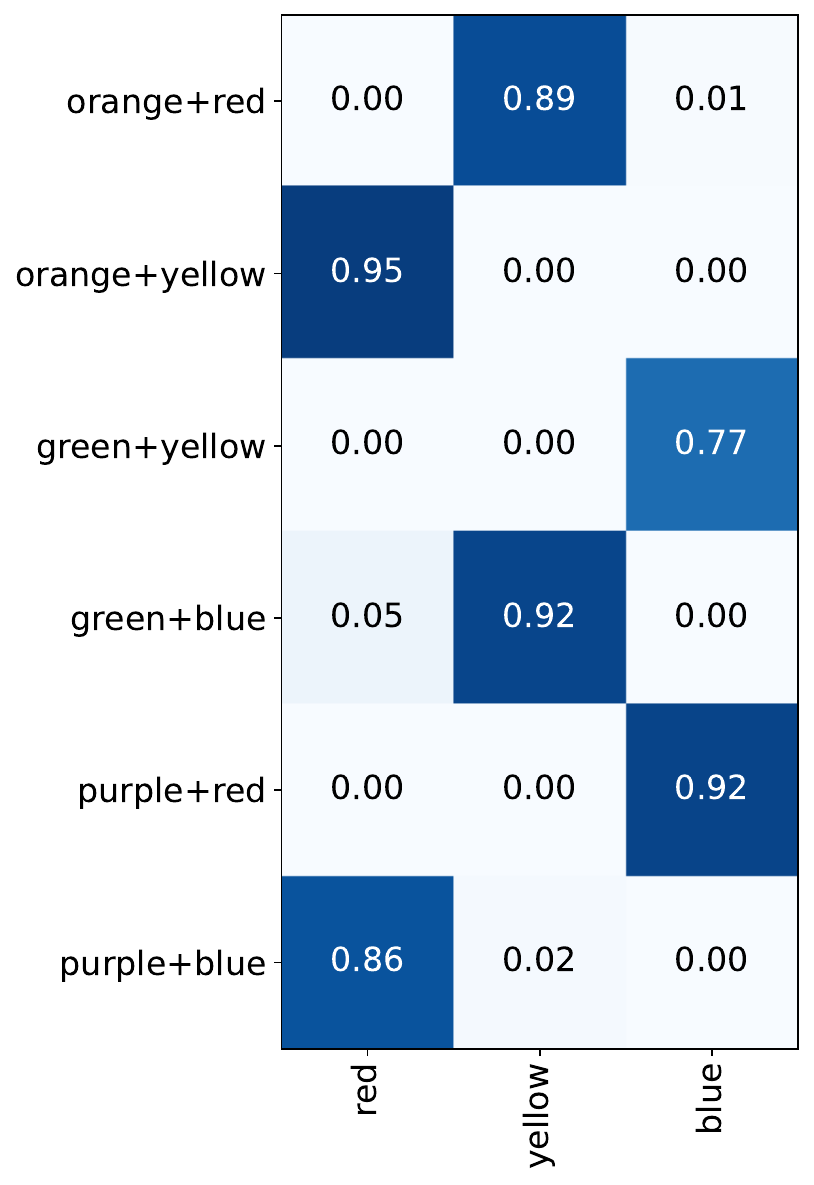}
    \caption{GPT-OSS-20B}
  \end{subfigure}\hfill
  \begin{subfigure}[t]{0.32\textwidth}
    \centering
    \includegraphics[width=\linewidth]{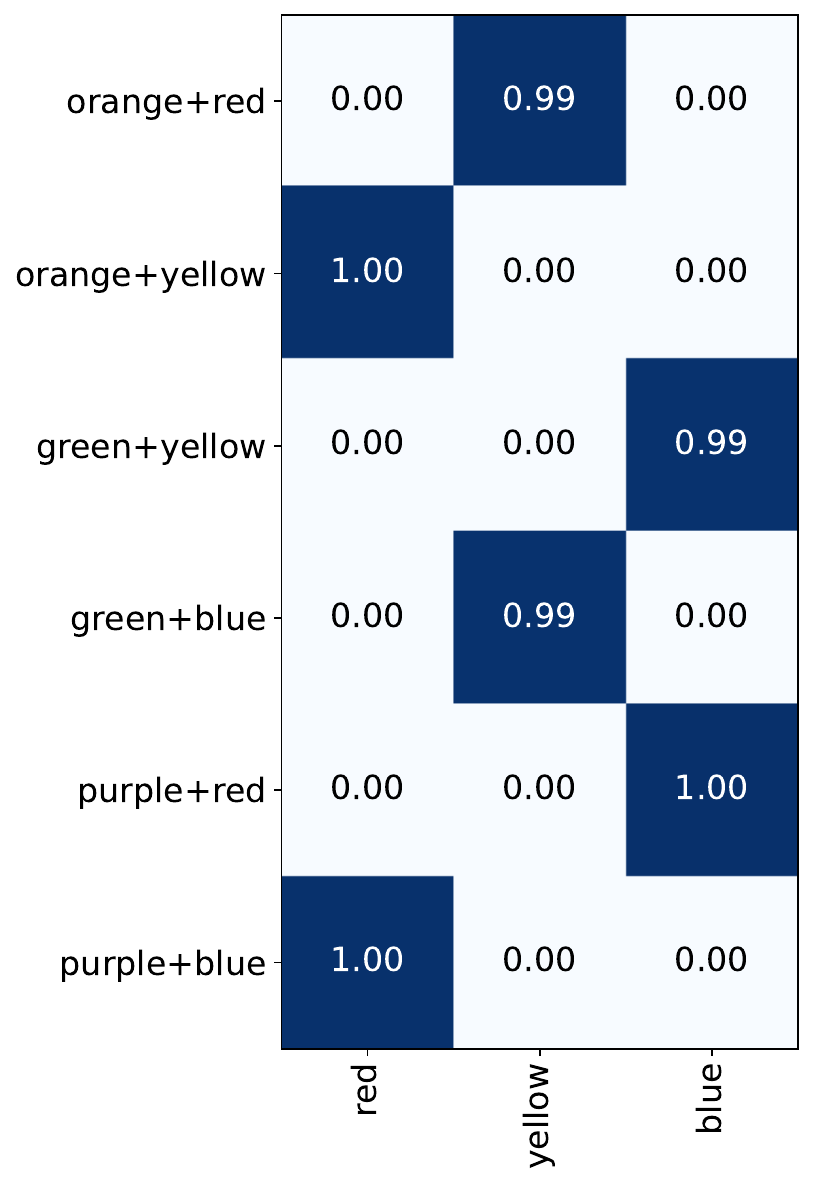}
    \caption{Mistral-7B-Instruct-v0.3}
  \end{subfigure}\hfill
  \begin{subfigure}[t]{0.32\textwidth}
    \centering
    \includegraphics[width=\linewidth]{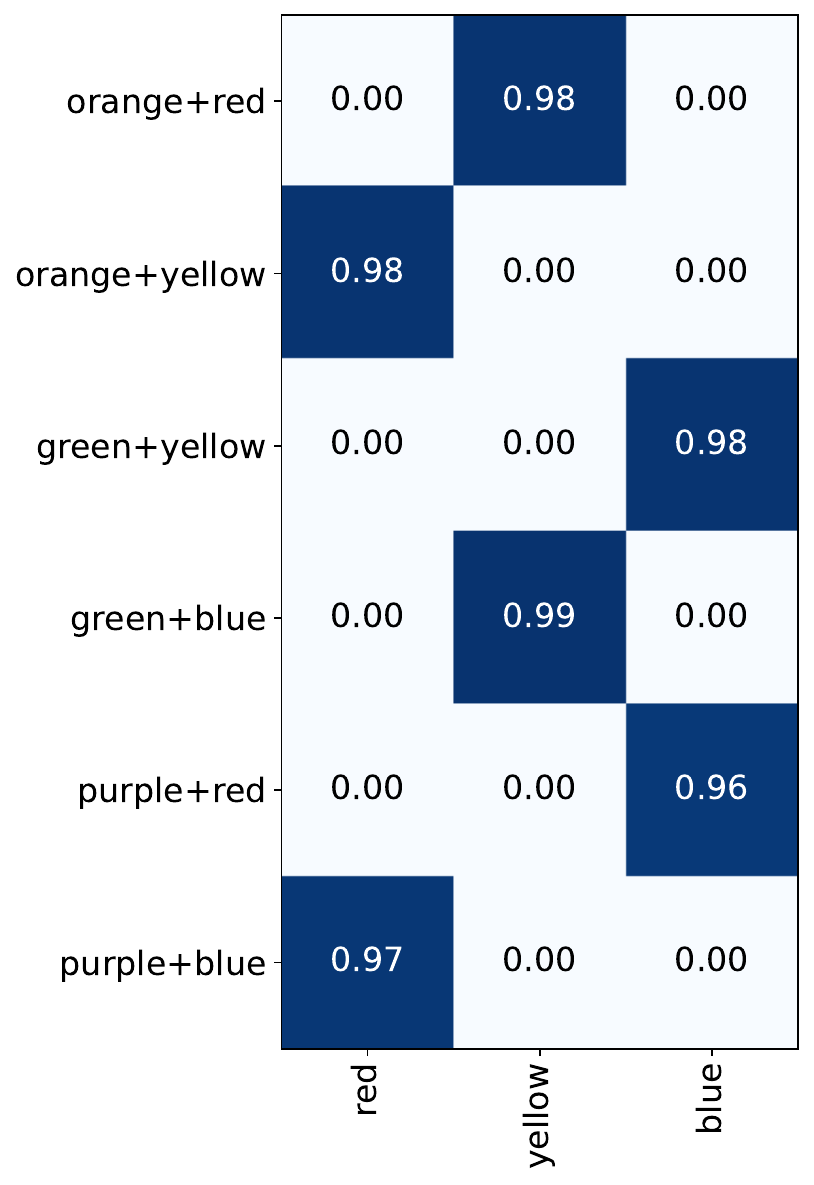}
    \caption{RWKV7-7.2B}
  \end{subfigure}
  \caption{Predicted next-token probability matrices for color-mix prompts. Row labels denote prompt contexts obtained by permuting \texttt{(<TOKEN1>,<TOKEN2>)}. Column labels are the candidate next tokens restricted to the primary colors. Each entry is the model's predicted probability for that token.}

  \label{fig:perm-probs}
\end{figure}

\begin{figure}[htbp]
  \centering
  \captionsetup[subfigure]{justification=centering}
  \begin{subfigure}[t]{0.32\textwidth}
    \centering
    \includegraphics[width=\linewidth]{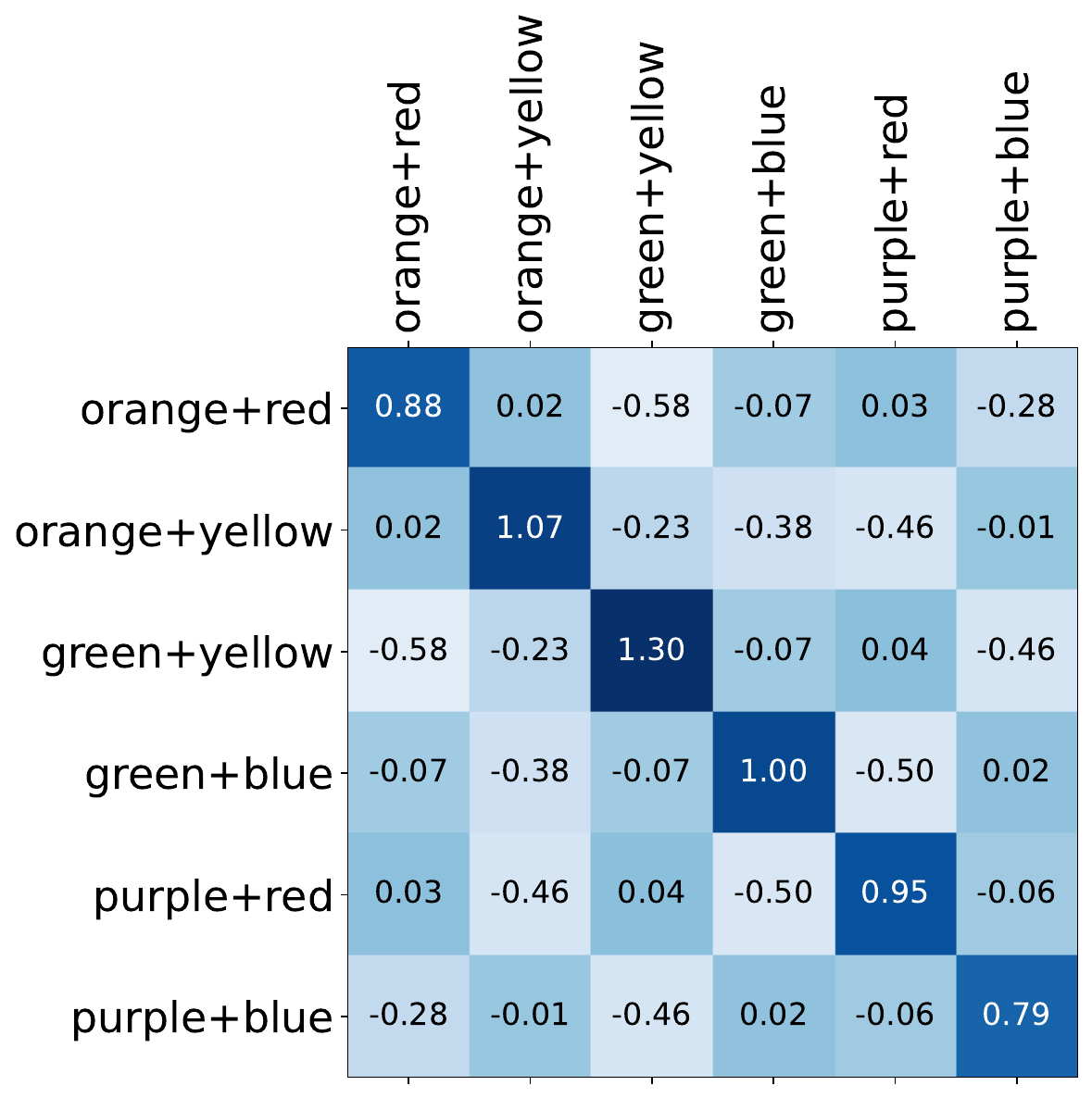}
    \caption{GPT-OSS-20B}
  \end{subfigure}\hfill
  \begin{subfigure}[t]{0.32\textwidth}
    \centering
    \includegraphics[width=\linewidth]{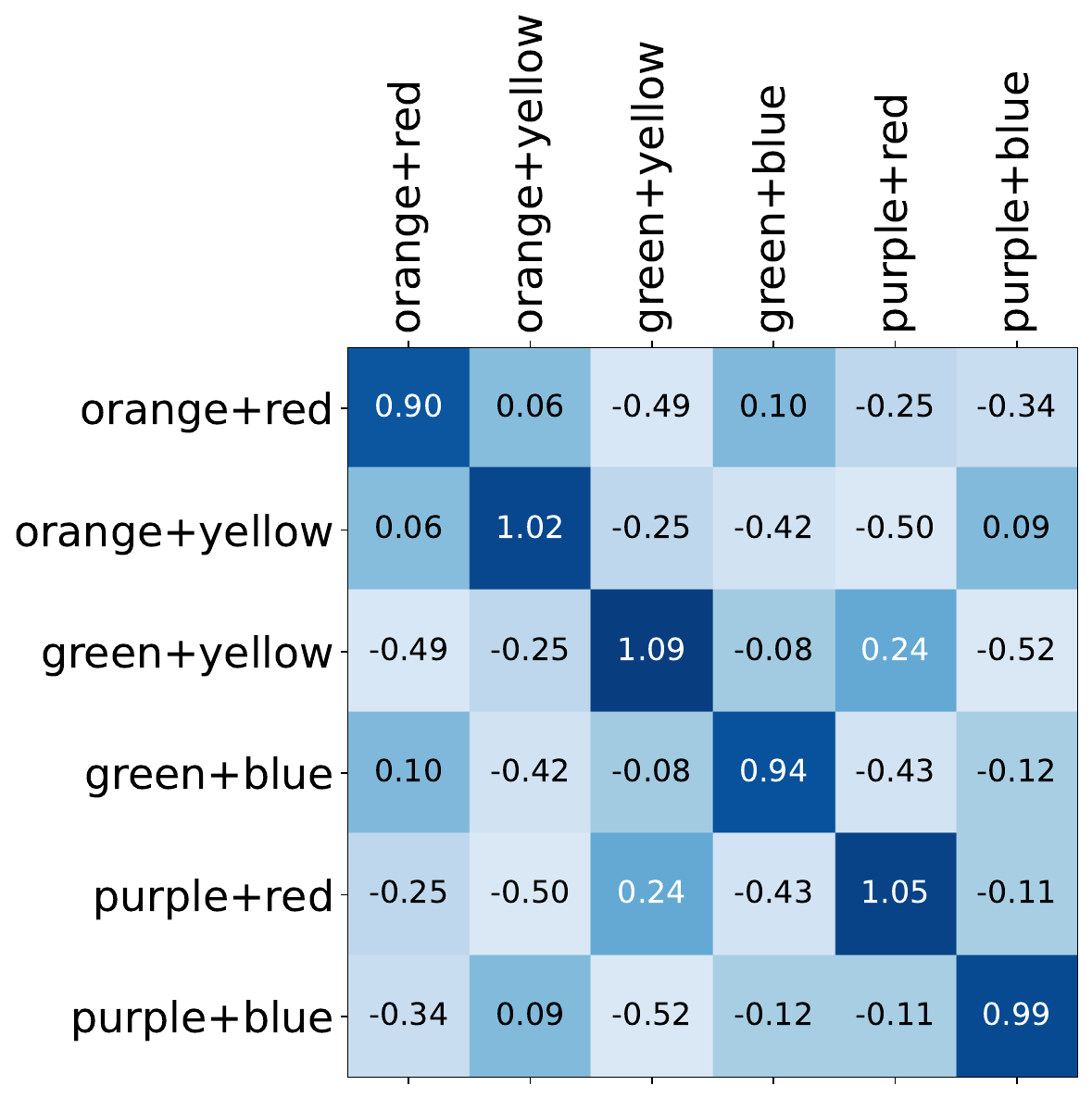}
    \caption{Mistral-7B-Instruct-v0.3}
  \end{subfigure}\hfill
  \begin{subfigure}[t]{0.32\textwidth}
    \centering
    \includegraphics[width=\linewidth]{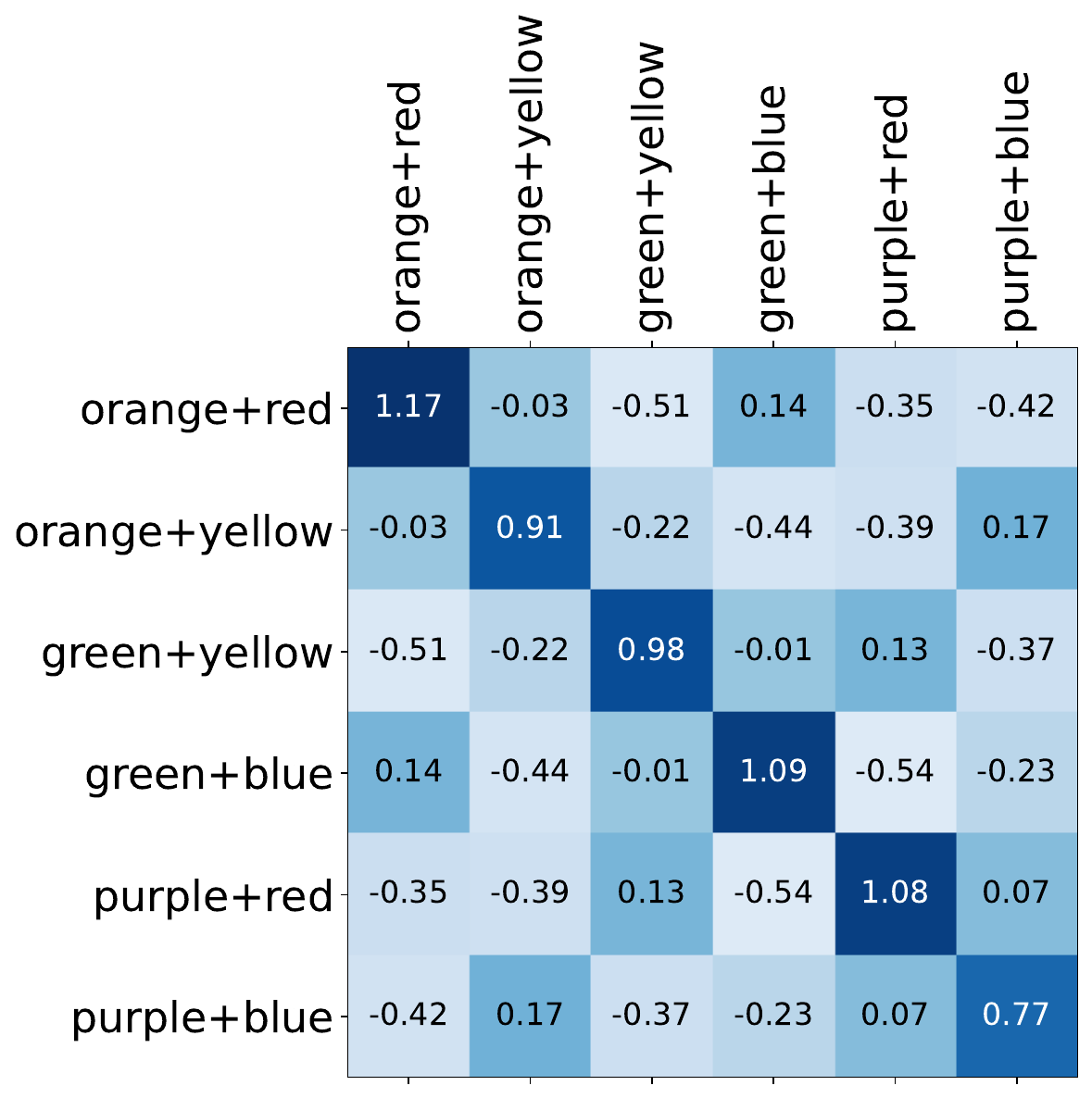}
    \caption{RWKV7-7.2B}
  \end{subfigure}
  \caption{Context embedding Grams for color-mix prompts.}
  \label{fig:perm-gram}
\end{figure}

\paragraph{Cyclic-shift prompts.}
For the prompts ``\texttt{Today is <TOKEN>. Tomorrow is \textvisiblespace}'' with \texttt{<TOKEN>} ranging from \{\texttt{Monday}, \texttt{Tuesday}, \texttt{Wednesday}, \texttt{Thursday}, \texttt{Friday}, \texttt{Saturday}, \texttt{Sunday}\}, the idealized cyclic-shift model motivates testing whether the next-token distributions approximately shift forward by one day, producing a near-circulant \(\boldsymbol{\sigma}(Z)\) with rows that are cyclic shifts of each other (\cref{fig:ctx-cyclic-probs}). Correspondingly, the diagnostics test whether the context embedding Gram \(H^\top H\) is close to circulant, depending on the day difference \((j-i)\bmod 7\) (\cref{thm:single_block_cyclic,fig:ctx-cyclic-gram}).

\begin{figure}[htbp]
  \centering
  \captionsetup[subfigure]{justification=centering}
  \begin{subfigure}[t]{0.32\textwidth}
    \centering
    \includegraphics[width=\linewidth]{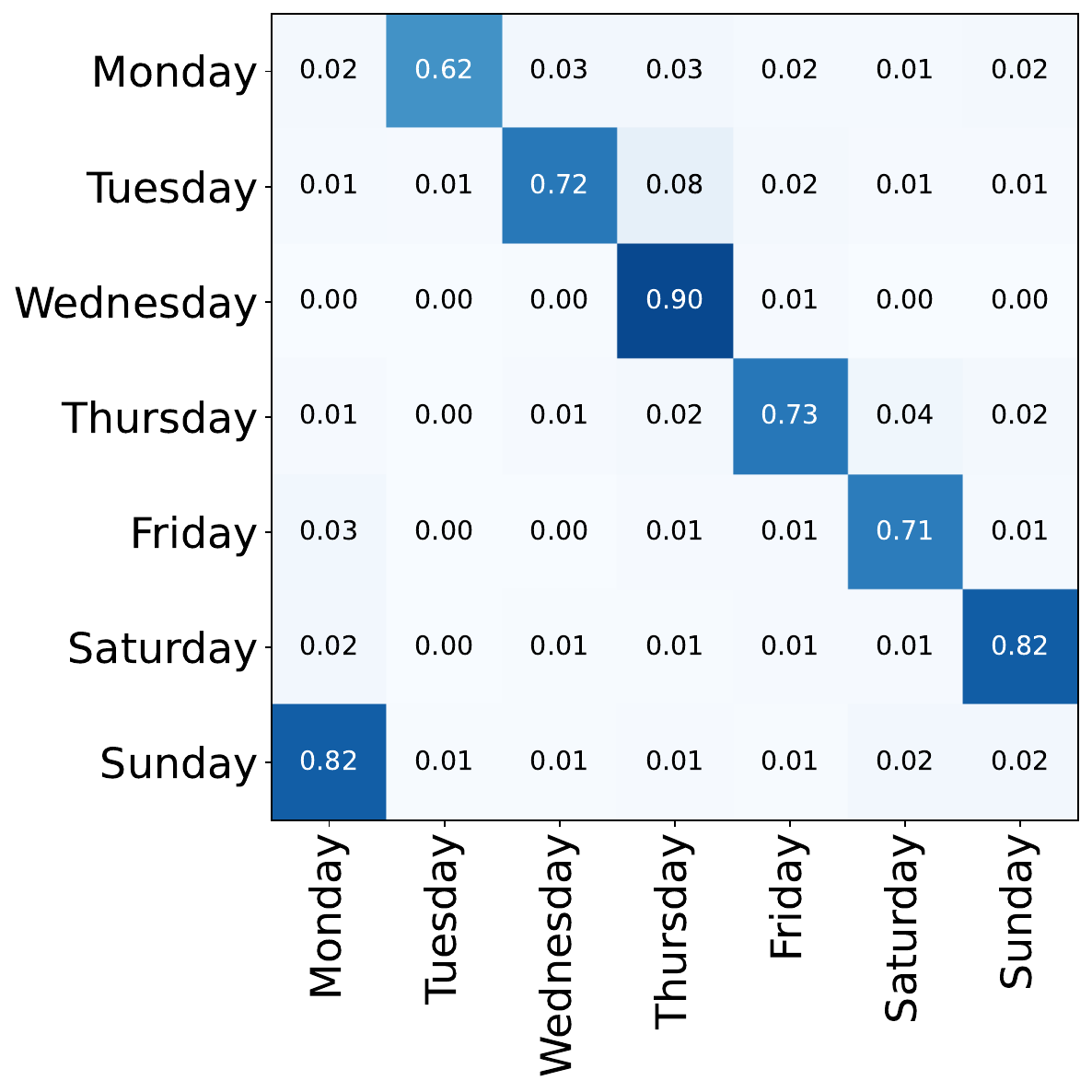}
    \caption{GPT-OSS-20B}
  \end{subfigure}\hfill
  \begin{subfigure}[t]{0.32\textwidth}
    \centering
    \includegraphics[width=\linewidth]{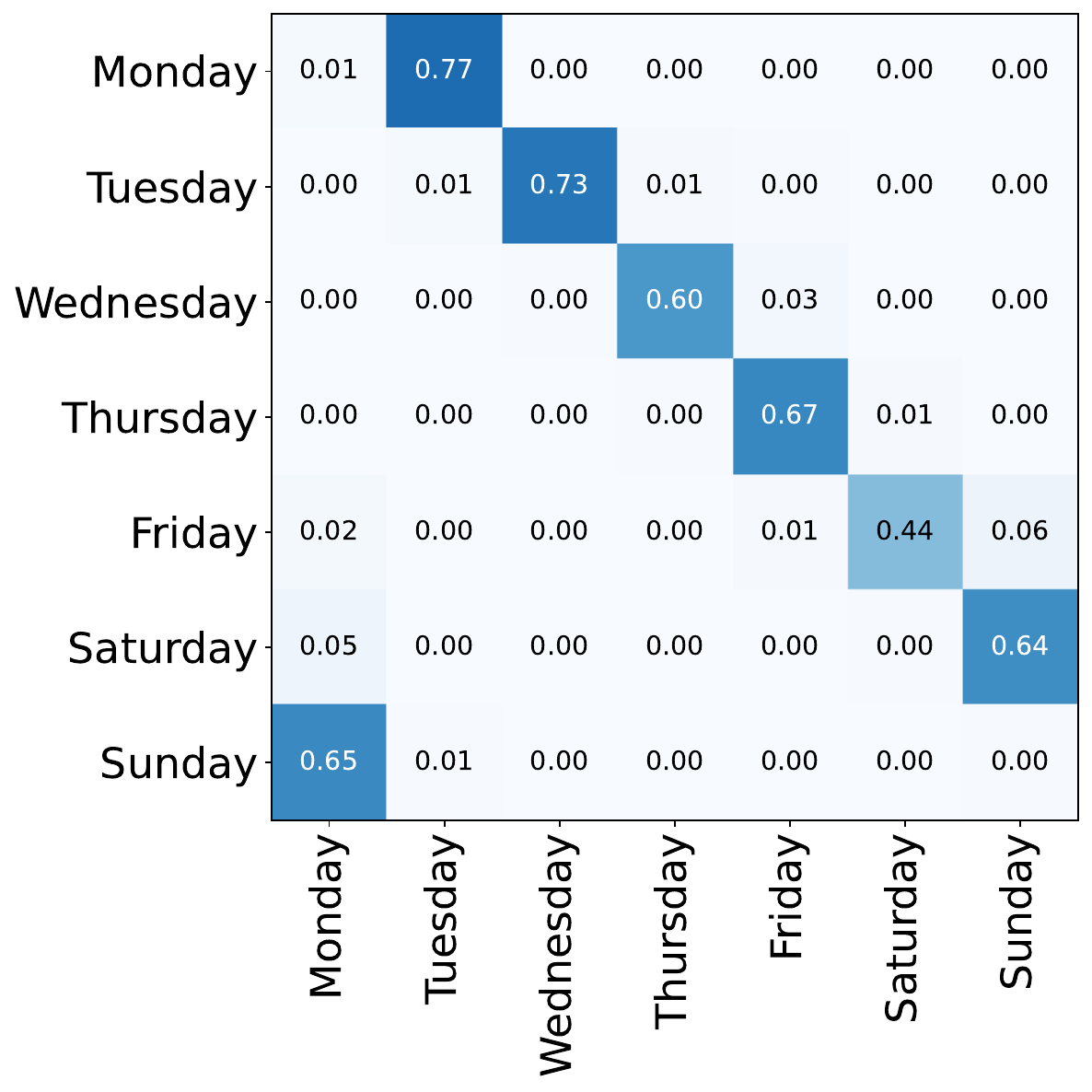}
    \caption{Mistral-7B-Instruct-v0.3}
  \end{subfigure}\hfill
  \begin{subfigure}[t]{0.32\textwidth}
    \centering
    \includegraphics[width=\linewidth]{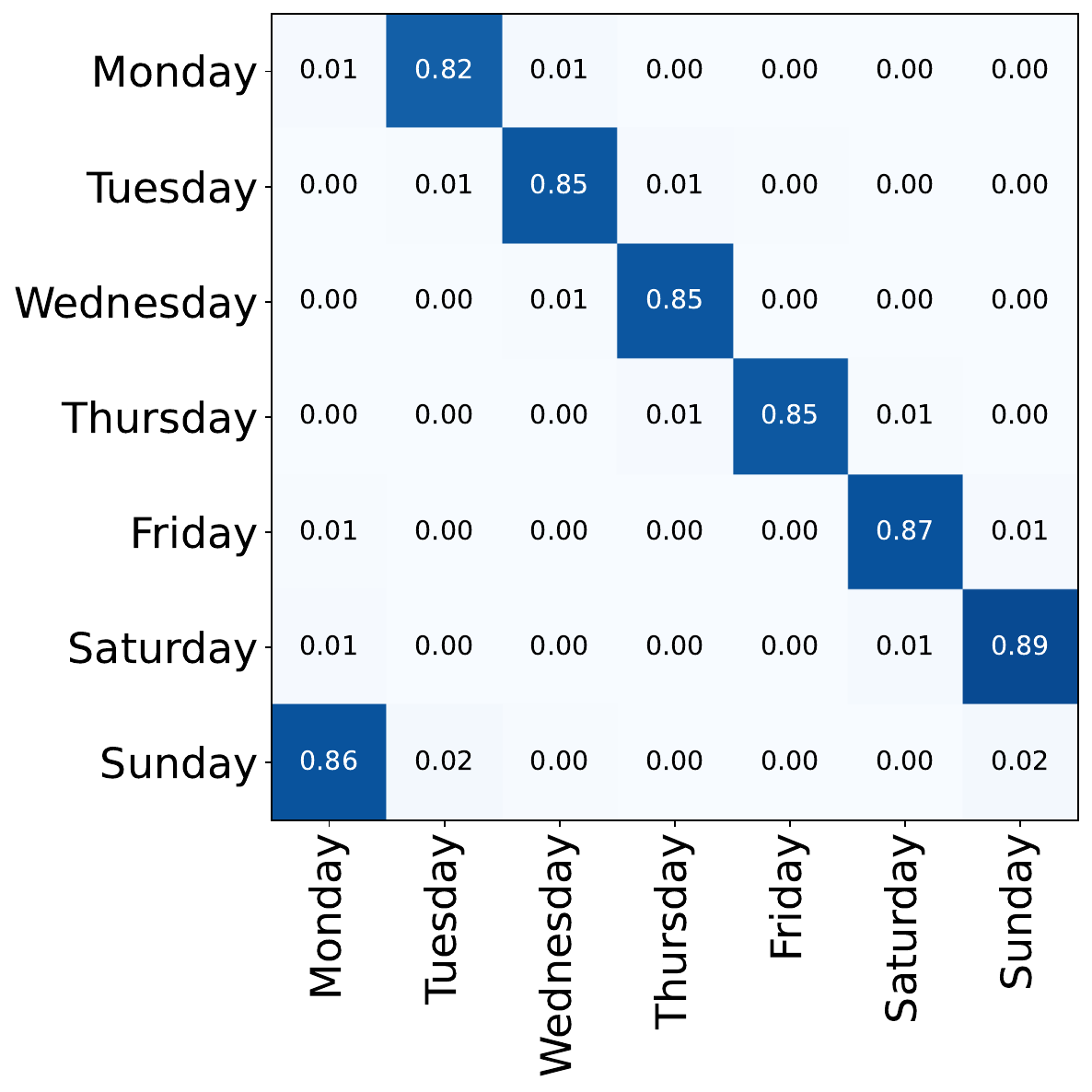}
    \caption{RWKV7-7.2B}
  \end{subfigure}
  \caption{Predicted next-token probability matrices for cyclic-shift weekday prompts. Row labels denote prompt contexts instantiated from the template. Column labels are the candidate next tokens restricted to the seven weekdays. Each cell shows the model's predicted probability for that token.}
  \label{fig:ctx-cyclic-probs}
\end{figure}

\begin{figure}[htbp]
  \centering
  \captionsetup[subfigure]{justification=centering}
  \begin{subfigure}[t]{0.32\textwidth}
    \centering
    \includegraphics[width=\linewidth]{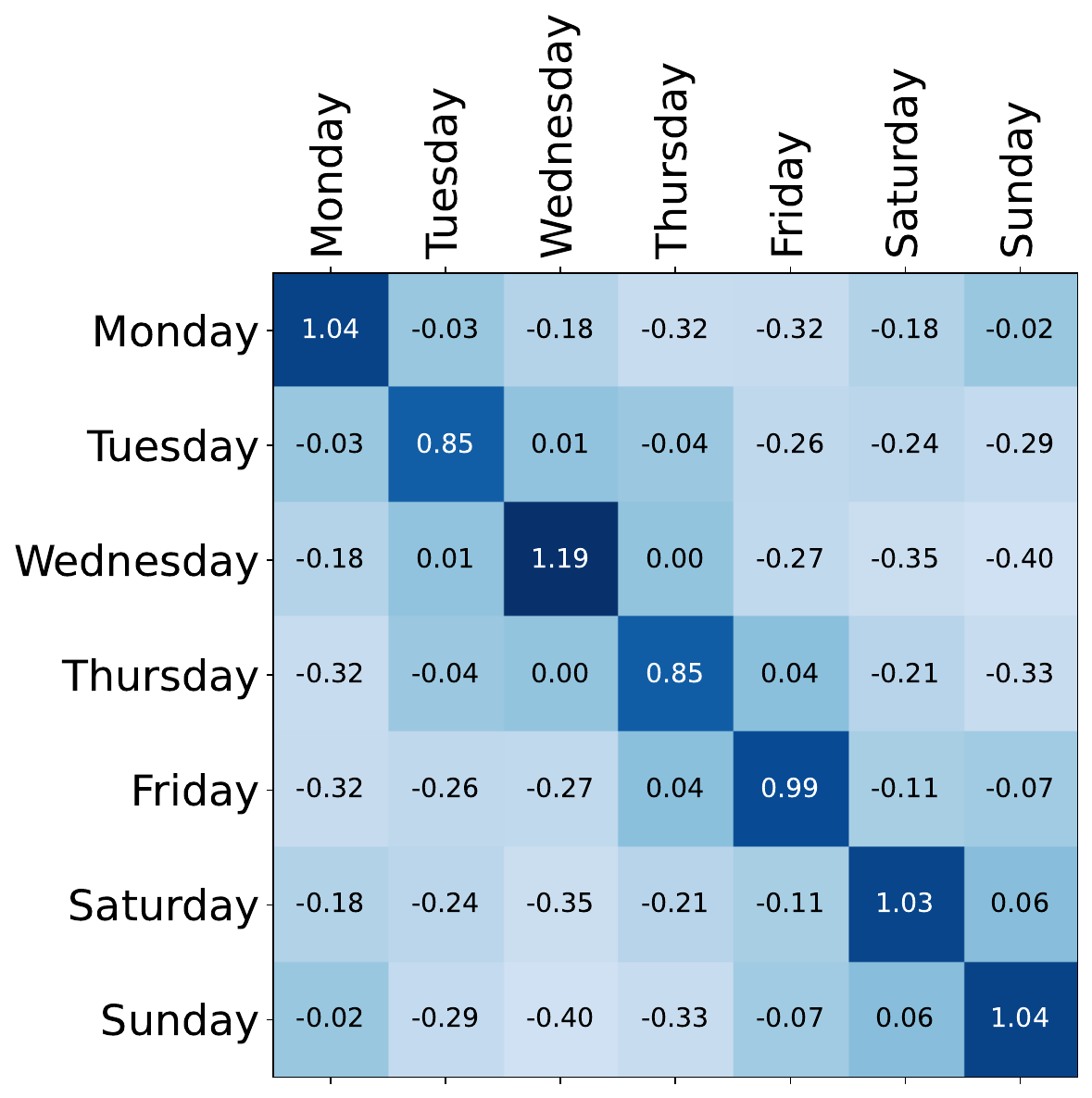}
    \caption{GPT-OSS-20B (\(\delta_{\mathrm{circ}}=0.171\))}
  \end{subfigure}\hfill
  \begin{subfigure}[t]{0.32\textwidth}
    \centering
    \includegraphics[width=\linewidth]{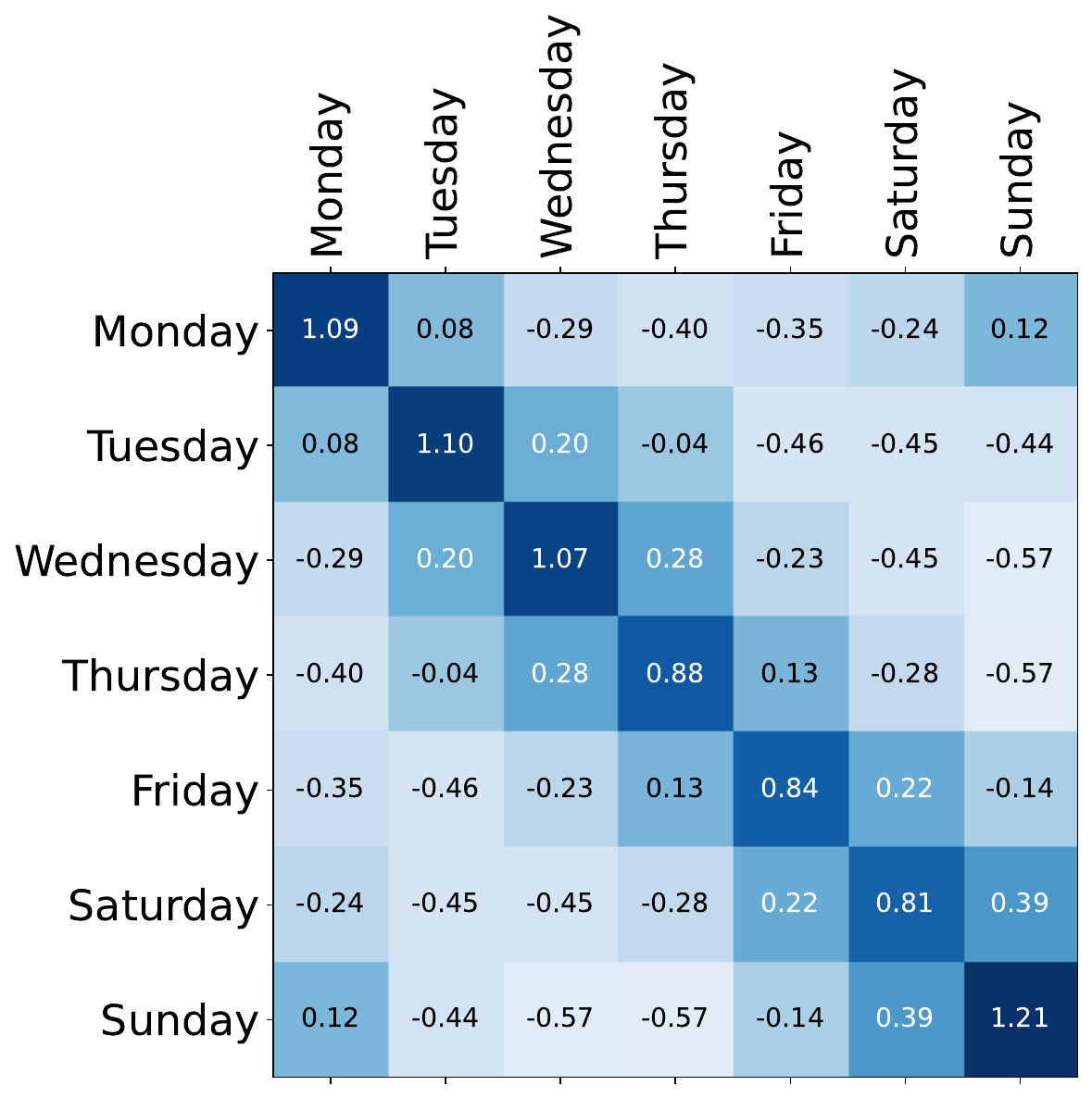}
    \caption{Mistral-7B-Instruct-v0.3 (\(\delta_{\mathrm{circ}}=0.214\))}
  \end{subfigure}\hfill
  \begin{subfigure}[t]{0.32\textwidth}
    \centering
    \includegraphics[width=\linewidth]{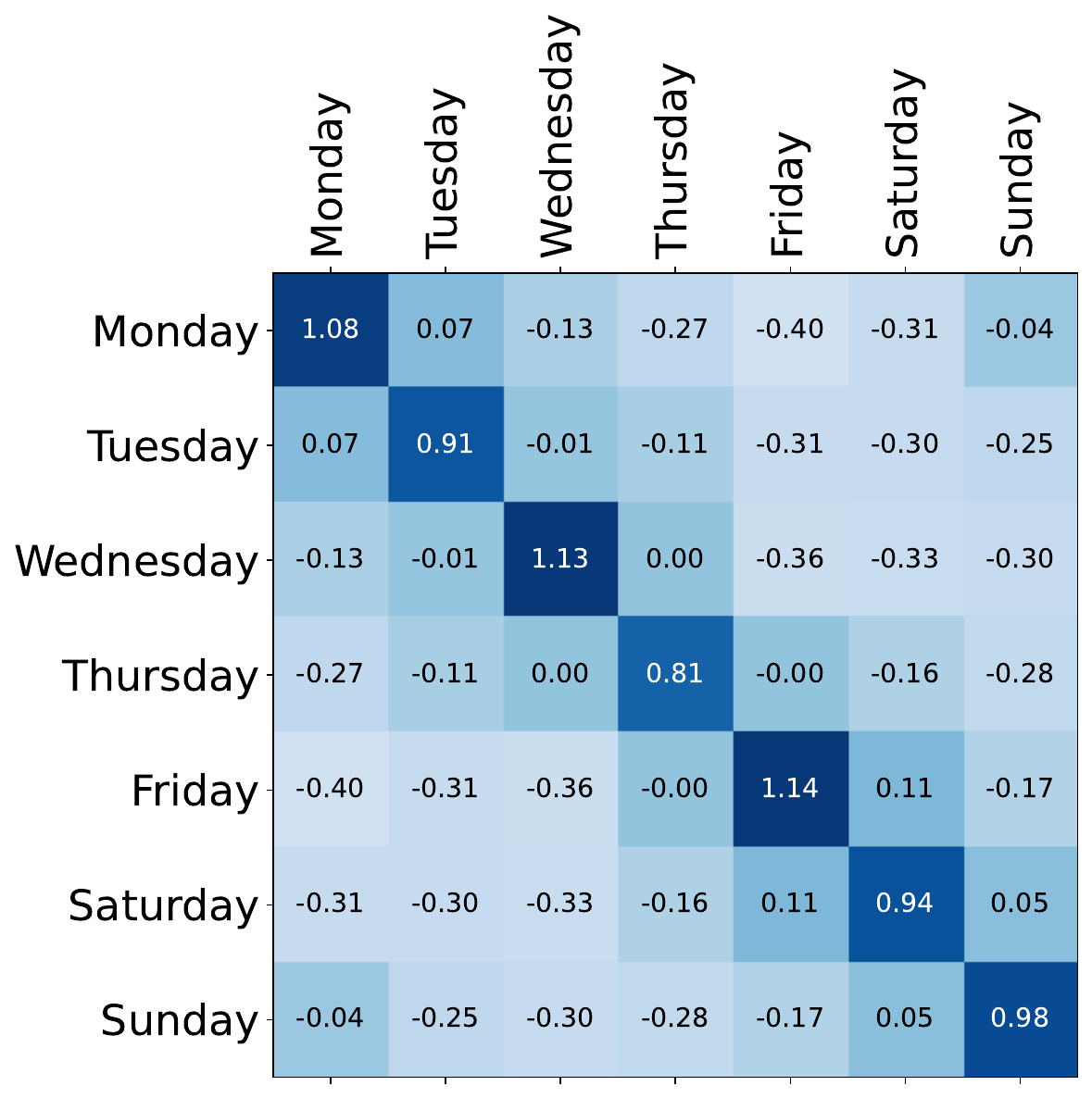}
    \caption{RWKV7-7.2B (\(\delta_{\mathrm{circ}}=0.165\))}
  \end{subfigure}
  \caption{Context embedding Grams for cyclic-shift weekday prompts.}
  \label{fig:ctx-cyclic-gram}
\end{figure}

\section{Conclusions}

In summary, we have investigated how exact target distribution symmetries govern the geometry of the global minimizers in a constrained layer-peeled model for next-token prediction. By analyzing this nonconvex optimization program, we demonstrated that when the target distribution matrix is invariant under a certain permutation group action, the globally optimal output projection matrix and context embeddings systematically inherit the same symmetric structure. We rigorously established this symmetry transfer phenomenon in two settings: cyclic-shift symmetries, which we proved lead to a circulant geometry, and permutation symmetries---including more general two-transitive group actions---which we showed induce a simplex ETF and matching orbit structures. We further generalized these geometric characterizations to multi-block orbit models. The core of our theoretical analysis relies on the exact reduction of the constrained nonconvex factorized problem to an explicit logit-level convex characterization and a sharp analysis of the optimal factorizations. Empirically, we verified that the last-layer weights and representations of open-source LLMs naturally self-organize into these mathematically predicted geometric configurations, even without explicit geometric regularization.

Our theoretical framework and empirical findings open up several directions for future research. First, from an optimization perspective, a natural generalization is to characterize the global minimizers under broader classes of group actions that induce other refined geometric structures, such as the composite permutation symmetries explored in our numerical experiments. Furthermore, our current global optimality guarantees rely on the overparameterized assumption that the hidden dimension is bounded below by the size of the local candidate token set (\(d \ge m\)). Because the entire token vocabulary in practice is exceedingly large while the hidden dimension remains comparatively small, extending our exact symmetry transfer results to the rank-constrained, underparameterized regime (\(d < m\)) is an important future direction. Third, our present analysis deliberately isolates local geometric properties by restricting the target matrix to small, semantically cohesive subsets of words; that is, we consider a small $m$. Extending this framework to include the entire vocabulary across diverse contexts implies transitioning from exact, local symmetries to global, approximate ones. Characterizing the stability of the optimal solutions under approximate symmetries is an interesting challenge for future work. Finally, in a broader context, it would be valuable to investigate how the explicitly symmetric patterns unveiled by our work can be leveraged for the mechanistic interpretability of LLMs.

\bibliographystyle{abbrvnat}
\bibliography{refs}

\clearpage
\appendix 
\appendixpage
\addappheadtotoc 

\section{Notation}
Bold lower-case letters denote column vectors, capital letters denote matrices, and calligraphic letters denote sets, groups, spaces, objectives, feasible domains, or operators. We use \(m\) for the number of candidate tokens, \(d\) for the hidden dimension, and \(n\) for the number of contexts; \(W\in\mathbb{R}^{m\times d}\) is the output projection matrix, \(H\in\mathbb{R}^{d\times n}\) is the context embedding matrix, \(Y\in\mathbb{R}^{m\times n}\) is the target distribution matrix, and \(Z=WH\) is the logit matrix. The Frobenius norm budgets are \(E_W,E_H>0\). For a matrix \(X\), \(X_{i,j}\) is its \((i,j)\)-entry, \(X_{:,j}\) is its \(j\)-th column, and \(X_{i,:}\) is its \(i\)-th row; \(X^\top\) and \(X^{\mathsf H}\) denote transpose and conjugate transpose. We write \(\boldsymbol{1}_m,\boldsymbol{0}_m,I_m,J_m=\boldsymbol{1}_m\boldsymbol{1}_m^\top\) for the all-ones vector, all-zeros vector, identity matrix, and all-ones matrix, and we write \(\mathbb R,\mathbb C\) for the real and complex fields. The norms \(\|\cdot\|_2,\|\cdot\|_F,\|\cdot\|_*,\|\cdot\|_{\mathrm{op}}\) are the Euclidean, Frobenius, nuclear, and operator norms; \(s_i(X)\), \(\mathrm{rank}(X)\), \(\mathrm{tr}(X)\), \(\mathrm{diag}(\cdot)\), and \(\langle\cdot,\cdot\rangle_F\) denote singular values, rank, trace, diagonal matrices, and Frobenius inner products. The probability simplex is \(\Delta^{m-1}\), the vector softmax is \(\sigma\), the columnwise matrix softmax is \(\boldsymbol{\sigma}\), with \(\boldsymbol{\sigma}(X)_{:,j}=\sigma(X_{:,j})\), the cross-entropy loss is \(\mathcal L(\boldsymbol q,\boldsymbol p)=-\sum_i p_i\log q_i\), and the main objective is \(\mathcal J(W,H;Y)=\sum_{j=1}^n\mathcal L(\sigma((WH)_{:,j}),Y_{:,j})\); unless stated otherwise, \(\log\) denotes the natural logarithm. Superscript \(^*\) denotes optimal quantities, such as \(W^*,H^*,Z^*\), \(\boldsymbol z^*\), or \(\boldsymbol{\alpha}^*\), while tildes denote symmetrized or projected quantities. The one-step cyclic-shift matrix is \(\Pi\), the cyclic-shift group is \(C_m\), and \(S_m\) is the symmetric group. For a group \(\mathcal G\), \(g\circ\boldsymbol{x}\) denotes its permutation action on a vector, \(\rho(g)\) is the corresponding permutation matrix, \(\Lambda(g)\) denotes the induced column permutation when needed, and \(|\mathcal G|\) is the group order. The zero-sum subspace is \(\mathcal V_0=\{\boldsymbol{x}\in\mathbb R^m:\boldsymbol{1}_m^\top\boldsymbol{x}=0\}\). In the permutation-symmetry analysis, \(A\) denotes an auxiliary probability matrix, \(C=A-Y\) denotes the residual matrix, \(\boldsymbol{\alpha}\) denotes an auxiliary probability vector, and \(k,\gamma,\lambda_1,\lambda_2,\nu\) denote scalar parameters or Lagrange multipliers defined locally. In cyclic and Fourier arguments, \(F_m\), \(\omega\), and diagonal matrices \(D,D_A,D_B,D_P\) are defined locally when diagonalizing circulant matrices. In multi-block models, \(n_{\mathrm{blk}}\) is the number of blocks and \(Y=[Y_1|\cdots|Y_{n_{\mathrm{blk}}}]\), with block-specific objects carrying matching subscripts. For simplex ETF notation, \(M^*\) denotes the canonical simplex ETF matrix and \(M=cM^*U^\top\) denotes a rotated and rescaled realization. In the experiments, \(G\) denotes a Gram matrix, \(\mathcal C\) is the circulant subspace, \(\mathcal P\) is the circulant projection, and \(\delta_{\mathrm{ETF}}\) and \(\delta_{\mathrm{circ}}\) are the relative distances defined in~\cref{app:measure}.

\section{Proofs}
\subsection{Proof of Symmetry Transfer~I, Multi-Block Extension}
\label{pf:multi_block_cyclic}

\begin{proof}[Proof of~\cref{thm:multi_block_cyclic}]
The proof is similar to the single-block case in~\cref{thm:single_block_cyclic}. We first establish that any optimal solution to the original optimization problem~\cref{eq:obj_function_constrained} must have a product matrix \(Z=WH\) that is block-circulant. Let the original optimization problem be:
\begin{equation} \label{eq:prob_original_multi_block}
\mathcal J_1^* = \min_{(W,H) \in \mathcal{F}_{\mathrm{feas}} } \mathcal J(W,H;Y),
\end{equation}
where, writing \(Z=WH=[Z_1|\cdots|Z_{n_{\mathrm{blk}}}]\), \(\mathcal J(W,H;Y) = \sum_{i=1}^{n_{\mathrm{blk}}} \sum_{j=1}^{m}\mathcal L(\sigma((Z_i)_{:,j}), (Y_i)_{:,j})\) and the feasible set is \(\mathcal{F}_{\mathrm{feas}}  = \{(W,H) \mid \|W\|_F^2\le E_W, \|H\|_F^2\le E_H\}\). Let the restricted problem, where the logit matrix \(WH\) must be block-circulant as described in the theorem, be:
\begin{equation} \label{eq:prob_circulant_multi_block}
\mathcal J_2^* = \min_{(W,H) \in \mathcal{F}_{\mathrm{feas}} , WH \text{ is block-circulant}} \mathcal J(W,H;Y).
\end{equation}
We will prove that any optimal solution to~\cref{eq:prob_original_multi_block} must have a block-circulant product matrix, which implies \(\mathcal J_1^* = \mathcal J_2^*\) and that the solution sets are equivalent. The inequality \(\mathcal J_2^* \ge \mathcal J_1^*\) is immediate, as the feasible set of~\cref{eq:prob_circulant_multi_block} is a subset of that of~\cref{eq:prob_original_multi_block}. The remaining task is to show that for any feasible pair \((W,H)\), we can construct a new feasible pair \((\widetilde{W}, \widetilde{H})\) with a block-circulant product \(\widetilde{Z}=\widetilde{W}\widetilde{H}\) such that \(\mathcal J(\widetilde{W},\widetilde{H};Y) \le \mathcal J(W,H;Y)\).

Let \((W,H)\) be any feasible solution to~\cref{eq:prob_original_multi_block}. Let the logit matrix be \(Z = WH\), partitioned into blocks corresponding to \(Y\), so \(Z = [Z_1 | Z_2 | \cdots | Z_{n_{\mathrm{blk}}}]\). We apply a symmetrization procedure to each block \(Z_i\) independently. For each block \(i\), define the average aligned logit vector \(\boldsymbol{q}_i\) by applying the inverse shift to each column of \(Z_i\) and averaging:
\[ \boldsymbol{q}_i := \frac{1}{m}\sum_{j=1}^{m} \Pi^{-(j-1)}(Z_i)_{:,j}. \]
From each \(\boldsymbol{q}_i\), we construct a circulant logit block \(\widetilde{Z}_i\) by setting
\[ (\widetilde{Z}_i)_{:,j}:=\Pi^{j-1}\boldsymbol q_i,\qquad j=1,\ldots,m. \]
The symmetrized block predictor is the concatenation of these blocks: \(\widetilde{Z} = [\widetilde{Z}_1 | \widetilde{Z}_2 | \cdots | \widetilde{Z}_{n_{\mathrm{blk}}}]\).
Equivalently, if \(I_{n_{\mathrm{blk}}}\otimes \Pi^{-(t-1)}\) denotes the block-diagonal matrix that applies the shift
\(\Pi^{-(t-1)}\) within each block of columns, then
\[
\widetilde Z
=
\frac1m\sum_{t=1}^{m} \Pi^{t-1} Z (I_{n_{\mathrm{blk}}}\otimes \Pi^{-(t-1)}).
\]
This makes clear that \(\widetilde Z\) is block-circulant.
For each block \(i\), define
\[
f_i(\boldsymbol u):=\mathcal L(\sigma(\boldsymbol u),\boldsymbol{y}_i),
\qquad
\boldsymbol z_{i,j}:=\Pi^{-(j-1)}(Z_i)_{:,j},\qquad j=1,\ldots,m.
\]
Then
\[
\mathcal J(W,H;Y)=\sum_{i=1}^{n_{\mathrm{blk}}}\sum_{j=1}^{m} f_i(\boldsymbol z_{i,j}),
\]
while the loss of the symmetrized predictor is
\[
\mathcal J(\widetilde W,\widetilde H;Y)
=
m\sum_{i=1}^{n_{\mathrm{blk}}} f_i \left(\frac1m\sum_{j=1}^{m}\boldsymbol z_{i,j}\right).
\]
By~\cref{lem:softmax_ce_strict_convexity_V0}, each \(f_i\) is convex, so Jensen's inequality applied blockwise gives
\[
\mathcal J(\widetilde W,\widetilde H;Y)\le \mathcal J(W,H;Y).
\]
We also need feasibility of the symmetrized predictor. Since \(\widetilde Z\) is an average of matrices obtained from \(Z\) by left and right multiplication by orthogonal matrices, the nuclear norm is nonincreasing:
\[
\|\widetilde Z\|_*
=
\left\|
\frac1m\sum_{t=1}^{m} \Pi^{t-1} Z (I_{n_{\mathrm{blk}}}\otimes \Pi^{-(t-1)})
\right\|_*
\le
\frac1m\sum_{t=1}^{m}\|\Pi^{t-1} Z (I_{n_{\mathrm{blk}}}\otimes \Pi^{-(t-1)})\|_*
=
\|Z\|_*.
\]
Since \(Z=WH\), \cref{lem:nuclear_frobenius_combined} gives
\[
\|Z\|_*\le \|W\|_F\|H\|_F\le \sqrt{E_WE_H},
\]
hence \(\|\widetilde Z\|_*\le \sqrt{E_WE_H}\). Therefore, exactly as in
~\cref{thm:single_block_cyclic}, taking an SVD \(\widetilde Z=U\Sigma V^\top\), choosing \(Q\in\mathbb R^{d\times r}\) with \(Q^\top Q=I_r\), and defining
\[
\widetilde W:=\sqrt[4]{\frac{E_W}{E_H}} U\Sigma^{1/2}Q^\top,
\qquad
\widetilde H:=\sqrt[4]{\frac{E_H}{E_W}} Q\Sigma^{1/2}V^\top,
\]
produces a feasible pair \((\widetilde W,\widetilde H)\) with
\(\widetilde W\widetilde H=\widetilde Z\).
Now let \((W^*,H^*)\) be an optimal solution to~\cref{eq:prob_original_multi_block}, and write
\(Z^*=W^*H^*=[Z_1^*|\cdots|Z_{n_{\mathrm{blk}}}^*]\).
Applying the symmetrization above to \(Z^*\) yields a feasible block-circulant predictor
\(\widetilde Z^*\) with
\[
\mathcal J(\widetilde W,\widetilde H;Y)\le \mathcal J(W^*,H^*;Y).
\]
By optimality, equality must hold.
Because equality holds in the blockwise Jensen inequalities, for each block \(i\) the aligned columns
\[
\boldsymbol z_{i,j}^*:=\Pi^{-(j-1)}(Z_i^*)_{:,j},\qquad j=1,\ldots,m,
\]
must satisfy the same equality condition as in the single-block proof:
\[
\boldsymbol z_{i,j}^*=\boldsymbol z_{i,1}^*+\beta_{i,j}^*\boldsymbol 1_m,
\qquad
\boldsymbol 1_m^\top \boldsymbol z_{i,1}^*=0.
\]
Define the circulant part of each block by
\[
(N_i^*)_{:,j}:=\Pi^{j-1}\boldsymbol z_{i,1}^*,\qquad j=1,\ldots,m,
\]
and let \(\boldsymbol\beta_i^*:=(\beta_{i,1}^*,\dots,\beta_{i,m}^*)^\top\).
Then
\[
Z_i^*=N_i^*+\boldsymbol 1_m(\boldsymbol\beta_i^*)^\top.
\]
Concatenating the blocks gives
\[
Z^*=N^*+\boldsymbol 1_m(\boldsymbol\beta^*)^\top,
\]
where
\[
N^*:=[N_1^*|\cdots|N_{n_{\mathrm{blk}}}^*],
\qquad
\boldsymbol\beta^*:=[(\boldsymbol\beta_1^*)^\top|\cdots|(\boldsymbol\beta_{n_{\mathrm{blk}}}^*)^\top]^\top.
\]
Since the loss is invariant under columnwise additive shifts, replacing the product \(WH=Z^*\) by \(N^*\) leaves the objective value unchanged.
Moreover, every column of \(N^*\) lies in \(\mathcal V_0\), because each generator
\(\boldsymbol z_{i,1}^*\) has zero sum and cyclic shifts preserve zero sum. Hence
\[
\boldsymbol 1_m^\top N^*=0,
\]
and therefore
\[
Z^{*\top}Z^*
=
N^{*\top}N^*
+
m \boldsymbol\beta^*(\boldsymbol\beta^*)^\top.
\]
If \(\boldsymbol\beta^*\neq \boldsymbol{0}_n\), then exactly the same Weyl inequality and trace argument as in
~\cref{thm:single_block_cyclic} shows that
\[
\|Z^*\|_*>\|N^*\|_*.
\]
Since \(\|N^*\|_*<\|Z^*\|_*\le \sqrt{E_WE_H}\), the SVD construction above produces a feasible pair
realizing \(N^*\) with strictly smaller Frobenius norms, contradicting
~\cref{ass:active_constraint}. Therefore \(\boldsymbol\beta^*=\boldsymbol{0}_n\), so each block \(Z_i^*\) is circulant.
Hence every optimal logit matrix is block-circulant, and the original problem is equivalent to the restricted block-circulant problem.
Therefore, we only need to solve the restricted problem~\cref{eq:prob_circulant_multi_block}. The optimal logit matrix \(Z\) must be block-circulant:
\[ Z = [Z_1 | Z_2 | \cdots | Z_{n_{\mathrm{blk}}}], \]
where each block is defined by the cyclic shifts of a generating vector:
\[ (Z_i)_{:,j}=\Pi^{j-1}\boldsymbol z_i,\qquad j=1,\ldots,m. \]
The set of these optimal generating vectors \(\{\boldsymbol{z}_{1}, \ldots, \boldsymbol{z}_{n_{\mathrm{blk}}}\}\) is found as the solution to the following convex optimization problem, which reformulates the original objective and constraint in terms of \(\{\boldsymbol{z}_{i}\}\):
\[
\begin{aligned}
\min_{\boldsymbol{z}_{1}, \ldots, \boldsymbol{z}_{n_{\mathrm{blk}}} \in \mathbb{R}^m} \quad & \sum_{i=1}^{n_{\mathrm{blk}}} \mathcal L \left(\sigma(\boldsymbol{z}_{i}), \boldsymbol{y}_i\right) \\
\text{s.t.} \quad & \left\| [ [\Pi^{j-1} \boldsymbol{z}_{1}]_{j=1}^{m} | \cdots | [\Pi^{j-1} \boldsymbol{z}_{n_{\mathrm{blk}}}]_{j=1}^{m} ] \right\|_* \le \sqrt{E_W E_H}.
\end{aligned}
\]
Once this convex problem is solved for \(\{\boldsymbol{z}_{i}\}\), the optimal logit matrix \(Z\) is determined. Furthermore, applying the same procedure in~\cref{thm:single_block_cyclic}, the products of the optimal factors must satisfy:
\[ WW^\top = \sqrt{\frac{E_W}{E_H}} (Z Z^\top)^{1/2}. \]
By~\cref{lem:svd_factor_symmetry_block}, the matrix \((Z Z^\top)^{1/2}\) is circulant.
This completes the proof.
\end{proof}

\subsection{Proof of Symmetry Transfer~II, 2-Transitive Group Extension}
\label{pf:2_transitive_optimal_embeddings}

\begin{proof}[Proof of~\cref{thm:2_transitive_optimal_embeddings}]
The proof structure largely follows that of~\cref{thm:single_block_perm}.
The derivation of the lower bound on the loss function \(\mathcal J(W,H;Y)\) proceeds similarly as in~\cref{thm:single_block_perm}. The cross-entropy loss involves terms \(\mathrm{tr}(C ^\top WH)\) and an entropy term related to \(\boldsymbol{\alpha}\). The number of columns is now \(|\mathcal G|\) instead of \(m!\).
The optimization for \(\boldsymbol{\alpha}\) leads to the system of equations:
\[ k (\alpha_i-y_i) + \log\alpha_i = \frac{1}{m}\sum_{r=1}^m\log\alpha_r, \quad \sum_{i=1}^m \alpha_i=1, \]
where \(k  = \sqrt{E_W E_H}/((m-1)\gamma) \) and \(\gamma  = \sqrt{|\mathcal G|/(m-1)}\|\boldsymbol{\alpha}-\boldsymbol{y}\|_2\). The existence and uniqueness of \(\boldsymbol{\alpha}^*\) solving this system (given \(\boldsymbol{y}\) is nonuniform and \(\boldsymbol{\alpha}\neq \boldsymbol{y}\)) is proven similarly by concavity arguments as in~\cref{thm:single_block_perm}. The lower bound on the loss is achieved when these optimal \(\boldsymbol{\alpha}^*\) are used.

\paragraph{Sufficiency.}
We choose \(\boldsymbol{\alpha}=\boldsymbol{\alpha}^*\) and define \(C =A-Y \). The proposed optimal \(W=\sqrt{E_W/(m-1)}UQ^\top\) and \(H=-\sqrt{E_H/(m-1)}QV^\top\) can achieve this lower bound, as in~\cref{thm:single_block_perm}.

\paragraph{Necessity.}
If \((W,H)\) is an optimal solution, it must saturate the bounds. This implies that \(\boldsymbol{\sigma}(WH)=A\). As in~\cref{thm:single_block_perm}, \cref{lem:nuclear_frobenius_combined} implies that the SVDs of \(W\) and \(H\) can be written as \(W=U_W\Sigma_W Q^\top\) and \(H=Q\Sigma_H V_H^\top\).
By~\cref{lem:von_neumann_rectangular} we can write
\[ WH = U\Sigma_{WH}V^\top, \]
where \(\Sigma_{WH} = \mathrm{diag}(\mu_1, \dots, \mu_{m-1})\) contains the singular values of \(WH\).
To establish the isotropy of \(\Sigma_{WH}\), we use the condition \(\boldsymbol{\sigma}(WH)=A\). The columns of \(A\) are \(g\circ\boldsymbol{\alpha}^*\), forming the orbit of the vector \(\boldsymbol{\alpha}^*\) under the group \(\mathcal G\). We then apply~\cref{lem:orbit_S_softmax_2_transitive}. This lemma states that if the columns of \(\boldsymbol{\sigma}(U\Sigma_{WH}V^\top)\) form such an orbit under the 2-transitive group \(\mathcal G\), then the diagonal matrix \(\Sigma_{WH}\) must be a scalar multiple of the identity matrix. Therefore,
\[ \Sigma_{WH} = k_0 I_{m-1} \]
for some scalar \(k_0 \ge 0\). The remainder of the proof follows the same algebraic steps as in the detailed proof for~\cref{thm:single_block_perm}. The derivations for the structures of \(WW^\top\), \(H^\top H\), and \(WH\) then follow directly as shown in the proof of~\cref{thm:embedding_products_structure}.
\end{proof}

\subsection{Proof of Symmetry Transfer~II, Multi-Block Extension}
\label{pf:2_transitive_multiblock}

\begin{proof}[Proof of~\cref{thm:2_transitive_multiblock}]
The argument largely follows that of~\cref{thm:single_block_perm}. We highlight the main adaptations. Set $\boldsymbol{v}_i:=\boldsymbol{\alpha}_i-\boldsymbol{y}_i$ and let \(C_i = A_i - Y_i = [g\circ\boldsymbol{v}_i]_{g \in \mathcal G_i}\) for each block \(i\). The overall matrix is \(C = [C_1|\dots|C_{n_{\mathrm{blk}}}]\).
The derivation of the lower bound for the loss function \(\mathcal J(W,H;Y)\) using Jensen's inequality (analogous to~\cref{eq:jensen} in prior proofs) remains structurally similar:
\[ \mathcal J(W,H;Y) \ge \mathrm{tr}(C^{\top}WH) - \sum_{i=1}^{n_{\mathrm{blk}}} |\mathcal G_i|\sum_{\ell=1}^{m}\alpha_{i\ell}\log\alpha_{i\ell}. \]
First, consider the structure of \(CC^\top\). Since \(\mathcal G_i\) is a 2-transitive group and \(\boldsymbol{1}_m^\top \boldsymbol{v}_i = 0\), by~\cref{lem:2_transitive_svd}, \(C_i C_i^\top = \lambda'_i (I_m - \boldsymbol{1}_m\boldsymbol{1}_m^\top/m)\) where \(\lambda'_i = |\mathcal G_i|\|\boldsymbol{v}_i\|_2^2/(m-1)\).
Summing over blocks:
\[ CC^{\top} = \sum_{i=1}^{n_{\mathrm{blk}}} C_i C_i^\top = \left(\sum_{i=1}^{n_{\mathrm{blk}}} \lambda'_i\right) \left(I_m-\frac1m\boldsymbol{1}_m\boldsymbol{1}_m^{\top}\right). \]
Let \(\lambda = \sum_i \lambda'_i \). Then \(CC^\top = \lambda (I_m - \boldsymbol{1}_m\boldsymbol{1}_m^\top/m)\). This confirms that \(C\) has a rank-$(m-1)$ SVD \(C=U(\gamma I_{m-1})V^\top\) with \(\gamma = \sqrt{\lambda}\). \( \gamma > 0 \) follows from the assumption that at least one \(\boldsymbol{y}_i\) is nonuniform, ensuring that at least one \(\boldsymbol{v}_i\) is nonzero.
The trace inequality \(\mathrm{tr}(C^{\top}WH)\ge-\sqrt{E_W E_H}\gamma\) follows as previously established. This leads to \(\mathcal J(W,H;Y)\ge\phi(\boldsymbol{\alpha}_1,\dots,\boldsymbol{\alpha}_{n_{\mathrm{blk}}})\) with
\[ \phi(\boldsymbol{\alpha}_1,\dots,\boldsymbol{\alpha}_{n_{\mathrm{blk}}}) = -\sqrt{E_W E_H}\gamma(\boldsymbol{\alpha}_1,\dots,\boldsymbol{\alpha}_{n_{\mathrm{blk}}}) - \sum_i |\mathcal G_i|\sum_{\ell=1}^{m}\alpha_{i\ell}\log\alpha_{i\ell}. \]
We now establish the properties of the maximizer \(\boldsymbol{\alpha}^* = (\boldsymbol{\alpha}_1^*, \ldots, \boldsymbol{\alpha}_{n_{\mathrm{blk}}}^*)\) of \(\phi(\boldsymbol{\alpha}_1,\dots,\boldsymbol{\alpha}_{n_{\mathrm{blk}}})\) subject to \(\boldsymbol{\alpha}_i \in \Delta^{m-1}\) for each \(i\). The domain \(\mathcal{A} = \times_{i=1}^{n_{\mathrm{blk}}} \Delta^{m-1}\) is compact and convex.
To establish uniqueness of the maximizer, we first examine the concavity of \(\phi\).
The first term depends on \(\gamma\), which is \(\gamma = \sqrt{\sum_{j=1}^{n_{\mathrm{blk}}}|\mathcal G_j|\|\boldsymbol{\alpha}_j-\boldsymbol{y}_j\|_{2}^{2}/(m-1)}\).
Let \(r_{j\ell} = \sqrt{|\mathcal G_j|/(m-1)}(\alpha_{j\ell}-y_{j\ell})\). Each \(r_{j\ell}\) is an affine function of \(\alpha_{j\ell}\). If we form a high-dimensional vector \(\boldsymbol{r}\) by concatenating all such \(r_{j\ell}\) components, the mapping from \(\{\boldsymbol{\alpha}_j\}\) to \(\boldsymbol{r}\) is an affine transformation.
Then, \(\gamma\) can be expressed as
\[ \gamma = \sqrt{\sum_{j=1}^{n_{\mathrm{blk}}}\sum_{\ell=1}^m r_{j\ell}^2} = \|\boldsymbol{r}\|_2. \]
Since the Euclidean norm is a convex function, and convexity is preserved under affine transformations of the input variables, \(\gamma\) is a convex function of \((\boldsymbol{\alpha}_1, \ldots, \boldsymbol{\alpha}_{n_{\mathrm{blk}}})\).
The term \(-\sqrt{E_W E_H}\gamma\) is therefore concave.
The Shannon entropy term for each block, \(\mathcal{H}(\boldsymbol{\alpha}_i) = -\sum_{\ell=1}^m \alpha_{i\ell}\log\alpha_{i\ell}\), is strictly concave on the interior of \(\Delta^{m-1}\). Since each \(|\mathcal G_i| > 0\), the sum \(\sum_{i=1}^{n_{\mathrm{blk}}} |\mathcal G_i|\mathcal{H}(\boldsymbol{\alpha}_i)\) is strictly concave with respect to the collection \(\{\boldsymbol{\alpha}_i\}\) on the interior of \(\mathcal{A}\).
Because \(\phi\) is the sum of a concave function (\(-\sqrt{E_W E_H}\gamma\)) and a strictly concave function (\(\sum |\mathcal G_i|\mathcal{H}(\boldsymbol{\alpha}_i)\)), it follows that \(\phi\) is strictly concave on the interior of its domain \(\mathcal{A}\).
Therefore, \(\phi\) attains a unique maximum value on \(\mathcal{A}\).
Next, the fact that this unique maximizer \(\boldsymbol{\alpha}^*\) lies in the interior of \(\mathcal{A}\) follows from the same directional derivative argument as in~\cref{thm:single_block_perm}.
Finally, since \(\boldsymbol{\alpha}^*\) is an interior point maximizer subject to the \(n_{\mathrm{blk}}\) constraints \(\sum_{\ell=1}^m \alpha_{i\ell}=1\) (for \(i=1,\dots,n_{\mathrm{blk}}\)), we use the method of Lagrange multipliers. The Lagrangian is:
\[ \mathscr{J}_{\boldsymbol{\alpha}}(\{\boldsymbol{\alpha}_j\}_{j=1}^{n_{\mathrm{blk}}}, \{\nu_j\}_{j=1}^{n_{\mathrm{blk}}}) = -\sqrt{E_W E_H}\gamma - \sum_{j=1}^{n_{\mathrm{blk}}} |\mathcal G_j|\sum_{\ell=1}^m \alpha_{j\ell}\log\alpha_{j\ell} - \sum_{j=1}^{n_{\mathrm{blk}}} \nu_j \left(\sum_{\ell=1}^m \alpha_{j\ell}-1\right). \]
Setting the partial derivative \(\partial \mathscr{J}_{\boldsymbol{\alpha}}/\partial \alpha_{i\ell} = 0\) for a specific block \(i\) and component \(\ell\):
\[ -\sqrt{E_W E_H} \frac{\partial \gamma}{\partial \alpha_{i\ell}} - |\mathcal G_i| (\log\alpha_{i\ell}+1) - \nu_i = 0. \]
We have \(\partial \gamma/\partial \alpha_{i\ell} = |\mathcal G_i|(\alpha_{i\ell}-y_{i\ell})/((m-1)\gamma)\). Substitution yields
\[ -\frac{\sqrt{E_W E_H} |\mathcal G_i|}{(m-1)\gamma} (\alpha_{i\ell}-y_{i\ell}) - |\mathcal G_i| (\log\alpha_{i\ell}+1) - \nu_i = 0. \]
Let \(k = \sqrt{E_W E_H}/((m-1)\gamma)\). The equation becomes:
\[ -k |\mathcal G_i| (\alpha_{i\ell}-y_{i\ell}) - |\mathcal G_i| (\log\alpha_{i\ell}+1) - \nu_i = 0. \]
Dividing by \(-|\mathcal G_i|\) (since \(|\mathcal G_i| > 0\)):
\begin{equation}
  \label{eq:alpha_i_ell}
  k (\alpha_{i\ell}-y_{i\ell}) + \log\alpha_{i\ell}+1 + \frac{\nu_i}{|\mathcal G_i|} = 0.
\end{equation}
Summing this equation over \(\ell=1,\dots,m\) for the fixed block \(i\), and using the fact that \(\sum_{\ell=1}^m (\alpha_{i\ell}-y_{i\ell}) = \sum_{\ell}\alpha_{i\ell} - \sum_{\ell}y_{i\ell} = 1-1=0\):
\[ \sum_{\ell=1}^m \log\alpha_{i\ell} + m\left(1 + \frac{\nu_i}{|\mathcal G_i|}\right) = 0. \]
This allows us to express the term involving the Lagrange multiplier for block \(i\) as \(1 + \nu_i/|\mathcal G_i| = -\sum_{r=1}^m \log\alpha_{ir}/m\).
Substituting this back into~\cref{eq:alpha_i_ell} gives:
\[ k (\alpha_{i\ell}-y_{i\ell}) + \log\alpha_{i\ell} - \frac{1}{m}\sum_{r=1}^m \log\alpha_{ir} = 0. \]
Rearranging gives \( k\left(\alpha_{i\ell}-y_{i\ell}\right) + \log\alpha_{i\ell} = \sum_{r=1}^{m}\log\alpha_{ir}/m \). This, along with \(\sum_{\ell}\alpha_{i\ell}=1\) and \(\alpha_{i\ell}>0\), is precisely the system of equations~\cref{eq:alpha_system_2trans} stated in~\cref{thm:2_transitive_multiblock}. This confirms that the unique, interior maximizer \(\boldsymbol{\alpha}^*\) is indeed determined by this system.

\paragraph{Sufficiency.}
Set \(\boldsymbol{\alpha}_i = \boldsymbol{\alpha}_i^*\), thereby defining \(A\) and \(C=A-Y\) as in the theorem. The forms \(W^*=\sqrt{E_W/(m-1)} UQ^{\top}\) and \(H^*=-\sqrt{E_H/(m-1)} QV^{\top}\) satisfy the constraint in~\cref{eq:obj_function_constrained}. The product is \(W^*H^* = -kC\). The conditions for the lower bound to be met (equality in trace bound and Jensen's inequality, implying \(\boldsymbol{\sigma}(W^*H^*)=A\)) are verified analogously to the detailed sufficiency proof for the single-block symmetric group case, using the system~\cref{eq:alpha_system_2trans} for \(\boldsymbol{\alpha}_i^*\). Thus \((W^*,H^*)\) is globally optimal.

\paragraph{Necessity.}
If \((W,H)\) is an optimal solution, all inequalities in deriving the lower bound become equalities. Thus, \(\boldsymbol{\sigma}(WH)=A\). The KKT conditions remain \(C H^\top + \lambda_1 W = 0\) and \(W^\top C + \lambda_2 H = 0\).
\Cref{lem:nuclear_frobenius_combined} (equality condition) implies \(W=U_W\Sigma_W Q^\top, H=Q\Sigma_H V_H^\top\) with \(\Sigma_W=\eta \Sigma_H\).
\Cref{lem:von_neumann_rectangular} implies that \(C\) and \(WH\) share singular directions \(U,V\). We can write
\[ WH = U\Sigma_{WH}V^\top, \]
where \(\Sigma_{WH} = \mathrm{diag}(\mu_1, \dots, \mu_{m-1})\) contains the singular values of \(WH\).

Now, we establish that \(\Sigma_{WH}\) is a scalar multiple of an identity matrix. The theorem's condition that at least one \(\boldsymbol{y}_i\) is nonuniform ensures that \(\boldsymbol{v}_i = \boldsymbol{\alpha}_i^* - \boldsymbol{y}_i \neq \boldsymbol{0}_m\) for at least one block \(i\). This, in turn, guarantees that the overall singular value \(\gamma\) of \(C\) is strictly positive (\(\gamma > 0\)). 
Since \(\boldsymbol{\sigma}(WH)=A\), each block of \(\boldsymbol{\sigma}(WH)\) equals the corresponding orbit block \(A_i=[g\circ\boldsymbol{\alpha}_i^*]_{g\in\mathcal G_i}\). Hence the columns of \((WH)_i\) form the same \(\mathcal G_i\)-orbit up to the softmax shift invariance, as detailed in the proof of~\cref{lem:orbit_S_softmax_2_transitive}. Let \(\rho_i(g)\) denote the permutation matrix satisfying \(\rho_i(g)\boldsymbol{x}=g\circ\boldsymbol{x}\) for this block action.
Write the right singular factor blockwise as
\[
V^\top=[V_1^\top|\cdots|V_{n_{\mathrm{blk}}}^\top],
\]
so that
\[
C_i = U(\gamma I_{m-1})V_i^\top,
\qquad
(WH)_i = U\Sigma_{WH}V_i^\top.
\]
For each \(i\) and \(g\in \mathcal G_i\), let \(\Lambda_i(g)\) be the permutation matrix on the \(|\mathcal G_i|\) columns of block \(i\) induced by the group action on the orbit indexing. Then the orbit structure gives
\[
\rho_i(g)C_i = C_i\Lambda_i(g),
\qquad
\rho_i(g)(WH)_i = (WH)_i\Lambda_i(g).
\]
Let
\[
R_{g,i}:=U^\top \rho_i(g)U.
\]
Substituting the SVD forms of \(C_i\) and \((WH)_i\) yields
\[
R_{g,i}(\gamma I_{m-1})V_i^\top = (\gamma I_{m-1})V_i^\top \Lambda_i(g),
\qquad
R_{g,i}\Sigma_{WH}V_i^\top = \Sigma_{WH}V_i^\top \Lambda_i(g).
\]
Since \(\gamma>0\), the first identity simplifies to
\[
R_{g,i}V_i^\top = V_i^\top \Lambda_i(g).
\]
Now choose a block \(i_0\) with \(C_{i_0}\neq 0\); such a block exists because \(\gamma>0\).
For this block, \(\lambda'_{i_0}>0\), hence
\[
C_{i_0}C_{i_0}^\top=\lambda'_{i_0} \left(I_m-\frac1m\boldsymbol 1_m\boldsymbol 1_m^\top\right)
\]
has rank \(m-1\). Therefore \(C_{i_0}\) has rank \(m-1\), and since
\[
C_{i_0}=U(\gamma I_{m-1})V_{i_0}^\top,
\]
the matrix \(V_{i_0}^\top\) has full row rank \(m-1\).
Using the two displayed identities for \(i=i_0\), we obtain
\[
R_{g,i_0}\Sigma_{WH}V_{i_0}^\top
=
\Sigma_{WH}V_{i_0}^\top \Lambda_{i_0}(g)
=
\Sigma_{WH}R_{g,i_0}V_{i_0}^\top,
\]
hence
\[
\bigl(R_{g,i_0}\Sigma_{WH}-\Sigma_{WH}R_{g,i_0}\bigr)V_{i_0}^\top=0.
\]
Because \(V_{i_0}^\top\) has full row rank, this implies
\[
R_{g,i_0}\Sigma_{WH}=\Sigma_{WH}R_{g,i_0}
\qquad\text{for all }g\in \mathcal G_{i_0}.
\]
Thus \(\Sigma_{WH}\) lies in the commutant of the representation
\(\{R_{g,i_0}\}_{g\in \mathcal G_{i_0}}\).
By~\cref{lem:irreducibility_2_transitive_action}, this representation is irreducible, and then~\cref{lem:two_orbit_matrix_form} implies that its commutant consists only of scalar multiples of the identity. Therefore
\[
\Sigma_{WH}=k_0 I_{m-1}
\]
for some \(k_0\ge 0\).
The remainder of the proof follows the algebraic steps analogous to those in~\cref{thm:single_block_perm}. The product forms in~\cref{eq:opt_products_2trans} are direct consequences.
\end{proof}

\section{Auxiliary Lemmas}

\begin{lemma}[von Neumann Trace Inequality]
  \label{lem:von_neumann_rectangular}
  Let \(C \in \mathbb{R}^{m \times n}\) and \(X \in \mathbb{R}^{m \times n}\) be arbitrary real matrices. Let \(s_i(C)\) and \(s_i(X)\) denote their singular values ordered nonincreasingly. Then:
  \[
  \left|\mathrm{tr}\left(C^\top X\right)\right|
  \le
  \sum_{i=1}^{\min(m,n)} s_i(C)s_i(X).
  \]
  Equality in the absolute-value bound holds if and only if \(C\) and \(X\) share singular vectors up to a common sign~\citep{carlsson2021neumann}, i.e., there exist \(\varepsilon \in \{-1,1\}\) and orthogonal matrices \(U \in \mathbb{R}^{m \times m}\), \(V \in \mathbb{R}^{n \times n}\) such that:
  \[
  C = U D_C V^\top, \quad
  \varepsilon X = U D_X V^\top,
  \]
  where \(D_C,D_X\in\mathbb{R}^{m\times n}\) are rectangular diagonal matrices with \((D_C)_{ii}=s_i(C)\) and \((D_X)_{ii}=s_i(X)\) for \(1\le i\le \min(m,n)\), and all other entries zero.
\end{lemma}

\begin{lemma}[Nuclear-Frobenius Norm Inequality~\citep{fazel2002matrix}]
  \label{lem:nuclear_frobenius_combined}
  Let \(W \in \mathbb{R}^{m \times d}\) and \(H \in \mathbb{R}^{d \times n}\) be nonzero matrices. Then
  \[
  \|W H\|_{*} \le \|W\|_{F} \|H\|_{F}.
  \]
  Moreover, \(\|W H\|_{*} = \|W\|_{F} \|H\|_{F}\) if and only if
  \begin{itemize}
  \item \(W\) and \(H\) share their singular vectors. That is, there exist SVDs
  \[
  W = U \Sigma_W V^\top, \qquad
  H = V \Sigma_H Q^\top
  \]
  for some common factor \(V \in \mathbb{R}^{d \times d}\),
  \item \(W\) and \(H\) have the same rank \(r\), and their nonzero singular values are proportional: there exists \(k>0\) such that
  \[
  s_i(W) = k s_i(H), \qquad i = 1, \dots, r.
  \]
  \end{itemize}
\end{lemma}

\begin{proof}[Proof of~\cref{lem:nuclear_frobenius_combined}]
The nuclear norm \( \|WH\|_* \) can be expressed as an optimization problem:
  \[
  \|WH\|_* = \max_{\|B\|_{\mathrm{op}} \le 1} \mathrm{tr}(B^\top WH),
  \]
  where \( \|B\|_{\mathrm{op}} \) is the operator norm (maximum singular value of \( B \)). Applying the Cauchy--Schwarz inequality for the Frobenius inner product gives
  \[
  \mathrm{tr}\left(B^\top W H\right) = \mathrm{tr}\left(H B^\top W\right)
  =
  \langle HB^\top,W^\top\rangle_F
  \le \|H B^\top\|_{F} \|W\|_{F}.
  \]
  Moreover, since \(\|B\|_{\mathrm{op}} \le 1\), multiplication by
  \(B^\top\) is a contraction in Euclidean norm. Applying this row by row to
  \(H\), it follows that
  \[
  \|H B^\top\|_{F} \le \|H\|_{F}   \|B^\top\|_{\mathrm{op}} \le \|H\|_{F}.
  \]
  Therefore,
  \[
  \mathrm{tr}\left(H B^\top W\right) \le \|H\|_{F}   \|W\|_{F}.
  \]
  Taking the maximum over all \(B\) with \(\|B\|_{\mathrm{op}} \le 1\), we conclude
  \[
  \|W H\|_{*} \le \|W\|_{F}    \|H\|_{F}.
  \]
  \paragraph{Necessity.}
  Conversely, assume that
  \[
  \|WH\|_* = \|W\|_{F}   \|H\|_{F}.
  \]
  Let
  \[
  WH = U_r \Sigma_r Q_r^\top
  \]
  be a compact SVD, where \(r = \mathrm{rank}(WH)\), \(U_r \in \mathbb{R}^{m \times r}\), \(Q_r \in \mathbb{R}^{n \times r}\), and \(\Sigma_r = \mathrm{diag}(s_1(WH), \dots, s_r(WH))\). Then
\(B_* = U_r Q_r^\top\) satisfies \(\|B_*\|_{\mathrm{op}} \le 1\) and is a maximizer in the nuclear-norm dual formula. Hence
  \[
\|WH\|_* = \mathrm{tr}(B_*^\top WH)
  = \mathrm{tr}(Q_r U_r^\top WH)
  = \mathrm{tr}(H Q_r U_r^\top W)
  = \langle W^\top, H Q_r U_r^\top \rangle_F.
  \]
  By the Cauchy--Schwarz inequality:
  \[
  \langle W^\top, H Q_r U_r^\top \rangle_F \le
  \|W\|_F   \|H Q_r U_r^\top \|_F.
  \]
  Combining this with the contraction bound used above gives
  \[
  \|WH\|_*
  \le \|W\|_F   \|H Q_r U_r^\top\|_F
  \le \|W\|_F   \|H\|_F
  =
  \|WH\|_*.
  \]
  Since \(W\neq 0\) and \(H\neq 0\), equality holds throughout. Therefore,
  \[
  W^\top = k   H Q_r U_r^\top,
  \]
  for some scalar \( k > 0 \); positivity follows because the common inner product equals the positive number \(\|WH\|_*\). Moreover,
  \[
  \|H Q_r U_r^\top \|_F = \|H\|_F.
  \]
  Now,
  \[
  \|H Q_r U_r^\top \|_F^2
  = \mathrm{tr}(U_r Q_r^\top H^\top H Q_r U_r^\top)
  = \mathrm{tr}(Q_r^\top H^\top H Q_r)
  = \|H Q_r Q_r^\top\|_F^2.
  \]
  Let \(P = Q_r Q_r^\top\), the orthogonal projector onto \(\mathrm{span}(Q_r)\). Since \(P\) is an orthogonal projector,
  \[
  \|H\|_F^2 = \|HP\|_F^2 + \|H(I-P)\|_F^2.
  \]
  Because \(\|HP\|_F = \|H\|_F\), it follows that
  \[
  H(I-P)=0.
  \]
  Hence the row space of \(H\) is contained in \(\mathrm{span}(Q_r)\). Therefore,
  \[
  \mathrm{rank}(H) \le r.
  \]
  On the other hand, \(r = \mathrm{rank}(WH) \le \mathrm{rank}(H)\), so \(\mathrm{rank}(H)=r\), and consequently
  \[
  \operatorname{range}(H^\top)=\mathrm{span}(Q_r).
  \]
  Let
  \[
  H = \widetilde U_H D_H V_H^\top
  \]
  be a compact SVD of \(H\), where \(\widetilde U_H \in \mathbb{R}^{d \times r}\), \(V_H \in \mathbb{R}^{n \times r}\), and
  \[
  D_H = \mathrm{diag}(s_1(H), \dots, s_r(H)).
  \]
  Since \(\mathrm{span}(V_H)=\operatorname{range}(H^\top)=\mathrm{span}(Q_r)\), there exists an orthogonal matrix \(O \in \mathbb{R}^{r \times r}\) such that \(V_H = Q_r O\). Hence
  \[
  H = \widetilde U_H D_H O^\top Q_r^\top.
  \]
  Substituting this into \(W^\top = k H Q_r U_r^\top\) gives
  \[
  W^\top = k   \widetilde U_H D_H O^\top U_r^\top,
  \]
  and therefore
  \[
  W = (U_r O)   (k D_H)   \widetilde U_H^\top.
  \]
  This is a compact SVD of \(W\). Thus the right singular vectors of \(W\) are the left singular vectors of \(H\), \(W\) and \(H\) have the same rank \(r\), and
  \[
  s_i(W) = k  s_i(H), \qquad i = 1, \dots, r.
  \]
  Extending these compact SVDs to full SVDs yields
  \[
  W = U \Sigma_W V^\top, \qquad H = V \Sigma_H Q^\top
  \]
  for some common factor \(V \in \mathbb{R}^{d \times d}\), as claimed.

  \paragraph{Sufficiency.}
  Conversely, suppose
  \[
  W = U \Sigma_W V^\top, \qquad H = V \Sigma_H Q^\top,
  \]
  and suppose that \(W\) and \(H\) have the same rank \(r\), and that there exists \(k>0\) such that
  \[
  s_i(W) = k s_i(H), \qquad i = 1, \dots, r.
  \]
  Let \(U_r\), \(V_r\), and \(Q_r\) denote the columns corresponding to the nonzero singular values. Then
  \[
  W = U_r D_W V_r^\top, \qquad H = V_r D_H Q_r^\top,
  \]
  where
  \[
  D_W = \mathrm{diag}(s_1(W), \dots, s_r(W)), \qquad
  D_H = \mathrm{diag}(s_1(H), \dots, s_r(H)),
  \]
  and \(D_W = k D_H\). Therefore,
  \[
  WH = U_r D_W D_H Q_r^\top = U_r (k D_H^2) Q_r^\top.
  \]
  Hence the nonzero singular values of \(WH\) are \(k s_i(H)^2\), so
  \[
  \|WH\|_* = k \sum_{i=1}^r s_i(H)^2 = k \|H\|_F^2.
  \]
  Moreover,
  \[
  \|W\|_F^2 = \sum_{i=1}^r s_i(W)^2 = k^2 \sum_{i=1}^r s_i(H)^2 = k^2 \|H\|_F^2,
  \]
  and therefore \(\|W\|_F = k \|H\|_F\). Consequently,
  \[
  \|WH\|_* = \|W\|_F \|H\|_F.
  \]
  This completes the proof.
\end{proof}

\begin{lemma}[Strict Convexity of Softmax Cross-Entropy on the Zero-Sum Subspace]
\label{lem:softmax_ce_strict_convexity_V0}
Let \(f(\boldsymbol u) \coloneqq \mathcal L(\sigma(\boldsymbol u), \boldsymbol p)\)
with \(\boldsymbol p\in\Delta^{m-1}\) fixed and \(\boldsymbol q \coloneqq \sigma(\boldsymbol u)\).
Then \(f\) is convex on \(\mathbb R^m\) with Hessian
\[
\nabla^2 f(\boldsymbol u)   =   \mathrm{diag}(\boldsymbol q) - \boldsymbol q \boldsymbol q^\top \succeq 0.
\]
Moreover, \(f\) is strictly convex along every direction in the zero-sum subspace
\[
\mathcal V_0 \coloneqq \{\boldsymbol v\in\mathbb R^m : \boldsymbol{1}_m^\top \boldsymbol v = 0 \}.
\]
Equivalently, for every nonzero \(\boldsymbol v\in \mathcal V_0\),
\(\boldsymbol v^\top \nabla^2 f(\boldsymbol u)   \boldsymbol v > 0\).
\end{lemma}

\begin{proof}[Proof of~\cref{lem:softmax_ce_strict_convexity_V0}]
First, we compute the gradient of \(f(\boldsymbol{u})\),
\[ \frac{\partial f}{\partial u_j} = \frac{\partial}{\partial u_j} \left( \log\left(\sum_{k=1}^m e^{u_k}\right) \right) - \frac{\partial}{\partial u_j} \left( \sum_{k=1}^m p_k u_k \right). \]
The derivative of the log-sum-exp term is
\[ \frac{\partial}{\partial u_j} \log\left(\sum_{k=1}^m e^{u_k}\right) = \frac{1}{\sum_{k=1}^m e^{u_k}} \cdot e^{u_j} = (\sigma(\boldsymbol{u}))_j. \]
The derivative of the linear term is simply \(p_j\).
Letting \(\boldsymbol{q} = \sigma(\boldsymbol{u})\), the gradient vector is
\[ \nabla_{\boldsymbol{u}} f(\boldsymbol{u}) = \boldsymbol{q} - \boldsymbol{p}. \]
Next, we compute the Hessian matrix, \(H_f = \nabla^2_{\boldsymbol{u}} f\), whose entries are \( (H_f)_{ij} = \partial^2 f/(\partial u_i \partial u_j) \). This is equivalent to finding the Jacobian matrix of the gradient \(\nabla_{\boldsymbol{u}} f = \boldsymbol{q} - \boldsymbol{p}\). Since \(\boldsymbol{p}\) is a constant vector, the Hessian is simply the Jacobian of \(\boldsymbol{q} = \sigma(\boldsymbol{u})\).
The partial derivative of a component \(q_j\) with respect to a logit \(u_i\) is
\[ \frac{\partial q_j}{\partial u_i} = \frac{\partial}{\partial u_i}\left( \frac{e^{u_j}}{\sum_k e^{u_k}} \right). \]
If \(i = j\),
    \[ \frac{\partial q_j}{\partial u_j} = \frac{e^{u_j}\left(\sum_k e^{u_k}\right) - e^{u_j}(e^{u_j})}{\left(\sum_k e^{u_k}\right)^2} = \frac{e^{u_j}}{\sum_k e^{u_k}} - \left(\frac{e^{u_j}}{\sum_k e^{u_k}}\right)^2 = q_j - q_j^2 = q_j(1-q_j). \]
If \(i \neq j\),
    \[ \frac{\partial q_j}{\partial u_i} = \frac{0 \cdot (\sum_k e^{u_k}) - e^{u_j}(e^{u_i})}{\left(\sum_k e^{u_k}\right)^2} = -\frac{e^{u_j}}{\sum_k e^{u_k}}\frac{e^{u_i}}{\sum_k e^{u_k}} = -q_i q_j. \]
Combining these results, the Hessian matrix is
\[ H_f = \nabla^2_{\boldsymbol{u}} f = \mathrm{diag}(\boldsymbol{q}) - \boldsymbol{q}\boldsymbol{q}^\top. \]
To prove that \(f(\boldsymbol{u})\) is convex, it suffices to show that \(H_f\) is positive semidefinite. For any \(\boldsymbol{v} \in \mathbb{R}^m\),
\[
\boldsymbol{v}^\top H_f \boldsymbol{v} = \boldsymbol{v}^\top (\mathrm{diag}(\boldsymbol{q}) - \boldsymbol{q}\boldsymbol{q}^\top) \boldsymbol{v} = \boldsymbol{v}^\top \mathrm{diag}(\boldsymbol{q}) \boldsymbol{v} - \boldsymbol{v}^\top \boldsymbol{q}\boldsymbol{q}^\top \boldsymbol{v}.
\]
Expanding both terms gives
\[
\boldsymbol{v}^\top H_f \boldsymbol{v} = \sum_{i=1}^m q_i v_i^2 - \left(\sum_{i=1}^m q_i v_i\right)^2.
\]
Thus \(\boldsymbol{v}^\top H_f \boldsymbol{v}\) is exactly the variance of the discrete random variable taking values \(\{v_1,\dots,v_m\}\) with weights \(\{q_1,\dots,q_m\}\):
\[
\boldsymbol{v}^\top H_f \boldsymbol{v} = \mathbb{E}[v^2] - (\mathbb{E}[v])^2 = \mathrm{Var}_q(v).
\]
Equivalently,
\[
\boldsymbol{v}^\top H_f \boldsymbol{v} = \mathrm{Var}_q(v) = \mathbb{E} \left[(v-\mathbb{E}[v])^2\right] \ge 0.
\]
Hence \(H_f \succeq 0\), so \(f\) is convex on \(\mathbb{R}^m\).

Now let \(\boldsymbol{v}\in \mathcal V_0\) be nonzero. To prove strict convexity on \(\mathcal V_0\), it remains to show that \(\mathrm{Var}_q(v) > 0\). Since \(\boldsymbol{q} = \sigma(\boldsymbol{u})\), every coordinate satisfies \(q_i > 0\). Because \(q_i>0\) for every \(i\), \(\mathrm{Var}_q(v)=0\) holds if and only if all coordinates \(v_i\) are equal. Therefore
\[ \boldsymbol{v}^\top H_f \boldsymbol{v} = 0 \quad \iff \quad v_1 = v_2 = \cdots = v_m. \]
Thus \(\boldsymbol{v}^\top H_f \boldsymbol{v} = 0\) if and only if \(\boldsymbol{v}\) is a multiple of \(\boldsymbol{1}_m\). But a nonzero vector in
\[
\mathcal V_0 = \{\boldsymbol{v} \in \mathbb{R}^m : \boldsymbol{1}_m^\top \boldsymbol{v} = 0\}
\]
cannot be a multiple of \(\boldsymbol{1}_m\). Therefore \(\boldsymbol{v}^\top H_f \boldsymbol{v} > 0\) for every nonzero \(\boldsymbol{v} \in \mathcal V_0\). Hence the Hessian is positive definite on \(\mathcal V_0\), and \(f\) is strictly convex along that subspace.
\end{proof}

\begin{lemma}[Weyl's Inequality for Eigenvalues of a Sum~\citep{horn2012matrix}]
\label{lem:weyl_inequality}
Let \(A\) and \(B\) be \(n \times n\) Hermitian matrices. Let the eigenvalues $\lambda_i$ for each matrix (\(A\), \(B\), and \(A+B\)) be real and ordered nonincreasingly.
Then for any indices \(i,j\) with \(1\le i,j\le n\),
\[
\lambda_{i+j-1}(A+B)\le \lambda_i(A)+\lambda_j(B)\quad \text{if } i+j-1\le n,
\]
and
\[
\lambda_i(A)+\lambda_j(B)\le \lambda_{i+j-n}(A+B)\quad \text{if } i+j-n\ge 1.
\]
\end{lemma}

\begin{lemma}[Fourier Diagonalization of Circulant Matrices]
\label{lem:fourier_diagonalization_circulant}
Let \(F_m\in\mathbb C^{m\times m}\) denote the unitary Discrete Fourier Transform matrix with entries
\[
(F_m)_{k\ell}=m^{-1/2}\exp(-2\pi \mathrm{i}(k-1)(\ell-1)/m),
\qquad 1\le k,\ell\le m.
\]
A matrix \(A\in\mathbb C^{m\times m}\) is circulant with respect to the shift matrix \(\Pi\) in~\cref{eq:cyclic} if and only if there exists a diagonal matrix \(D\in\mathbb C^{m\times m}\) such that
\[
A=F_m^{\mathsf{H}}DF_m.
\]
In particular, every circulant matrix is normal, and its singular values are the absolute values of its eigenvalues.
\end{lemma}

\begin{proof}
Let \(\omega=\exp(-2\pi \mathrm{i}/m)\), and let \(\Omega\) be the diagonal matrix with entries \(\Omega_{kk}=\omega^{k-1}\) for \(k=1,\ldots,m\). With the convention for \(\Pi\) in~\cref{eq:cyclic}, the columns of \(F_m^{\mathsf{H}}\) are eigenvectors of \(\Pi\), so
\[
\Pi=F_m^{\mathsf{H}}\Omega F_m.
\]
If \(A\) is circulant with first column \(\boldsymbol a=(a_1,\ldots,a_m)^\top\), then
\[
A_{:,j}=\Pi^{j-1}\boldsymbol a,\qquad j=1,\ldots,m,
\]
and hence
\[
A=\sum_{r=1}^{m}a_r\Pi^{r-1}
  =F_m^{\mathsf{H}}\left(\sum_{r=1}^{m}a_r\Omega^{r-1}\right)F_m,
\]
which has the required form with \(D=\sum_{r=1}^{m}a_r\Omega^{r-1}\). Conversely, if \(A=F_m^{\mathsf{H}}DF_m\) with \(D=\mathrm{diag}(\lambda_1,\ldots,\lambda_m)\), then the inverse Discrete Fourier Transform gives coefficients \(a_1,\ldots,a_m\) such that
\[
D=\sum_{r=1}^{m}a_r\Omega^{r-1}.
\]
Therefore \(A=\sum_{r=1}^{m}a_r\Pi^{r-1}\), so \(A\) is circulant. Finally, the displayed diagonalization is unitary, hence every circulant matrix is normal, and the singular values are the absolute values of the diagonal entries of \(D\), equivalently the absolute values of the eigenvalues of \(A\).
\end{proof}

\begin{lemma}[Symmetry of SVD Factors for Circulant Matrices]
\label{lem:svd_factor_symmetry}
Let \(Z \in \mathbb{R}^{m \times m}\) be a real circulant matrix, and let \(Z=U\Sigma V^\top\) be any real singular value decomposition of \(Z\), where \(U, V \in \mathbb{R}^{m \times m}\) are orthogonal and \(\Sigma \in \mathbb{R}^{m \times m}\) is the diagonal matrix of singular values.
Then the matrices \(U\Sigma U^\top\) and \(V\Sigma V^\top\) are both circulant, and are equal to each other and to the principal square root of \(ZZ^\top\):
\[ U\Sigma U^\top = V\Sigma V^\top = (ZZ^\top)^{1/2}. \]
\end{lemma}

\begin{proof}[Proof of~\cref{lem:svd_factor_symmetry}]
By~\cref{lem:fourier_diagonalization_circulant}, for two real circulant matrices \(A\) and \(B\), we can write
\[ A = F_m^{\mathsf{H}} D_A F_m \quad \text{and} \quad B = F_m^{\mathsf{H}} D_B F_m, \]
where \(D_A\) and \(D_B\) are diagonal matrices containing the eigenvalues of \(A\) and \(B\), respectively.
Consider the product \(AB\):
\[ AB = (F_m^{\mathsf{H}} D_A F_m)(F_m^{\mathsf{H}} D_B F_m) = F_m^{\mathsf{H}} (D_A D_B) F_m. \]
Now consider the product in the reverse order, \(BA\):
\[ BA = (F_m^{\mathsf{H}} D_B F_m)(F_m^{\mathsf{H}} D_A F_m) = F_m^{\mathsf{H}} (D_B D_A) F_m. \]
Because diagonal matrices always commute (\(D_A D_B = D_B D_A\)), we can conclude that
\[ AB = BA, \]
which shows that \(A\) and \(B\) commute. Since both \(AB\) and \(BA\) are expressed in the form \(F_m^{\mathsf{H}} D F_m\) for some diagonal matrix \(D\), it follows that \(AB\) is also a circulant matrix.
Since both \(Z\) and \(Z^\top\) are circulant matrices, they must commute. Let us define the matrix \(P\) as this product: \(P := ZZ^\top = Z^\top Z\), which is circulant because it is the product of two circulant matrices.
We can write \(P\) in two ways using the SVD of \(Z = U\Sigma V^\top\):
\[ P = ZZ^\top = (U\Sigma V^\top)(U\Sigma V^\top)^\top = U\Sigma V^\top V \Sigma^\top U^\top = U\Sigma^2 U^\top = (U\Sigma U^\top)^2. \]
\[ P = Z^\top Z = (U\Sigma V^\top)^\top(U\Sigma V^\top) = V\Sigma^\top U^\top U \Sigma V^\top = V\Sigma^2 V^\top = (V\Sigma V^\top)^2. \]
A symmetric positive semidefinite matrix has a unique symmetric positive semidefinite square root. Since both \(U\Sigma U^\top\) and \(V\Sigma V^\top\) are symmetric, PSD, and square to the same matrix \(P\), they must be equal to this unique square root, which we denote \(P^{1/2}\).
\[ U\Sigma U^\top = V\Sigma V^\top = P^{1/2}. \]
Since \(P\) is a real symmetric circulant matrix, \cref{lem:fourier_diagonalization_circulant} gives its unitary DFT diagonalization. Let the eigendecomposition of \(P\) be:
\[
P = F_m^{\mathsf{H}} D_P F_m,
\]
where \(D_P = \mathrm{diag}(\lambda_1, \ldots, \lambda_m)\) is the diagonal matrix of the real eigenvalues of \(P\). As \(P\) is positive semidefinite, all \(\lambda_i \ge 0\). The unique symmetric PSD square root of \(P\) is constructed using this same diagonalization:
\[
P^{1/2} = F_m^{\mathsf{H}} D_P^{1/2} F_m, \quad \text{where} \quad D_P^{1/2} = \mathrm{diag}(\sqrt{\lambda_1}, \ldots, \sqrt{\lambda_m}).
\]
This resulting matrix \(P^{1/2}\) is also circulant because any matrix that can be written in the form \(F_m^{\mathsf{H}} D F_m\) for some diagonal matrix \(D\) (here, \(D = D_P^{1/2}\)) is circulant.
\end{proof}

\begin{lemma}[SVD Factor Structure for Block-Circulant Matrices]
\label{lem:svd_factor_symmetry_block}
Let \(Z \in \mathbb{R}^{m \times n}\) be the concatenation of \(n_{\mathrm{blk}}\) real circulant matrices, \(Z = [C_{1}, C_{2}, \cdots, C_{n_{\mathrm{blk}}}]\), and \(n=n_{\mathrm{blk}}m\). Let the thin singular value decomposition (SVD) of \(Z\) be
\[ Z = U\Sigma V^\top, \]
where \(U \in \mathbb{R}^{m \times r}\) and \(V \in \mathbb{R}^{n \times r}\) are matrices with orthonormal columns (\(r = \mathrm{rank}(Z)\)), and \(\Sigma \in \mathbb{R}^{r \times r}\) is the diagonal matrix of positive singular values.
Then the \(m \times m\) matrix \(U\Sigma U^\top\) is a circulant matrix.
\end{lemma}

\begin{proof}
First, we compute the matrix \(P = ZZ^\top\). Given the block structure of \(Z\), the product is:
\[
P = ZZ^\top = \left( [C_1 | \cdots | C_{n_{\mathrm{blk}}}] \right) \begin{pmatrix} C_1^\top \\ \vdots \\ C_{n_{\mathrm{blk}}}^\top \end{pmatrix} = \sum_{i=1}^{n_{\mathrm{blk}}} C_i C_i^\top.
\]
Each block \(C_i\) is a real circulant matrix. As established in~\cref{lem:svd_factor_symmetry}, each term \(C_i C_i^\top\) is a \(m \times m\) circulant matrix. Since the set of circulant matrices is closed under addition, their sum \(P = \sum_i C_i C_i^\top\) is also a circulant matrix.
Furthermore, \(P\) is symmetric and positive semidefinite.
Next, we express \(P\) in terms of the SVD of \(Z = U\Sigma V^\top\). Note that \(V\) is a rectangular matrix with orthonormal columns, so \(V^\top V = I_r\).
\[
P = ZZ^\top = (U\Sigma V^\top)(U\Sigma V^\top)^\top = U\Sigma V^\top V \Sigma^\top U^\top = U\Sigma^2 U^\top.
\]
Now, consider the matrix \(A := U\Sigma U^\top\):
\[ A^2 = (U\Sigma U^\top)(U\Sigma U^\top) = U\Sigma(U^\top U)\Sigma U^\top = U\Sigma^2 U^\top = P. \]
The matrix \(A\) is also symmetric and positive semidefinite.
A symmetric positive semidefinite matrix \(P\) has a unique symmetric positive semidefinite square root, denoted \(P^{1/2}\). Since \(A = U\Sigma U^\top\) is a symmetric PSD matrix whose square is \(P\), it must be this unique square root:
\[ U\Sigma U^\top = P^{1/2} = (ZZ^\top)^{1/2}. \]
As established in the proof of~\cref{lem:svd_factor_symmetry}, the unique symmetric PSD square root of a symmetric PSD circulant matrix is itself circulant.
Therefore, the matrix \(U\Sigma U^\top = (ZZ^\top)^{1/2}\) must be a circulant matrix.
\end{proof}

\begin{lemma}[Gram Spectrum from 2-Transitive Group Actions]
\label{lem:2_transitive_svd}
Let \(m \ge 2\). Let \(\mathcal G\) be a 2-transitive group of permutations acting on the set of indices \(\mathcal X = \{1,2,\dots,m\}\). Let \(\boldsymbol{v} = (v_1,\cdots,v_m) \in \mathbb{C}^m\) be a nonzero vector such that its components sum to zero, i.e., \(\boldsymbol{1}_m^{\mathsf H}\boldsymbol{v} = 0\), ensuring \(\boldsymbol{v} \in \mathcal V_0\). For each \(g \in \mathcal G\), let \(C = [g\circ\boldsymbol{v}]_{g \in \mathcal G} \in \mathbb{C}^{m \times |\mathcal G|}\). Equivalently, \(C=[\rho(g)\boldsymbol v]_{g\in\mathcal G}\), where \(\rho(g)\) is the permutation matrix satisfying \(\rho(g)\boldsymbol{x}=g\circ\boldsymbol{x}\). Then \((CC^{\mathsf H})\boldsymbol{1}_m=\boldsymbol{0}_m\), and there exists a scalar \(\lambda\) such that \((CC^{\mathsf H})\boldsymbol{y}=\lambda \boldsymbol{y}\) for every \(\boldsymbol{y}\in \mathcal V_0\). Moreover,
\[
\mathrm{tr}(CC^{\mathsf H})=|\mathcal G|\sum_{k=1}^m |v_k|^2.
\]
Equivalently, the Gram operator acts as a scalar multiple of the identity on \(\mathcal V_0\), with
\[
\lambda=\frac{|\mathcal G|}{m-1}\sum_{k=1}^m |v_k|^2.
\]
Thus the nonzero singular values of \(C\), when one chooses an SVD compatible with this Gram decomposition, are all equal to
\[
\gamma=\sqrt{\lambda}
=\left(\frac{|\mathcal G|}{m-1}\sum_{k=1}^m |v_k|^2\right)^{1/2}.
\]
\end{lemma}

\begin{proof}[Proof of~\cref{lem:2_transitive_svd}]
For each \(g\in \mathcal G\), let \(\rho(g)\in\mathbb{C}^{m\times m}\) be the matrix realization of the action \(g\circ\). Since \(\rho\) is a homomorphism, \(\rho(h)\rho(g)=\rho(hg)\), and permutation matrices are unitary, i.e., \(\rho(g)^{\mathsf H}=\rho(g)^{-1}\). The singular values of \(C\) are the square roots of the eigenvalues of the Hermitian positive semidefinite matrix \(CC^{\mathsf H}\). Let \(M_0 = \boldsymbol{v}\boldsymbol{v}^{\mathsf H}\). Then
\[
CC^{\mathsf H} = \sum_{g \in \mathcal G} \rho(g) M_0 \rho(g)^{\mathsf H}.
\]
Given \(\boldsymbol{1}_m^{\mathsf H} \boldsymbol{v} = 0\), it follows that \(M_0 \boldsymbol{1}_m = \boldsymbol{v}(\boldsymbol{v}^{\mathsf H} \boldsymbol{1}_m) = \boldsymbol{v}(0) = \boldsymbol{0}_m\). Since \(\rho(g)\) is a permutation matrix, hence unitary, \(\rho(g)^{\mathsf H} \boldsymbol{1}_m = \boldsymbol{1}_m\). Therefore,
\[
(CC^{\mathsf H})\boldsymbol{1}_m
= \sum_{g \in \mathcal G} \rho(g) M_0 \rho(g)^{\mathsf H} \boldsymbol{1}_m
= \sum_{g \in \mathcal G} \rho(g) (M_0 \boldsymbol{1}_m)
= \sum_{g \in \mathcal G} \rho(g)\boldsymbol{0}_m
= \boldsymbol{0}_m.
\]
Thus, \(\boldsymbol{1}_m\) is an eigenvector of \(CC^{\mathsf H}\) with eigenvalue \(0\). This corresponds to one zero singular value of \(C\).
The subspace \(\mathcal V_0 = \{\boldsymbol{y} \in \mathbb{C}^m : \boldsymbol{1}_m^{\top}\boldsymbol{y} = 0\}\) is the \((m-1)\)-dimensional orthogonal complement of \(\mathrm{span}(\boldsymbol{1}_m)\). As stated in~\cref{lem:irreducibility_2_transitive_action}, \(\mathcal V_0\) is \(\mathcal G\)-invariant under the permutation representation \(\rho\). The vector \(\boldsymbol{v}\) itself belongs to \(\mathcal V_0\).

Next, we show that \(CC^{\mathsf H}\) commutes with all representation matrices \(\rho(h)\) for \(h \in \mathcal G\), i.e., \(\rho(h)(CC^{\mathsf H}) = (CC^{\mathsf H})\rho(h)\). First, consider
\[
\rho(h)(CC^{\mathsf H}) \rho(h)^{\mathsf H}
= \rho(h)\left(\sum_{g \in \mathcal G} \rho(g) M_0 \rho(g)^{\mathsf H}\right)\rho(h)^{\mathsf H}
= \sum_{g \in \mathcal G} \rho(h)\rho(g) M_0 \rho(g)^{\mathsf H}\rho(h)^{\mathsf H}.
\]
Since \(\rho\) is a representation, \(\rho(h)\rho(g) = \rho(hg)\). Also, as \(\rho(x)\) is unitary, \(\rho(x)^{\mathsf H} = (\rho(x))^{-1} = \rho(x^{-1})\), so \(\rho(g)^{\mathsf H}\rho(h)^{\mathsf H} = (\rho(h)\rho(g))^{\mathsf H} = (\rho(hg))^{\mathsf H}\). Hence
\[
\rho(h)(CC^{\mathsf H}) \rho(h)^{\mathsf H}
= \sum_{g \in \mathcal G} \rho(hg) M_0 (\rho(hg))^{\mathsf H}.
\]
Let \(g' = hg\). As \(g\) ranges over all elements of \(\mathcal G\), so does \(g'\). Therefore,
\[
\sum_{g' \in \mathcal G} \rho(g') M_0 \rho(g')^{\mathsf H} = CC^{\mathsf H},
\]
so \(\rho(h)(CC^{\mathsf H}) \rho(h)^{\mathsf H} = CC^{\mathsf H}\). Right-multiplying by \(\rho(h)\) and using \(\rho(h)^{\mathsf H}\rho(h)=I_m\), we get \(\rho(h)(CC^{\mathsf H}) = (CC^{\mathsf H})\rho(h)\).

Since \(\mathcal G\) is 2-transitive, \cref{lem:irreducibility_2_transitive_action} states that the representation of \(\mathcal G\) restricted to \(\mathcal V_0\) is irreducible. The matrix \(CC^{\mathsf H}\) maps \(\mathcal V_0\) to itself: indeed, \(CC^{\mathsf H}\) is Hermitian and \((CC^{\mathsf H})\boldsymbol{1}_m=\boldsymbol{0}_m\), so for any \(\boldsymbol{y}\in \mathcal V_0=\boldsymbol{1}_m^\perp\),
\[
\boldsymbol{1}_m^{\mathsf H}(CC^{\mathsf H})\boldsymbol{y}
=((CC^{\mathsf H})\boldsymbol{1}_m)^{\mathsf H}\boldsymbol{y}=0.
\]
Thus the commutation relation restricts to \(\mathcal V_0\). As \(CC^{\mathsf H}\) commutes with this irreducible representation of \(\mathcal G\) on \(\mathcal V_0\), we apply~\cref{lem:two_orbit_matrix_form}. This implies that \(CC^{\mathsf H}\) restricted to \(\mathcal V_0\) is a scalar multiple of the identity operator on \(\mathcal V_0\). Thus, for any \(\boldsymbol{y} \in \mathcal V_0\), \((CC^{\mathsf H})\boldsymbol{y} = \lambda \boldsymbol{y}\) for some scalar \(\lambda\). This means \(\lambda\) is an eigenvalue of \(CC^{\mathsf H}\) with multiplicity \(m-1\), and its eigenspace is \(\mathcal V_0\).

To determine \(\lambda\), we compute the trace of \(CC^{\mathsf H}\) in two ways. First, from its eigenvalues, \(0\) with multiplicity \(1\) and \(\lambda\) with multiplicity \(m-1\),
\[
\mathrm{tr}(CC^{\mathsf H}) = (1 \cdot 0) + ((m-1)\cdot \lambda) = (m-1)\lambda.
\]
Alternatively, from the expansion of \(CC^{\mathsf H}\),
\[
\mathrm{tr}(CC^{\mathsf H})
= \mathrm{tr}\left(\sum_{g \in \mathcal G} \rho(g) M_0 \rho(g)^{\mathsf H}\right)
= \sum_{g \in \mathcal G} \mathrm{tr}(\rho(g) M_0 \rho(g)^{\mathsf H}).
\]
Using cyclicity of trace, \(\mathrm{tr}(\rho(g) M_0 \rho(g)^{\mathsf H}) = \mathrm{tr}(M_0 \rho(g)^{\mathsf H}\rho(g)) = \mathrm{tr}(M_0 I_m) = \mathrm{tr}(M_0)\). The trace of \(M_0 = \boldsymbol{v}\boldsymbol{v}^{\mathsf H}\) is \(\mathrm{tr}(\boldsymbol{v}\boldsymbol{v}^{\mathsf H}) = \boldsymbol{v}^{\mathsf H}\boldsymbol{v} = \|\boldsymbol{v}\|_2^2 = \sum_{k=1}^m |v_k|^2\). So
\[
\mathrm{tr}(CC^{\mathsf H}) = \sum_{g \in \mathcal G} \|\boldsymbol{v}\|_2^2 = |\mathcal G| \|\boldsymbol{v}\|_2^2.
\]
Equating the two expressions for the trace gives \((m-1)\lambda = |\mathcal G| \|\boldsymbol{v}\|_2^2\), hence \(\lambda = (|\mathcal G|/(m-1))\|\boldsymbol{v}\|_2^2\). Since \(\boldsymbol{v} \neq \boldsymbol{0}_m\) and \(\mathcal G\) is nontrivial, indeed \(|\mathcal G| \ge 1\) and for 2-transitivity on \(m \ge 2\) elements \(|\mathcal G| \ge 2\), we have \(\lambda > 0\).

The nonzero eigenvalues of \(CC^{\mathsf H}\) on \(\mathcal V_0\) are thus all equal to \(\lambda\). The singular values of \(C\) associated with this Gram decomposition are the square roots of these eigenvalues, hence they are all equal to \(\gamma = \sqrt{\lambda}\):
\[
\gamma = \sqrt{\frac{|\mathcal G|}{m-1} \|\boldsymbol{v}\|_2^2}
= \sqrt{\frac{|\mathcal G|}{m-1}} \|\boldsymbol{v}\|_2.
\]
This completes the proof.
\end{proof}

\begin{lemma}[Orbit Structure of Softmax Columns for 2-Transitive Group Actions]
\label{lem:orbit_S_softmax_2_transitive}
Let \(m \ge 2\) and let \(\mathcal G\) be a 2-transitive permutation group on \(\mathcal X=\{1,2,\ldots,m\}\). For \(g\in \mathcal G\), let \(\rho(g)\in\mathbb{R}^{m\times m}\) be the permutation matrix satisfying \(\rho(g)\boldsymbol{x}=g\circ\boldsymbol{x}\). Let \(\boldsymbol{v}\in\mathbb{R}^m\) be nonzero with \(\boldsymbol{1}_m^{\top}\boldsymbol{v}=0\). Fix an enumeration \(\mathcal G=\{g_1,\ldots,g_{|\mathcal G|}\}\), and set \(C=[g_j\circ\boldsymbol{v}]_{j=1}^{|\mathcal G|}\in\mathbb{R}^{m\times |\mathcal G|}\). From~\cref{lem:2_transitive_svd}, choose the compatible SVD \(C=\gamma U V^{\top}\) whose left singular vectors form an orthonormal basis of \(\mathcal V_0=\{\boldsymbol{x}:\boldsymbol{1}_m^\top \boldsymbol{x}=0\}\). Thus \(U\in\mathbb{R}^{m\times(m-1)}\), \(V\in\mathbb{R}^{|\mathcal G|\times(m-1)}\) have orthonormal columns and \(\gamma=\sqrt{|\mathcal G|/(m-1)}\|\boldsymbol{v}\|_2\). For any diagonal \(\Sigma=\mathrm{diag}(s_1,\ldots,s_{m-1})\) with \(s_k\ge 0\), define \(Z=U\Sigma V^{\top}\in\mathbb{R}^{m\times |\mathcal G|}\), and let \(\Xi\in\mathbb{R}^{m\times|\mathcal G|}\) have columns \(\Xi_{:,j}=\sigma(Z_{:,j})\). Then the columns of \(\Xi\) form the orbit of a single vector under \(\mathcal G\), i.e., there exists \(\boldsymbol{y}_{\Xi} \in \mathbb{R}^m\) such that \(\Xi_{:,j} = g_j\circ\boldsymbol{y}_{\Xi}\) for all \(j=1,\ldots,|\mathcal G|\), if and only if \(\Sigma\) is a scalar multiple of the identity matrix \(I_{m-1}\).
\end{lemma}

\begin{proof}[Proof of~\cref{lem:orbit_S_softmax_2_transitive}]
All matrices are real up to the brief complexification step below. We use the fixed enumeration \(\mathcal G=\{g_1,\ldots,g_{|\mathcal G|}\}\) from the lemma statement.

\paragraph{Sufficiency.}
Assume \(\Sigma = sI_{m-1}\) for some scalar \(s \ge 0\). Then
\[
Z = U(sI_{m-1})V^{\top} = s(UV^{\top}).
\]
From the SVD \(C = \gamma U V^{\top}\), and since \(\gamma > 0\) because \(\boldsymbol{v} \neq \boldsymbol{0}_m\), we have \(UV^{\top} = C/\gamma\). Substituting gives
\[
Z = \frac{s}{\gamma}C.
\]
Let \(\boldsymbol{x}_Z = (s/\gamma)\boldsymbol{v}\). Since the columns of \(C\) are \(C_{:,j} = g_j\circ\boldsymbol{v}\), the columns of \(Z\) are
\[
Z_{:,j}
= \frac{s}{\gamma}(g_j\circ\boldsymbol{v})
= g_j\circ\left(\frac{s}{\gamma}\boldsymbol{v}\right)
= g_j\circ\boldsymbol{x}_Z.
\]
The \(j\)-th column of \(\Xi\) is \(\Xi_{:,j} = \sigma(Z_{:,j}) = \sigma(g_j\circ\boldsymbol{x}_Z)\). The softmax function is permutation equivariant: for any group element \(g\) and vector \(\boldsymbol{z}\), \(\sigma(g\circ\boldsymbol{z}) = g\circ\sigma(\boldsymbol{z})\). Thus
\[
\Xi_{:,j} = g_j\circ\sigma(\boldsymbol{x}_Z).
\]
Let \(\boldsymbol{y}_{\Xi} = \sigma(\boldsymbol{x}_Z)\). Then \(\Xi_{:,j} = g_j\circ\boldsymbol{y}_{\Xi}\). This shows that the columns of \(\Xi\) form the orbit of the single vector \(\boldsymbol{y}_{\Xi}\) under the action of \(\mathcal G\).

\paragraph{Necessity.}
Assume the columns of \(\Xi\) form the orbit of a single vector \(\boldsymbol{y}_{\Xi}\), i.e., \(\Xi_{:,j} = g_j\circ\boldsymbol{y}_{\Xi}\). Let \(j_e\) be the index such that \(g_{j_e}=e\), and set \(\boldsymbol{z}_{e}:=Z_{:,j_e}\). Then \(\boldsymbol{y}_{\Xi} = \sigma(\boldsymbol{z}_{e})\), since \(e\circ\boldsymbol{x}=\boldsymbol{x}\). So \(\Xi_{:,j} = g_j\circ\sigma(\boldsymbol{z}_{e})\). Using permutation equivariance, \(g_j\circ\sigma(\boldsymbol{z}_{e}) = \sigma(g_j\circ\boldsymbol{z}_{e})=\sigma(\rho(g_j)\boldsymbol{z}_{e})\). Thus
\[
\sigma(Z_{:,j}) = \sigma(g_j\circ\boldsymbol{z}_{e}).
\]
A property of the softmax function is that if \(\sigma(\boldsymbol{a}) = \sigma(\boldsymbol{b})\), then \(\boldsymbol{a} = \boldsymbol{b} + c\boldsymbol{1}_m\) for some scalar \(c\). Therefore, for each column \(j\), associated with \(g_j \in \mathcal G\),
\[
Z_{:,j} = g_j\circ\boldsymbol{z}_{e} + c_j\boldsymbol{1}_m
\]
for some scalar \(c_j\), which may depend on \(j\).

The columns of \(Z = U\Sigma V^{\top}\) lie in the column space of \(U\). Since the columns of \(U\) are a basis for \(\mathcal V_0 = \{\boldsymbol{y} \in \mathbb{R}^m : \boldsymbol{1}_m^\top\boldsymbol{y} = 0\}\), they sum to zero, i.e., \(\boldsymbol{1}_m^\top U = \boldsymbol{0}_{(m-1)}^\top\). Thus, for any column \(Z_{:,j}\) of \(Z\),
\[
\boldsymbol{1}_m^\top Z_{:,j}
= (\boldsymbol{1}_m^\top U)\Sigma (V^{\top})_{:,j}
= \boldsymbol{0}_{(m-1)}^\top \Sigma (V^{\top})_{:,j}
= 0.
\]
Applying this to \(Z_{:,j} = g_j\circ\boldsymbol{z}_{e} + c_j\boldsymbol{1}_m\), we get
\[
\boldsymbol{1}_m^\top Z_{:,j}
= \boldsymbol{1}_m^\top (g_j\circ\boldsymbol{z}_{e}) + \boldsymbol{1}_m^\top (c_j\boldsymbol{1}_m).
\]
Since \(g_j\circ\) permutes entries, the sum remains the same, and \(\boldsymbol{1}_m^\top \boldsymbol{z}_{e} = 0\), as \(\boldsymbol{z}_{e}\) is a column of \(Z\) and thus in \(\mathcal V_0\), this gives
\[
0 = \boldsymbol{1}_m^\top \boldsymbol{z}_{e} + c_j (\boldsymbol{1}_m^\top \boldsymbol{1}_m) = 0 + c_j m.
\]
Since \(m \ge 2\), this implies \(c_j=0\) for all \(j\). Therefore, \(Z_{:,j} = g_j\circ\boldsymbol{z}_{e}\) for all \(j=1,\ldots,|\mathcal G|\). This means the columns of \(Z\) form the orbit of the single vector \(\boldsymbol{z}_{e}\) under the action of \(\mathcal G\).

Now we show that this property of \(Z\) implies \(\Sigma = sI_{m-1}\). The condition \(Z_{:,j} = g_j\circ\boldsymbol{z}_{e}\) for all \(j=1,\ldots,|\mathcal G|\) implies an intertwining relation for \(Z\). For any \(h \in \mathcal G\), applying the matrix realization \(\rho(h)\) to all columns of \(Z\) results in a permutation of these columns:
\[
\rho(h)Z
= \rho(h)[g_j\circ\boldsymbol{z}_{e}]_{j=1}^{|\mathcal G|}
= [\rho(h)\rho(g_j)\boldsymbol{z}_{e}]_{j=1}^{|\mathcal G|}
= [\rho(hg_j)\boldsymbol{z}_{e}]_{j=1}^{|\mathcal G|}.
\]
This means there exists a permutation matrix \(\Lambda(h) \in \mathbb{R}^{|\mathcal G| \times |\mathcal G|}\), which permutes the columns of \(Z\) according to the permutation of the fixed enumeration induced by left multiplication by \(h\), such that
\[
\rho(h)Z = Z\Lambda(h).
\]
More precisely, \(\Lambda(h)\) is the column permutation matrix induced by left multiplication on the fixed enumeration of \(\mathcal G\). The same \(\Lambda(h)\) also satisfies \(\rho(h)C=C\Lambda(h)\), because the columns of \(C\) are \(g_j\circ\boldsymbol{v}\) and \(\rho(h)(g_j\circ\boldsymbol{v})=(hg_j)\circ\boldsymbol{v}\).

Substituting \(Z=U\Sigma V^{\top}\) gives
\[
\rho(h)U\Sigma V^{\top} = U\Sigma V^{\top} \Lambda(h).
\]
Left-multiplying by \(U^{\top}\), we obtain
\[
U^{\top} \rho(h)U\Sigma V^{\top}
= (U^{\top} U)\Sigma V^{\top} \Lambda(h)
= \Sigma V^{\top} \Lambda(h),
\]
since \(U^{\top} U=I_{m-1}\). Let \(R_h = U^{\top} \rho(h)U\). The compatible choice of \(U\) means that \(\rho(h)U=UR_h\) and \(U^{\top}\rho(h)=R_hU^{\top}\), so \(R_h\) is the \((m-1)\times(m-1)\) matrix representation of \(h \in \mathcal G\) acting on \(\mathcal V_0\) in the basis defined by \(U\). Thus \(R_h \Sigma V^{\top} = \Sigma V^{\top} \Lambda(h)\).

To relate \(\Lambda(h)\) to \(R_h\), we use the SVD \(C = \gamma U V^{\top}\) and the fact that \(C\) also satisfies \(\rho(h)C = C\Lambda(h)\), as its columns \(g\circ\boldsymbol{v}\) are permuted in the same way:
\[
\rho(h)\gamma U V^{\top} = \gamma U V^{\top} \Lambda(h).
\]
Since \(\boldsymbol{v}\ne\boldsymbol{0}\), we have \(\gamma>0\), so dividing by \(\gamma\) gives \(\rho(h)U V^{\top} = U V^{\top} \Lambda(h)\). Left-multiplying by \(U^{\top}\), we get \(U^{\top}\rho(h)U V^{\top} = U^{\top} U V^{\top} \Lambda(h)\), i.e., \(R_h V^{\top} = V^{\top} \Lambda(h)\). Now substitute \(V^{\top} \Lambda(h) = R_h V^{\top}\) into the equation for \(\Sigma\):
\[
R_h \Sigma V^{\top} = \Sigma (R_h V^{\top}).
\]
Since \(V^{\top} V = I_{m-1}\), right-multiplying by \(V\) yields
\[
R_h \Sigma = \Sigma R_h \quad \text{for all } h \in \mathcal G.
\]

At this point we pass to the complexification: the same commutation relation holds over \(\mathbb{C}\) for the complexified representation on \(\mathcal V_0 \otimes \mathbb{C}\). By the compatible choice of \(U\) above, \(R_h=U^{\top}\rho(h)U\) is exactly the matrix of the restricted real \(\mathcal G\)-action on \(\mathcal V_0\) in this orthonormal basis, and its complex-linear extension is the corresponding matrix of the complexified action. By~\cref{lem:irreducibility_2_transitive_action}, this complex representation is irreducible; hence, by~\cref{lem:two_orbit_matrix_form}, any matrix commuting with all \(R_h\) is a scalar multiple of the identity on \(\mathcal V_0 \otimes \mathbb{C}\). Therefore \(\Sigma = s I_{m-1}\) over \(\mathbb{C}\). Since \(\Sigma\) is a real diagonal matrix, \(s \in \mathbb{R}\), yielding \(\Sigma = s I_{m-1}\) in the real setting. This completes the necessity proof.
\end{proof}

\begin{lemma}[Irreducibility of the Standard Module for a 2-Transitive Group~\citep{serre1977linear}]
\label{lem:irreducibility_2_transitive_action}
Let \(\mathcal G\) be a permutation group on \(\mathcal X=\{1,2,\ldots,m\}\), \(m\ge 2\), acting on
\(\mathbb C^m\) by \(g\circ\boldsymbol{x}:=\rho(g)\boldsymbol{x}\), where \(\rho(g)\) is the corresponding permutation matrix.
Then
\[
\mathbb C^m=\mathcal W_1\oplus \mathcal V_0,
\qquad
\mathcal W_1=\mathbb C\boldsymbol{1}_m,
\qquad
\mathcal V_0=\left\{\boldsymbol{y}\in\mathbb C^m:\sum_{i=1}^m y_i=0\right\},
\]
and both \(\mathcal W_1\) and \(\mathcal V_0\) are \(\mathcal G\)-invariant.
If the action of \(\mathcal G\) on \(\mathcal X\) is 2-transitive, then the restriction of the
permutation representation to \(\mathcal V_0\) is irreducible.
That is, the only \(\mathcal G\)-invariant complex subspaces of \(\mathcal V_0\) are \(\{0\}\) and \(\mathcal V_0\) itself.
\end{lemma}

\begin{proof}
The vector \(\boldsymbol{1}_m\) is fixed by every permutation, so \(\mathcal W_1\) is
\(\mathcal G\)-invariant. Permutations preserve coordinate sums, so \(\mathcal V_0\) is also
\(\mathcal G\)-invariant. Clearly \(\mathcal W_1\cap \mathcal V_0=\{0\}\), and
\(\dim \mathcal W_1+\dim \mathcal V_0=1+(m-1)=m\), hence
\(\mathbb C^m=\mathcal W_1\oplus \mathcal V_0\).
Let \(\mathcal Y\subseteq \mathcal V_0\) be a nonzero \(\mathcal G\)-invariant subspace. Choose
\(\boldsymbol{y}\in \mathcal Y\setminus\{0\}\), and pick \(a\) with \(y_a\neq 0\). Let
\[
\mathcal G_a=\{g\in \mathcal G:g(a)=a\}
\]
be the stabilizer of \(a\), and define
\[
\boldsymbol w=\sum_{g\in \mathcal G_a}g\circ\boldsymbol{y}.
\]
Then \(\boldsymbol w\in \mathcal Y\). Its \(a\)-th coordinate is
\[
w_a=\sum_{g\in \mathcal G_a}(g\circ\boldsymbol{y})_a
=\sum_{g\in \mathcal G_a} y_a
=|\mathcal G_a|\,y_a\neq 0,
\]
so \(\boldsymbol w\neq 0\). Since the action is 2-transitive, \(\mathcal G_a\) acts
transitively on \(\mathcal X\setminus\{a\}\). Hence all coordinates \(w_i\) with
\(i\neq a\) are equal. Because \(\boldsymbol w\in \mathcal V_0\), we must have
\[
\boldsymbol w=c\left(\boldsymbol e_a-\frac1m\boldsymbol{1}_m\right)
\]
for some nonzero scalar \(c\).
Now \(\mathcal Y\) is \(\mathcal G\)-invariant and \(\mathcal G\) is transitive on \(\mathcal X\), so \(\mathcal Y\)
contains
\[
\boldsymbol e_i-\frac1m\boldsymbol{1}_m
\qquad\text{for every } i=1,\ldots,m.
\]
Therefore \(\mathcal Y\) contains every difference \(\boldsymbol e_i-\boldsymbol e_j\), and these differences
span \(\mathcal V_0\). Hence \(\mathcal Y=\mathcal V_0\). So the restricted representation on \(\mathcal V_0\) is
irreducible.
\end{proof}

\begin{lemma}[Two-Orbit Form for Matrices Invariant under a 2-Transitive Action~\citep{serre1977linear}]
\label{lem:two_orbit_matrix_form}
Let \(\mathbb F\in\{\mathbb R,\mathbb C\}\). Let \(\mathcal G\) act 2-transitively on
\(\mathcal X=\{1,\ldots,m\}\), with \(m\ge 2\), and let \(\rho(g)\) be the permutation matrix satisfying \(\rho(g)\boldsymbol{x}=g\circ\boldsymbol{x}\). If \(T\in\mathbb F^{m\times m}\) satisfies
\[
\rho(g)T\rho(g)^{-1}=T
\qquad\text{for all } g\in \mathcal G,
\]
then there exist scalars \(a,b\in\mathbb F\) such that
\[
T=(a-b)I_m+bJ_m,
\qquad J_m=\boldsymbol{1}_m\boldsymbol{1}_m^{\top}.
\]
Moreover, for every \(\boldsymbol{y}\in \mathcal V_0 =\left\{\boldsymbol{y}\in\mathbb F^m:\sum_{i=1}^m y_i=0\right\}\),
\[
T\boldsymbol{y}=(a-b)\boldsymbol{y}.
\]
\end{lemma}

\begin{proof}
Because the action is transitive, all diagonal entries of \(T\) are equal; call
this common value \(a\). Because the action is 2-transitive, any two ordered
pairs of distinct indices lie in the same \(\mathcal G\)-orbit, so all off-diagonal
entries are equal; call this common value \(b\). Therefore
\[
T=(a-b)I_m+bJ_m.
\]
Now \(J_m\boldsymbol{y}=(\sum_{i=1}^m y_i)\boldsymbol{1}_m=0\) for every
\(\boldsymbol{y}\in \mathcal V_0\), so
\[
T\boldsymbol{y}=((a-b)I_m+bJ_m)\boldsymbol{y}=(a-b)\boldsymbol{y}.
\]
In particular, \(T\) acts as a scalar on \(\mathcal V_0\).
\end{proof}

\section{Distance Metrics and Projection Operators}
\label{app:measure}

To measure the proximity of a matrix \(G\in\mathbb{R}^{m\times m}\) to the predicted geometric structures, namely the simplex ETF Gram matrix and the subspace of circulant matrices, we introduce the following scale-invariant distances based on Frobenius norm projections.

\begin{definition}[Distance to a Simplex ETF]
\label{def:dis_etf}
Let \(M^*\) be the canonical simplex ETF matrix in~\cref{eq:simplex-etf-mstar}. For a nonzero Gram matrix \(G\), define the best Frobenius scalar fit of \(G\) to \(M^*\) by
\[
c^*  =  \arg\min_{c\in\mathbb{R}}   \left\| cG - M^* \right\|_F
=  \frac{\langle G, M^*\rangle_F}{\langle G, G\rangle_F}.
\]
The relative (scale-invariant) ETF distance is
\[
\delta_{\mathrm{ETF}}(G)  =  \frac{\left\| c^* G - M^* \right\|_F}{\left\| M^* \right\|_F}.
\]
This vanishes iff \(G\) is a positive scalar multiple of \(M^*\). Moreover, if \(G\mapsto \alpha G\) for any \(\alpha>0\), then \(c^*\mapsto c^*/\alpha\) and \(\delta_{\mathrm{ETF}}(G)\) is unchanged.
\end{definition}

\begin{definition}[Distance to the Circulant Subspace]\label{def:circ-distance}
Let \(\mathcal{C}\subset\mathbb{R}^{m\times m}\) denote the linear subspace of real circulant matrices.
Given a nonzero symmetric Gram matrix \(G\in\mathbb{R}^{m\times m}\), define its nearest circulant approximation in Frobenius norm by
\[
C^*  \in  \arg\min_{C\in\mathcal{C}}  \| G-C \|_F   .
\]
Then \( C^* \) is given by
\[
C^*  =  \frac{1}{m}\sum_{k=1}^{m}\Pi^{k-1} G (\Pi^{k-1})^{ \top},
\]
where \(\Pi\) is the one-step cyclic-shift permutation matrix:
\[
\Pi=\begin{bmatrix}
0 & 0 & \cdots & 0 & 1\\
1 & 0 & \cdots & 0 & 0\\
0 & 1 & \cdots & 0 & 0\\
\vdots & \vdots & \ddots & \vdots & \vdots\\
0 & 0 & \cdots & 1 & 0
\end{bmatrix}
 \in \mathbb{R}^{m\times m}.
\]
The associated relative (scale-invariant) circulant distance is
\[
\delta_{\mathrm{circ}}(G)  =  \frac{\| G-C^* \|_F}{\| G \|_F}  .
\]
Moreover, \(\delta_{\mathrm{circ}}(\alpha G)=\delta_{\mathrm{circ}}(G)\) for any \(\alpha>0\),
so the metric is invariant under global rescaling of \(G\).
\end{definition}

\begin{proof}[Proof of the form of \(C^*\) in~\cref{def:circ-distance}]
Define the averaging operator
\[
\mathcal{P}:\mathbb{R}^{m\times m}\to\mathbb{R}^{m\times m},
\qquad
\mathcal{P}(X) = \frac{1}{m}\sum_{k=1}^{m}\Pi^{k-1}X(\Pi^{k-1})^{ \top}.
\]
For any \(X\), each term \(\Pi^{k-1}X(\Pi^{k-1})^{ \top}\) cyclically shifts the rows and columns of \(X\).
Hence \(\mathcal{P}(X)\) is invariant under simultaneous row/column cyclic shift:
\[
\Pi \mathcal{P}(X) \Pi^{ \top}
=\frac{1}{m}\sum_{k=1}^{m}\Pi^{k}X(\Pi^{k})^{ \top}
=\mathcal{P}(X),
\]
so \(\mathcal{P}(X)\) is circulant and \(\mathrm{range}(\mathcal{P})\subseteq\mathcal{C}\).
For any \(X,Y\),
\[
\left\langle \Pi^{k-1}X(\Pi^{k-1})^{ \top}, \Pi^{k-1}Y(\Pi^{k-1})^{ \top}\right\rangle_F
=\langle X,Y\rangle_F,
\]
since conjugation by an orthogonal matrix preserves the Frobenius inner product.
Averaging over \(k\) yields
\[
\langle \mathcal{P}(X), Y\rangle_F
=\frac{1}{m}\sum_{k=1}^{m}\left\langle \Pi^{k-1}X(\Pi^{k-1})^{ \top}, Y\right\rangle_F
=\frac{1}{m}\sum_{k=1}^{m}\left\langle X, \Pi^{k-1}Y(\Pi^{k-1})^{ \top}\right\rangle_F
=\langle X, \mathcal{P}(Y)\rangle_F,
\]
so \(\mathcal{P}\) is self-adjoint. Moreover,
\[
\mathcal{P}(\mathcal{P}(X))
=\frac{1}{m}\sum_{k=1}^{m}\Pi^{k-1}\left(\frac{1}{m}\sum_{\ell=1}^{m}\Pi^{\ell-1}X(\Pi^{\ell-1})^{ \top}\right)(\Pi^{k-1})^{ \top}
=\frac{1}{m^2}\sum_{k,\ell=1}^{m}\Pi^{k+\ell-2}X(\Pi^{k+\ell-2})^{ \top}.
\]
Because the map \((k,\ell)\mapsto r\), where \(r\in\{1,\ldots,m\}\) is determined by \(r-1\equiv k+\ell-2\pmod m\), hits each \(r\) exactly \(m\) times,
\[
\mathcal{P}(\mathcal{P}(X))
=\frac{1}{m}\sum_{r=1}^{m}\Pi^{r-1}X(\Pi^{r-1})^{ \top}
=\mathcal{P}(X),
\]
so \(\mathcal{P}\) is idempotent. Combined with self-adjointness, \(\mathcal{P}\) is the orthogonal projector
(with respect to \(\langle\cdot,\cdot\rangle_F\)) onto its range.
If \(Y\in\mathcal{C}\) is circulant, then \(\Pi Y\Pi^{ \top}=Y\), hence
\(\Pi^{k-1}Y(\Pi^{k-1})^{ \top}=Y\) for all \(k=1,\ldots,m\), so \(\mathcal{P}(Y)=Y\).
Thus \(\mathcal{C}\subseteq\mathrm{range}(\mathcal{P})\), and 
\(\mathrm{range}(\mathcal{P})=\mathcal{C}\).
Since \(\mathcal{C}\) is a closed linear subspace of the Hilbert space
\((\mathbb{R}^{m\times m},\langle\cdot,\cdot\rangle_F)\),
the orthogonal projection onto \(\mathcal{C}\) is the unique minimizer of the Frobenius distance, hence
\[
C^*=\mathcal{P}(G)
=\frac{1}{m}\sum_{k=1}^{m}\Pi^{k-1}G(\Pi^{k-1})^{ \top}
\quad\text{and}\quad
\|G-C^*\|_F=\min_{C\in\mathcal{C}}\|G-C\|_F.
\]
This completes the proof.
\end{proof}

\section{Additional Experiments}\label{app:more_exp}

All experimental settings are identical to those in the main text; see~\cref{sec:exp}. Here we report additional results with GPT-OSS-120B~\citep{openai2025gptoss120bgptoss20bmodel}, Llama-3.2-3B-Instruct~\citep{dubey2024llama}, GPT2-XL~\citep{radford2019language}, and Gemma-2-9B-IT~\citep{team2024gemma}.

\subsection{Output projections}

\paragraph{Simplex ETF pattern.} First, we report kitchen utensil results for the additional models \{\texttt{fork}, \texttt{knife}, \texttt{plate}, \texttt{cup}, \texttt{glass}\}. The theory motivates comparing the output projection Gram matrices with the simplex ETF reference (\cref{fig:app-word-kitchen}). 

\begin{figure}[htbp]
  \centering
  \captionsetup[subfigure]{justification=centering}
  \begin{subfigure}[t]{0.24\textwidth}
    \centering
    \includegraphics[width=\linewidth]{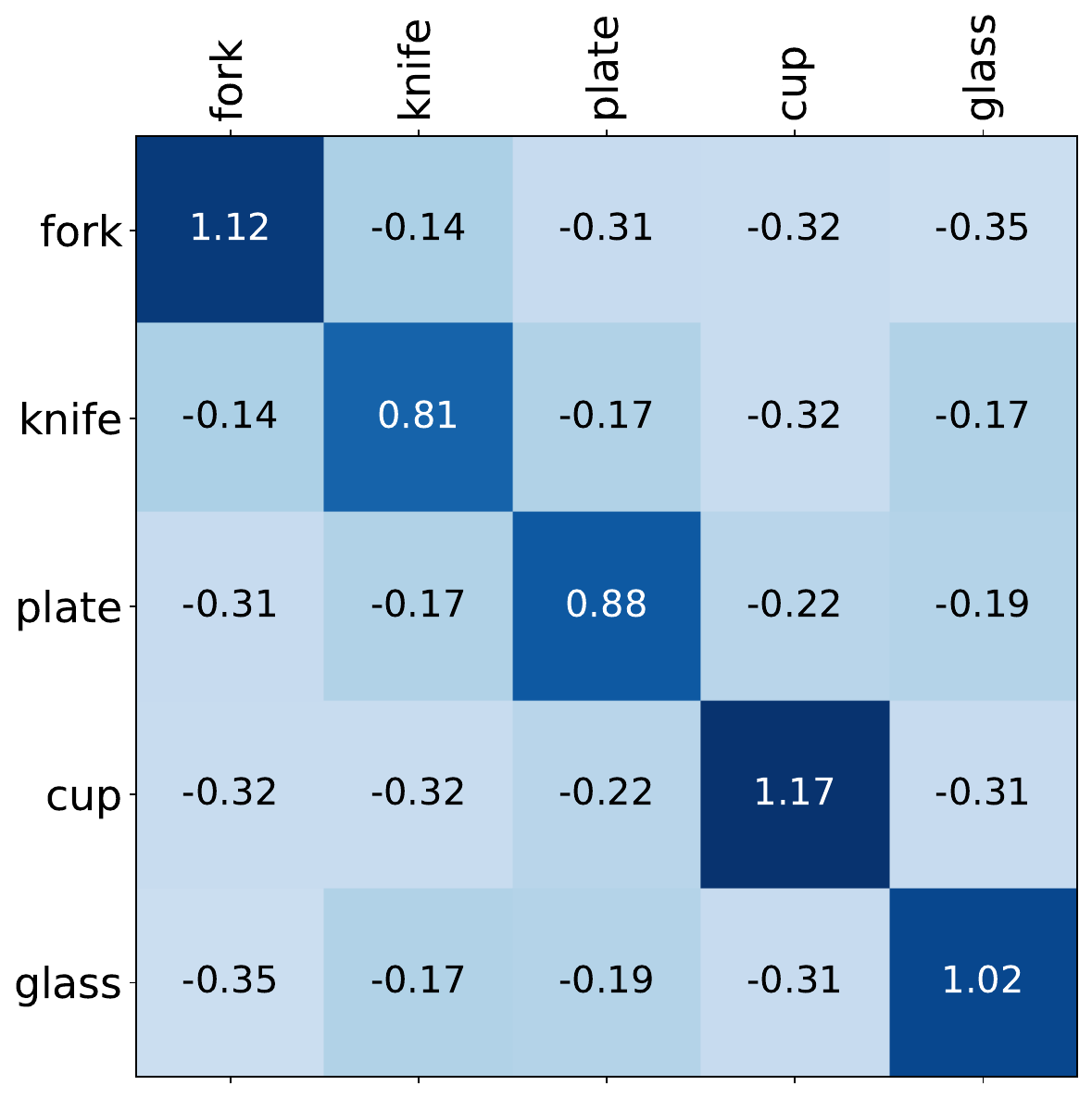}
    \caption{GPT-OSS-120B (\(\delta_{\mathrm{ETF}}=0.180\))}
  \end{subfigure}\hfill
  \begin{subfigure}[t]{0.24\textwidth}
    \centering
    \includegraphics[width=\linewidth]{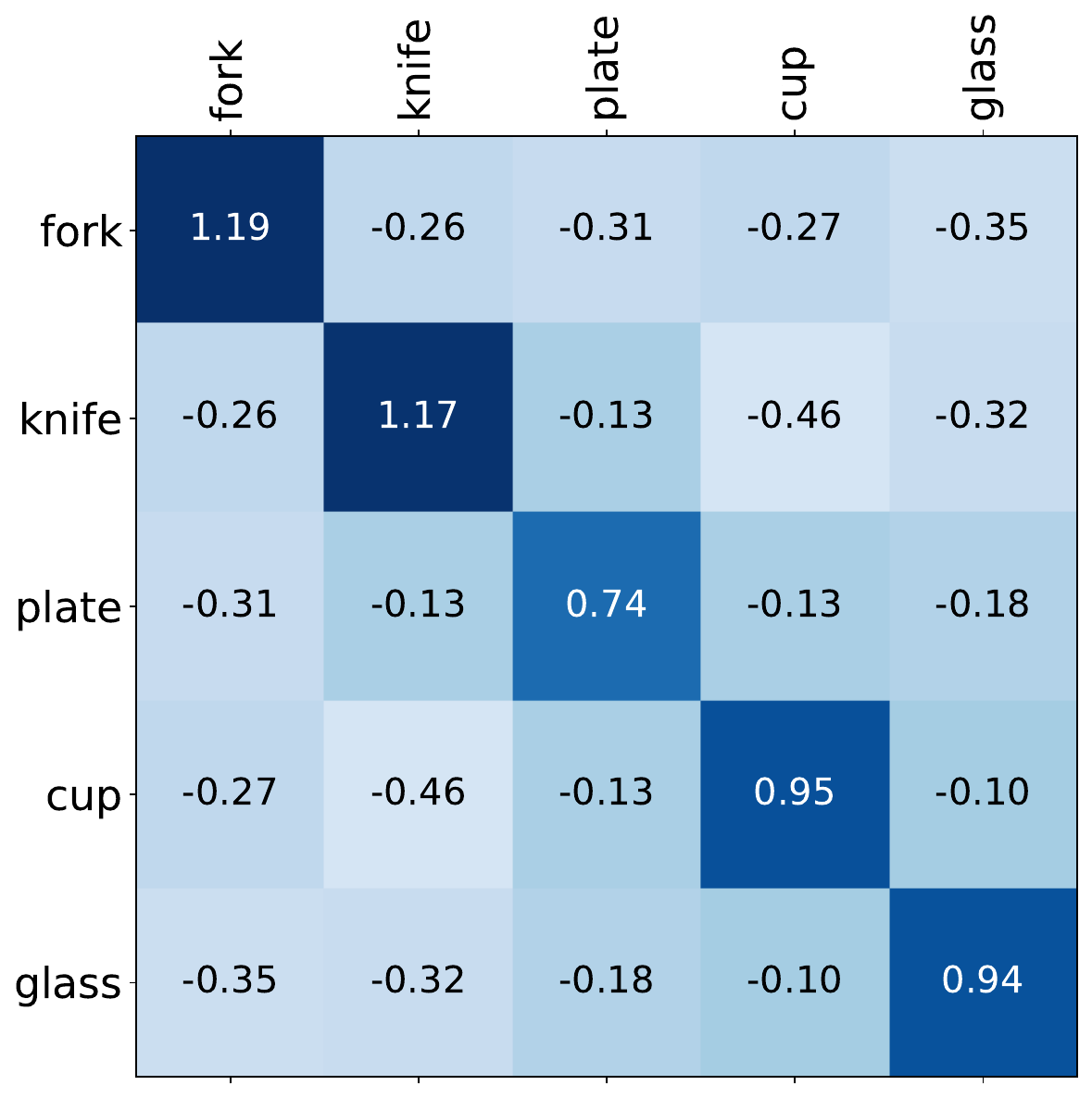}
    \caption{Llama-3.2-3B-Instruct (\(\delta_{\mathrm{ETF}}=0.242\))}
  \end{subfigure}\hfill
  \begin{subfigure}[t]{0.24\textwidth}
    \centering
    \includegraphics[width=\linewidth]{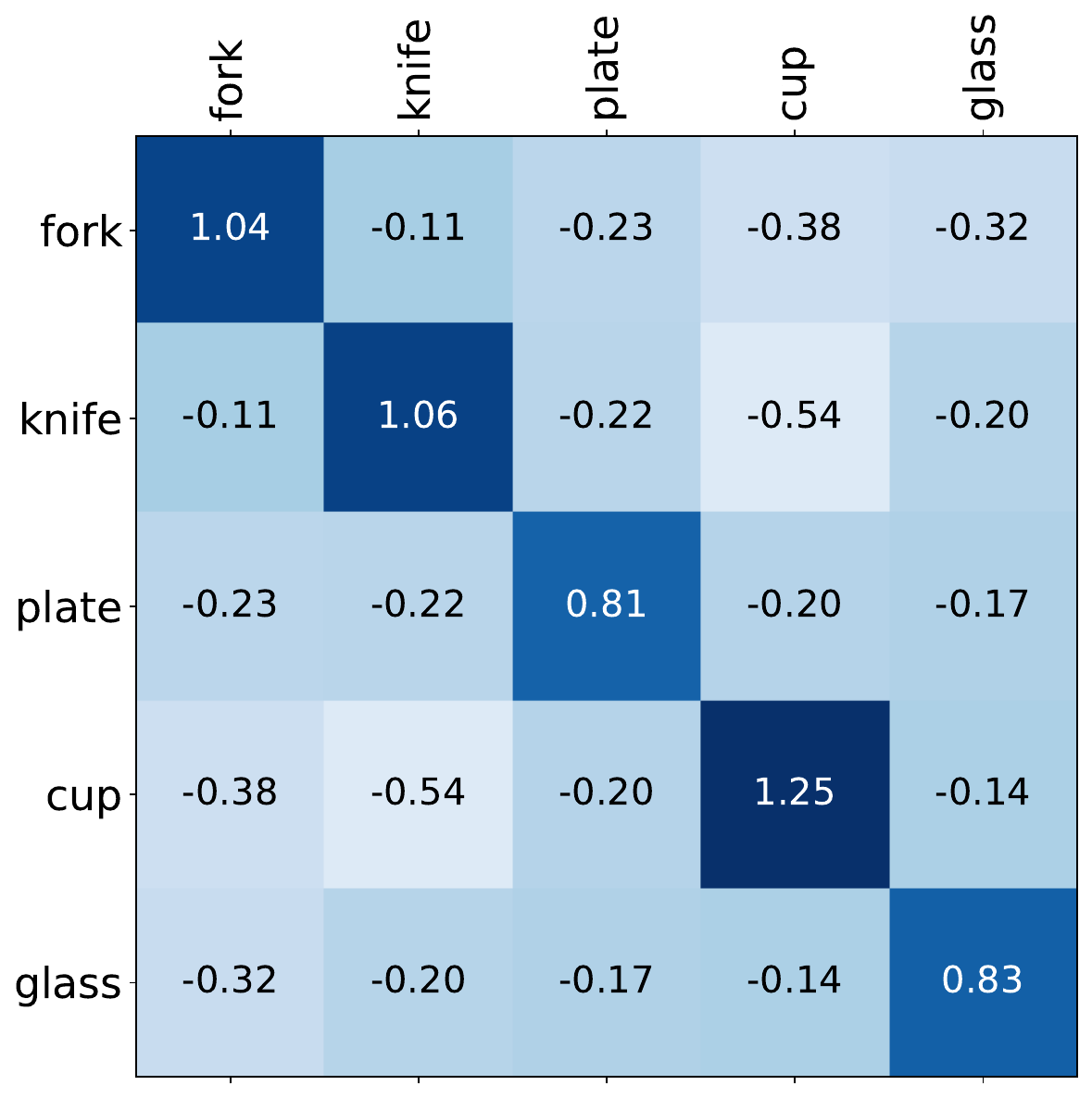}
    \caption{GPT2-XL (\(\delta_{\mathrm{ETF}}=0.256\))}
  \end{subfigure}\hfill
  \begin{subfigure}[t]{0.24\textwidth}
    \centering
    \includegraphics[width=\linewidth]{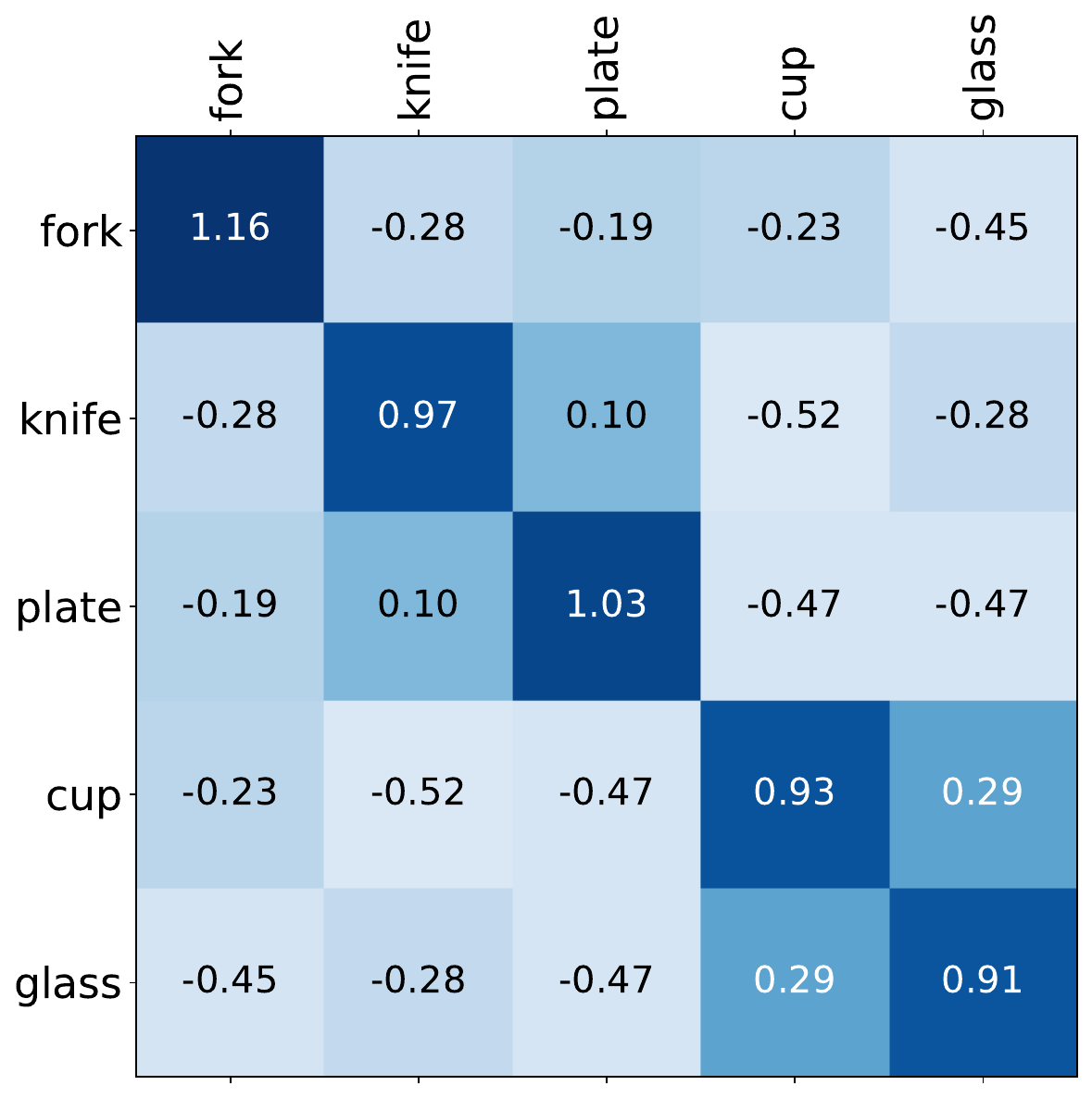}
    \caption{Gemma-2-9B-IT (\(\delta_{\mathrm{ETF}}=0.414\))}
  \end{subfigure}
  \caption{Output projection Grams for kitchen utensils.}
  \label{fig:app-word-kitchen}
\end{figure}

Next, we consider an additional word group, modes of transportation \{\texttt{car}, \texttt{bus}, \texttt{bicycle}, \texttt{train}\}, for all models. \Cref{fig:word-transport,fig:app-word-transport} show the output projection Gram matrices for these words, which we compare with the simplex ETF reference.

\begin{figure}[htbp]
  \centering
  \captionsetup[subfigure]{justification=centering}
  \begin{subfigure}[t]{0.32\textwidth}
    \centering
    \includegraphics[width=\linewidth]{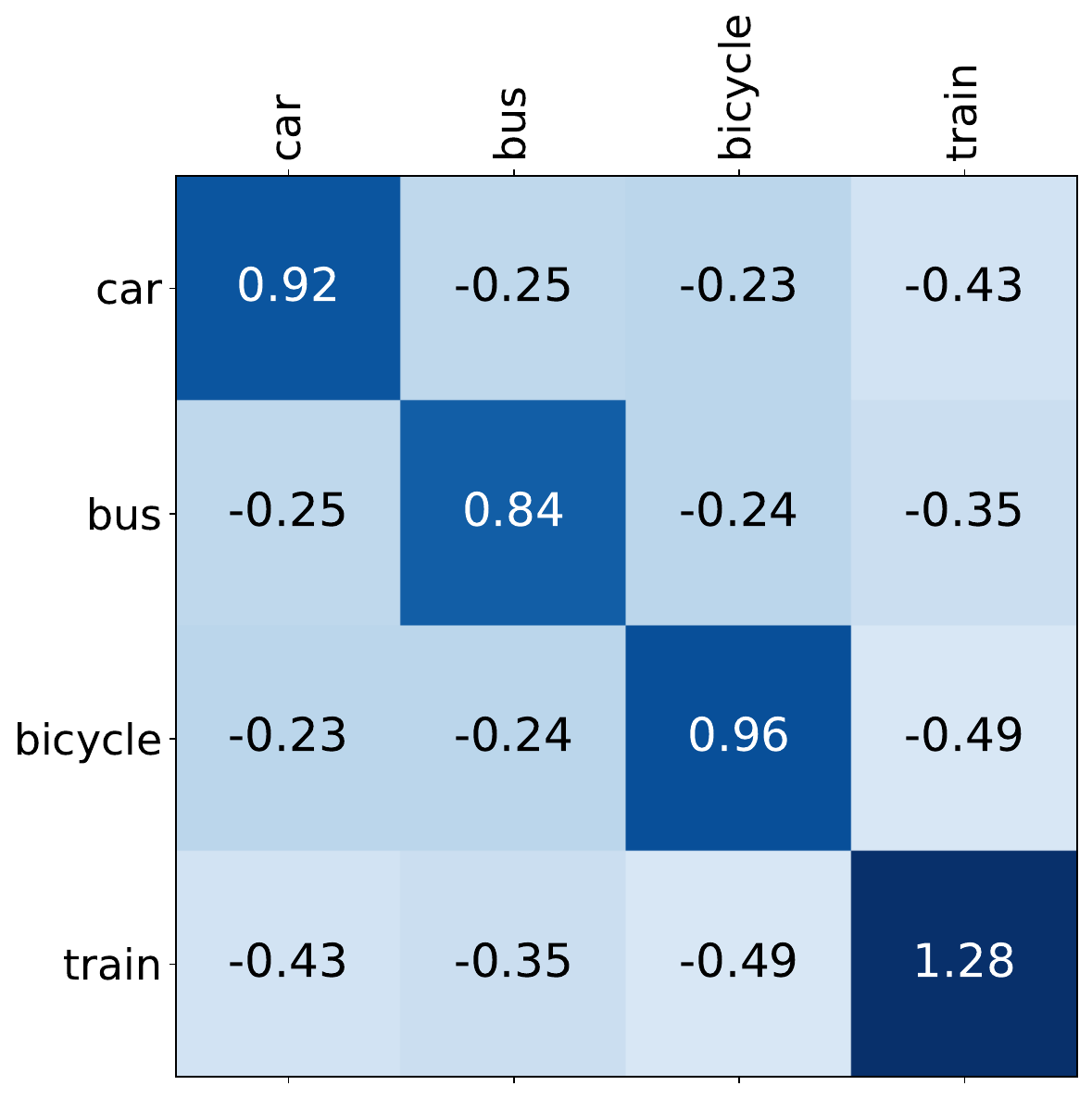}
    \caption{GPT-OSS-20B (\(\delta_{\mathrm{ETF}}=0.204\))}
  \end{subfigure}\hfill
  \begin{subfigure}[t]{0.32\textwidth}
    \centering
    \includegraphics[width=\linewidth]{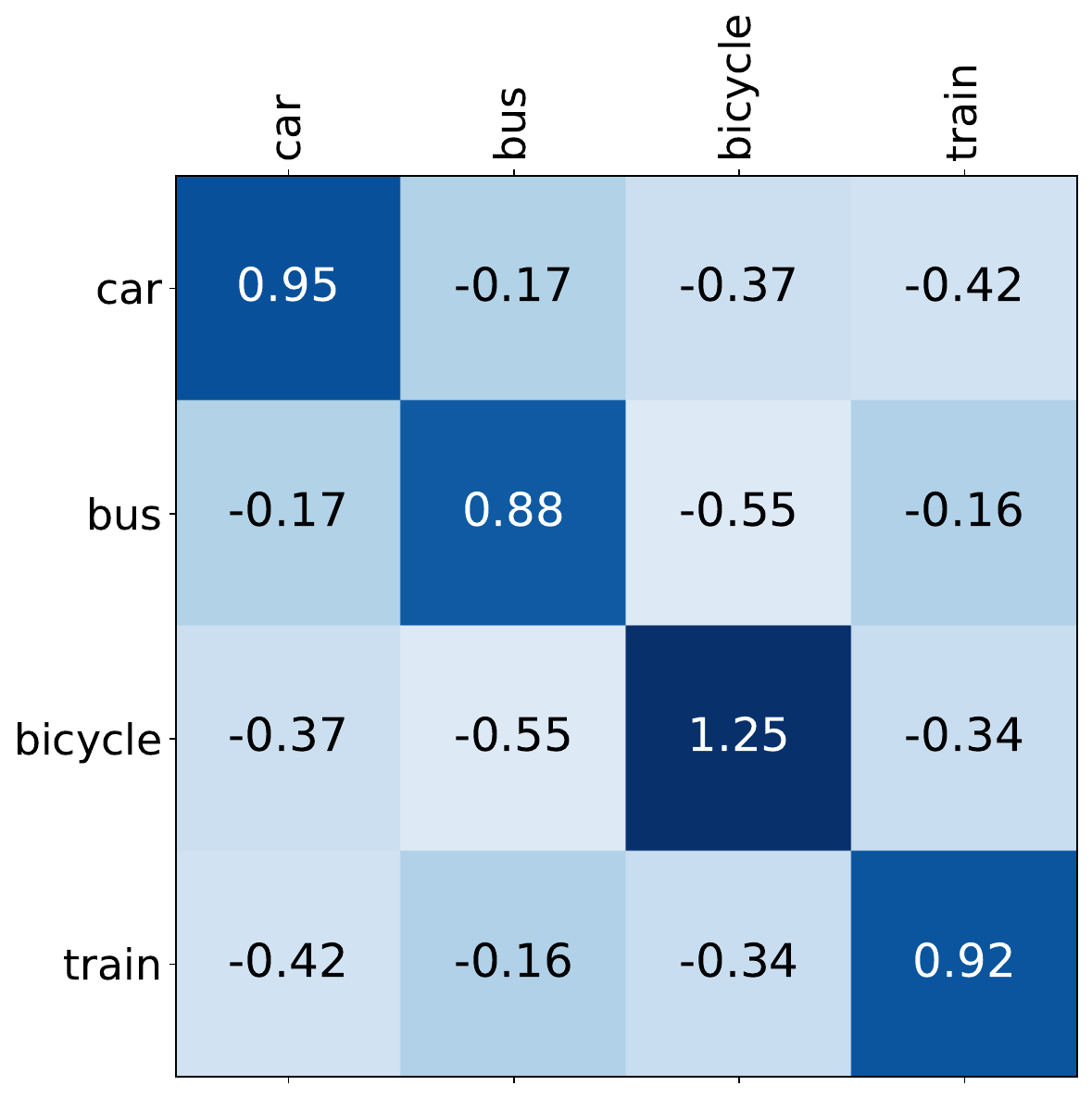}
    \caption{Mistral-7B-Instruct-v0.3 (\(\delta_{\mathrm{ETF}}=0.233\))}
  \end{subfigure}\hfill
  \begin{subfigure}[t]{0.32\textwidth}
    \centering
    \includegraphics[width=\linewidth]{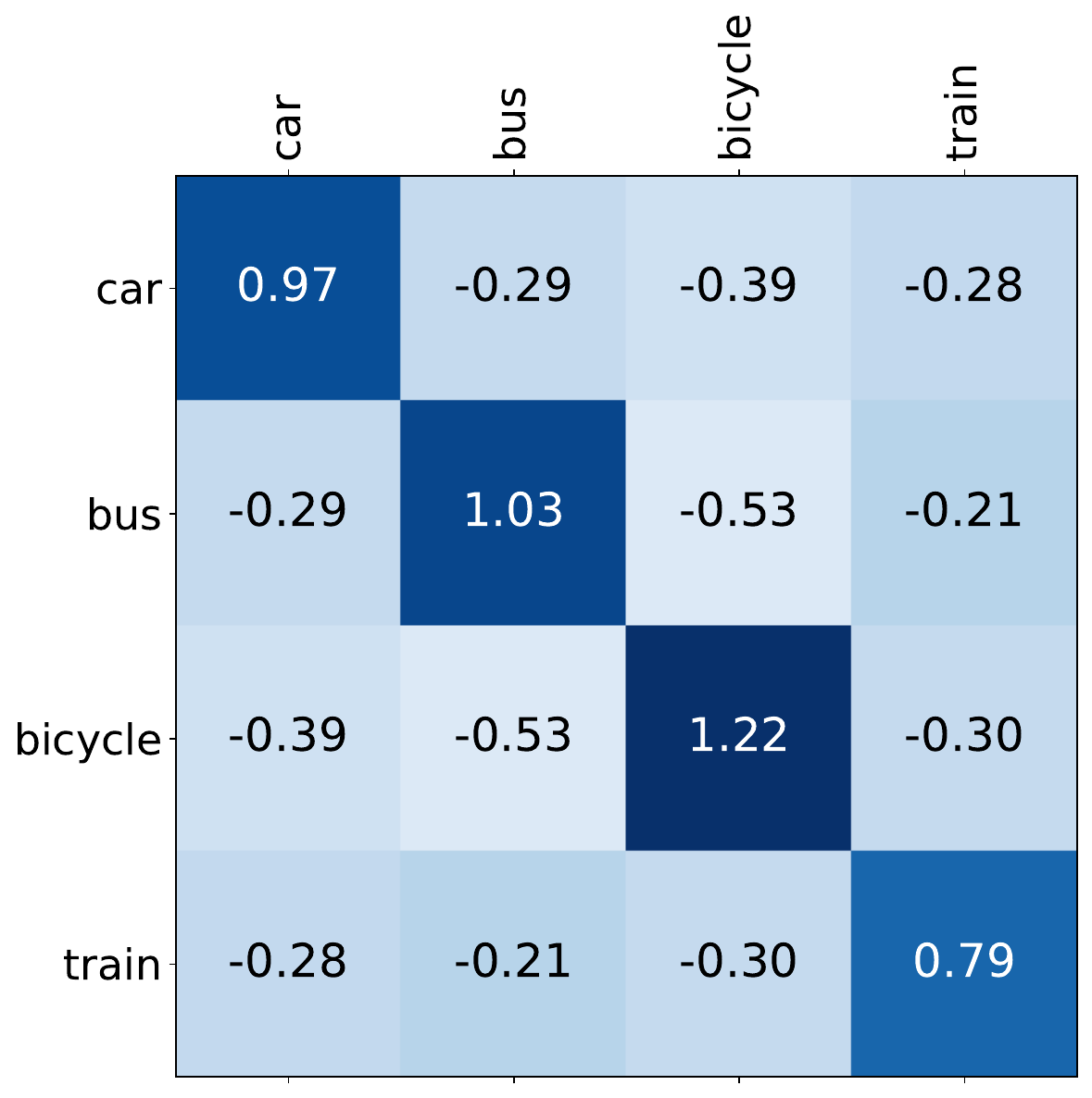}
    \caption{RWKV7-7.2B (\(\delta_{\mathrm{ETF}}=0.199\))}
  \end{subfigure}
  \caption{Output projection Grams for modes of transportation.}
  \label{fig:word-transport}
\end{figure}

\begin{figure}[htbp]
  \centering
  \captionsetup[subfigure]{justification=centering}
  \begin{subfigure}[t]{0.24\textwidth}
    \centering
    \includegraphics[width=\linewidth]{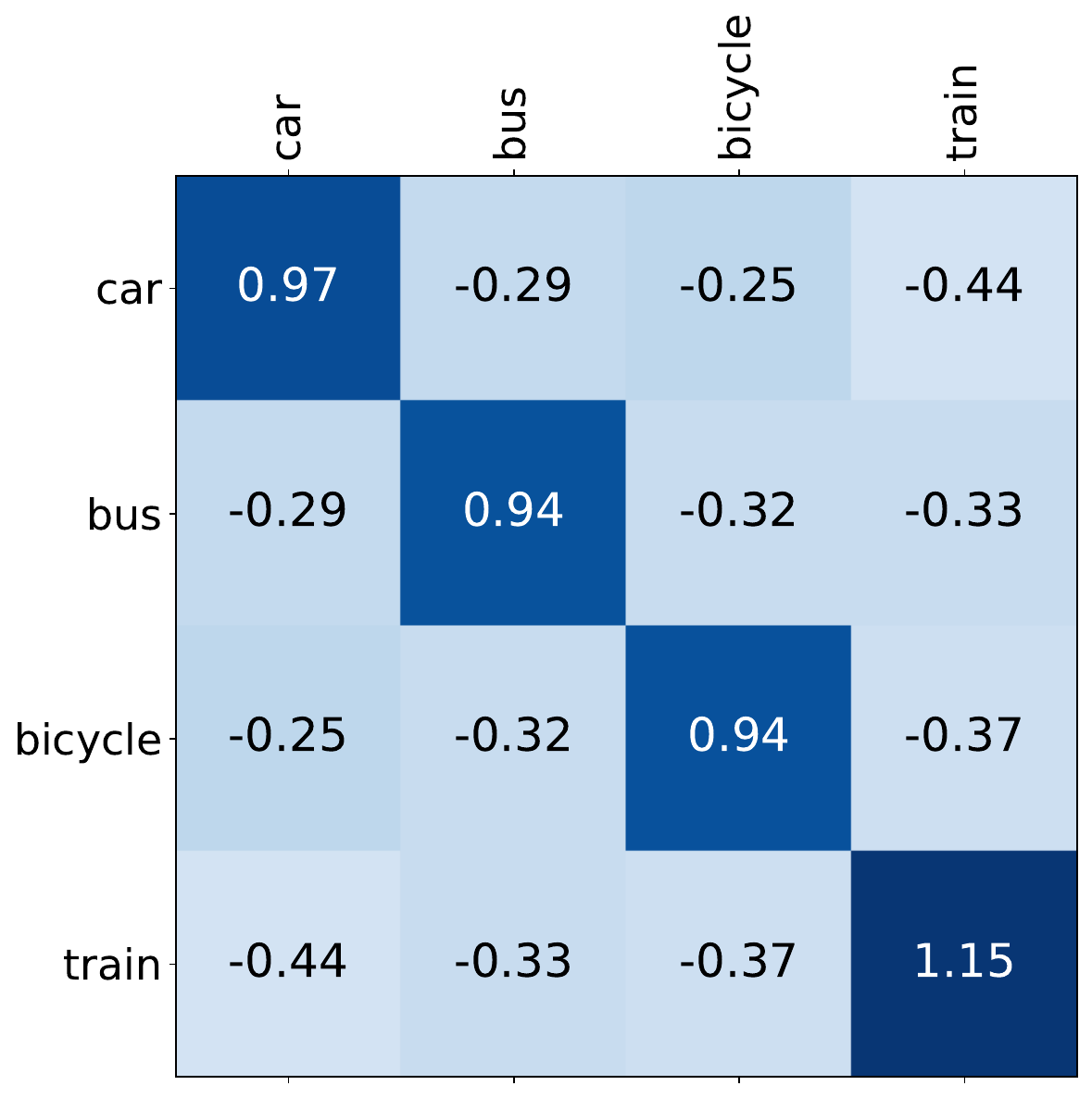}
    \caption{GPT-OSS-120B (\(\delta_{\mathrm{ETF}}=0.118\))}
  \end{subfigure}\hfill
  \begin{subfigure}[t]{0.24\textwidth}
    \centering
    \includegraphics[width=\linewidth]{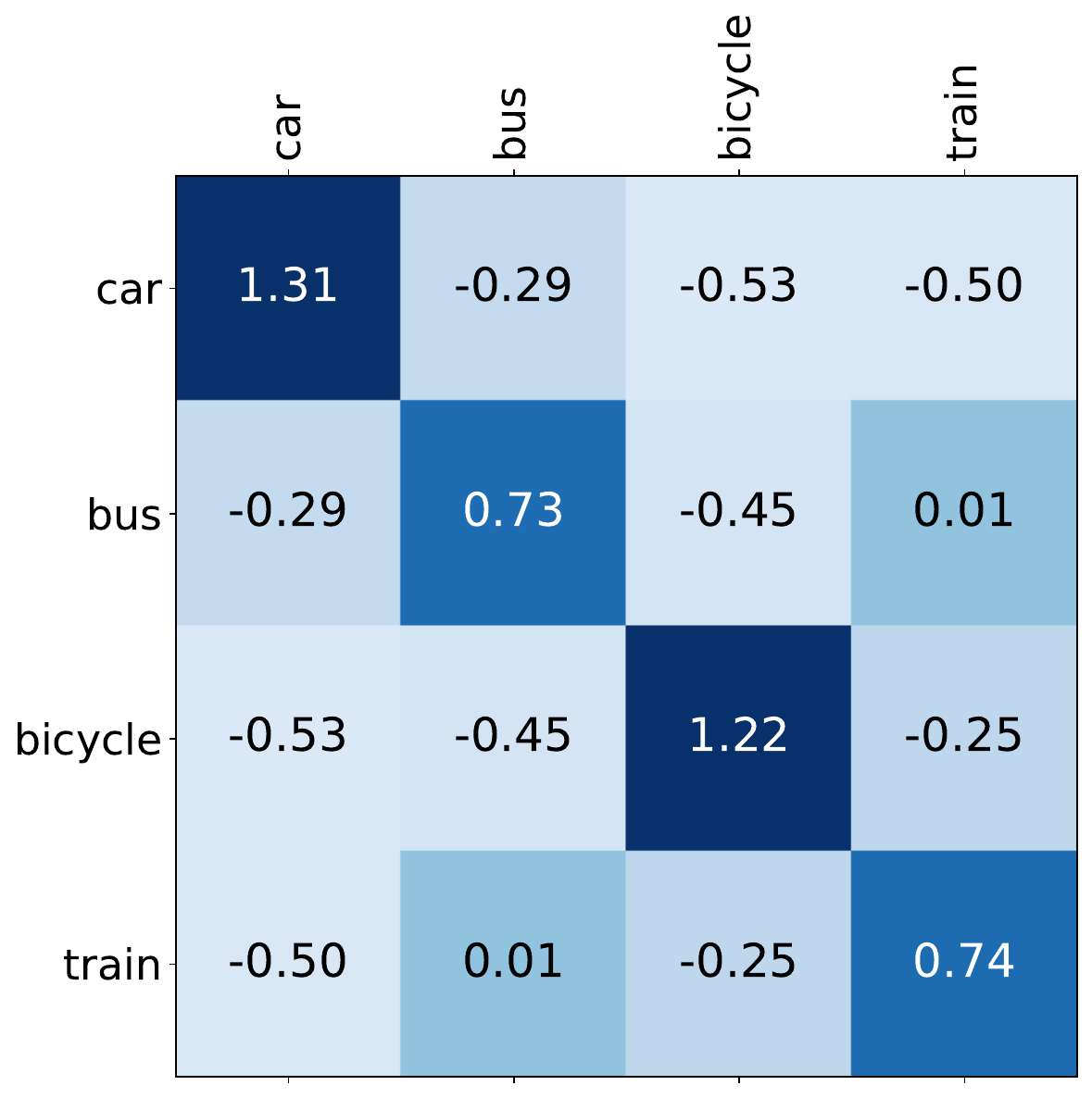}
    \caption{Llama-3.2-3B-Instruct (\(\delta_{\mathrm{ETF}}=0.340\))}
  \end{subfigure}\hfill
  \begin{subfigure}[t]{0.24\textwidth}
    \centering
    \includegraphics[width=\linewidth]{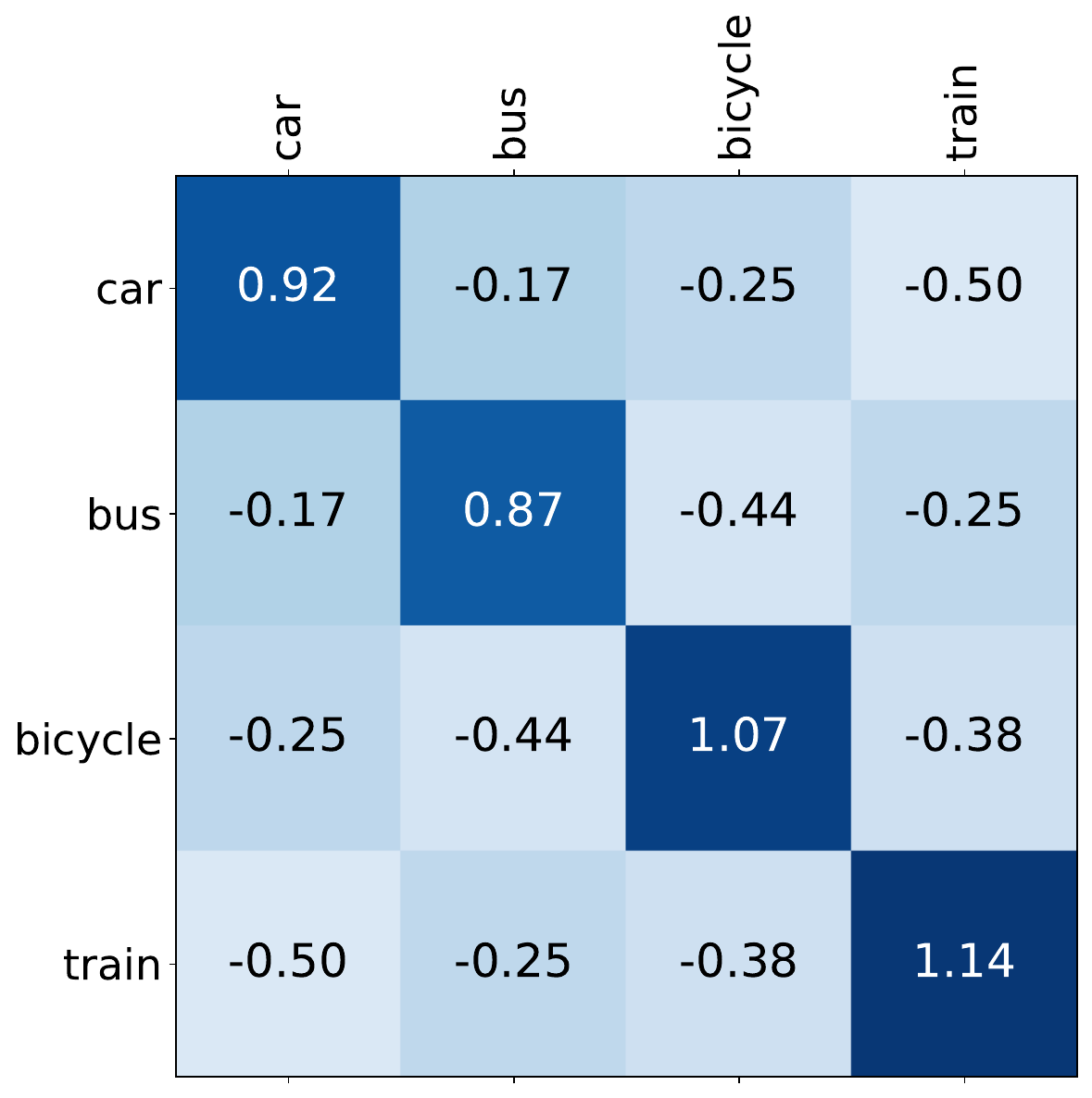}
    \caption{GPT2-XL (\(\delta_{\mathrm{ETF}}=0.196\))}
  \end{subfigure}\hfill
  \begin{subfigure}[t]{0.24\textwidth}
    \centering
    \includegraphics[width=\linewidth]{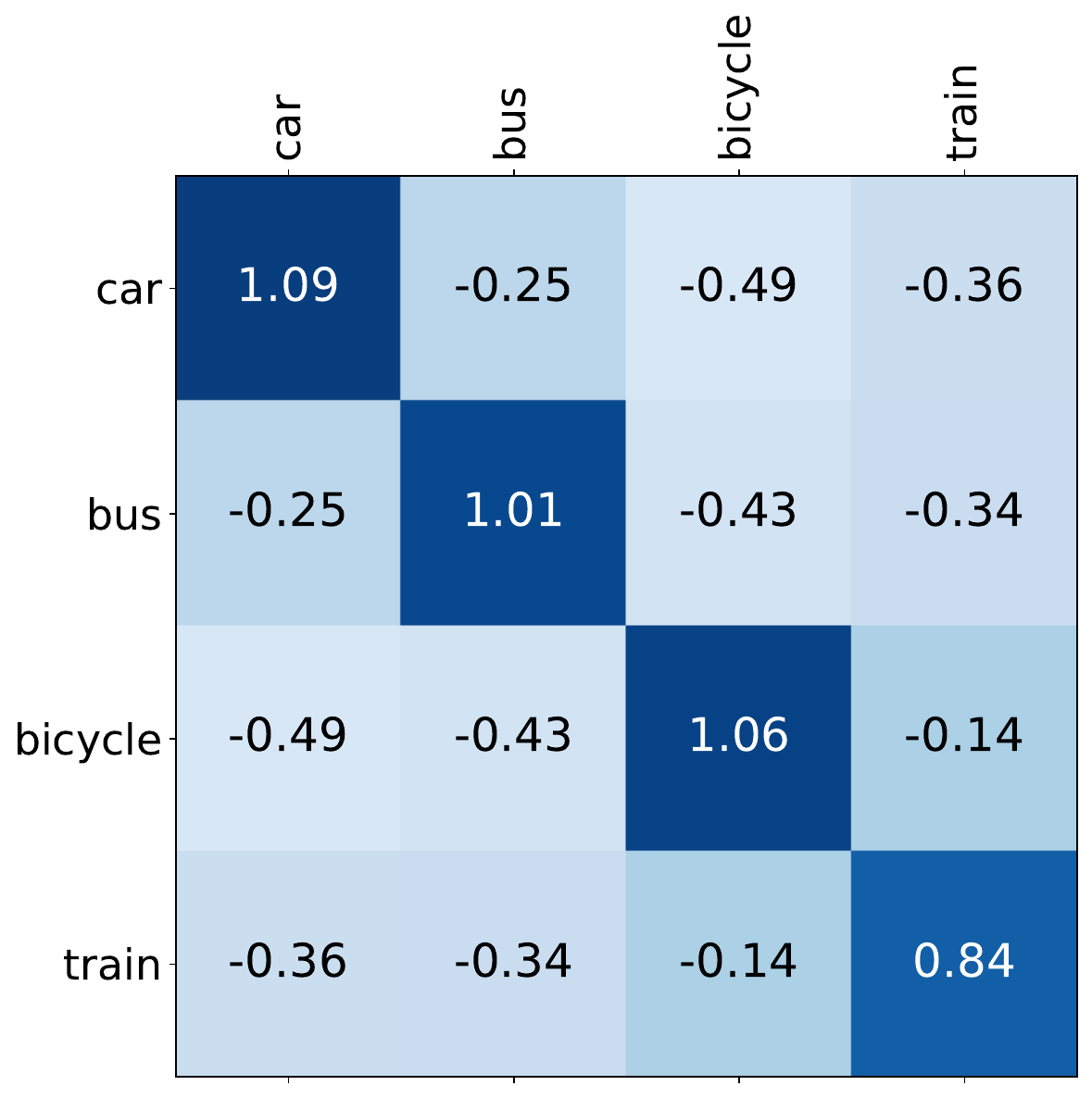}
    \caption{Gemma-2-9B-IT (\(\delta_{\mathrm{ETF}}=0.189\))}
  \end{subfigure}
  \caption{Output projection Grams for modes of transportation.}
  \label{fig:app-word-transport}
\end{figure}

\paragraph{Circulant pattern.} For weekdays \{\texttt{Monday}, \texttt{Tuesday}, \texttt{Wednesday}, \texttt{Thursday}, \texttt{Friday}, \texttt{Saturday}, \texttt{Sunday}\}, the theory motivates testing whether the output projection Gram matrices are approximately circulant (\cref{fig:app-word-weekdays}).

\begin{figure}[htbp]
  \centering
  \captionsetup[subfigure]{justification=centering}
  \begin{subfigure}[t]{0.24\textwidth}
    \centering
    \includegraphics[width=\linewidth]{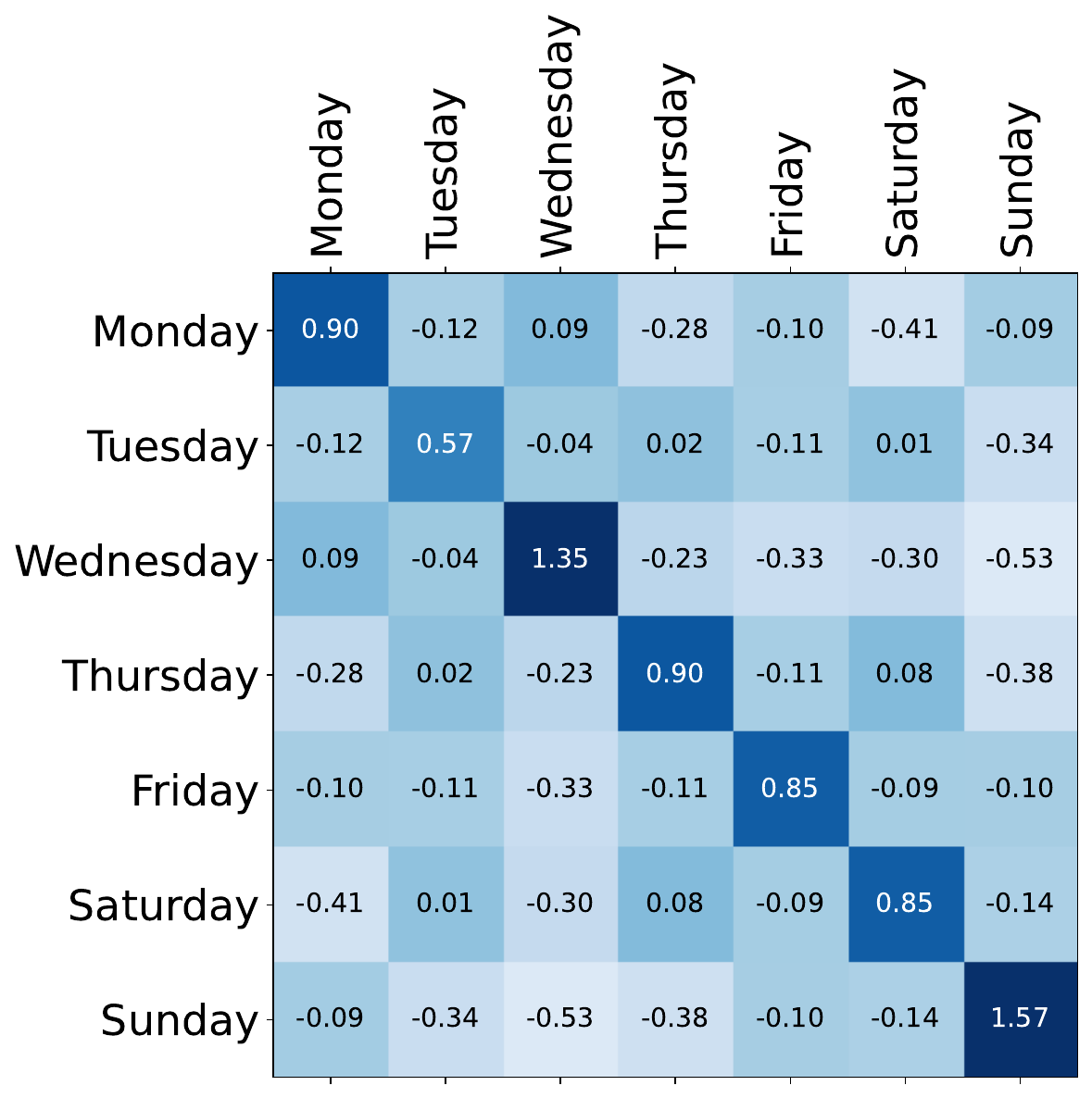}
    \caption{GPT-OSS-120B (\(\delta_{\mathrm{circ}}=0.413\))}
  \end{subfigure}\hfill
  \begin{subfigure}[t]{0.24\textwidth}
    \centering
    \includegraphics[width=\linewidth]{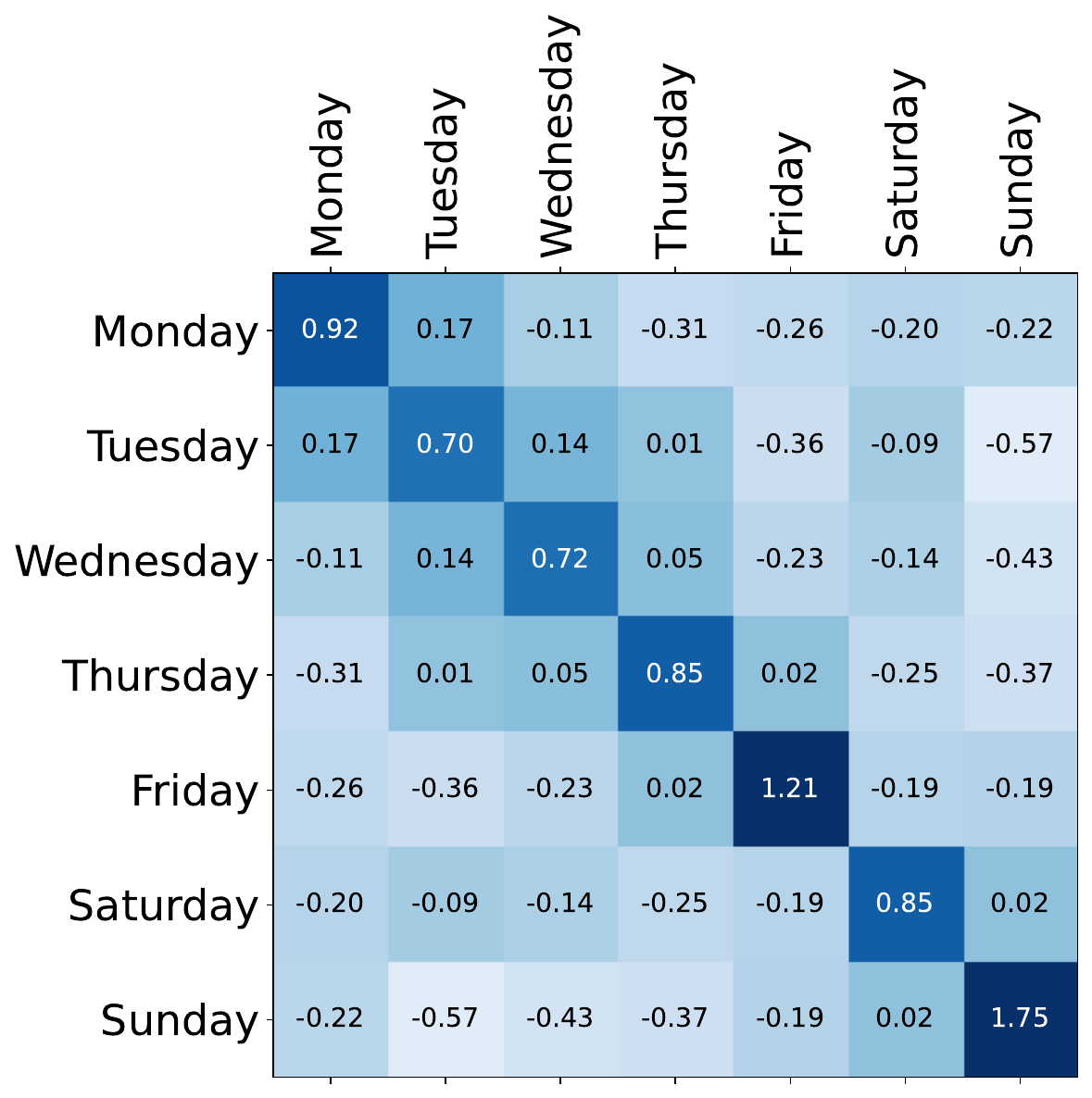}
    \caption{Llama-3.2-3B-Instruct (\(\delta_{\mathrm{circ}}=0.401\))}
  \end{subfigure}\hfill
  \begin{subfigure}[t]{0.24\textwidth}
    \centering
    \includegraphics[width=\linewidth]{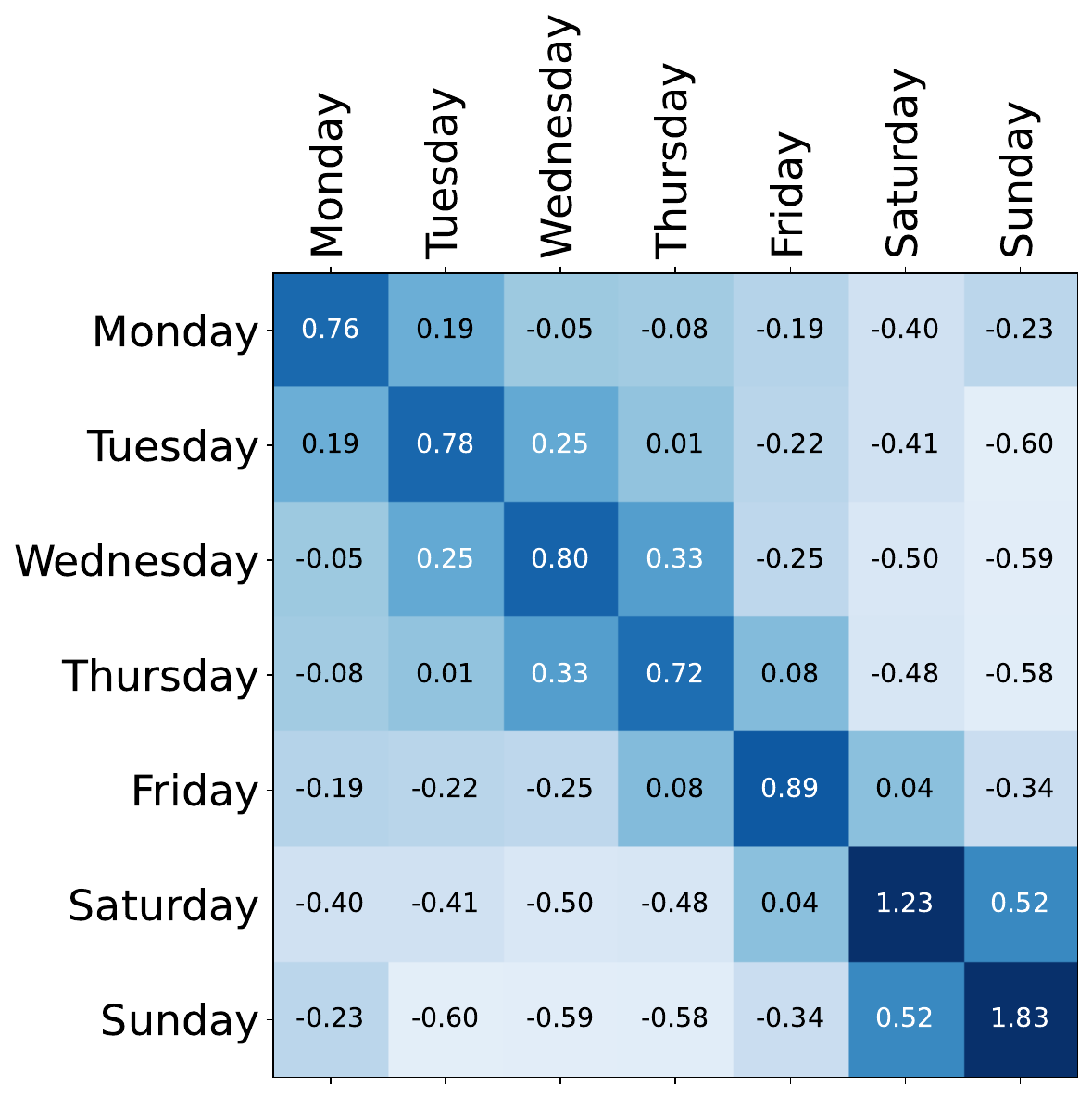}
    \caption{GPT2-XL (\(\delta_{\mathrm{circ}}=0.456\))}
  \end{subfigure}\hfill
  \begin{subfigure}[t]{0.24\textwidth}
    \centering
    \includegraphics[width=\linewidth]{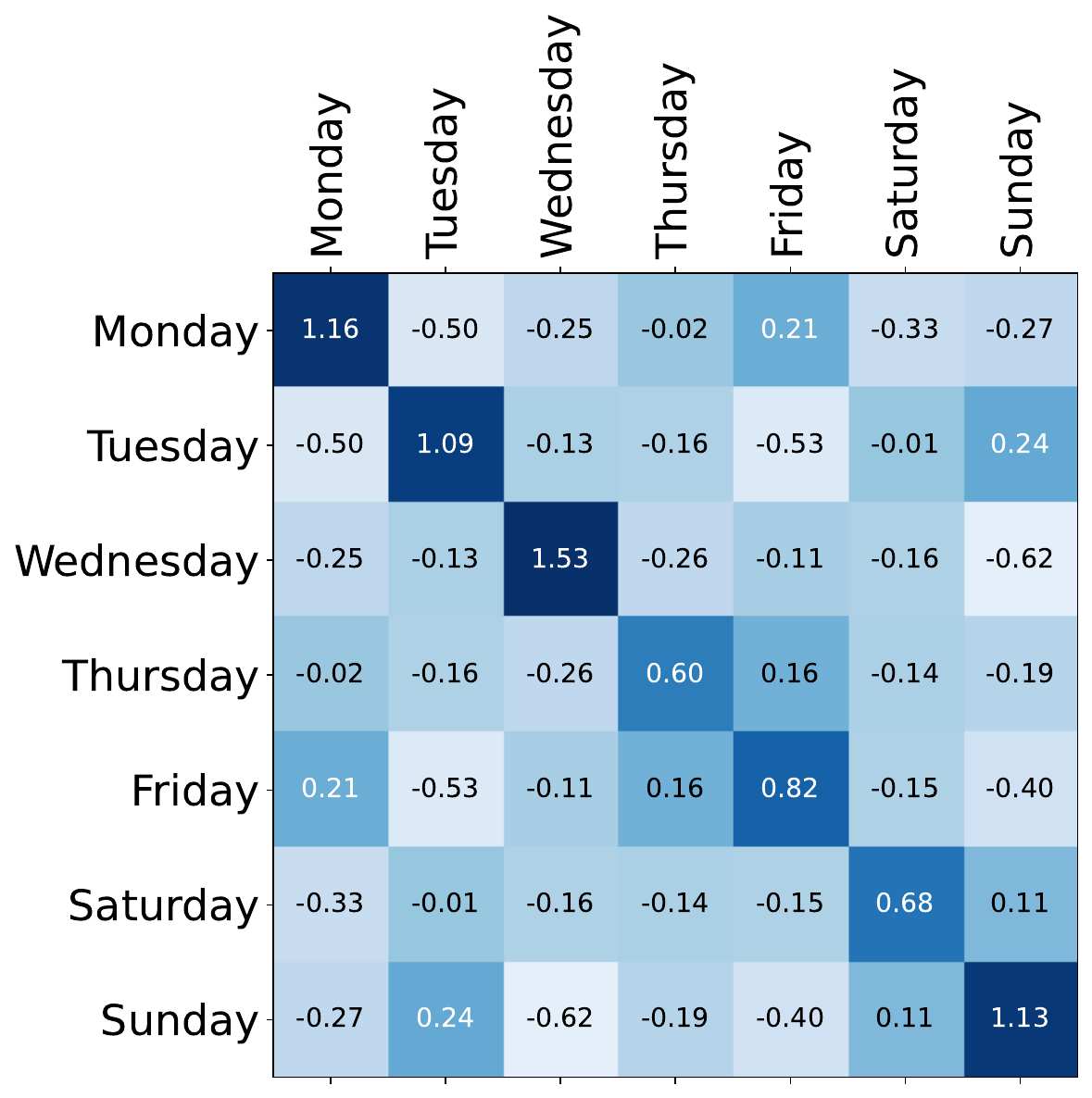}
    \caption{Gemma-2-9B-IT (\(\delta_{\mathrm{circ}}=0.505\))}
  \end{subfigure}
  \caption{Output projection Grams for weekdays.}
  \label{fig:app-word-weekdays}
\end{figure}

\subsection{Context Embeddings and Logits}

\paragraph{Permutation.}
``\texttt{Among three primary colors, <TOKEN1> is a mix of <TOKEN2> and \textvisiblespace}''. The six contexts permute (\texttt{<TOKEN1>}, \texttt{<TOKEN2>}) as (\texttt{orange}, \texttt{red}), (\texttt{orange}, \texttt{yellow}), (\texttt{green}, \texttt{yellow}), (\texttt{green}, \texttt{blue}), (\texttt{purple}, \texttt{red}), and (\texttt{purple}, \texttt{blue}). \Cref{fig:app-perm-probs} shows the predicted next-token probability matrices; except for GPT2-XL, the models assign most probability to the correct color. The larger instruction-tuned models are broadly consistent with the predicted context embedding Gram pattern in~\cref{fig:app-perm-gram}, while GPT2-XL deviates more substantially.

\begin{figure}[htbp]
  \centering
  \captionsetup[subfigure]{justification=centering}
  \begin{subfigure}[t]{0.24\textwidth}
    \centering
    \includegraphics[width=\linewidth]{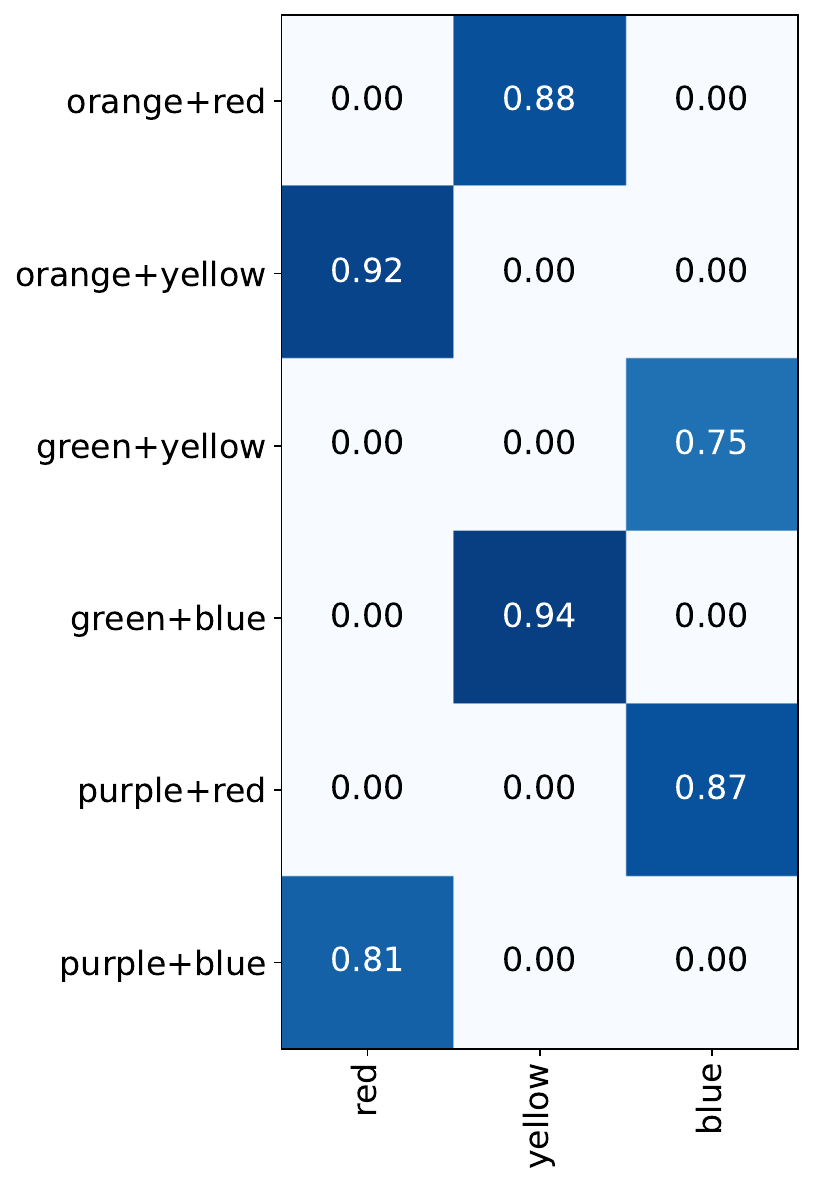}
    \caption{GPT-OSS-120B}
  \end{subfigure}\hfill
  \begin{subfigure}[t]{0.24\textwidth}
    \centering
    \includegraphics[width=\linewidth]{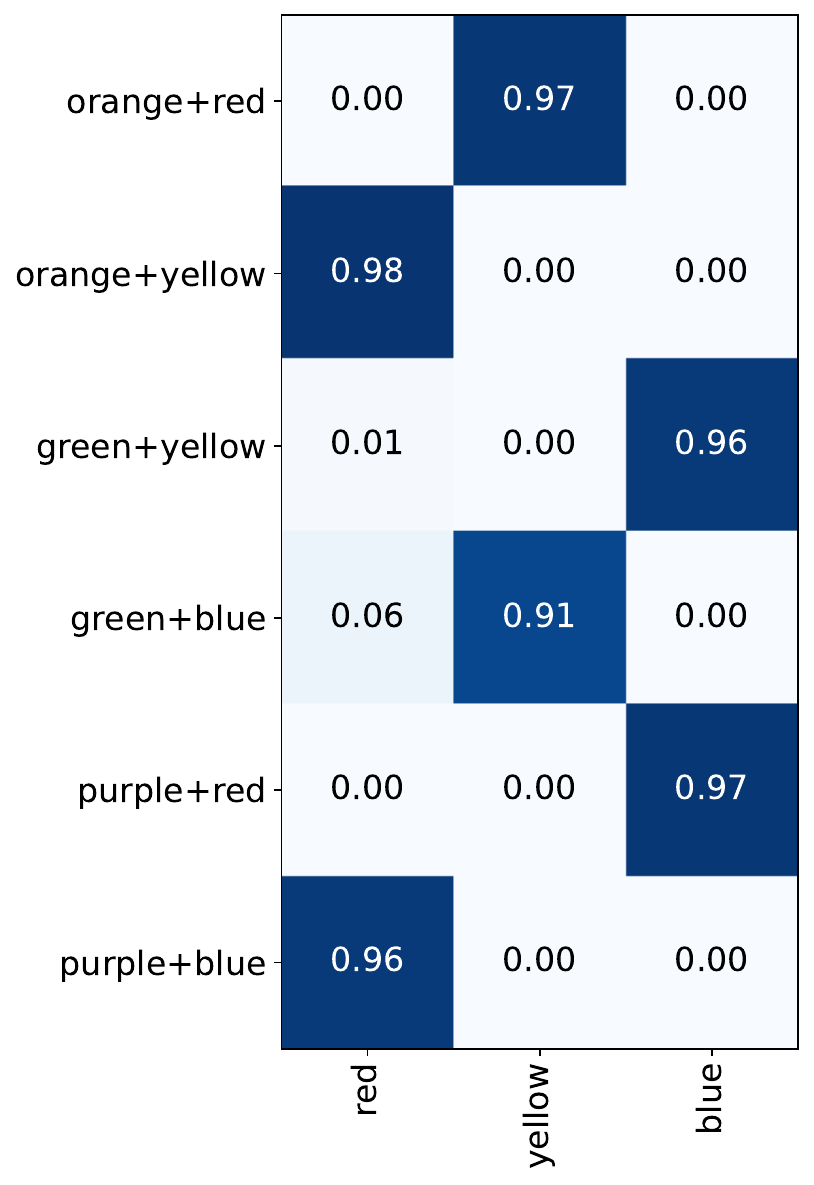}
    \caption{Llama-3.2-3B-Instruct}
  \end{subfigure}\hfill
  \begin{subfigure}[t]{0.24\textwidth}
    \centering
    \includegraphics[width=\linewidth]{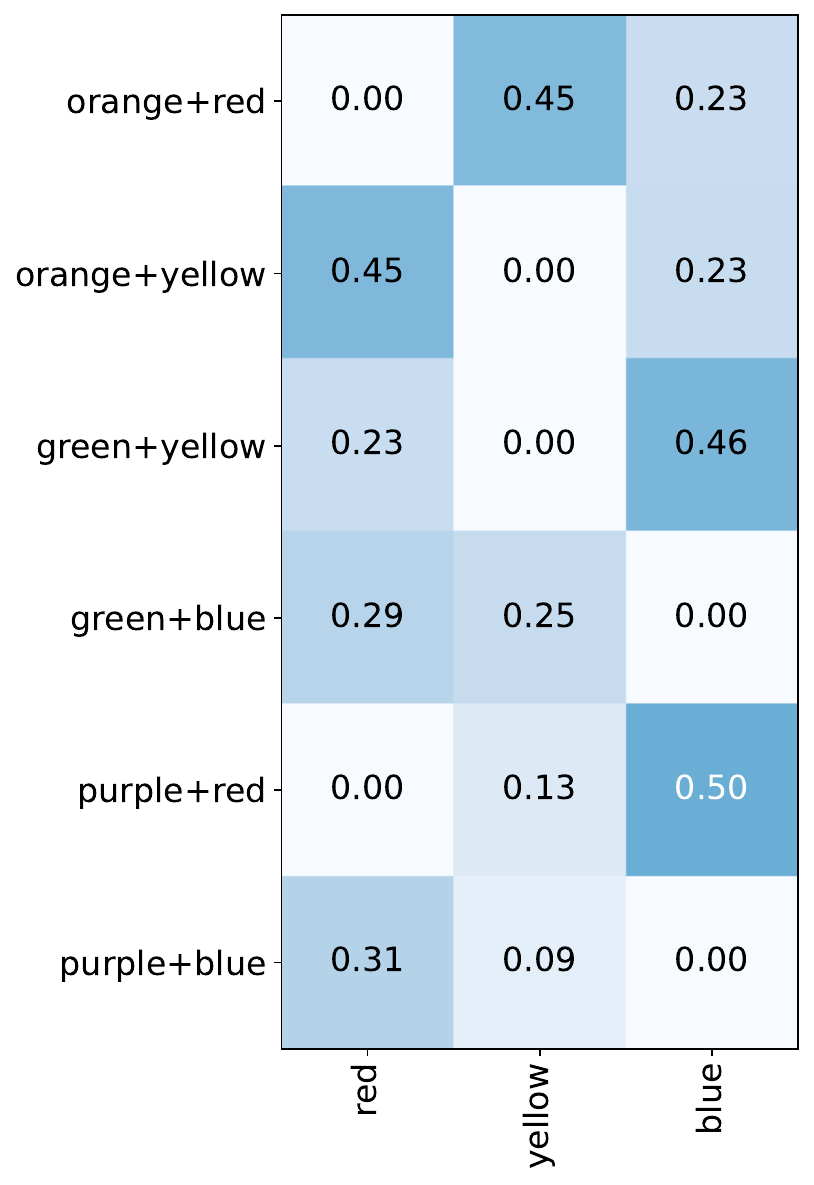}
    \caption{GPT2-XL}
  \end{subfigure}\hfill
  \begin{subfigure}[t]{0.24\textwidth}
    \centering
    \includegraphics[width=\linewidth]{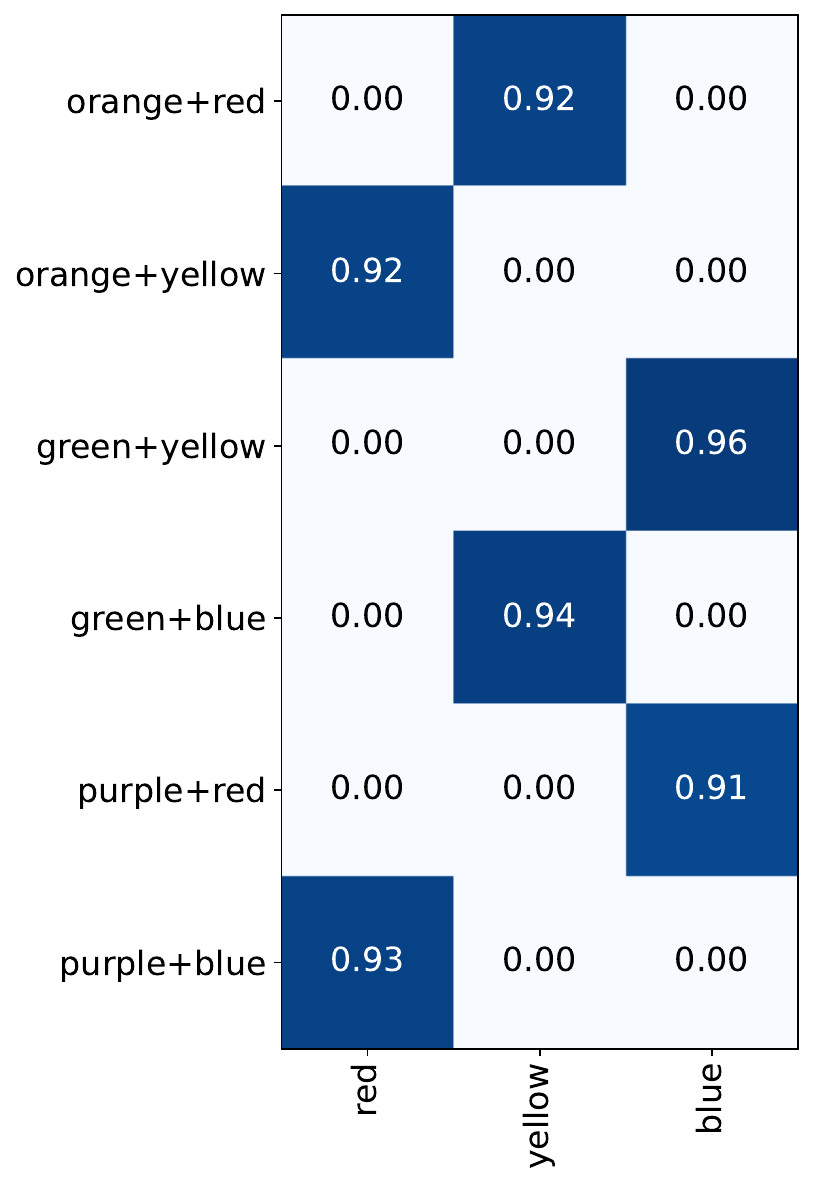}
    \caption{Gemma-2-9B-IT}
  \end{subfigure}
  \caption{Predicted next-token probability matrices for color-mix prompts.}
  \label{fig:app-perm-probs}
\end{figure}

\begin{figure}[htbp]
  \centering
  \captionsetup[subfigure]{justification=centering}
  \begin{subfigure}[t]{0.24\textwidth}
    \centering
    \includegraphics[width=\linewidth]{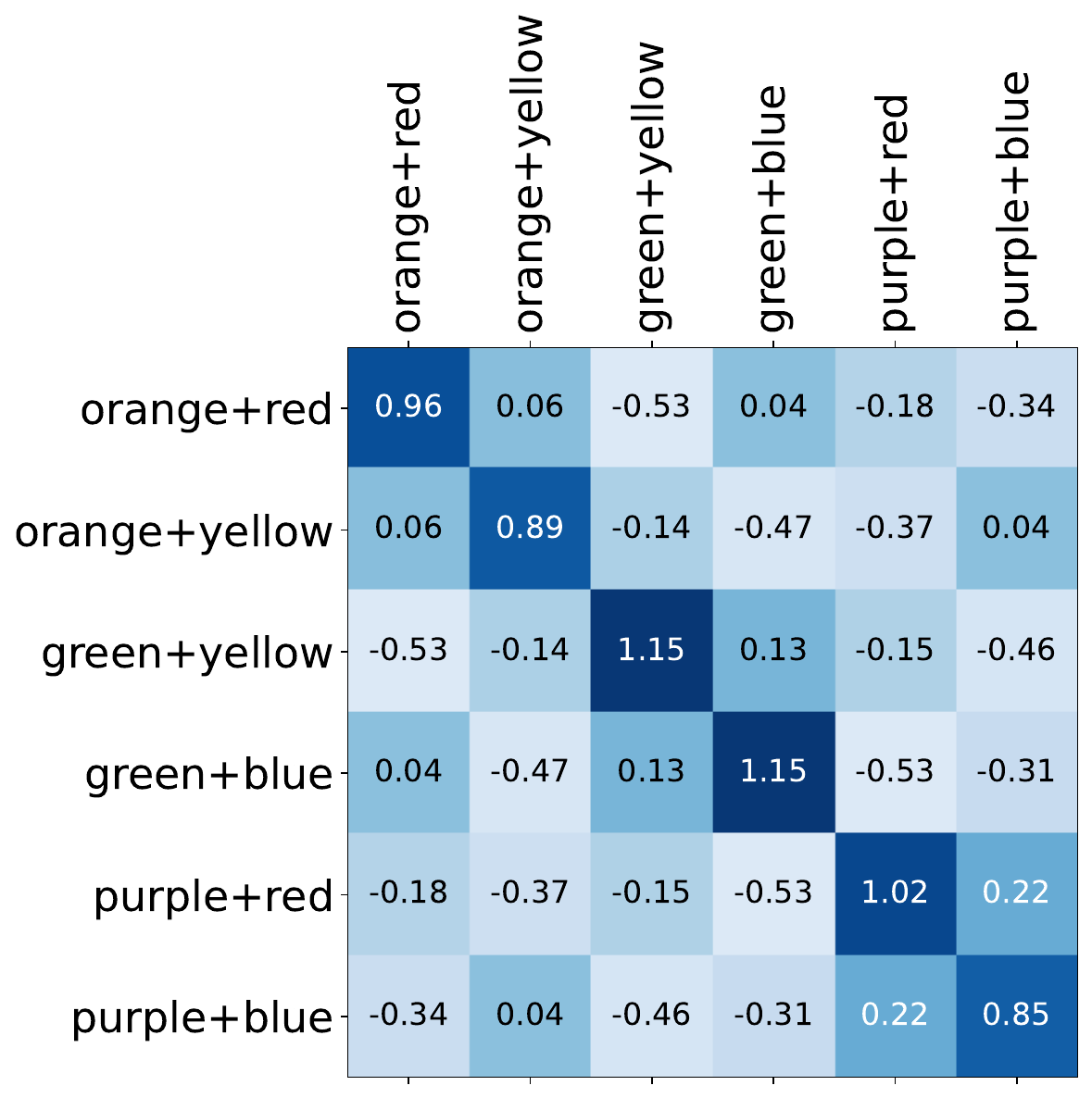}
    \caption{GPT-OSS-120B}
  \end{subfigure}\hfill
  \begin{subfigure}[t]{0.24\textwidth}
    \centering
    \includegraphics[width=\linewidth]{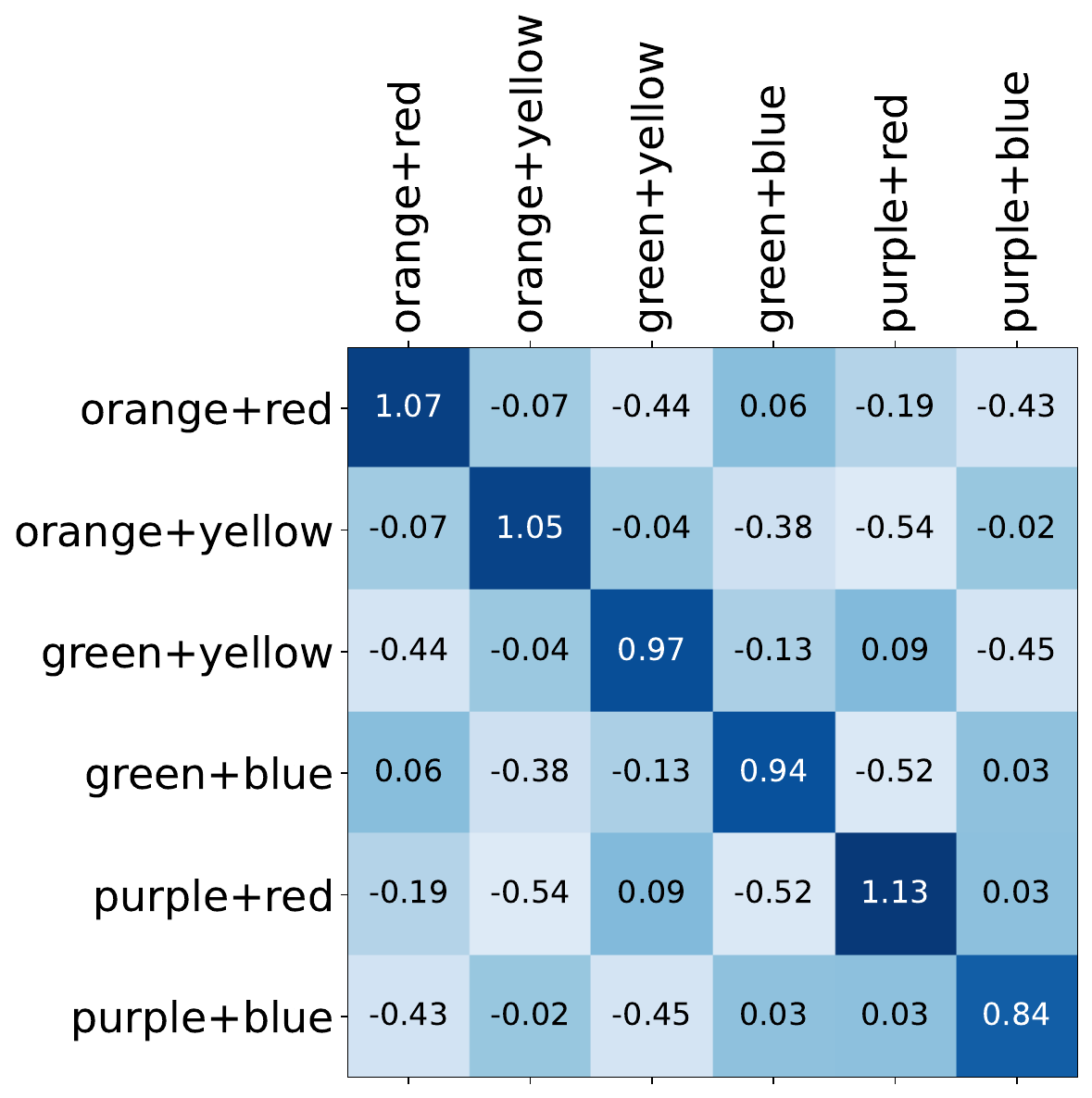}
    \caption{Llama-3.2-3B-Instruct}
  \end{subfigure}\hfill
  \begin{subfigure}[t]{0.24\textwidth}
    \centering
    \includegraphics[width=\linewidth]{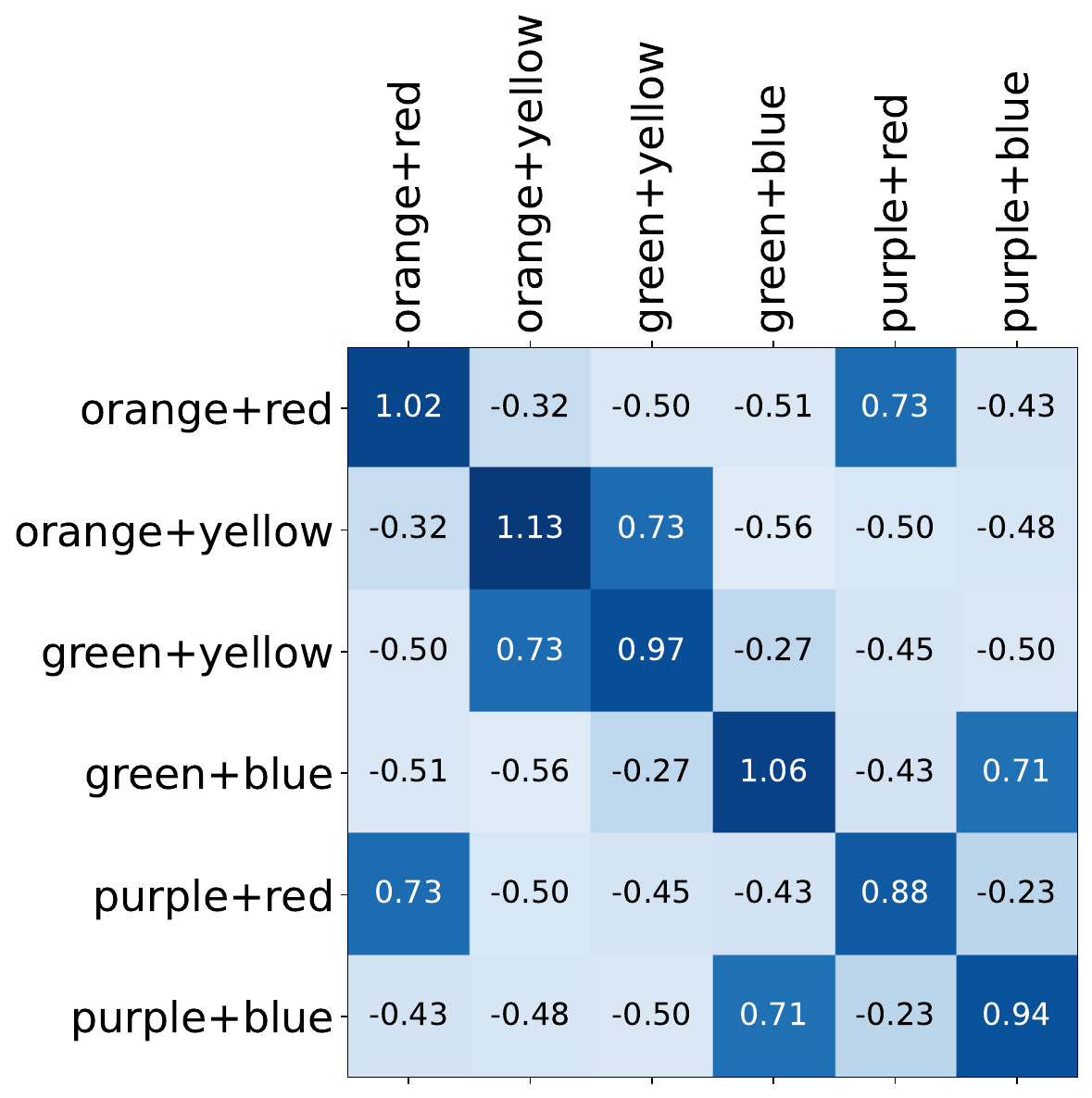}
    \caption{GPT2-XL}
  \end{subfigure}\hfill
  \begin{subfigure}[t]{0.24\textwidth}
    \centering
    \includegraphics[width=\linewidth]{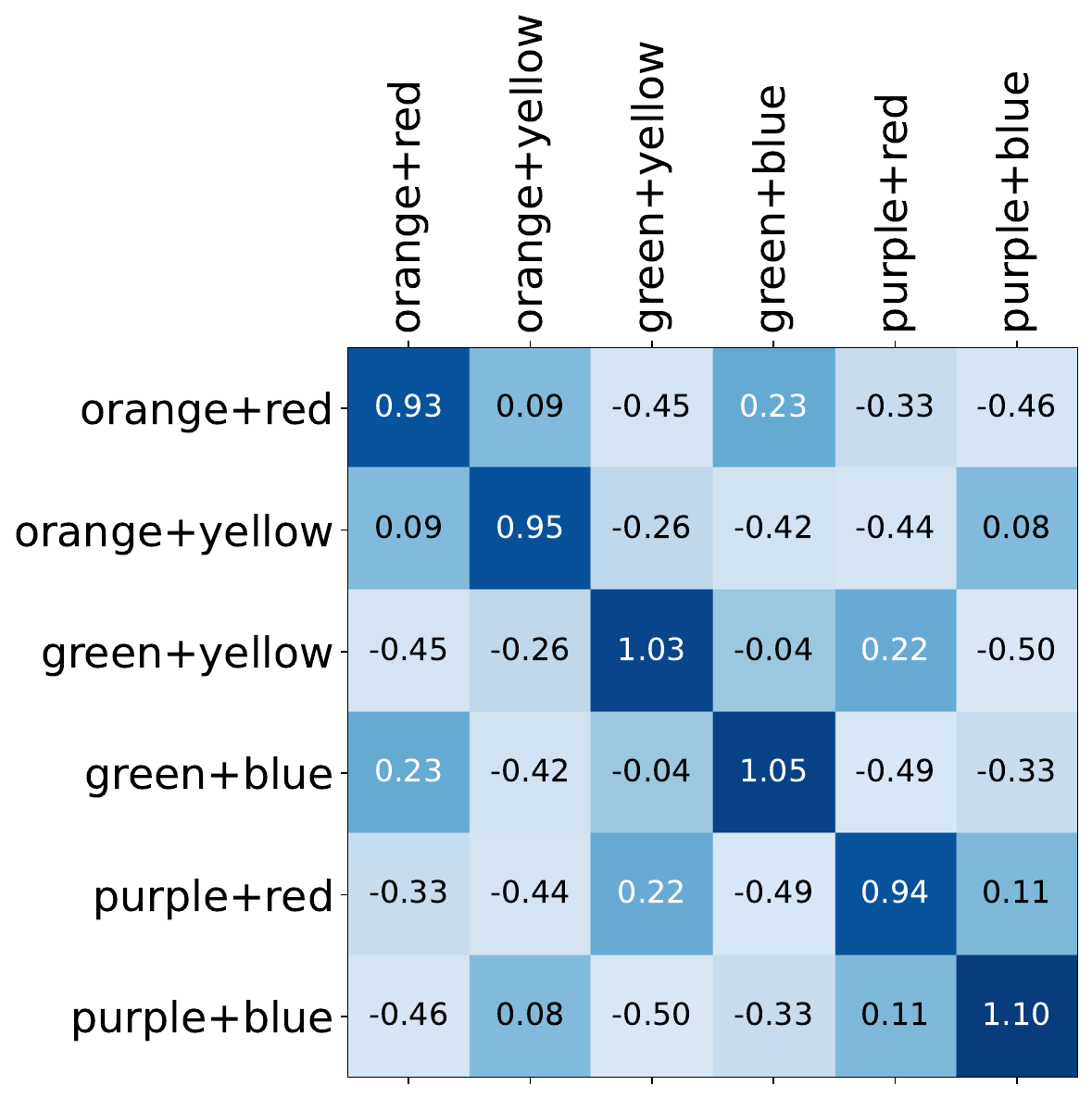}
    \caption{Gemma-2-9B-IT}
  \end{subfigure}
  \caption{Context embedding Grams for color-mix prompts.}
  \label{fig:app-perm-gram}
\end{figure}

\paragraph{Cyclic-shift prompts.}
``\texttt{Today is <TOKEN>. Tomorrow is \textvisiblespace}'' with \texttt{<TOKEN>} ranging over \{\texttt{Monday}, \texttt{Tuesday}, \texttt{Wednesday}, \texttt{Thursday}, \texttt{Friday}, \texttt{Saturday}, \texttt{Sunday}\}. \Cref{fig:app-ctx-cyclic-probs} shows the predicted next-token probability matrices, and~\cref{fig:app-ctx-cyclic-gram} shows the context embedding Grams. The cyclic-shift model motivates testing whether both displayed matrices are approximately circulant.

\begin{figure}[htbp]
  \centering
  \captionsetup[subfigure]{justification=centering}
  \begin{subfigure}[t]{0.24\textwidth}
    \centering
    \includegraphics[width=\linewidth]{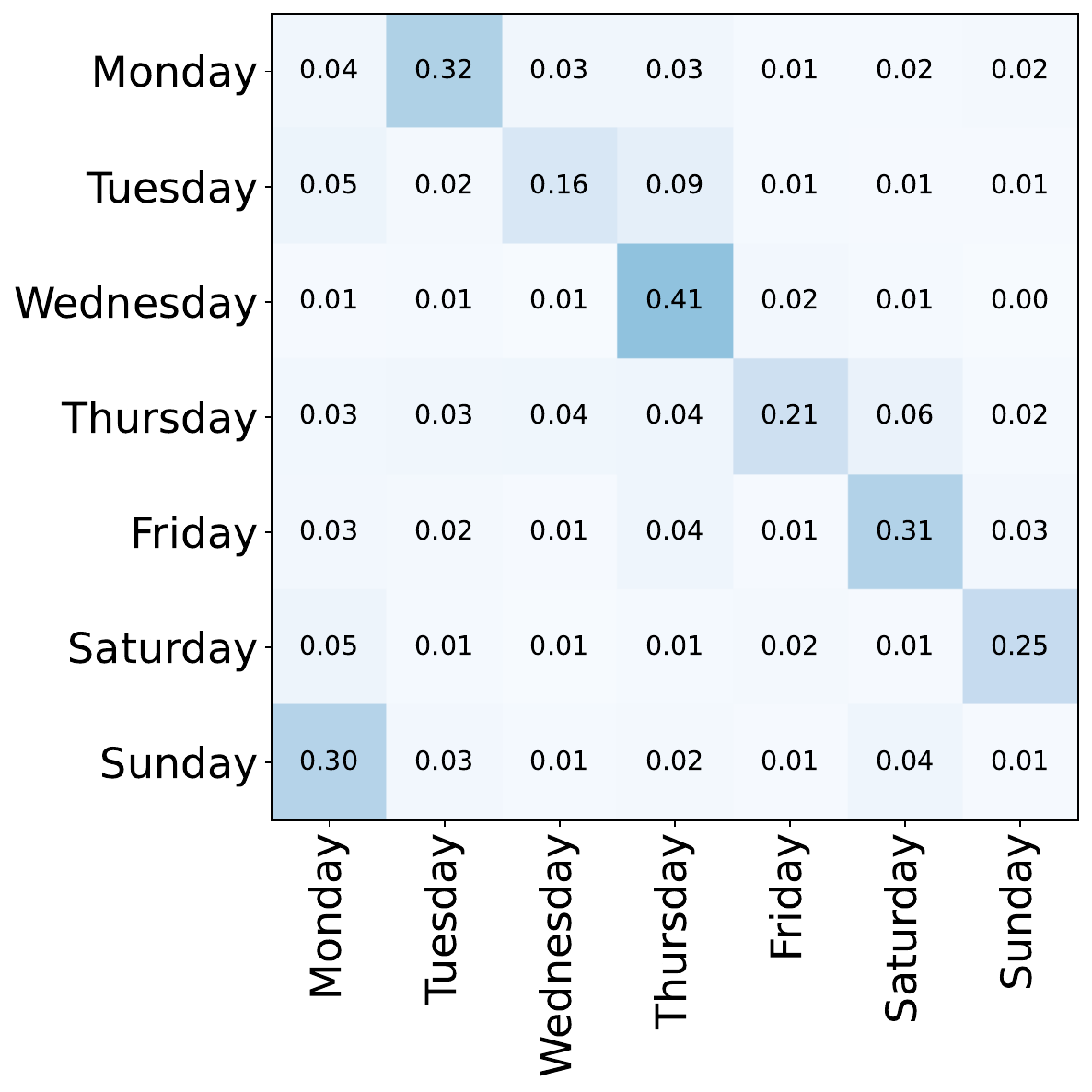}
    \caption{GPT-OSS-120B}
  \end{subfigure}\hfill
  \begin{subfigure}[t]{0.24\textwidth}
    \centering
    \includegraphics[width=\linewidth]{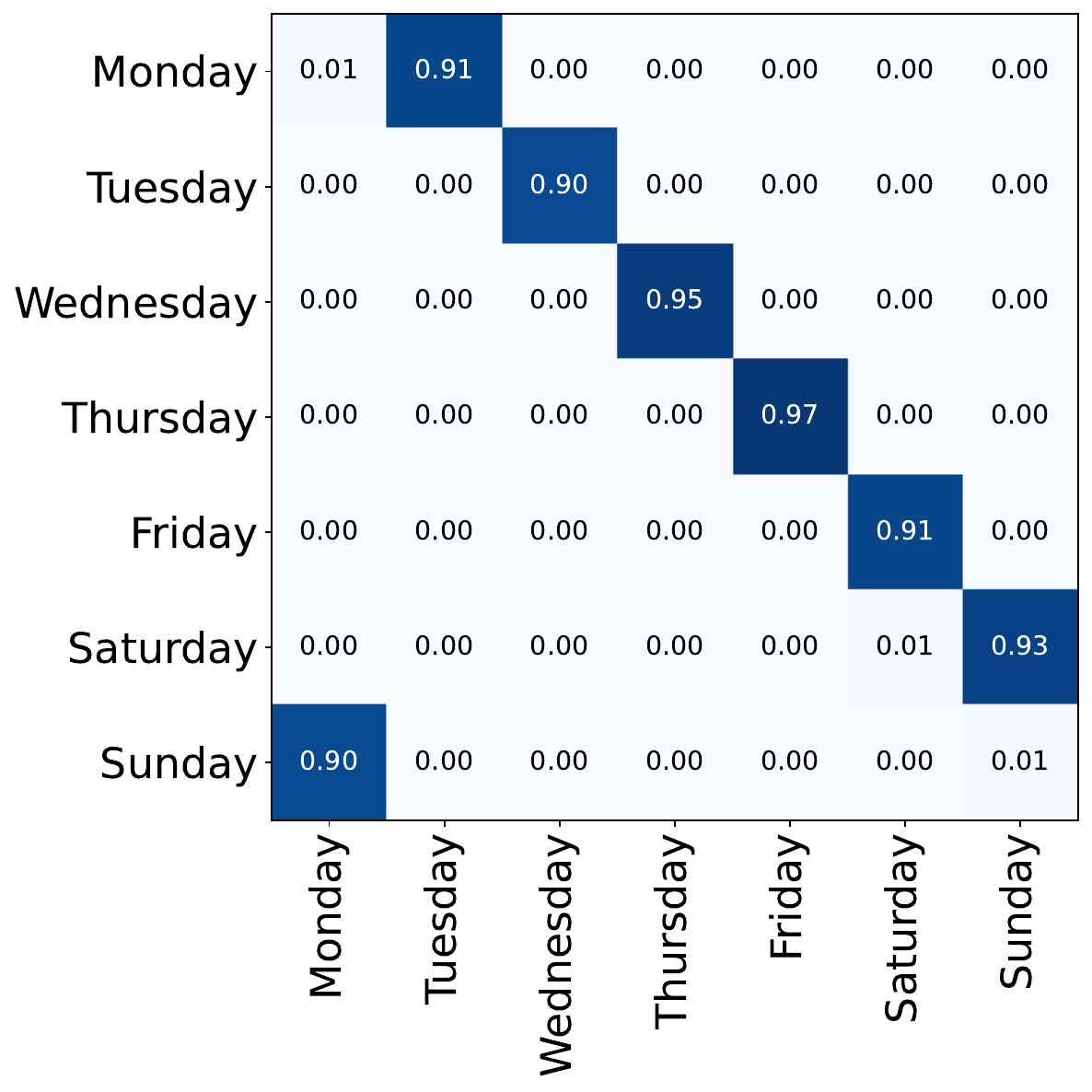}
    \caption{Llama-3.2-3B-Instruct}
  \end{subfigure}\hfill
  \begin{subfigure}[t]{0.24\textwidth}
    \centering
    \includegraphics[width=\linewidth]{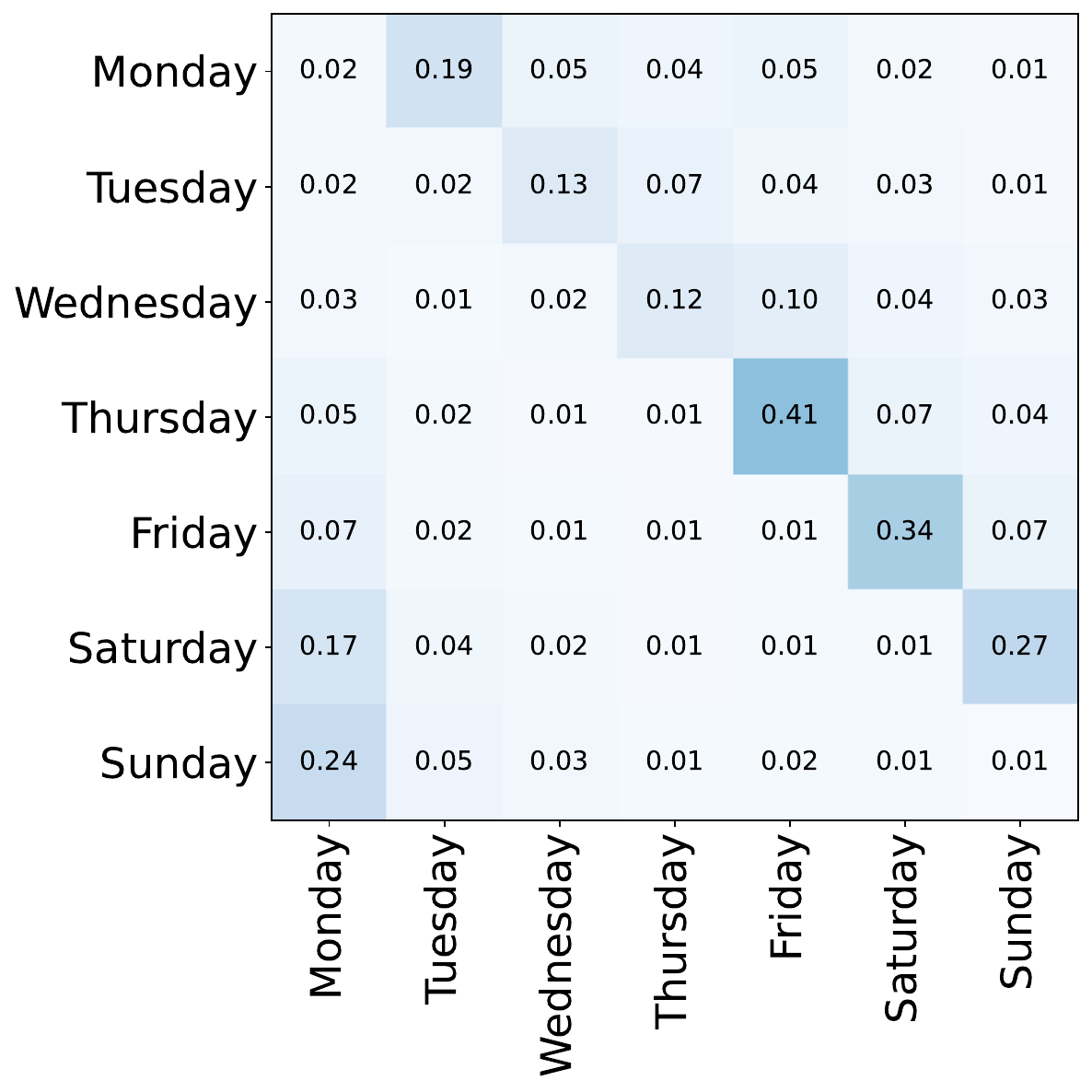}
    \caption{GPT2-XL}
  \end{subfigure}\hfill
  \begin{subfigure}[t]{0.24\textwidth}
    \centering
    \includegraphics[width=\linewidth]{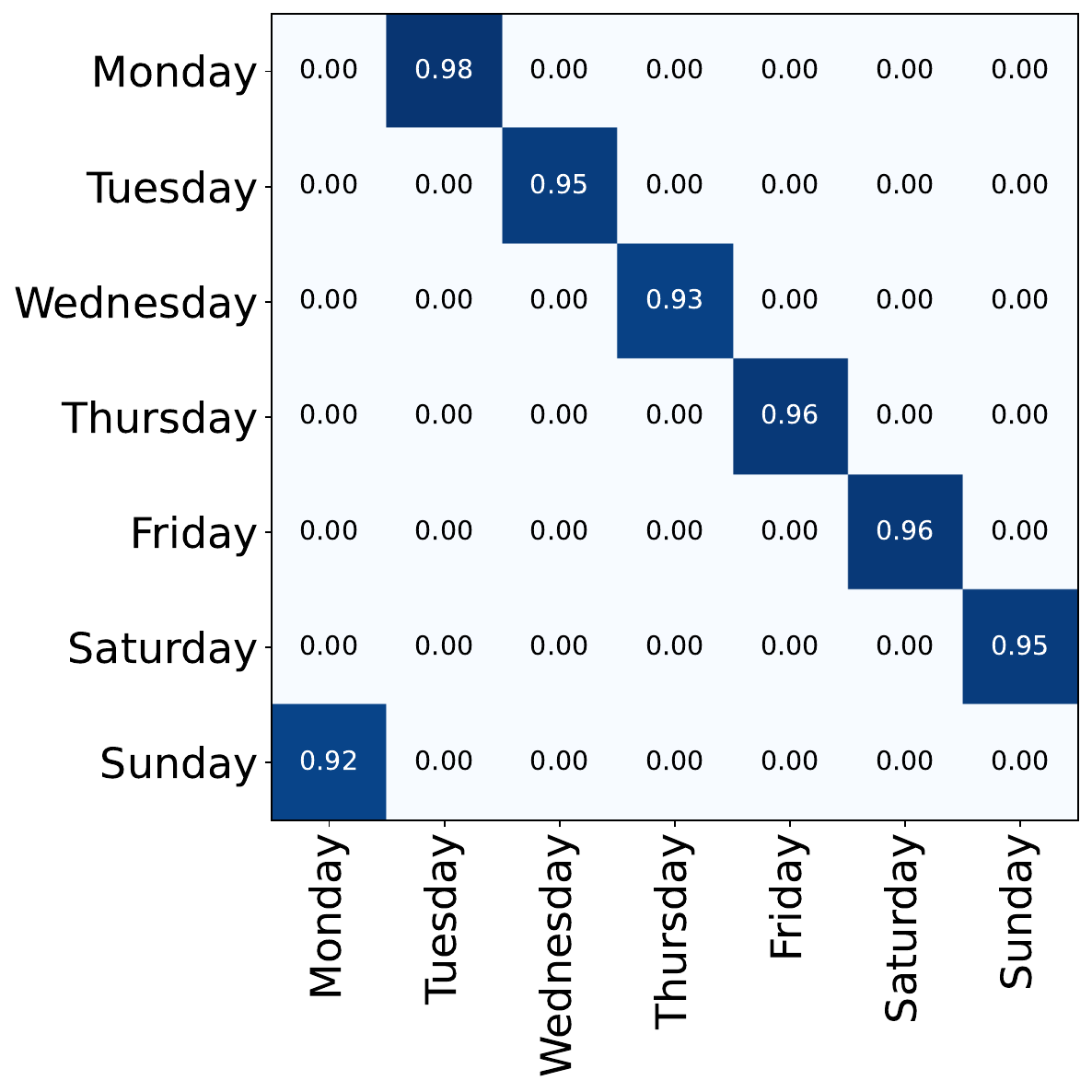}
    \caption{Gemma-2-9B-IT}
  \end{subfigure}
  \caption{Predicted next-token probability matrices for cyclic-shift weekday prompts.}
  \label{fig:app-ctx-cyclic-probs}
\end{figure}

\begin{figure}[htbp]
  \centering
  \captionsetup[subfigure]{justification=centering}
  \begin{subfigure}[t]{0.24\textwidth}
    \centering
    \includegraphics[width=\linewidth]{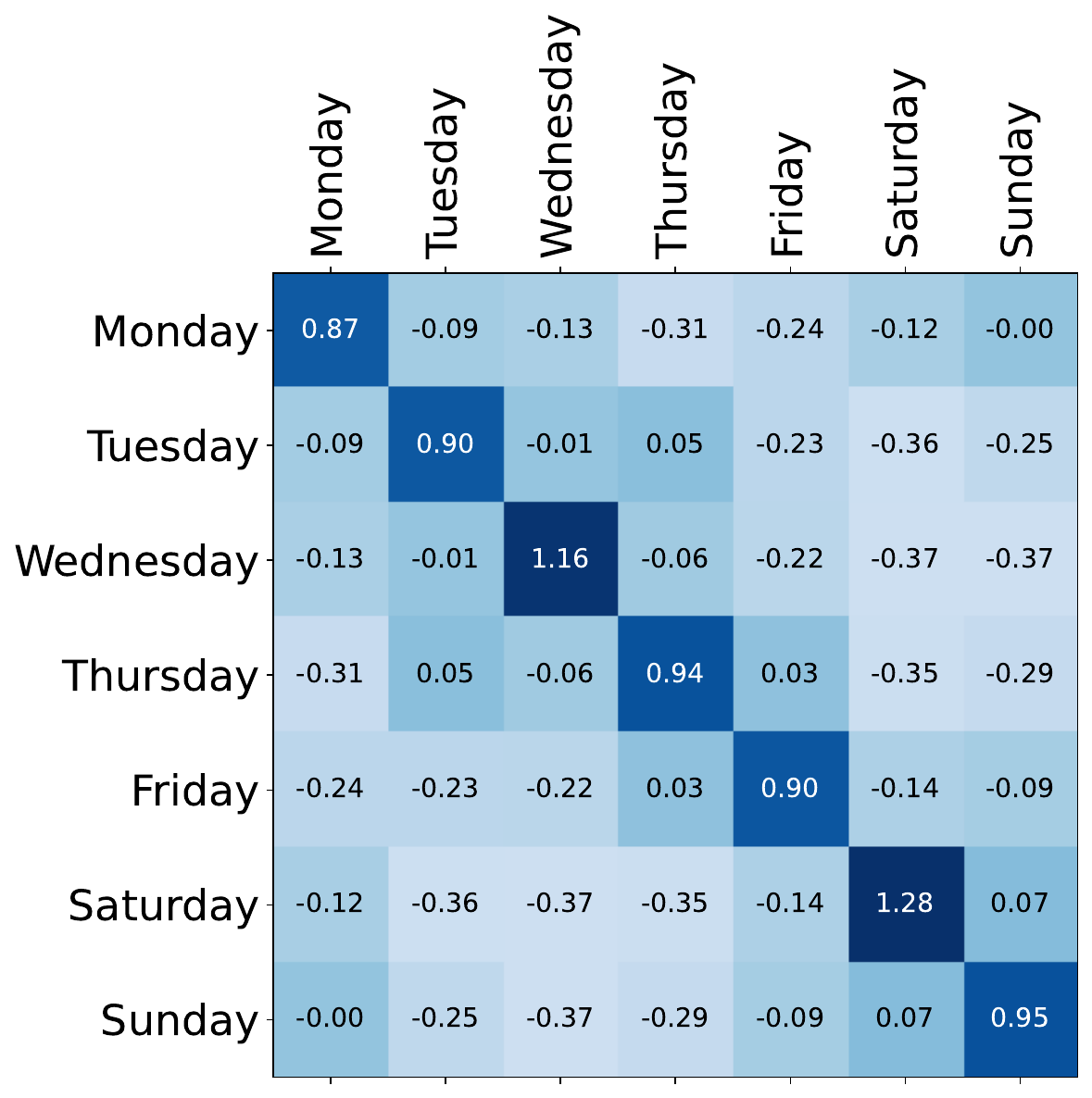}
    \caption{GPT-OSS-120B (\(\delta_{\mathrm{circ}}=0.222\))}
  \end{subfigure}\hfill
  \begin{subfigure}[t]{0.24\textwidth}
    \centering
    \includegraphics[width=\linewidth]{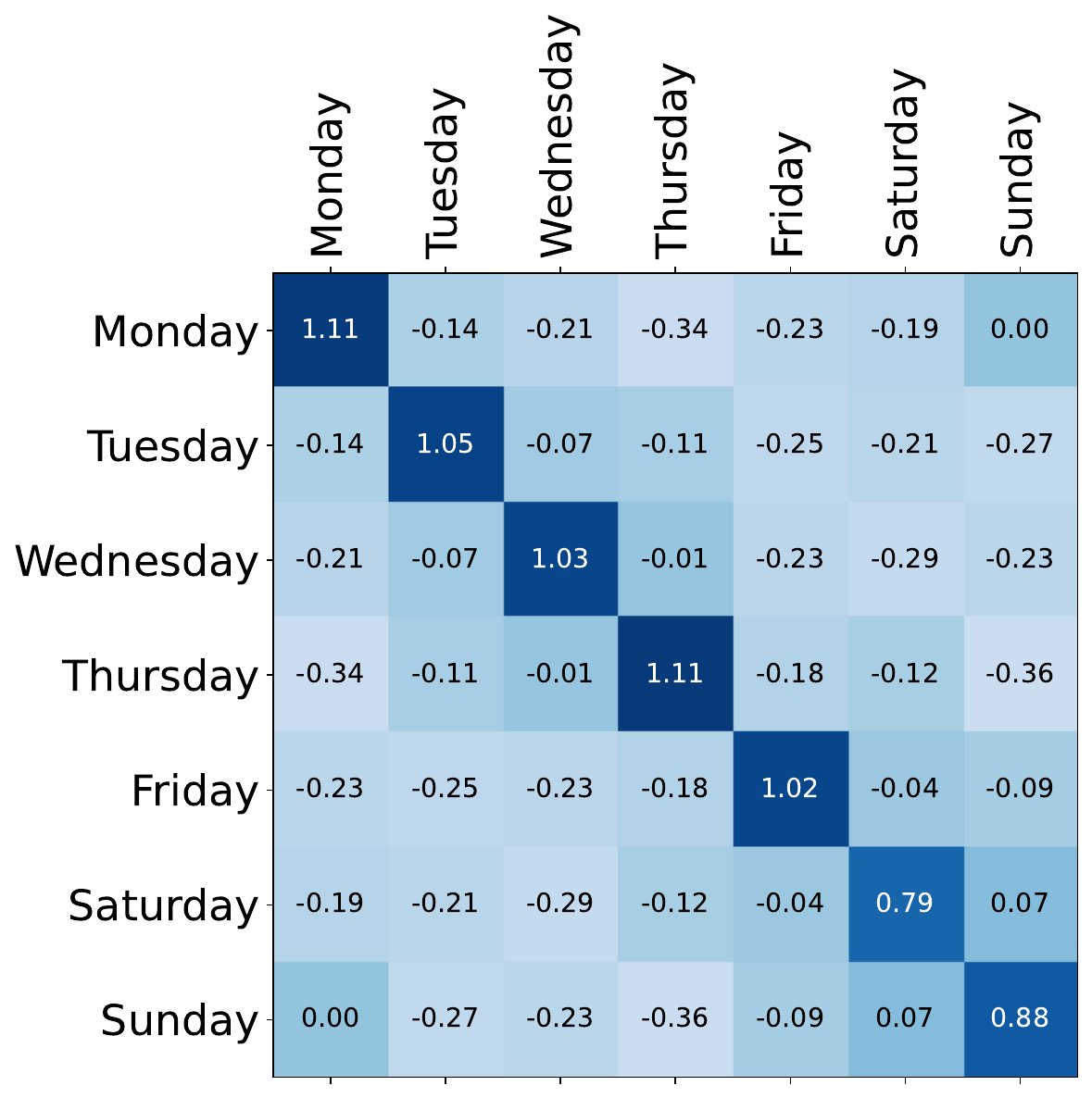}
    \caption{Llama-3.2-3B-Instruct (\(\delta_{\mathrm{circ}}=0.178\))}
  \end{subfigure}\hfill
  \begin{subfigure}[t]{0.24\textwidth}
    \centering
    \includegraphics[width=\linewidth]{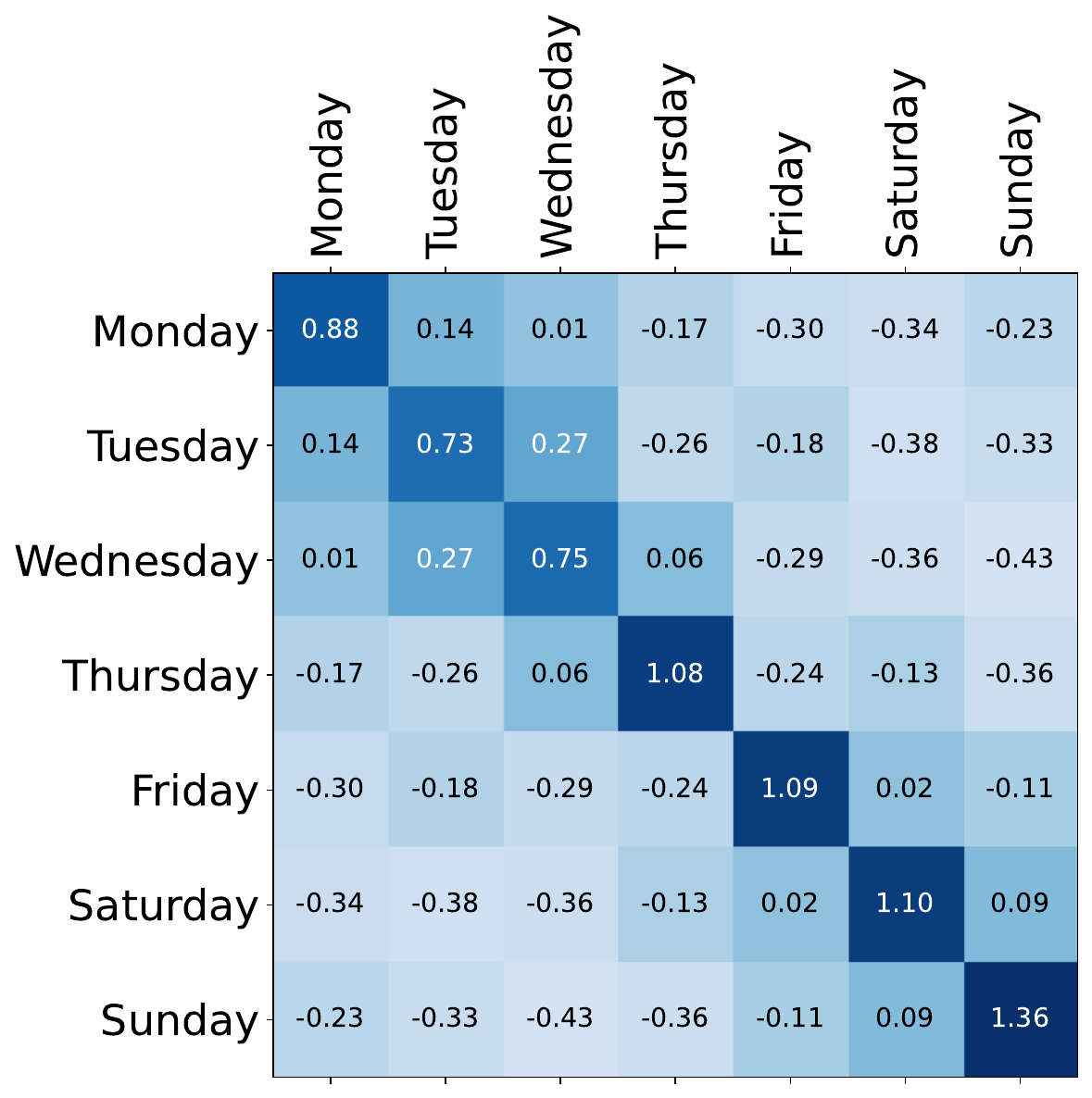}
    \caption{GPT2-XL (\(\delta_{\mathrm{circ}}=0.325\))}
  \end{subfigure}\hfill
  \begin{subfigure}[t]{0.24\textwidth}
    \centering
    \includegraphics[width=\linewidth]{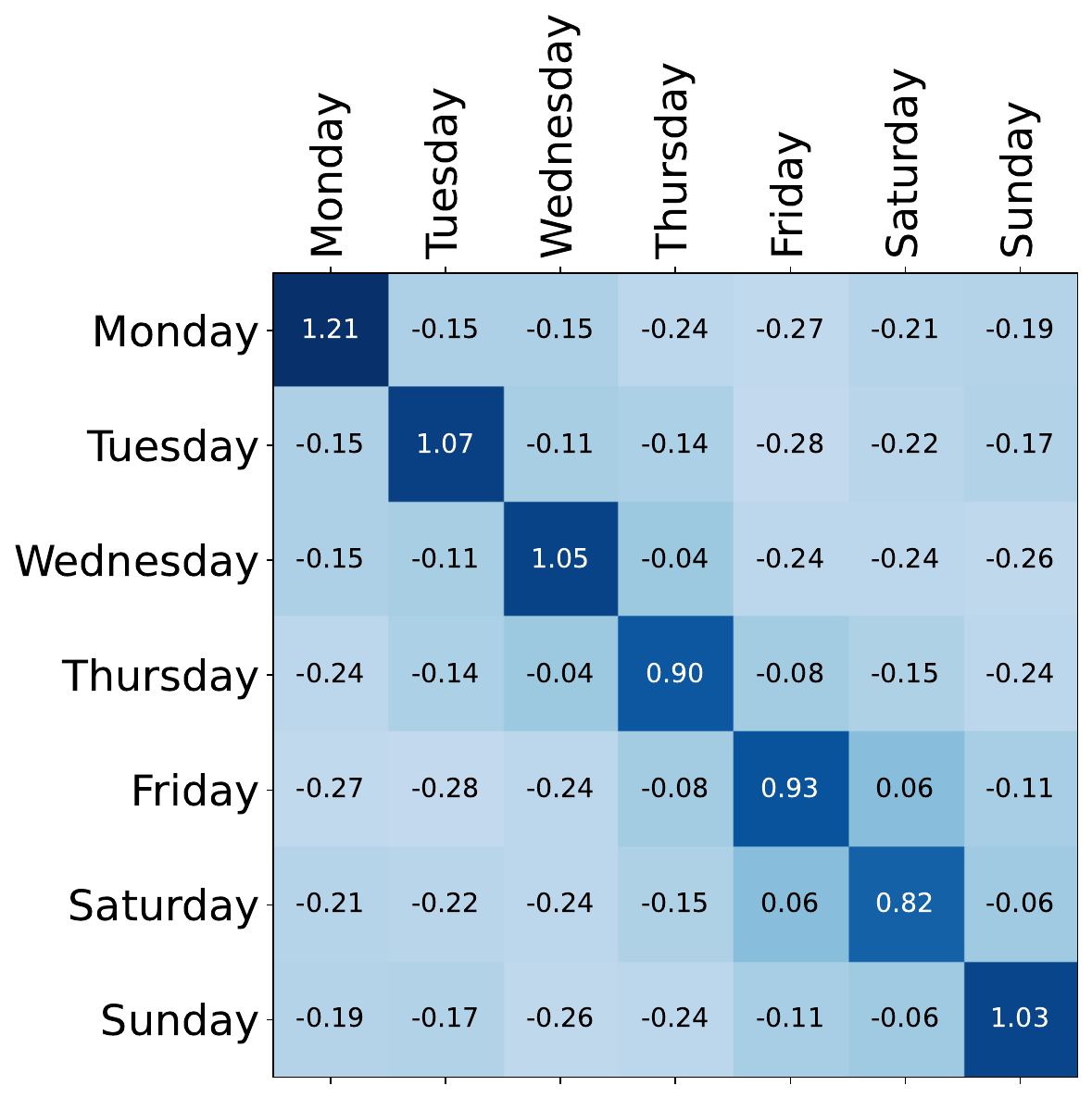}
    \caption{Gemma-2-9B-IT (\(\delta_{\mathrm{circ}}=0.154\))}
  \end{subfigure}
  \caption{Context embedding Grams for cyclic-shift weekday prompts.}
  \label{fig:app-ctx-cyclic-gram}
\end{figure}

\section{Numerical Observations for Composite Permutation Symmetries}

In this section, we report numerical results for several composite group actions whose output projection Gram matrices \(WW^{\top}\) display special structure. These patterns are not yet proved; we record them as numerical observations and leave rigorous analysis for future work.

We still consider the optimization problem~\cref{eq:obj_function_constrained}, but the matrix \(Y\) is generated by three types of groups: direct sum, direct product, and wreath product, built from symmetric groups. We adopt a convex relaxation of~\cref{eq:obj_function_constrained} by lifting to a positive semidefinite block Gram variable \(X\in\mathbb{R}^{(n+m)\times(n+m)}\) that encodes the context embedding Gram, logits, and output projection Gram:
\[
X = 
\begin{bmatrix}
\underbrace{H^{\top}H}_{\text{context embedding Gram }(n\times n)} &
\underbrace{H^{\top}W^{\top}}_{\text{logits }(n\times m)}\\[4pt]
\underbrace{WH}_{\text{logits }(m\times n)} &
\underbrace{WW^{\top}}_{\text{output projection Gram }(m\times m)}
\end{bmatrix}
\succeq 0.
\]
Let \(\boldsymbol{y}_k\in\Delta^{m-1}\) be the target for context \(k\). Since the bottom-left block of \(X\) is the logit matrix \(Z=WH\), write \(\boldsymbol z_k \coloneqq X_{n+1:n+m,k}\in\mathbb{R}^{m}\). We solve
\[
\min_{X\succeq 0} \sum_{k=1}^{n} \mathcal L  \left(\sigma(\boldsymbol z_k),  \boldsymbol{y}_k\right)
\quad\text{s.t.}\quad
\sum_{k=1}^{n}X_{kk}\le E_H,
\sum_{i=1}^{m}X_{n+i,n+i}\le E_W,
\]
which is convex because \(X\mapsto z_k\) is linear and \(\log  \sum_j e^{(\cdot)}\) is convex. We implement this in CVXPY~\citep{diamond2016cvxpy,agrawal2018rewriting} and solve with MOSEK~\citep{mosek}; the blocks of \(X^*\) provide relaxed analogues of the context embedding Gram (top-left block), the logits (off-diagonal blocks), and the output projection Gram (bottom-right block) visualized in the figures.

\subsection{Direct Sums}

We split the \(m\) coordinates into blocks of sizes \(\left(m_1,\dots,m_r\right)\) and write a base probability vector as
\[
\boldsymbol{y}=\left[ (\boldsymbol{y}^{(1)})^{\top}\mid\cdots\mid(\boldsymbol{y}^{(r)})^{\top} \right]^{\top},
\qquad
\boldsymbol{y}^{(i)}\in\mathbb{R}^{m_i}_{\ge 0},\quad
\sum_{i=1}^r \boldsymbol{1}_{m_i}^{\top}\boldsymbol{y}^{(i)}=1.
\]
The group \(\mathcal G_{\mathrm{sum}}=\prod_{i=1}^r S_{m_i}\) acts by permuting within each block: for \(\boldsymbol{\tau}=\left(\tau_1,\dots,\tau_r\right)\in\mathcal G_{\mathrm{sum}}\), define
\[
\boldsymbol{\tau}\circ\boldsymbol{y}
:=\left[  (P_{\tau_1}\boldsymbol{y}^{(1)})^{\top}\mid \cdots \mid (P_{\tau_r}\boldsymbol{y}^{(r)})^{\top}  \right]^{\top},
\qquad
Y=\left[  \boldsymbol{\tau}\circ\boldsymbol{y}  \right]_{\boldsymbol{\tau}\in\mathcal G_{\mathrm{sum}}},\quad
n=\left|\mathcal G_{\mathrm{sum}}\right|=\prod_i m_i!.
\]
Here \(P_{\tau_i}\) is the concrete permutation matrix implementing the component action \(\tau_i\circ\) on block \(i\).
In the first experiment, two blocks of sizes \(a=2\) and \(b=3\) are used with
\[
\boldsymbol{y}=\left[(\boldsymbol{y}^{(1)})^{\top}\mid (\boldsymbol{y}^{(2)})^{\top}\right]^{\top}
= \left[  0, \frac14 \ \mid\ \frac12, \frac16, \frac1{12}  \right]^{\top},
\]
the acting group is \(S_2\times S_3\), the full orbit gives \(n=12\), and
\(
Y=\left[  \boldsymbol{\tau}\circ\boldsymbol{y}  \right]_{\boldsymbol{\tau}=(\tau_1,\tau_2)\in S_2\times S_3}
\)
(see~\cref{fig:ds1}).

\begin{figure}[H]
  \centering
  \includegraphics[width=\textwidth]{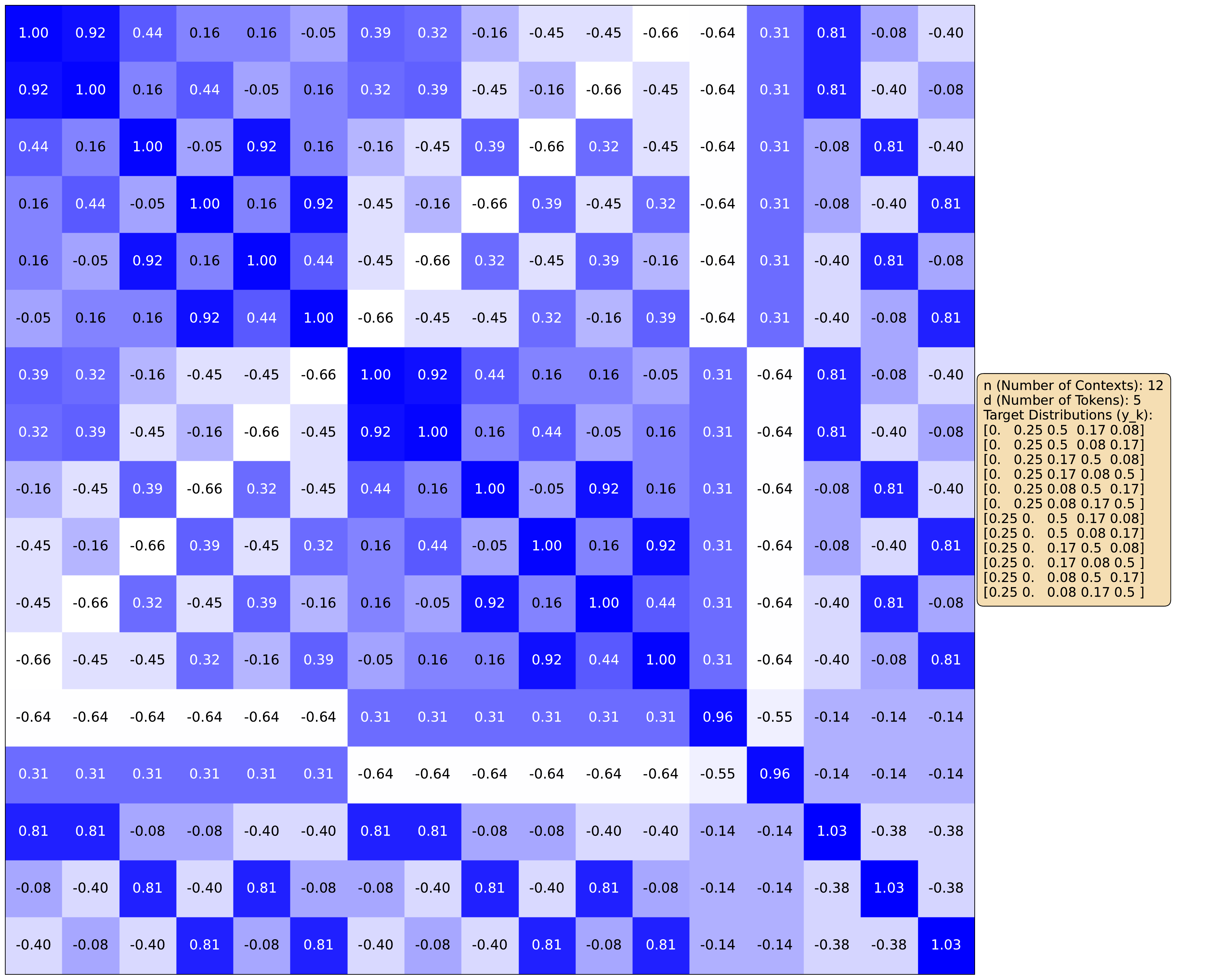}
  \caption{Direct sum \(\left(S_2\times S_3\right)\); full orbit \(n=12\).}
  \label{fig:ds1}
\end{figure}

In the second experiment, three blocks of sizes \(2,3,2\) are used with
\[
\boldsymbol{y}=\left[(\boldsymbol{y}^{(1)})^{\top}\mid (\boldsymbol{y}^{(2)})^{\top}\mid (\boldsymbol{y}^{(3)})^{\top}\right]^{\top}
=\left[  \frac1{16}, \frac13 \ \mid\ \frac14, \frac3{16}, 0 \ \mid\ \frac18, \frac1{24}  \right]^{\top},
\]
the acting group is \(S_2\times S_3\times S_2\), the full orbit gives \(n=24\), and
\(
Y=\left[  \boldsymbol{\tau}\circ\boldsymbol{y}  \right]_{\boldsymbol{\tau}=(\tau_1,\tau_2,\tau_3)\in S_2\times S_3\times S_2}
\)
(see~\cref{fig:ds2}).

\begin{figure}[H]
  \centering
  \includegraphics[width=\textwidth]{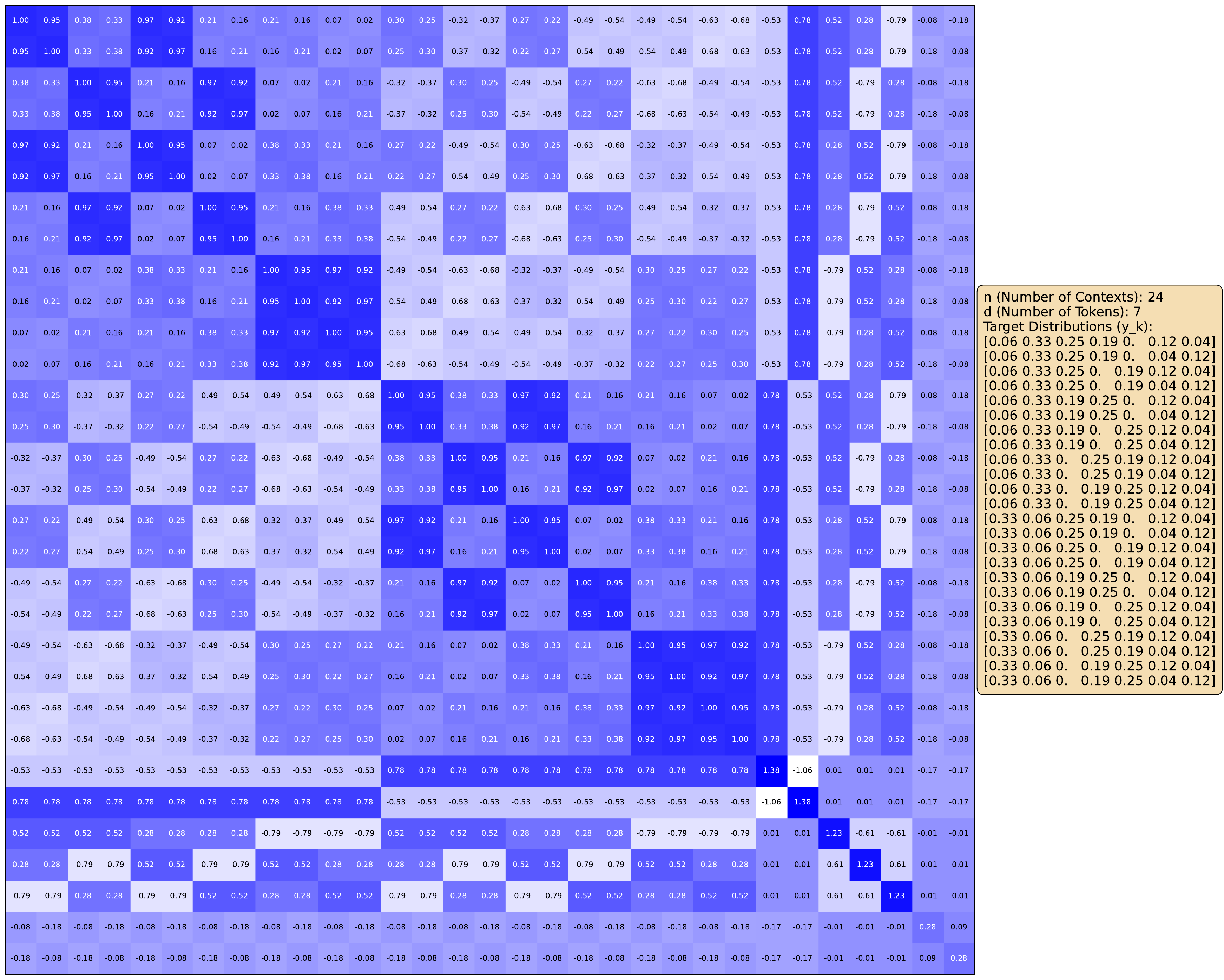}
  \caption{Direct sum \(\left(S_2\times S_3\times S_2\right)\); full orbit \(n=24\).}
  \label{fig:ds2}
\end{figure}

In the numerical solutions, \(WW^{\top}\) has diagonal blocks with a two-level within-block pattern and off-diagonal blocks with constant entries. Concretely, for the partition \(m=m_1+\cdots+m_r\),
\[
WW^{\top}  =
\begin{bmatrix}
A_1 & \kappa_{12} \boldsymbol{1}_{m_1}\boldsymbol{1}_{m_2}^{\top} & \cdots & \kappa_{1r} \boldsymbol{1}_{m_1}\boldsymbol{1}_{m_r}^{\top}\\
\kappa_{21} \boldsymbol{1}_{m_2}\boldsymbol{1}_{m_1}^{\top} & A_2 & \cdots & \kappa_{2r} \boldsymbol{1}_{m_2}\boldsymbol{1}_{m_r}^{\top}\\
\vdots & \vdots & \ddots & \vdots\\
\kappa_{r1} \boldsymbol{1}_{m_r}\boldsymbol{1}_{m_1}^{\top} & \kappa_{r2} \boldsymbol{1}_{m_r}\boldsymbol{1}_{m_2}^{\top} & \cdots & A_r
\end{bmatrix},
\qquad
A_i  =  \alpha_i I_{m_i} + \beta_i\left(J_{m_i}-I_{m_i}\right),
\]
where each diagonal block \(A_i\) has constant diagonal and constant off-diagonal entries, each off-diagonal block is a constant matrix with all entries \(\kappa_{ij}\), and \(J_{q}:=\boldsymbol{1}_{q}\boldsymbol{1}_{q}^{\top}\).

\subsection{Direct Products}

We identify coordinates with an \(a\times \ell\) grid so that \(m=a\ell\), and reshape a base probability vector
\(\boldsymbol{y}\in\Delta^{m-1}\) into a matrix
\(P\in\mathbb{R}^{a\times \ell}_{\ge 0}\) satisfying
\(\sum_{i=1}^{a}\sum_{j=1}^{\ell}P_{ij}=1\).
We flatten \(\boldsymbol{y}=\mathrm{vec}(P)\) in row-major order.
The group \(\mathcal G_{\mathrm{prod}}=S_a\times S_{\ell}\) acts by row/column relabeling. For \(\boldsymbol{\tau}=(\tau_1,\tau_2)\), define
\[
\boldsymbol{\tau}\circ\boldsymbol{y}
=
\mathrm{vec}\!\left(P_{\tau_1} P P_{\tau_2}^{\top}\right),
\qquad
Y=\left[  \boldsymbol{\tau}\circ\boldsymbol{y}  \right]_{\boldsymbol{\tau}=(\tau_1,\tau_2)\in S_a\times S_{\ell}},
\quad
n=a!\ell!.
\]
Here \(P_{\tau_1}\) and \(P_{\tau_2}\) are the concrete permutation matrices implementing the row and column actions in \(\boldsymbol{\tau}\circ\boldsymbol{y}\).
The acting group is \(S_2\times S_3\), and \(n=12\) (see~\cref{fig:dp1}).

\begin{figure}[H]
  \centering
  \includegraphics[width=\textwidth]{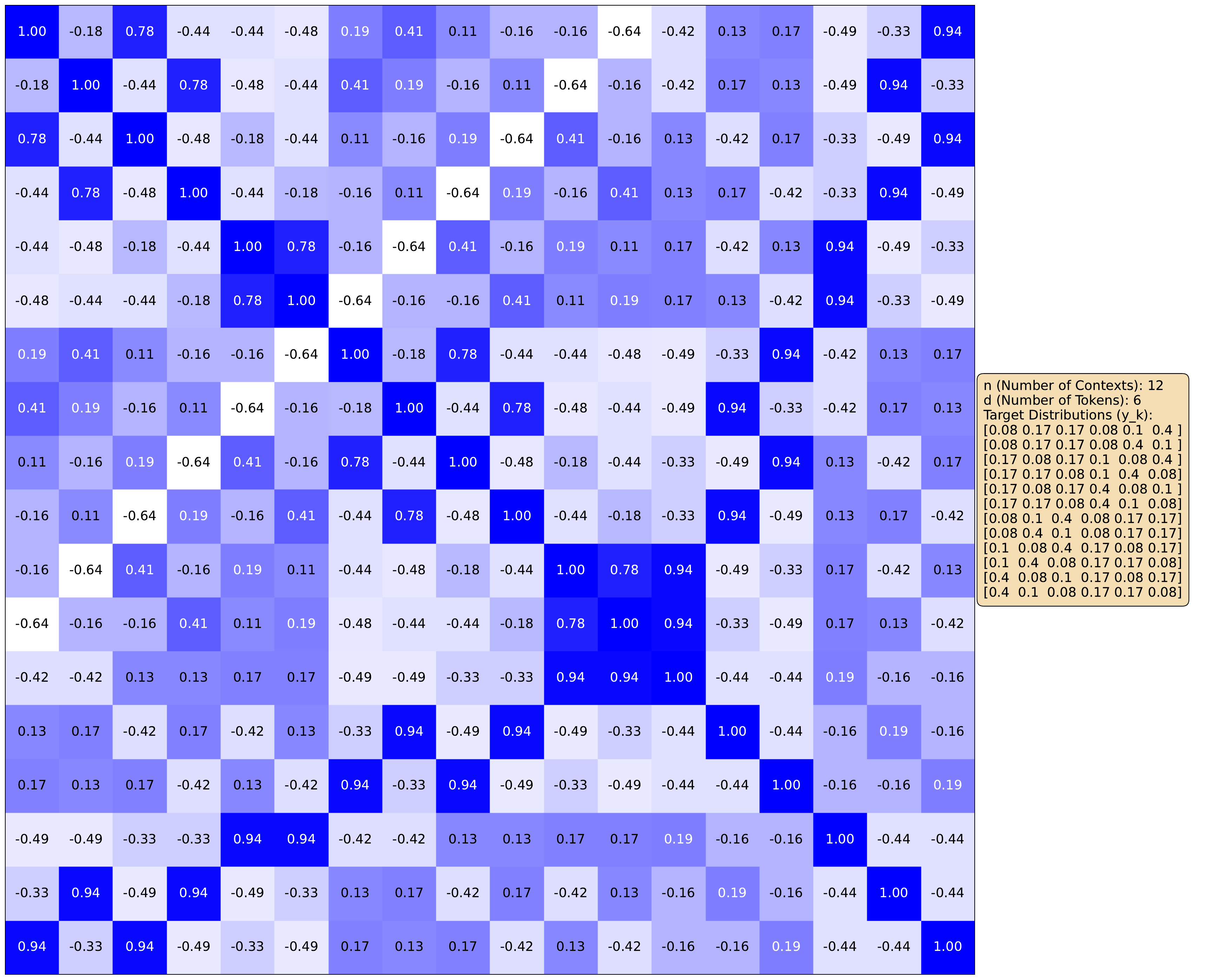}
  \caption{Direct product \(\left(S_2\times S_3\right)\); full orbit \(n=12\).}
  \label{fig:dp1}
\end{figure}

The second configuration reshapes the same entries as
\[
P=\begin{bmatrix}
\frac1{12} & \frac2{12}\\[4pt]
\frac16 & \frac1{12}\\[4pt]
\frac1{10} & \frac4{10}
\end{bmatrix},
\qquad
\boldsymbol{y}=\mathrm{vec} \left(P\right),
\]
The acting group is \(S_3\times S_2\), and \(n=12\) (see~\cref{fig:dp2}).

\begin{figure}[H]
  \centering
  \includegraphics[width=\textwidth]{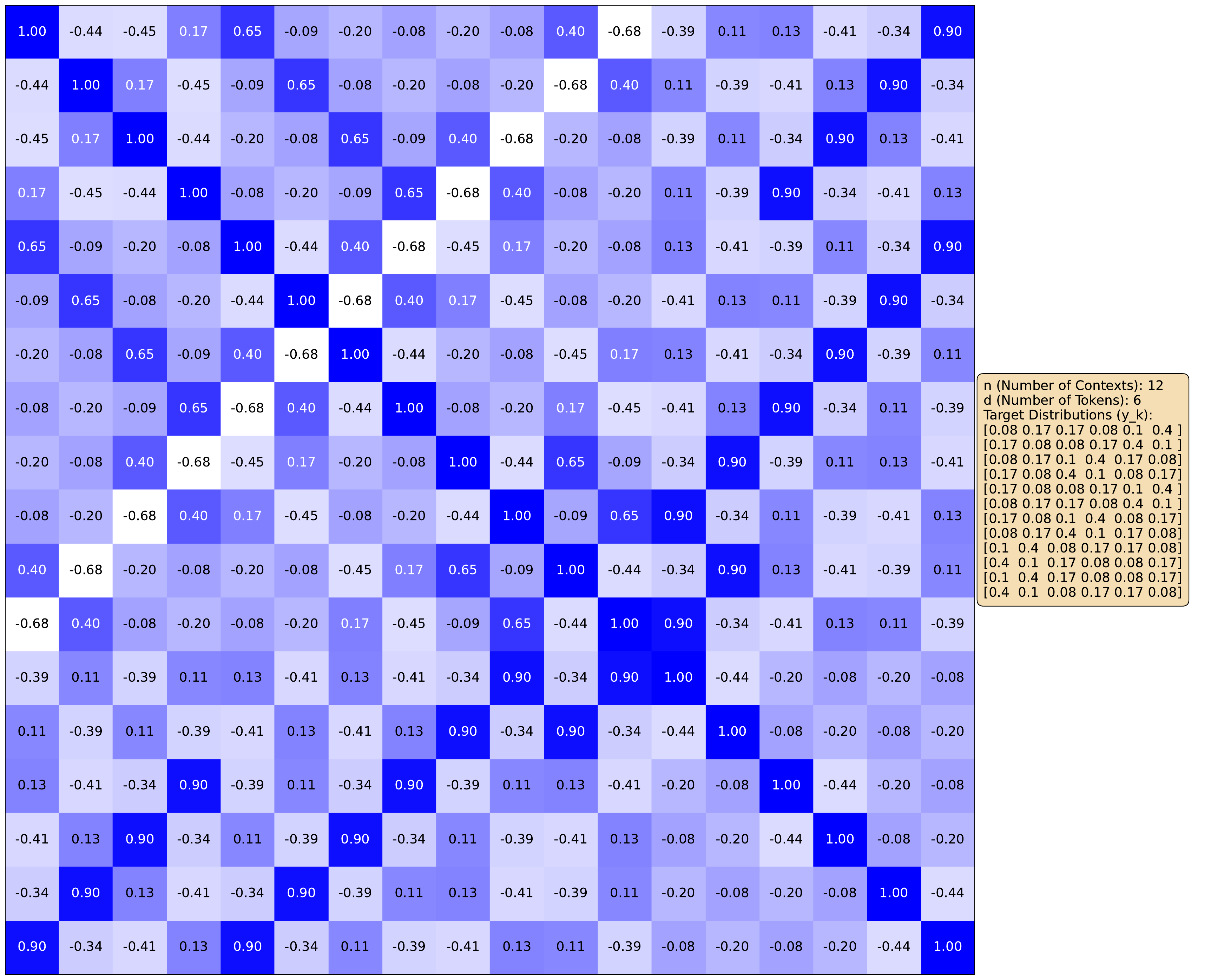}
  \caption{Direct product \(\left(S_3\times S_2\right)\); full orbit \(n=12\).}
  \label{fig:dp2}
\end{figure}

The third configuration uses
\[
P=\begin{bmatrix}
0 & \frac1{9} & \frac1{18}\\[4pt]
\frac1{18} & \frac1{9} & \frac16\\[4pt]
\frac1{12} & \frac16 & \frac14
\end{bmatrix},
\qquad
\boldsymbol{y}=\mathrm{vec} \left(P\right),
\]
The acting group is \(S_3\times S_3\), and \(n=36\) (see~\cref{fig:dp3}).

\begin{figure}[H]
  \centering
  \includegraphics[width=\textwidth]{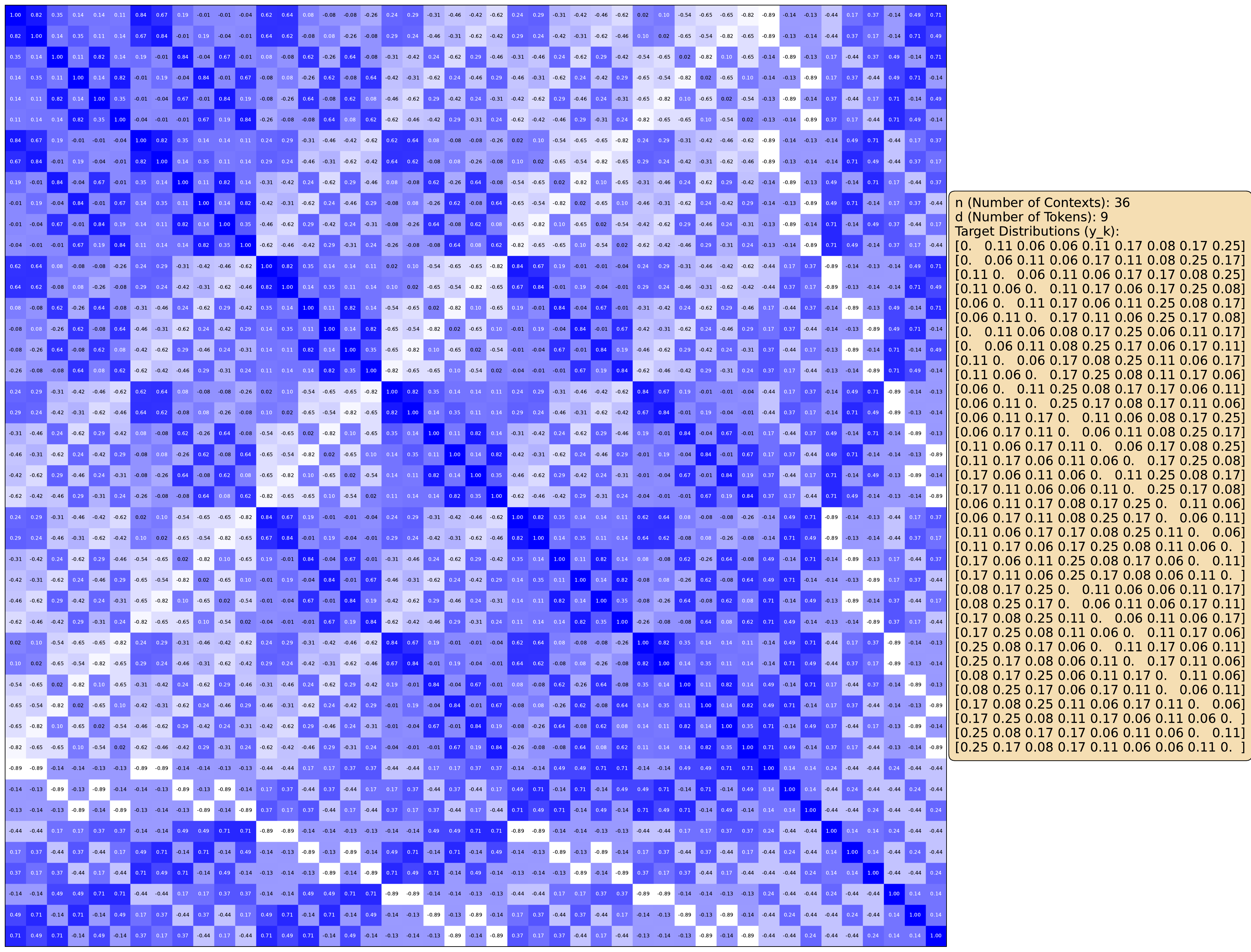}
  \caption{Direct product \(\left(S_3\times S_3\right)\); full orbit \(n=36\).}
  \label{fig:dp3}
\end{figure}

In the computed solution, we partition \(WW^{\top}\) into \(a\times a\) blocks, each of size \(\ell\times \ell\), following the row-major \(a\times \ell\) grid convention. Every diagonal block equals
\[
\alpha I_{\ell} + \beta \left(J_{\ell} - I_{\ell}\right),
\]
and every off-diagonal block equals
\[
\alpha' I_{\ell} + \beta' \left(J_{\ell} - I_{\ell}\right).
\]
Moreover, all diagonal blocks are identical to one another, and all off-diagonal blocks are identical to one another.

\subsection{Wreath Products}

We partition the \(m = b s\) coordinates into \(b\) blocks of size \(s\) and write
\[
\boldsymbol{y}=\left[  (\boldsymbol{y}^{(1)})^{\top}\mid\cdots\mid (\boldsymbol{y}^{(b)})^{\top}  \right]^{\top},\qquad
\boldsymbol{y}^{(i)}\in\mathbb{R}^{s}_{\ge 0}.
\]
An element \(\boldsymbol{\tau}=\left(\tau_1,\ldots,\tau_b; \tau_{\mathrm{blk}}\right)\in\left(S_s\right)^b\rtimes S_b\) acts by first permuting within blocks and then permuting blocks:
\[
\boldsymbol{\tau}\circ\boldsymbol{y}
 :=  B_{\tau_{\mathrm{blk}}}  \mathrm{diag} \left(P_{\tau_1},\ldots,P_{\tau_b}\right)  \boldsymbol{y},
\qquad
Y=\left[  \boldsymbol{\tau}\circ\boldsymbol{y}  \right]_{\boldsymbol{\tau}\in S_s\wr S_b},\quad
n=\left(s!\right)^b  b!,
\]
where the concrete action \(\boldsymbol{\tau}\circ\boldsymbol{y}\) is implemented by the \(bs\times bs\) block permutation matrix \(B_{\tau_{\mathrm{blk}}}\) induced by \(\tau_{\mathrm{blk}}\), and \(\mathrm{diag}(\cdot)\) denotes a block-diagonal concatenation. The first experiment uses three 2-dimensional vectors
\[
\boldsymbol{y}^{(1)}=\left[\frac1{12}, \frac14\right]^{\top},\quad
\boldsymbol{y}^{(2)}=\left[\frac14, 0\right]^{\top},\quad
\boldsymbol{y}^{(3)}=\left[\frac16, \frac14\right]^{\top},
\]
so that \(\boldsymbol{y}=\left[  (\boldsymbol{y}^{(1)})^{\top} \mid (\boldsymbol{y}^{(2)})^{\top} \mid (\boldsymbol{y}^{(3)})^{\top}  \right]^{\top}\),
with action by \(S_2\wr S_3\) and \(n=48\) (see~\cref{fig:wr1}).

\begin{figure}[H]
  \centering
  \includegraphics[width=\textwidth]{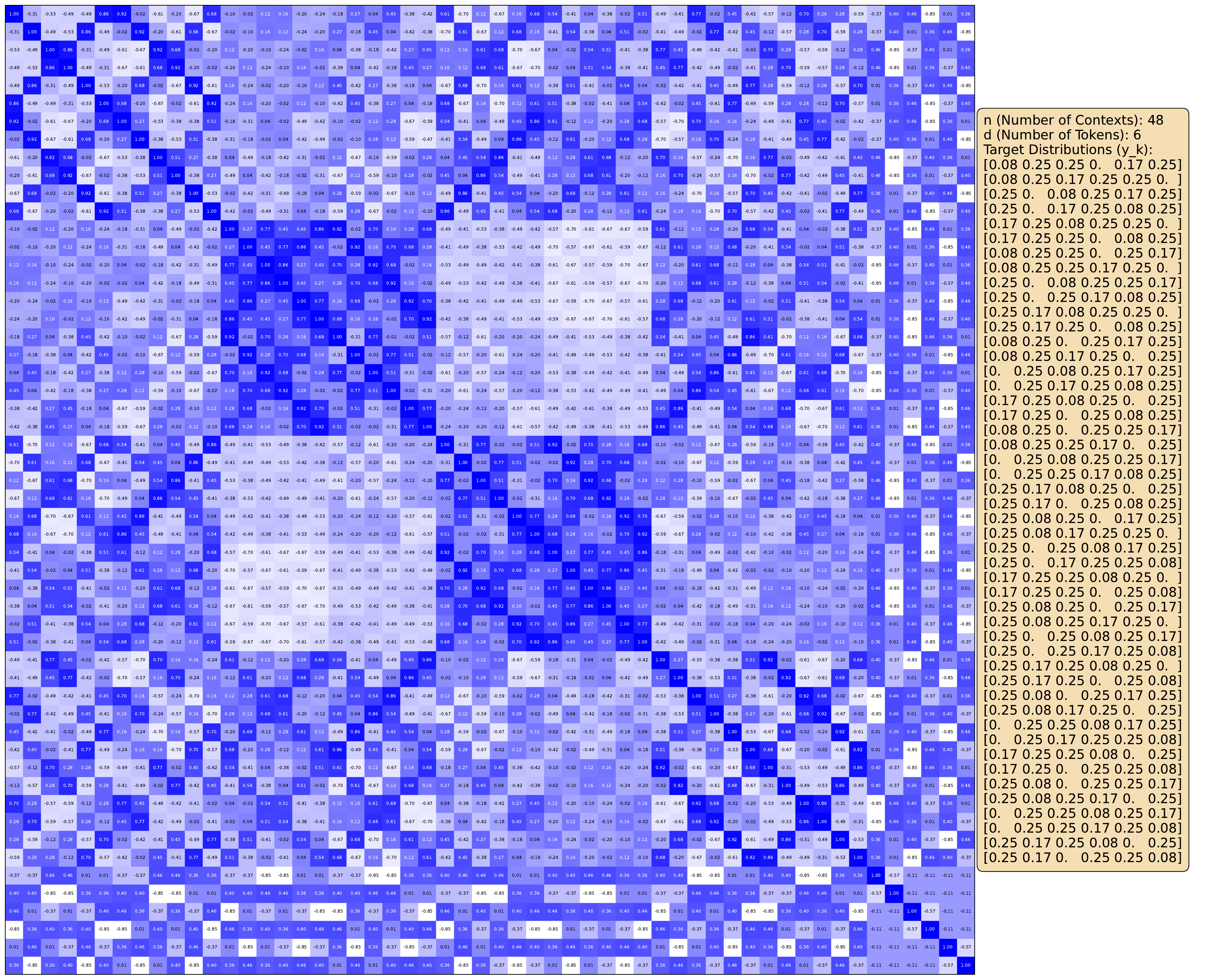}
  \caption{Wreath \(\left(S_2\wr S_3\right)\); full orbit \(n=48\).}
  \label{fig:wr1}
\end{figure}

The second experiment uses two 3-dimensional vectors
\[
\boldsymbol{y}^{(1)}=\left[\frac1{12}, \frac14, \frac14\right]^{\top},\qquad
\boldsymbol{y}^{(2)}=\left[0, \frac16, \frac14\right]^{\top},
\]
so that \(\boldsymbol{y}=\left[  (\boldsymbol{y}^{(1)})^{\top} \mid (\boldsymbol{y}^{(2)})^{\top}  \right]^{\top}\),
with action by \(S_3\wr S_2\) and \(n=72\) (see~\cref{fig:wr2}).

\begin{figure}[H]
  \centering
  \includegraphics[width=\textwidth]{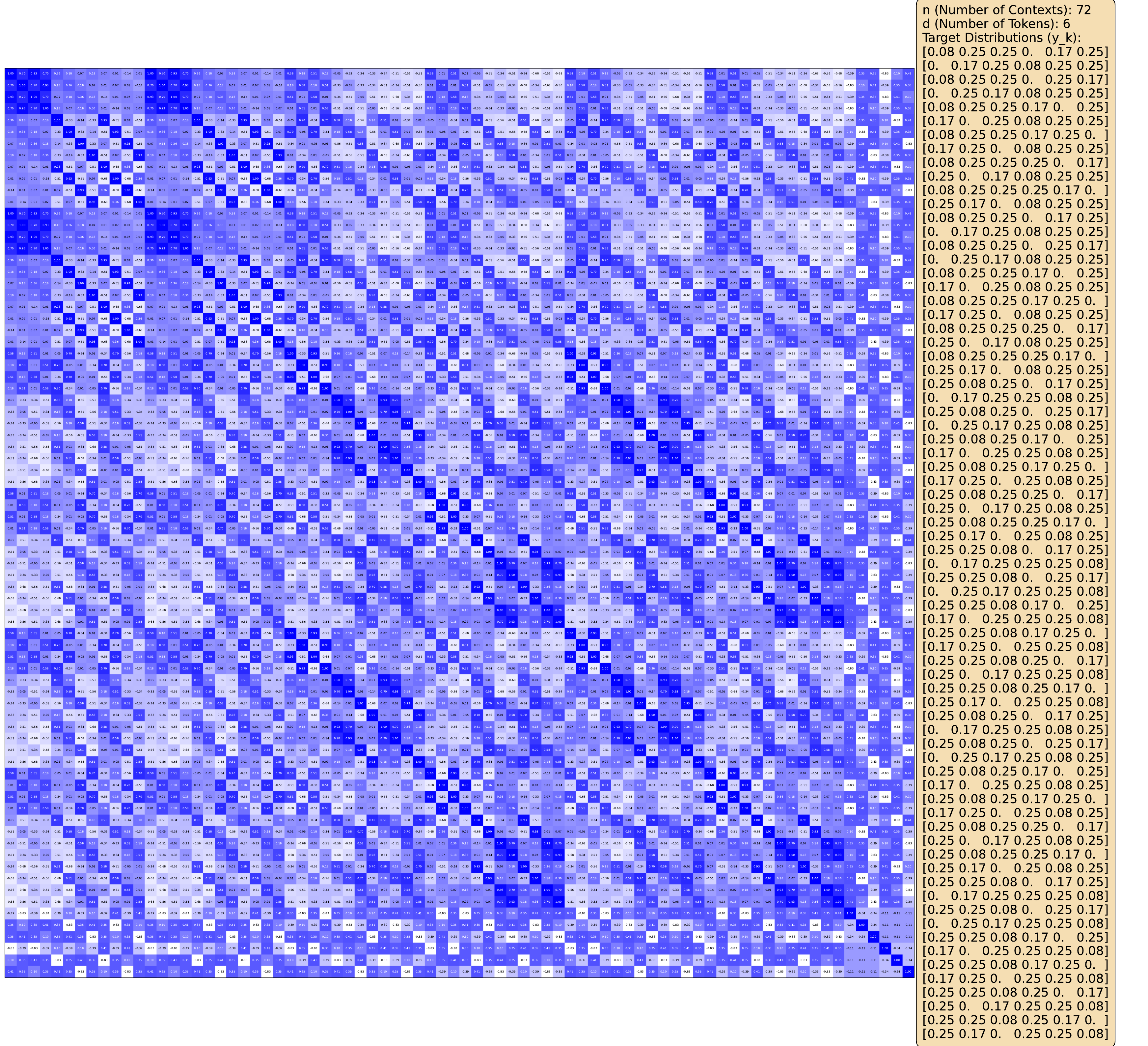}
  \caption{Wreath \(\left(S_3\wr S_2\right)\); full orbit \(n=72\).}
  \label{fig:wr2}
\end{figure}

In the computed solution, we partition \(WW^{\top}\) into \(s\times s\) blocks. Every diagonal block equals
\[
\alpha I_{s} + \beta \left(J_{s} - I_{s}\right),
\]
and every off-diagonal block equals
\[
\alpha' I_{s} + \beta' \left(J_{s} - I_{s}\right).
\]
Moreover, all diagonal blocks are identical to one another, and all off-diagonal blocks are identical to one another. Thus the observed block pattern coincides with the one described above for the direct product experiments. 

\end{document}